%% file: manuscript_arxiv.tex
\documentclass[aos]{imsart}

\RequirePackage{amsthm,amsmath,amsfonts,amssymb}
\RequirePackage[numbers]{natbib}
\RequirePackage[colorlinks,citecolor=blue,urlcolor=blue]{hyperref}
\RequirePackage{graphicx}
\usepackage{subcaption}

\startlocaldefs
\usepackage{amsmath,amsfonts,amssymb,amsthm,mathrsfs}
\usepackage{paralist}
\usepackage{enumitem}
\usepackage{algorithm}
\usepackage{algpseudocode}
\usepackage{tikz}
\usetikzlibrary{arrows}
\usetikzlibrary{decorations.pathreplacing}
\tikzset{mybrace/.style={decoration={brace,raise=1.8mm},decorate}}
\usepackage{titlesec} % Used to redefine format of section heading
\usepackage{etoolbox} % Use '\AfterEndEnvironment' to avoid indent after theorems
\usepackage{float}
\usepackage{graphics}
\usepackage{booktabs}
\usepackage{multirow}
\usepackage{xcolor}
\usepackage{sectsty}
\sectionfont{\bfseries\Large\raggedright}

\makeatletter

\usepackage{color}

\numberwithin{equation}{section}
\allowdisplaybreaks

\def\@noindentfalse{\global\let\if@noindent\iffalse}
\def\@noindenttrue {\global\let\if@noindent\iftrue}
\def\@aftertheorem{%
	\@noindenttrue
	\everypar{%
		\if@noindent%
		\@noindentfalse\clubpenalty\@M\setbox\z@\lastbox%
		\else%
		\clubpenalty \@clubpenalty\everypar{}%
		\fi}}

\theoremstyle{plain}
\newtheorem{theorem}{Theorem}[section]
\AfterEndEnvironment{theorem}{\@aftertheorem}
\newtheorem{definition}[theorem]{Definition}
\AfterEndEnvironment{definition}{\@aftertheorem}
\newtheorem{lemma}[theorem]{Lemma}
\AfterEndEnvironment{lemma}{\@aftertheorem}
\newtheorem{corollary}[theorem]{Corollary}
\AfterEndEnvironment{corollary}{\@aftertheorem}
\newtheorem{proposition}[theorem]{Proposition}
\AfterEndEnvironment{proposition}{\@aftertheorem}

\theoremstyle{definition}
\newtheorem{remark}[theorem]{Remark}
\AfterEndEnvironment{remark}{\@aftertheorem}

\AfterEndEnvironment{example}{\@aftertheorem}

\AfterEndEnvironment{proof}{\@aftertheorem}

\newtheorem{assumption}[theorem]{Assumption}
\AfterEndEnvironment{assumption}{\@aftertheorem}

\titleformat{name=\section}
{\center\small\sc}{\thesection\kern1em}{0pt}{\MakeTextUppercase}
\titleformat{name=\section, numberless}
{\center\small\sc}{}{0pt}{\MakeTextUppercase}

\titleformat{name=\subsection}
{\bf\mathversion{bold}}{\thesubsection\kern1em}{0pt}{}
\titleformat{name=\subsection, numberless}
{\bf\mathversion{bold}}{}{0pt}{}

\titleformat{name=\subsubsection}
{\normalsize\itshape}{\thesubsubsection\kern1em}{0pt}{}
\titleformat{name=\subsubsection, numberless}
{\normalsize\itshape}{}{0pt}{}

\def\note#1{\par\smallskip%
	\noindent\kern-0.01\hsize%
	{\setlength\fboxrule{0pt}\fbox{\setlength\fboxrule{0.5pt}\fbox{%
				\llap{$\boldsymbol\Longrightarrow$ }%
				\vtop{\hsize=0.98\hsize\parindent=0cm\small\rm #1}%
				\rlap{$\enskip\,\boldsymbol\Longleftarrow$}
	}}}%
}

\usepackage{delimset}
\def\given{\mskip 0.5mu plus 0.25mu\vert\mskip 0.5mu plus 0.15mu}
\newcounter{bracketlevel}%
\def\@bracketfactory#1#2#3#4#5#6{%
	\expandafter\def\csname#1\endcsname##1{%
		\global\advance\c@bracketlevel 1\relax%
		\global\expandafter\let\csname @middummy\alph{bracketlevel}\endcsname\given%
		\global\def\given{\mskip#5\csname#4\endcsname\vert\mskip#6}\csname#4l\endcsname#2##1\csname#4r\endcsname#3%
		\global\expandafter\let\expandafter\given\csname @middummy\alph{bracketlevel}\endcsname%
		\global\advance\c@bracketlevel -1\relax%
	}%
}
\def\bracketfactory#1#2#3{%
	\@bracketfactory{#1}{#2}{#3}{relax}{0.5mu plus 0.25mu}{0.5mu plus 0.15mu}
	\@bracketfactory{b#1}{#2}{#3}{big}{1mu plus 0.25mu minus 0.25mu}{0.6mu plus 0.15mu minus 0.15mu}
	\@bracketfactory{bb#1}{#2}{#3}{Big}{2.4mu plus 0.8mu minus 0.8mu}{1.8mu plus 0.6mu minus 0.6mu}
	\@bracketfactory{bbb#1}{#2}{#3}{bigg}{3.2mu plus 1mu minus 1mu}{2.4mu plus 0.75mu minus 0.75mu}
	\@bracketfactory{bbbb#1}{#2}{#3}{Bigg}{4mu plus 1mu minus 1mu}{3mu plus 0.75mu minus 0.75mu}
}
\bracketfactory{clc}{\lbrace}{\rbrace}
\bracketfactory{clr}{(}{)}
\bracketfactory{cls}{[}{]}
\bracketfactory{abs}{\lvert}{\rvert}
\bracketfactory{norm}{\Vert}{\Vert}
\bracketfactory{floor}{\lfloor}{\rfloor}
\bracketfactory{ceil}{\lceil}{\rceil}
\bracketfactory{angle}{\langle}{\rangle}

\newcounter{ctr}\loop\stepcounter{ctr}\edef\X{\@Alph\c@ctr}%
\expandafter\edef\csname s\X\endcsname{\noexpand\mathscr{\X}}
\expandafter\edef\csname c\X\endcsname{\noexpand\mathcal{\X}}
\expandafter\edef\csname b\X\endcsname{\noexpand\boldsymbol{\X}}
\expandafter\edef\csname I\X\endcsname{\noexpand\mathbb{\X}}
\ifnum\thectr<26\repeat

\let\@IE\IE\let\IE\undefined
\newcommand{\IE}{\mathop{{}\@IE}\mathopen{}}
\let\@IP\IP\let\IP\undefined
\newcommand{\IP}{\mathop{{}\@IP}}

\def\^#1{\relax\ifmmode {\mathaccent"705E #1} \else {\accent94 #1}\fi}
\def\~#1{\relax\ifmmode {\mathaccent"707E #1} \else {\accent"7E #1}\fi}
\def\*#1{\relax#1^\ast}
\edef\-#1{\relax\noexpand\ifmmode {\noexpand\bar{#1}} \noexpand\else \-#1\noexpand\fi}
\def\>#1{\vec{#1}}
\def\.#1{\dot{#1}}

\def\atop{\@@atop}

\renewcommand{\leq}{\leqslant}
\renewcommand{\geq}{\geqslant}

\newcommand{\ub}{\mathbf{u}}

\newcommand{\ro}{\mathrm{o}}
\newcommand{\rO}{\mathrm{O}}

\newcommand{\ri}{\mathrm{i}}

\newcommand\indep{\protect\mathpalette{\protect\@indep}{\perp}}
\def\@indep#1#2{\mathrel{\rlap{$#1#2$}\mkern2mu{#1#2}}}

\def\parsetime#1#2#3#4#5#6{#1#2:#3#4}
\def\parsedate#1:20#2#3#4#5#6#7#8+#9\empty{20#2#3-#4#5-#6#7 \parsetime #8}
\def\moddate{\expandafter\parsedate\pdffilemoddate{\jobname.tex}\empty}

\theoremstyle{definition}

\theoremstyle{remark}

\theoremstyle{definition}

\theoremstyle{plain}

\theoremstyle{plain}

\theoremstyle{plain}

\theoremstyle{plain}

\allowdisplaybreaks[4]

\makeatother

\providecommand{\conditionname}{Condition}
\providecommand{\definitionname}{Definition}
\providecommand{\lemmaname}{Lemma}
\providecommand{\propositionname}{Proposition}
\providecommand{\remarkname}{Remark}
\providecommand{\corollaryname}{Corollary}
\providecommand{\theoremname}{Theorem}

\endlocaldefs

\begin{document}

\begin{frontmatter}
\title{Multiplier Bootstrap and Edge Phase Transitions of High-Dimensional Covariance Matrices}
%\title{A sample article title with some additional note\thanksref{t1}}
\runtitle{Spectral-Edge Fluctuations under Multiplier Bootstrap}
%\thankstext{T1}{A sample additional note to the title.}

\begin{aug}
%%%%%%%%%%%%%%%%%%%%%%%%%%%%%%%%%%%%%%%%%%%%%%%
%% Only one address is permitted per author. %%
%% Only division, organization and e-mail is %%
%% included in the address.                  %%
%% Additional information can be included in %%
\author[A]{\fnms{Jiahui}~\snm{Xie}\ead[label=e1]{ jihxie@ucdavis.edu}}
%%%%%%%%%%%%%%%%%%%%%%%%%%%%%%%%%%%%%%%%%%%%%%
%% Addresses                                %%
%%%%%%%%%%%%%%%%%%%%%%%%%%%%%%%%%%%%%%%%%%%%%%
\address[A]{Department of Statistics, University of California, Davis\printead[presep={,\ }]{e1}}

\end{aug}

\begin{abstract}
In this paper, we study the effects of employing multiplier bootstrap to analyze the asymptotic distributions of the largest eigenvalues of high-dimensional sample covariance matrices in both spiked and non-spiked models. Our findings demonstrate that the multiplier bootstrap establishes several phase transitions in the limiting edge distributions of both unconditional and conditional bootstrapped covariance matrices, provided the different classes of multipliers. In the nonspiked setting, unbounded multipliers lead to Fréchet or Gumbel limits for the largest eigenvalue of the bootstrapped covariance matrix, both conditionally on the observed data and unconditionally. For bounded multipliers, the unconditional model exhibits transitions among Tracy–Widom, Gaussian, or Weibull limits, determined jointly by the aspect ratio $p/n$, the upper-endpoint behavior of the multipliers, and the population covariance matrix. The conditional model displays analogous Gaussian and Weibull regimes; in contrast, the conditional counterpart of the unconditional Tracy–Widom regime collapses to a point mass. In the spiked setting, under suitable signal-strength conditions, the leading eigenvalues of both the unconditional and conditional bootstrapped sample covariance matrices are asymptotically Gaussian for bounded as well as unbounded multipliers, under some mild assumptions. Our theoretical results also clarify the feasibility and adaptability of the multiplier bootstrap for spectral inference in high-dimensional sample covariance models. Numerical simulations confirm the accuracy of our results and the effectiveness of the proposed spectral inference procedures, which may be of independent interest.
\end{abstract}



\end{frontmatter}
%%%%%%%%%%%%%%%%%%%%%%%%%%%%%%%%%%%%%%%%%%%%%%
%% Please use \tableofcontents for articles %%
%% with 50 pages and more                   %%
%%%%%%%%%%%%%%%%%%%%%%%%%%%%%%%%%%%%%%%%%%%%%%

\setcounter{tocdepth}{2}
\begin{center}
{\small\scshape Contents}
\end{center}
\makeatletter
\@starttoc{toc}
\makeatother

\section{Introduction}\label{sec:intro}
The bootstrap method \cite{efron1979bootstrap} is a powerful and widely used tool in multivariate statistics and machine learning. By resampling a single dataset to generate numerous simulated samples, it facilitates inference even when little is known about the properties of the data-generating distribution. This approach is particularly appealing when theoretical derivations based on asymptotic analysis are complex or rely on restrictive assumptions. On a related note, the covariance matrix plays a central role in virtually every aspect of multivariate data analysis. Over the past few decades, technological advancements have spurred growing interest in developing methodologies and tools to address the challenges posed by high-dimensionality and complexity \cite{yao2015sample}. In particular, extreme eigenvalues of sample covariance matrices are critical in principal component analysis (PCA) \cite{abdi2010principal,johnstone2018pca}. However, current theoretical findings on these extreme eigenvalues often depend on unknown parameters of the population covariance matrix and are typically intricate \cite{bao2022statistical,yao2015sample}, rendering conventional methods impractical in many cases.

An intuitive approach to overcome the above challenges is to apply bootstrap methodology to covariance matrices in the high-dimensional regime.
Then, a natural question arises: 
\vspace*{2pt}

\textit{Could the asymptotics of extreme sample eigenvalues be effectively approximated using the multiplier bootstrap method? If not, additionally, how does a random multiplier change the edge limits of a high-dimensional sample covariance matrix?} 
\vspace*{2pt}

Addressing this question is highly non-trivial, as the extreme sample eigenvalues of large covariance matrices exhibit complex limiting distributions that depend on the structure of the population covariance matrix. For example, non-spiked covariance matrices often follow the Tracy-Widom distribution \cite{Bao2015,Ding&Yang2018,9779233,fan2022tracy,el2007tracy,John2001,knowles2017anisotropic,lee2016tracy,PillaiandYin2014}, whereas spiked covariance matrices tend to follow some Gaussian distribution \cite{BaiandYao2008,bao2022statistical,CHP,johnstone2018pca,paul2007asymptotics}. 
In all these distributions, the asymptotic results often involve unknown and complex quantities, making it challenging to apply these results in practice. Moreover, as highlighted in \cite{el2019non,karoui2016bootstrap}, directly applying standard bootstrap methods in high-dimensional regimes can sometimes yield erratic results for statistical inference involving extreme eigenvalues. This issue becomes particularly pronounced in scenarios where population spikes are weak or entirely absent; see Section \ref{subsec_existingresult} below for further discussion of these challenges.

Motivated by these difficulties, this paper provides a comprehensive analysis of the effects of multiplier bootstrap procedures on the asymptotic behavior of the top eigenvalues of sample covariance matrices in high-dimensional settings. It also proposes practical and effective methods for specific statistical tasks using multiplier bootstrap techniques. We demonstrate that preserving information from the original sample covariance matrix through multiplier bootstrap procedures poses significant challenges. These challenges often necessitate careful selection of multipliers, a large number of bootstrap replications, and, in some cases, additional bias corrections to improve statistical estimates.

Before going to the details, we first generate our bootstrap procedure for high-dimensional sample covariance matrices as follow. Consider a sequence of data $\mathbf{s}_i \sim \mathbf{s} \in \mathbb{R}^p,  1 \leq i \leq n,$ which are i.i.d. observations of a random vector $\mathbf{s} $ such that
\begin{equation}\label{eq_generatingmodel}
\mathbf{s}= \Sigma^{1/2} \mathbf{x} \in \mathbb{R}^p, 
\end{equation}
where $\Sigma \in\mathbb{R}^{p\times p}$ is deterministic representing the covariance structures in the dataset and $\mathbf{x} \in \mathbb{R}^p$ is a random vector containing i.i.d. centered random variables with variance $n^{-1}$. For high-dimensionality, we mean that $p$ and $n$ are comparably large.  Now, given a sequence of data $\mathbf{s}_i=\Sigma^{1/2}\mathbf{x}_i,1 \leq i\leq n$, we resample $\mathbf{s}_i$'s via a sequence of random multipliers $\xi_i\sim\xi\in\mathbb{R}$ that are independent of $\{\mathbf{x}_i\}$. The resampled data can be described as 
\begin{equation}\label{eq_bootstrapdata}
\mathbf{y}_i=\xi_i \Sigma^{1/2} \mathbf{x}_i \in \mathbb{R}^p, \quad 1\leq i\leq n.
\end{equation}
 We may write the resampled data matrix as $Y=\Sigma^{1/2} XD,$ where $X=(\mathbf{x}_i)$ and $D$ is a diagonal matrix containing $\{\xi_i\}.$ Then the bootstrapped sample covariance matrix can be constructed as follows 
\begin{equation}\label{eq_samplecov}
Q:=YY^* \equiv \Sigma^{1/2} XD^2X^* \Sigma^{1/2}.
\end{equation}
Throughout the paper, we study both the unconditional and conditional bootstrap for analyzing the largest eigenvalues of $Q$, with the latter being the commonly used form of the multiplier bootstrap. The terms unconditional and conditional refer to whether we condition on $X$.

Technically, understanding the asymptotic behavior of the largest eigenvalues of $Q$ is crucial for analyzing the performance of the multiplier bootstrap for the edge eigenvalues of the sample covariance matrix. In what follows, we first provide a summary of some related results in Section \ref{subsec_existingresult}. Then we offer an overview of our contributions in Section \ref{sec_overviewofresults}. 

%This is an example of a new parapgraph with a numbered footnote\footnote{\url{https://data.gov.uk/}} and a second footnote marker.\footnote{Example of footnote text.}

\subsection{Summary of some existing related results}\label{subsec_existingresult}

In this section, we summarize results related to the bootstrap methodology. While the bootstrap is extensively studied in the literature \cite{alemayehu1988bootstrapping,bickel1981some,davison1997bootstrap,efron1979bootstrap,efron1981nonparametric,politis1999subsampling}, our focus is specifically on aspects relevant to sample covariance matrices. The use of the bootstrap for studying sample covariance matrices dates back to \cite{beran1985bootstrap,diaconis1983computer,eaton1991wielandt} and has since evolved into a powerful tool in multivariate analysis \cite{abdi2010principal,davison1997bootstrap,olive2017robust}. Key findings from these studies show that nonparametric bootstrap methods can effectively approximate the eigenvalue distribution of sample covariance matrices in low-dimensional settings, where the sample size approaches infinity while the data dimension remains fixed or grows slowly.

More recently, the research has focused on evaluating whether the bootstrap methods can reliably capture the asymptotic properties of sample covariance matrices in high-dimensional settings. We summarize some closely related literature as follows. For the global behavior of the spectrum of sample covariance matrices, \cite{lopes2019bootstrapping} investigated the problem of bootstrapping linear spectral statistics for datasets with the structure (\ref{eq_generatingmodel}). Subsequently, \cite{wang2023bootstrap} extended this study to the bootstrap of linear spectral statistics in the high-dimensional elliptical model. Moreover, \cite{lopes2023bootstrapping} explored the efficiency of bootstrapping the operator norm under various population decay profiles, while \cite{zhang2024covariance} developed a universal bootstrap statistic based on the covariance operator norm for testing covariance matrices.  Additionally, \cite{dette2024nonparametric} proposed a nonparametric sampling-with-replacement bootstrap for eigenvalue statistics of high-dimensional sample covariance matrices.

Regarding individual eigenvalues, much less attention has been given to high-dimensional settings, except for a few cases under certain structural assumptions. These assumptions generally ensure that the individual eigenvalues of sample covariance matrices exhibit Gaussian behavior. For instance, \cite{han2018gaussian} studied the multiplier bootstrap for the largest eigenvalue in a moderately diverging dimension setting, assuming $p = \mathrm{o}(n^{1/9})$, while \cite{yao2023rates} investigated the standard bootstrap for the largest eigenvalue, assuming that the eigenvalues of $\Sigma$ decay exponentially. Similar assumptions and results for eigenvectors were established in \cite{naumov2019bootstrap}. More recently, \cite{yu2024testing} examined the standard bootstrap under a factor model, assuming strong factor strength as in \cite{karoui2016bootstrap}. However, it remains unclear and challenging to determine whether the multiplier bootstrap can perform effectively in high-dimensional settings without relying on such strong structural assumptions, as questioned in \cite{el2019non,karoui2016bootstrap}.

Finally, we note that (\ref{eq_samplecov}) is often referred to as a \emph{separable sample covariance matrix} in the context of random matrix theory (RMT). In the literature, such models have been studied primarily under scenarios where both $\Sigma$ and $D$ are bounded and deterministic; see, for instance, \cite{couillet2014analysis,9779233,Karoui2009,paul2009no,yang2019edge,Zhanggeneral}. However, our focus on the random matrix model in \eqref{eq_samplecov} differs from the aforementioned studies, as we treat $D^2$ as random multipliers, whose range may also be unbounded. Existing results addressing the case of random $D$ are limited \cite{Karoui2009,Zhanggeneral}, and these are confined to analyzing the limiting spectral distribution under specific conditions. In this regard, our findings contribute to the RMT literature by providing the limiting distributions of the edge eigenvalues of a novel class of separable sample covariance matrices with random structures, which may be of independent interest.

\subsection{An overview of our results and contributions}\label{sec_overviewofresults}

In this section, we provide an informal overview of our results and highlight the main contributions and novelties of our paper. At a high level, our findings and proposed algorithms demonstrate that the conditional multiplier bootstrap can analyze the asymptotics of the largest few eigenvalues of sample covariance matrices, whether spiked or not, provided that the multipliers are appropriately chosen and the bootstrap procedures are properly modified. We elaborate on this in more detail below.

In Section \ref{sec_boostrapeffect_nonspike}, we examine the feasibility of the multiplier bootstrap for non-spiked sample covariance matrices in high dimensions. For clarity, we present both the unconditional and conditional versions of the bootstrap. In the unconditional bootstrap, we study the distribution of the largest eigenvalues of $Q$ without conditioning, while in the conditional bootstrap we study these eigenvalues conditional on $X$, which corresponds to the commonly used form of the multiplier bootstrap. First, Theorem \ref{thm_main_unbounded} shows that the unconditional multiplier bootstrap fails to replicate the asymptotic distribution for the largest eigenvalues of non-spiked sample covariance matrices when the multipliers are unbounded. Instead, in this setting, the largest eigenvalues of the bootstrapped sample covariance matrices follow either a Fréchet or a Gumbel distribution, and therefore do not capture the Tracy–Widom behavior. Moreover, the counterpart for the conditional bootstrap under this setup is given in Theorem \ref{thm_main_unbounded_conditional}, which shows that the unconditional and conditional bootstrap share the same asymptotic behavior when the multipliers are unbounded and therefore cannot be applied directly to study the Tracy–Widom distribution either.

Second, for bounded multipliers, Theorem \ref{thm_main_bounded} shows that the distribution of the largest eigenvalues of the unconditional bootstrapped sample covariance matrices can undergo several phase transitions and follow Tracy–Widom, Gaussian, or Weibull distributions, depending on the aspect ratio $p/n$, the edge behavior of the multipliers, and the population covariance matrix $\Sigma$. Even though the unconditional bootstrap is not directly applicable, it makes it possible to study the asymptotic distribution of the largest eigenvalues of sample covariance matrices with appropriately chosen bounded multipliers. We then establish the conditional multiplier bootstrap in Theorem \ref{thm_main_bounded_conditional}. The distribution of the largest eigenvalues also exhibits several phase transitions. In particular, the Weibull and Gaussian distributions arise in patterns similar to those in the unconditional bootstrap, although the counterpart of the Tracy–Widom (TW) law degenerates to a point mass. Although the conditional bootstrap is not directly consistent, part (3) of Theorem \ref{thm_main_bounded_conditional} still motivates modified inference for the right edge in \eqref{eq_twresultoriginal}. Remark \ref{remark_important} explains how bounded multipliers with suitable concentration and decaying variance can be used to construct confidence intervals for $E_+$, while Corollary \ref{cor_joint_gaussian_edge} identifies the common multiplier-induced shift of the leading eigenvalues. The numerical implementation uses the multipliers in \eqref{eq_xiconstruction}.

%, provided that the parameters of the Tracy-Widom law are accurately estimated.
%Third, to achieve this, we propose a new modified procedure, Algorithm \ref{algtw}, which generates a large number of bootstrapped sample covariance matrices using carefully designed bounded multipliers. Theoretically, Corollary \ref{col_recoverTW} establishes that our modified multiplier bootstrap procedure can effectively capture the Tracy-Widom law for non-spiked sample covariance matrices—

In Section \ref{sec_boostrapeffect_spike}, we examine the effectiveness of the multiplier bootstrap for spiked sample covariance matrices. In summary, the largest eigenvalues of both the unconditional and conditional bootstrapped sample covariance matrices under the spiked model are asymptotically Gaussian, regardless of whether the multipliers are bounded, under mild assumptions.  In Theorem \ref{thm_main_spike_unconditional}, we establish the result for the unconditional bootstrap. Specifically, for given multipliers, once the spikes exceed a certain threshold—possibly depending on the multipliers—the largest eigenvalues are asymptotically Gaussian. However, since it is impossible to find multipliers that equalize the asymptotic means of the largest eigenvalues of the sample covariance matrix and the unconditional bootstrapped covariance matrix (see Remark \ref{rem_asymGaussian_spike}). Therefore, the unconditional bootstrap is not only practically inapplicable but also fundamentally ineffective in this setting.

We then develop two versions of the conditional multiplier bootstrap in Theorem \ref{thm_main_spike_conditional}. A direct counterpart to the unconditional bootstrap (and to the limiting CLT for the sample covariance matrix) is established in part (1). However, this version is practically inapplicable, as the asymptotic mean involves the unobservable population covariance matrix. To address this issue, part (2) provides a modified version. Once the multipliers satisfy a suitable mean--variance constraint (cf. (\ref{eq_xicon})), the multiplier bootstrap can be applied directly and is able to recover the asymptotic mean and variance of the Gaussian limit for the largest eigenvalues of the sample covariance matrices (cf. Remark \ref{rem_conditionalspike}). Our results further show that the conditional bootstrap becomes valid—after an appropriate choice of multipliers and a suitable modification of the bootstrap—once the spike size is much larger than $\mathrm{O}(n^{1/4})$. This threshold is substantially smaller than the order $\mathrm{O}(n^{1/2})$ predicted in \cite{el2019non}.

We note that, technically, this problem reduces to studying the largest eigenvalues of both the unconditional and conditional bootstrapped sample covariance matrices (\ref{eq_samplecov}) under various assumptions on the multipliers, for both spiked and non-spiked models. Our analysis reveals that the asymptotic distribution of these eigenvalues, both unconditional and conditional, can take various forms, including the three extreme value distributions for sequences of i.i.d. random variables—Gumbel, Fréchet, and Weibull \cite{beirlant2004statistics}—as well as the Tracy-Widom (TW) law, Gaussian, or a mixture of TW and Gaussian distributions. The specific distribution depends on $D^2$ (i.e., the multipliers), $\Sigma$, the presence of spikes, and the aspect ratio $p/n$.
These theoretical findings are of independent interest, offering insights into how bootstrap mechanisms influence the spectral limits of covariance matrices. Consequently, they provide guidance for designing appropriate multiplier mechanisms to effectively bootstrap the largest eigenvalues in both spiked and non-spiked models for various statistical applications.

The rest of this article is organized as follows. In Section \ref{sec_modelsetup}, we give the details of our model and some basic assumptions. In Section \ref{sec_boostrapeffect_nonspike}, we present the main results of multiplier bootstrap for the non-spiked covariance matrix model. In Section \ref{sec_boostrapeffect_spike}, we study the multiplier bootstrap for the spiked covariance matrix model.  
%In Section \ref{sec_application_factorselection}, we consider application of multiplier bootstrap methodologies in common factor selection.
Numerical simulations are provided in Section \ref{sec_simulations} to show the accuracy of our results. The proof strategies are summarized in Section \ref{sec_proofstrategy}. Technique proof and details are deferred to our supplementary material. 
 %\cite{suppl}.

\vspace{10pt}
\noindent {\bf Conventions.} Let $\mathbb{C}_+$ be the complex upper half plane. We denote $C>0$ as a generic constant whose value may change from line to line. For two sequences of deterministic positive values $\{a_n\}$ and $\{b_n\},$ we write $a_n=\rO(b_n)$ if $a_n \leq C b_n$ for some positive constant $C>0.$ In addition, if both $a_n=\rO(b_n)$ and $b_n=\rO(a_n),$ we write $a_n \asymp b_n.$ Moreover, we write $a_n=\ro(b_n)$ if $a_n \leq c_n b_n$ for some positive sequence $c_n \downarrow 0.$ In addition, for a sequence of random variables $\{x_n\}$ and positive real values $\{a_n\},$ we use $x_n=\rO_{\mathbb{P}}(a_n)$ to state that $x_n/a_n$ is stochastically bounded. Similarly, we use $x_n=\ro_{\mathbb{P}}(a_n)$ to say that $x_n/a_n$ converges to zero in probability. For a sequence of positive random variables $\{y_n\},$ we use $y_{(k)}, 1 \leq k \leq n,$ for its order statistics with $y_{(1)} \geq y_{(2)} \geq \cdots \geq y_{(n)}>0.$

\section{The model and basic assumptions}\label{sec_modelsetup}
 
In this section, we introduce our model and some assumptions. As discussed around (\ref{eq_samplecov}),  we consider bootstrapped data matrix of the following form
\begin{equation}\label{eq_datamatrix}
Y=\Sigma^{1/2} XD,
\end{equation}
where $\Sigma$ is a $p \times p$ deterministic positive definite matrix, $D$ is an $n \times n$ diagonal random matrix containing i.i.d. multipliers, and $X$ is a $p \times n$ random matrix independent of $D$ whose entries satisfying the following assumption.  
%Motivated by statistical applications and for the purpose of definiteness, we consider the following separable i.i.d. data matrix in the form of (\ref{eq_datamatrix}).
\begin{assumption}\label{assum_model}
Throughout the paper, we assume that the entries of $X=(x_{ij})$ are centered  i.i.d. random variables satisfying that for $1 \leq i \leq p, 1 \leq j \leq n,$
\begin{equation}\label{eq_standard1n}
\mathbb{E} x_{ij}=0, \ \mathbb{E} x_{ij}^2=\frac{1}{n}. 
\end{equation} 
Moreover, we assume that for all $k \in \mathbb{N},$ there exists some constant $C_k>0$ so that $\mathbb{E}|\sqrt{n} x_{ij}|^k \leq C_k. $ 
%Finally, for $T,$ we assume that $p_1=p$ and $T=\Sigma^{1/2}$ for some positive definite matrix $\Sigma.$
\end{assumption}

%To illustrate the generality and usefulness of the concerned models in Assumption \ref{assum_model}, we provide a few examples and discuss their concrete applications in the statistical literature, beyond multiplier bootstrap. 
%
%\begin{example}\label{exam_boostrapping}
%The model of Assumption \ref{assum_model} has been used in many different contexts. We can write the columns of $Y, \{\mathbf{y}_i\}, 1 \leq i \leq n,$ as follows
%\begin{equation*}
%\mathbf{y}_i=\xi_i \Sigma^{1/2} \mathbf{x}_i. 
%\end{equation*}
%First, when $\mathbf{x}_i$'s are Gaussian, the data has been used in \cite{el2018impact,el2013robust} to study the performance of high dimensional robust regressions and used in \cite{hu2019central} to study the large MIMO systems in wireless communications. Second, when the entries $\mathbf{x}_i$ have more general distributions,  in \cite{li2018structure,yang2021testing}, the model has been utilized to study the covariance structures in various settings. Third, when $\xi_i$'s are chosen as the random sampling weights, the model has been used to study the high dimensional bootstrap in \cite{el2019non,naumov2019bootstrap,yu2024testing}. Finally, model (\ref{eq_datamatrix}) appears frequently in deep neural networks and are closely related to the  input-output Jacobian matrices \cite{PLone, PLthree, PLtwo}.       
%\end{example}

\quad \quad For the multipliers, we impose the following mild assumptions.  
%In the rest of this subsection, we introduce the  two main technical assumptions. The first assumption (cf. Assumption \ref{assum_D}) is imposed on our bootstrap multipliers, say the random diagonal matrix $D.$
\begin{assumption}\label{assum_D}
Let $D^2=\operatorname{diag} \left\{ \xi_1^2, \cdots, \xi_n^2 \right\}.$ Moreover, for its entries, we assume $$\xi_i^2 \sim \xi^2, \ 1 \leq i \leq n,$$ are i.i.d. generated from a nonnegative and non-degenerated continuous random variable $\xi^2$ satisfying the following assumptions.
\begin{enumerate}
\item[(i)] {\bf Unbounded support case.} We assume that $\xi^2$ has an unbounded support and  satisfies either of the following two conditions:  \\
 (a). $\xi^2$ is a  regularly varying random variable  \cite{resnick2008extreme}  that {
	\begin{equation}\label{ass3.1}
	    \mathbb{P}(\xi^2>x)=\frac{L}{x^{\alpha}},\quad x\uparrow\infty,
	\end{equation} }
	for some $\alpha\in(2,+\infty)$, where $L>0$ is some universal constant.
	%$L(x)$ is a slowly varying function in the sense that for all $t>0,$ $\lim_{x \rightarrow \infty} L(tx)/L(x)=1.$

\vspace{3pt}
	
\noindent(b). $\xi^2$ has an exponentially decaying tail in the sense that there exist some universal constants $L>0, t>0$ and $0<\beta<2$  so that {
\begin{align}\label{ass3.2}
\mathbb{P}(\xi^2 > x) = L e^{-t x^{\beta}}, \ x \uparrow \infty. 
\end{align} }
%
%
% for some constant $\beta>0$ and any fixed constant $t>0$
%\begin{equation}\label{ass3.2}
%\mathbb{E}e^{t\xi^{2\beta}}<\infty. 
%\end{equation}
%{\color{red} [need to add more assumption here, in some cases, the distribution is Tracy-Widom]}
\item[(ii)] {\bf Bounded support case}.  We assume that
 $\xi^2$ has a bounded support on $(0,l]$ for some fixed constant
$l>0.$ Moreover, for some constant $d>-1,$ we assume that  
%the tail probability around the hard edge $l$ satisfies
\begin{equation}\label{ass3.4}
\mathbb{P}( l-\xi^2 \leq x) \asymp x^{d+1}.
\end{equation}
Finally, let $F(x)$ be the cumulative distribution function (CDF) of $\xi^2,$ we assume that  
\begin{equation}\label{eq_defnbfrak}
0< \mathfrak{b}:= \lim_{x \uparrow l} \frac{1-F(x)}{(l-x)^{d+1}}<\infty. 
\end{equation}
\end{enumerate}

%Without loss of generality, we assume that 

\end{assumption}

\begin{remark}\label{rmk_multiplier}
Several remarks are in order. First, for the unbounded  multipliers, (\ref{ass3.1}) indicates that the tails of $\xi^2$ decay polynomially. Many commonly used distributions are included in this category. To name but a few, Pareto distribution, $F$ distribution and student-$t$ distribution. Moreover, according to extreme value theory (see Lemma \ref{lem_summaryevt} of our supplement), when (\ref{ass3.1}) is satisfied, $\xi_{(1)}^2$ follows a Fr{\' e}chet distribution asymptotically. Second, for the unbounded setting, (\ref{ass3.2}) implies that the tails of $\xi^2$ decay exponentially.  In fact, with some additional efforts, one can generalize (\ref{ass3.2}) to 
%, it is necessarily that the CDF of $\xi^2$ admits 
\begin{equation}\label{eq_defng}
\mathbb{P}(\xi^2>x)=\exp(-\mathsf{g}(x)),
\end{equation}
for some positive function $\mathsf{g}(x)>0$ exhibits polynomial growth which allows the constant $L$ to be slightly generalized to a broader class of functions. Furthermore, if \begin{equation}\label{assum_gg}
\mathsf{g} \in C^{\infty}([0, \infty)), \ \ \lim_{x \uparrow \infty} (1/\mathsf{g}'(x))'=0,
\end{equation}     
we see from Lemma \ref{lem_summaryevt} of our supplement that $\xi_{(1)}^2$ follows a Gumbel distribution asymptotically. In fact, many commonly used distributions, for instance, Chi-squared distribution, exponential distribution and Gamma distribution, satisfy these conditions. 

Third, for the bounded  multipliers, (\ref{ass3.4}) indicates that $\xi^2$ has a possible polynomial decay behavior near the edge. Under the assumption of (\ref{eq_defnbfrak}), we see from Lemma \ref{lem_summaryevt} of our supplement that $\xi_{(1)}^2$ obeys a Weibull distribution asymptotically. The conditions allow for many distributions like (shifted) Beta distribution, uniform distribution and U-quadratic distribution. In summary,  Assumption \ref{assum_D} is mild and covers many commonly used multipliers. 

%In contrast, as mentioned in Section \ref{subsec_existingresult}, existing literature only handles deterministic or nearly deterministic $\xi_i^2, 1 \leq i \leq n.$           
\end{remark}

\quad The following assumption introduces some mild conditions on the aspect ratio $p/n$ and the population covariance matrix $\Sigma.$

\begin{assumption}\label{assumption_techincial}
We assume the following conditions hold true for some small constant $0<\tau<1$. 
\begin{enumerate}
\item[(i)] {\bf On dimensionality}. Throughout the paper, we consider the high dimensional regime that 
\begin{equation}\label{ass1}
 \tau \leq \phi:=\frac{p}{n} \leq \tau^{-1}. 
\end{equation}
\item[(ii)] {\bf On $\Sigma.$} For the population covariance matrix $\Sigma,$ we assume that it admits the following spectral decomposition 
\begin{equation}\label{eq_sigmaspectraldecomposition}
\Sigma=\sum_{j=1}^p \sigma_j \mathbf{v}_j \mathbf{v}_j^*,
\end{equation}
where 
\begin{equation}\label{ass2}
\tau \leq \sigma_p \leq \sigma_{p-1} \leq \cdots \leq \sigma_2 \leq \sigma_1 \leq \tau^{-1}, 
\end{equation}
are the eigenvalues and $\{\mathbf{v}_j\}$ are the associated eigenvectors. 
\end{enumerate}
\end{assumption}

\quad We remark that (\ref{ass1}) is commonly used in random matrix theory and high dimensional statistics literature for quantifying the high dimensionality. (\ref{ass2}) states the eigenvalues of the population covariance matrix are bounded from above and below. On the one hand, when $\xi^2$ has unbounded support as in Case (i) of Assumption \ref{assum_D}, (\ref{ass2}) is the only assumption imposed on $\Sigma.$  On the other hand, when $\xi^2$ has bounded support as in Case (ii) of Assumption \ref{assum_D}, we will require an additional mild assumption, Assumption \ref{assum_additional_techinical}, to exclude potential spikes.   

In the statistical literature, motivated by real applications, one often adds some spikes to $\Sigma$ which result in the famous spiked covariance matrix model \cite{ding2021spiked11,John2001}. To construct such a model, one can introduce a perturbed version of $\Sigma,$ denoted as $\widetilde{\Sigma}$ whose spectral decomposition follows 
\begin{equation}\label{eq_truemodelspiked}
\widetilde{\Sigma}=\sum_{j=1}^p \widetilde{\sigma}_j \mathbf{v}_j \mathbf{v}_j^*, 
\end{equation}
where for some fixed constant $r>0,$ {$\widetilde{\sigma}_1 \geq \widetilde{\sigma}_2 \geq \cdots \geq \widetilde{\sigma}_r>\widetilde{\sigma}_{r+1}$ are $r$  values representing the larger spikes while the rest $\widetilde{\sigma}_j=\sigma_j, j\geq r+1$'s are relatively small and bounded. For simplicity and definiteness, we assume that for some constant $\tilde{\tau}>0$
\begin{equation}\label{eq_separation}
\frac{\widetilde{\sigma}_i}{\widetilde{\sigma}_{i+1}} \geq 1+ \tilde{\tau},  \  1 \leq i \leq r. 
\end{equation} }
Based on $\widetilde{\Sigma},$ the counterpart of the bootstrapped data matrix (\ref{eq_datamatrix}) can be written as 
\begin{equation*}
\widetilde{Y}=\widetilde{\Sigma}^{1/2}XD.
\end{equation*}
Consequently, the bootstrapped sample covariance matrix (i.e., the counterpart of (\ref{eq_samplecov})) can be written as 
\begin{equation}\label{eq_samplecov_spike}
\widetilde{Q}:=\widetilde{Y} \widetilde{Y}^* \equiv \widetilde{\Sigma}^{1/2}XD^2 X^*  \widetilde{\Sigma}^{1/2}. 
\end{equation}

As explained earlier, the understanding of the multiplier boostrap on the edge eigenvalues of sample covariance matrices boils down to  the study of 
%the extreme singular values of $Y$ in  (\ref{eq_datamatrix}) i.e., 
the first few largest eigenvalues of the $p \times p$ bootstrapped sample covariance matrices $Q$ in (\ref{eq_samplecov}) or $\widetilde{Q}$ in (\ref{eq_samplecov_spike}).  To clarify the notations used for the various matrices, we summarize them in Table \ref{table_notations}.

\begin{table}[!ht]
\begin{tabular}{llll}
\toprule
      Model           & Quantity  &Sample version & Bootstrapped version\\ \midrule
\multirow{2}{*}{Non-spiked} & Matrix & $S:=\Sigma^{1/2}XX^{*}\Sigma^{1/2}$ & $Q:=\Sigma^{1/2}XD^2X^{*}\Sigma^{1/2}$\\ \cmidrule(l){3-4} 
                  & Eigenvalues & $\{\widehat{\lambda}_i\}$ &  $\{\lambda_i\}$ \\ \midrule
\multirow{2}{*}{Spiked} &Matrix & $\widetilde{S}:=\widetilde{\Sigma}^{1/2}XX^{*}\widetilde{\Sigma}^{1/2}$ & $\widetilde{Q}:=\widetilde{\Sigma}^{1/2}XD^2X^{*}\widetilde{\Sigma}^{1/2}$  \\ \cmidrule(l){3-4}
                  & Eigenvalues & $\{\widehat{\mu}_i\}$ & $\{\mu_i\}$\\\bottomrule
\end{tabular}
\caption{Summary of some important notations.}\label{table_notations}
\end{table}

%\begin{remark}

%\end{remark}

%\section{Main results}\label{sec_mainresults}
%\subsection{Extreme eigenvalues}

%We emphasize that the theoretical foundation offers valuable insights into the performance of the multiplier bootstrap across various classes of multipliers. Additionally, we explore scenarios where the bootstrap method performs exceptionally well, instances where it falls short, and strategies to improve its accuracy in addressing testing problems within the spiked covariance matrix model framework. }
\section{Multiplier bootstrap meets the non-spiked covariance matrix model}\label{sec_boostrapeffect_nonspike}
%The first part of the results are summarized in Theorems \ref{thm_main_unbounded} and \ref{thm_main_bounded} below, which demonstrate the asymptotic distributions for top eigenvalues of bootstrapped sample covariance matrix through unbounded or bounded multipliers, respectively.

In this section, we present the first part of the main results by evaluating the effectiveness of the bootstrap method for different classes of multipliers. Specifically, we establish the asymptotic distributions of the largest eigenvalues of both the conditional and unconditional bootstrapped sample covariance matrix (\ref{eq_samplecov}) when the population covariance matrix $\Sigma$ does not contain  large spikes. In this scenario, the extreme eigenvalues of the sample covariance matrix $S$ in Table \ref{table_notations} follow the Tracy-Widom distribution \cite{lee2016tracy}. More explicitly, there exist some constants $\gamma_0, E_+$ so that when $n$ is sufficiently large
\begin{equation}\label{eq_twresultoriginal}
\left|\mathbb{P} \left( \gamma_0 n^{2/3}(\widehat{\lambda}_1-E_+) \leq x  \right) -\mathrm{TW}(x) \right|=\mathrm{o}(1),
\end{equation}
where $\mathrm{TW}(x)$ denotes the CDF of the type-1 Tracy--Widom distribution. For concrete definitions of $\gamma_0$ and $E_+,$ we refer the readers to (\ref{eq_def_Eplus}) and (\ref{eq_def_gamma0}) of the supplement.

In Section \ref{sec_nonspike_thebad}, we show that when the multipliers are unbounded, the multiplier bootstrap will destroy the Tracy-Widom shape in terms of  (\ref{eq_twresultoriginal}), whether considered unconditionally or conditionally. Subsequently, in Section \ref{sec_nonspike_thegood}, we study the impact of bounded multipliers with polynomial-type edge behavior on the bootstrapped sample covariance matrices, for both the unconditional and conditional bootstrap. We identify phase transitions among the limiting spectral-edge laws and briefly discuss their implications for further downstream statistical inference in terms of understanding (\ref{eq_twresultoriginal}). 

For notational simplicity, throughout the remainder of the paper, and following conventions in the bootstrap literature (see \cite{lopes2025improved}), we use $\mathbb{P}(\cdot \mid X)$ to denote probabilities conditional on $\mathbf{x}_1, \ldots, \mathbf{x}_n$.

%{\color{red}[introduce the notations for conditional boostrapping]}

\subsection{The impact of unbounded multipliers}\label{sec_nonspike_thebad}
We first provide the results for the extreme eigenvalues when the multiplier $\xi^2$ has unbounded support in the sense that (i) of Assumption \ref{assum_D} holds. Denote
\begin{equation}\label{eq_somenotations}
\bar{\sigma}_1=\frac{1}{p} \sum_{i=1}^p \sigma_i,\quad {\bar{\sigma}_2=\frac{1}{p}\sum_{i=1}^p\sigma_i^2}, \quad \varphi:=\phi\bar{\sigma}_1. 
\end{equation}
Recall $F(x)$ is the CDF of $\xi^2_i, 1 \leq i \leq n.$ Denote 
\begin{equation}\label{eq_defnbn}
b_n:=\inf \left\{ x: 1-F(x) \leq \frac{1}{n} \right\}.
\end{equation}

Recall from Table \ref{table_notations} that $\lambda_1$ is the largest eigenvalue of $Q$. The first result concerns the unconditional version of multiplier bootstrap. 
\begin{theorem}[Unconditional bootstrap]\label{thm_main_unbounded} Suppose Assumptions \ref{assum_model}, \ref{assumption_techincial} and (i) of Assumption \ref{assum_D} hold. Recall (\ref{eq_somenotations}).  When (\ref{ass3.1}) holds and $n$ is sufficiently large, we have that for all $x \geq 0$
 %when $n$ is sufficiently large,
%where for $\phi$ in (\ref{ass1}),  Consequently, when (\ref{ass3.1}) holds, $\varphi^{-1}\lambda_1 $ follows the Fr{\'e}chet distribution asymptotically in the sense that 
%for $x \geq 0$
{
\begin{equation}\label{eq_mainresultunboundedfrechet}
\left| \mathbb{P} \left( \frac{\lambda_1}{ \varphi b_n} \leq x \right)-\exp\left(-x^{-\alpha} \right) \right|=\mathrm{o}(1).
\end{equation} }
Moreover, when (\ref{ass3.2}) holds and $n$ is sufficiently large, we have that for all $x \in \mathbb{R}$  {
%$ \varphi^{-1}\lambda_1 $ follows the Gumbel distribution asymptotically in the sense that 
%for $x \in \mathbb{R}$
 \begin{equation}\label{eq_mainresultunboundedgumbel}
% \lim_{n \rightarrow \infty}
 \left| \mathbb{P}\left( \mathsf{g}'(b_n) \left[\varphi^{-1}\lambda_1-(b_n+c_0) \right] \leq x \right)-\exp \left(-e^{-x} \right) \right|=\mathrm{o}(1),
 \end{equation}}
 where we recall $\mathsf{g}(x)=e^{\log L} t x^\beta$ and {$c_0:=\varphi^{-1}\times\mathbb{E}\xi^2\times \bar{\sigma}_2/\bar{\sigma}_1$.} 
\end{theorem}

\begin{remark}\label{rmk_mainresults_unbounded}
%{\color{blue}
%
%explain the role of $c_0,$ both technically and empirically.  
%}
Two remarks are in order. First, Theorem \ref{thm_main_unbounded} states that when $\xi^2$ has unbounded support, $\lambda_1$ will be divergent. In fact, as can be seen in the proof of Proposition \ref{lem: eigenvalue rigidity} of the supplement, after being properly centered and scaled, $\lambda_1$ will have a similar behavior to $\xi_{(1)}^2$ in the sense that
\begin{equation}\label{eq_defnvarphi}
\frac{\lambda_1}{\xi_{(1)}^2}=\varphi+\ro_{\mathbb{P}}(1).
\end{equation}
Especially, when $\xi^2$ has a polynomial decay tail as in (\ref{ass3.1}), we can obtain the Fr{\' e}chet limit and when $\xi^2$ has an exponential decay tail as in (\ref{ass3.2}), we can get the Gumbel limit. 

{Second, Theorem \ref{thm_main_unbounded} also indicates that if we select unbounded multipliers, the typical Tracy-Widom (TW) limit of largest eigenvalues from sample covariance matrices has no chance to be reproduced.} Third, the above results can be generalized to the joint distribution of $k$ largest eigenvalues for any fixed integer $k.$ That is, for all $s_i \in \mathbb{R}, 1 \leq i \leq k,$ (\ref{eq_mainresultunboundedfrechet}) can be generalized to 
\begin{equation*}
\lim_{n \rightarrow \infty} \mathbb{P}  \left( \left( \frac{\lambda_i}{\varphi b_n} \leq s_i \right)_{1 \leq i \leq k} \right)=\lim_{n \rightarrow \infty}\mathbb{P} \left( \left( \frac{\xi^2_{(i)}}{b_n} \leq s_i \right)_{1 \leq i \leq k} \right),
\end{equation*}
and (\ref{eq_mainresultunboundedgumbel}) can be generalized to 
\begin{equation*}
\lim_{n \rightarrow \infty} \mathbb{P}  \left( \mathsf{g}'(b_n)\left(\varphi^{-1} \lambda_i-(b_n+c_0) \leq s_i \right)_{1 \leq i \leq k} \right)=\lim_{n \rightarrow \infty} \mathbb{P}  \left( \mathsf{g}'(b_n)\left(\xi_{(i)}^2-(b_n+c_0) \leq s_i \right)_{1 \leq i \leq k} \right).
\end{equation*}
Since the joint distribution of the order statistics of $\{\xi_i^2\}$ can be computed explicitly \cite{coles2001introduction}, the above formulas give a complete description of the finite-dimensional correlation
functions of the extremal eigenvalues. Finally, we mention that the Fr{\'e}chet distribution and Gumbel distribution also appear in the literature in heavy-tailed sample covariance matrices, see \cite{auffinger2009poisson,heiny2017eigenvalues,heiny2021large,heiny2021point} for example.

%Third, Theorem \ref{thm_main_unbounded} shows that even when $\Sigma$ has no spikes, due to the effect of multiplier bootstrap of $\{\xi_i^2\},$ the first few eigenvalues of $Q$ can also be divergent. Consequently, in order the spikes to be properly detected, if exist, the true spikes of $\Sigma$ have to be divergent; see Theorem \ref{thm_main_spike} for more detail.       

\end{remark}

 The second result provides the conditional counterpart to Theorem \ref{thm_main_unbounded}. 
%Now, we consider the largest eigenvalue of $Q$ conditional on $X$. To this end, let $\mathcal{F}_X=\sigma(X)$, we work on a high probability event $\Omega_X\in\mathcal{F}_X$ depending only on $X$ with $\mathbb{P}(\Omega_X)\rightarrow1$ on which $X$ satisfies the regularity conditions as in Assumption \ref{ass1}. Conditional on $\mathcal{F}_X$ (equivalently, for $X$ restricted to $\Omega_X$), we study $\lambda_1(Q)$ which is a function of $D$, under the conditional probability $\mathbb{P}(\cdot|X)\equiv\mathbb{P}(\cdot|\mathcal{F}_X)$. 

\begin{theorem}[Conditional bootstrap]\label{thm_main_unbounded_conditional}
Suppose Assumptions \ref{assum_model}, \ref{assumption_techincial} and (i) of Assumption \ref{assum_D} hold. Then, when (\ref{ass3.1}) holds and $n$ is sufficiently large,  with probability at least $1-\ro(1),$ 
for all $x \geq 0$
\begin{equation}\label{eq_mainresultunboundedfrechet_conditional}
\left|\mathbb{P} \left( \frac{\lambda_1}{ \varphi b_n} \leq x \Big|X\right)-\exp\left(-x^{-\alpha} \right) \right|=\mathrm{o}(1).
\end{equation}
%holds with high probability
%, where $\varphi:=\phi\bar{\sigma}_1$ for $\phi$ in (\ref{ass1}).

Moreover, when (\ref{ass3.2}) holds and $n$ is sufficiently large, we have that the following statement holds with probability at least $1-\ro(1),$  for all $x \in \mathbb{R}$
 \begin{equation}\label{eq_mainresultunboundedgumbel_conditional}
 \left|\mathbb{P}\left( \mathsf{g}'(b_n) \left[\varphi^{-1}\lambda_1-(b_n+c_0) \right] \leq x \Big|X\right)-\exp \left(-e^{-x} \right) \right|=\mathrm{o}(1),
 \end{equation}
 where $\mathsf{g}$ and $c_0$ are defined in (\ref{eq_mainresultunboundedgumbel}). 
% holds with high probability,  where $\mathsf{g}$ is defined in (\ref{eq_defng}) and {$c_0:=1+\varphi^{-1}\times\mathbb{E}\xi^2\times \bar{\sigma}_2/\bar{\sigma}_1$.} 
\end{theorem}

\begin{remark}
Theorem \ref{thm_main_unbounded_conditional} shows that, under a high-probability event, the conditional bootstrap results align with the unconditional bootstrap results stated in Theorem \ref{thm_main_unbounded}. The construction of this high-probability event is provided in Section \ref{sec_characterization_OmegaX} of the supplement. Similar to the discussion in Remark \ref{rmk_mainresults_unbounded}, Theorem \ref{thm_main_unbounded_conditional} also indicates that, when unbounded multipliers are used, the Tracy--Widom (TW) limit for the largest eigenvalues of sample covariance matrices in (\ref{eq_twresultoriginal}) cannot be recovered.
\end{remark}

\subsection{The impact of bounded multipliers}\label{sec_nonspike_thegood}
\quad Next, we state the results when the multiplier $\xi^2$ has bounded support in the sense that (ii) of Assumption \ref{assum_D} holds. We first prepare some notations. Recall $F(x)$ is the CDF of $\xi^2.$ Let  $(m_{1n,c}(z), m_{2n,c}(z), m_{n,c}(z))$ be the unique solutions of the following system of equations
\begin{align}
    m_{1n,c}(z)=\frac{1}{n}\sum_{i=1}^p & \frac{\sigma_i}{-z(1+\sigma_im_{2n,c}(z))},\quad m_{2n,c}(z)=\int_0^l \frac{s}{-z(1+s m_{1n,c}(z))} \mathrm{d} F(s), \nonumber\\
   & m_{n,c}(z)=\frac{1}{p}\sum_{i=1}^p\frac{1}{-z(1+\sigma_im_{2n,c}(z))}. \label{eq_systemequationsm1m2intergrate}
\end{align}
As will be shown in Theorem \ref{lem_solutionsystem1} below, $m_{n,c}(z)$ is the Stieltjes transform of a probability density function, denoted by $\widetilde{\rho}$. Moreover, we denote
\begin{equation}\label{edge_edge_edge}
\operatorname{supp}(\widetilde{\rho}) = [L_-, L_+].
\end{equation}
Recall $l$ from (\ref{ass3.4}). For notional convenience, we denote
%
%
%and $l$ from
%
%and using the definitions of $m_{1n,c}$ and $L_+$ from (\ref{eq_systemequationsm1m2intergrate}) and (\ref{edge_edge_edge}) below, we denote 
\begin{equation}\label{eq_phasetransition}
\begin{split}
    &\mathsf{s}_1:=\int_0^l\frac{l^2s^2}{(l-s)^2}\mathrm{d}F(s),\quad \quad \mathsf{s}_2:= \int_0^l\frac{ls}{l-s}\mathrm{d}F(s),\\
    &\mathsf{s}_3:=\frac{1}{p} \sum_{i=1}^p \frac{\sigma_i^2 \mathsf{s}_1}{(L_+-\sigma_i \mathsf{s}_2)^2}, \quad \quad \mathsf{s}_4:=\frac{1}{n}\sum_{i=1}^p \frac{\sigma_i}{(L_+-\sigma_i \mathsf{s}_2)^2},
\end{split}
\end{equation} 
\begin{equation}\label{eq_defnvariance}
\mathsf{v}:= \left( \frac{\varsigma(2)}{\varsigma(1)} \right)^2 \operatorname{Var} \left( \frac{\xi^2}{1+\xi^2 m_{1n,c}(L_+)} \right),
%%\int \left( \frac{s}{1+s m_{1n,c}(L_+)} \right)^2 \mathrm{d}F(s)-\left( \int \frac{s}{1+s m_{1n,c}(L_+)}  \mathrm{d}F(s) \right)^2.
\end{equation}
where for $k \in \mathbb{N}$
\begin{equation*}
\varsigma(k):=\sum_{i=1}^p\sigma_i^k \left(L_+-\sigma_i \int_0^l \frac{s}{1+sm_{1n,c}(L_{+})} \mathrm{d} F(s)\right)^{-2}.
\end{equation*}

With the above preparation, we are now ready to state the results. Our analysis relies on the following assumption, which avoids certain singular behaviors of $\widetilde{\rho}$ and is commonly used in the random matrix theory literature; see, for example, \cite{Bao2015,Ding&Yang2018,ding2021spiked,9779233,el2007tracy,knowles2017anisotropic,lee2016tracy}.
\begin{assumption}\label{assum_additional_techinical} When (ii) of Assumption \ref{assum_D} holds, for $\Sigma$ satisfying Assumption \ref{assumption_techincial}, we assume that for some universal constant $\mathfrak{a}>0$ 
\begin{equation*}
\min_{1 \leq i \leq p} |1+\sigma_i m_{2n,c}(L_+)| \geq \mathfrak{a}. 
\end{equation*}  
\end{assumption}

The first result concerns the unconditional version of multiplier bootstrap. Recall that for two sequences $a_n, b_n>0,$ we write $a_n \gtrsim b_n$ if $b_n \asymp a_n$ or $b_n=\mathrm{o}(a_n)$ holds.

\begin{theorem}[Unconditional bootstrap]\label{thm_main_bounded} Suppose Assumptions \ref{assum_model}, \ref{assumption_techincial}, \ref{assum_additional_techinical} and (ii) of Assumption \ref{assum_D} hold. Recall the exponent $d$ in (\ref{ass3.4}), $\phi$ in (\ref{ass1}) and the quantities defined in (\ref{eq_phasetransition}). When $n$ is sufficiently large: 
\begin{enumerate}
\item[(1).]  When $d>1$ and $\phi^{-1}>\mathsf{s}_3$, we have that for all
%
%we have that $L_+$ satisfies
%\begin{equation}\label{eq_onlyoneequationdecide}
%    1=\frac{1}{n}\sum_{i=1}^p \frac{-l\sigma_i}{-L_{+}+\sigma_i \mathsf{s}_2}.
%\end{equation}
%Moreover, we have that 
%\begin{equation}\label{eq_boundednnnnn}
%n^{\frac{1}{d+1}}\left| \left( \frac{\lambda_1-L_+}{\mathsf{s}_4^{-1}(1-\phi \mathsf{s}_3)} \right)-\left( \xi_{(1)}^2-l \right) \right|=\ro_{\mathbb{P}}(1). 
%\end{equation} 
%Consequently, we have that $\lambda_1-L_+$ follows Weibull distribution with parameter $d+1$ asymptotically in the sense that for 
$x \leq 0$ {
\begin{equation}\label{eq_distributionresultweibull}
\left| \mathbb{P} \left(l^2 \frac{(\mathfrak{b}n)^{1/(d+1)}}{\mathsf{s}_4^{-1}(1-\phi \mathsf{s}_3)}(\lambda_1-L_+) \leq x \right)-\exp\left(-|x|^{d+1} \right) \right|=\mathrm{o}(1),
\end{equation}}
where $\mathfrak{b}$ is defined in (\ref{eq_defnbfrak}). 
\item[(2).] When $d>1$ and $\phi^{-1}<\mathsf{s}_3,$ we have that for $\mathsf{v}$ defined in (\ref{eq_defnvariance}) and $x \in \mathbb{R}$
%$\lambda_1$ is asymptotically Gaussian in the sense that 
{
\begin{equation}\label{eq_cltresult}
\left| \mathbb{P} \left( \sqrt{n \mathsf{v}^{-1}} (\lambda_1-L_+) \leq x \right)-\Phi(x)\right|=\mathrm{o}(1),
\end{equation}}
where $\Phi(x)$ is the CDF of a real standard Gaussian random variable.  
\item[(3).] When $-1<d \leq 1$, {on the one hand, if $\mathsf{v} \gtrsim n^{-1/3+\iota},$  for some sufficiently small constant $\iota>0$, then (\ref{eq_cltresult}) holds. On the other hand, if $\mathsf{v}= \mathrm{O}(n^{-1/3-\iota})$, then for some $\gamma$ defined in (\ref{eq_gammadefinition}) of the supplement, for all $x \in \mathbb{R}$
\begin{equation}\label{eq_twllllaaalllaaa}
\left|\mathbb{P}\left(n^{2/3}\gamma (\lambda_1-L_+) \leq x \right)-\mathrm{TW}(x)\right|=\mathrm{o}(1),
\end{equation}
where $\mathrm{TW}(x)$ denotes the CDF of the type-1 Tracy--Widom distribution. }
 %Moreover, 
\end{enumerate}

\end{theorem}

\begin{remark}\label{rem_nonspike_bounded}
%{\color{blue} \begin{equation*}
%\lambda_1-L_+=\nu_1+\nu_2+\rO_{\mathbb{P}}(n^{-1}),
%\end{equation*}}
Theorem \ref{thm_main_bounded} shows that, for the unconditional multiplier bootstrap, when $\xi^2$ has bounded support as in (\ref{ass3.4}), $\lambda_1$ will be bounded and can have several phase transitions depending on the exponent $d,$ aspect ratio $\phi$ and the threshold $\mathsf{s}_3$ which encodes the information of $\Sigma$ and the distribution of $\xi^2.$

First, in the setting when $d>1,$ on the one hand, when $\phi^{-1}>\mathsf{s}_3,$ after being properly centered and scaled, $\lambda_1$ will have similar asymptotics as $\xi_{(1)}^2$ and Weibull limit will be obtained. On the other hand when $\phi^{-1}<\mathsf{s}_3,$ $\lambda_1$ will be influenced by all $\{\xi_i^2\}$ and hence asymptotically Gaussian. For the critical case $\phi^{-1}=\mathsf{s}_3,$ we believe there will be a phase transition connecting Gaussian and Weibull. Since this is out of the scope of the paper which focuses on statistical applications, we will pursue this direction in the future works. 

Second, when $-1<d \leq 1,$ the limiting ESD of $Q$  will have a square root decay behavior.  In this setting, $\lambda_1$ will be influenced by two components, the TW part $\nu_1$ and the Gaussian part $\nu_2$ that $$
\lambda_1-L_+=\nu_1+\nu_2+\rO_{\mathbb{P}}(n^{-1}).$$
The TW part is due to the square root behavior and the Gaussian is due to the fact that $\lambda_1$ will be potentially influenced by all $\{\xi_i^2\};$ see Section \ref{sec_prf_nonspike_thegood} of our supplement  for more details. We mention that the variance of the Gaussian part can potentially decay and $\nu_1$ and $\nu_2$ are in generally dependent. Moreover, by appropriately choosing the multipliers to ensure the variance of the Gaussian component diminishes, the fluctuation of $\lambda_1 - L_{+}$ is predominantly governed by the Tracy-Widom distribution. This approach demonstrates that the Tracy-Widom distribution can be recovered using the unconditional bootstrap method with carefully selected bounded multipliers.

Finally, as discussed in Remark \ref{rmk_mainresults_unbounded}, we can generalize the results of Theorem \ref{thm_main_bounded} to the joint distribution of $k$ largest eigenvalues for any fixed $k.$ We omit the details.   
\end{remark}

The second result provides the conditional counterpart to Theorem \ref{thm_main_bounded}. Recall that for two sequences $a_n, b_n>0,$ we write $a_n \gtrsim b_n$ if $b_n \asymp a_n$ or $b_n=\mathrm{o}(a_n)$ holds.

\begin{theorem}[Conditional bootstrap]\label{thm_main_bounded_conditional}
 Suppose Assumptions \ref{assum_model}, \ref{assumption_techincial}, \ref{assum_additional_techinical} and (ii) of Assumption \ref{assum_D} hold. We have that for $n$ is sufficiently large, the following statements hold with probability at least $1-\ro(1):$ 
    \begin{enumerate}
\item[(1).]  When $d>1$ and $\phi^{-1}>\mathsf{s}_3$, we have that for all $x \leq 0$ 
%$L_+$ satisfies
%\begin{equation}\label{eq_location_L+}
%    1=\frac{1}{n}\sum_{i=1}^p \frac{-l\sigma_i}{-L_{+}+\sigma_i \mathsf{s}_2}.
%\end{equation}
%And 
\begin{equation*}
\left| \mathbb{P} \left( \frac{l^2(\mathfrak{b}n)^{1/(d+1)}}{\mathsf{s}_4^{-1}(1-\phi \mathsf{s}_3)}(\lambda_{1}-L_+) \leq x\Big|X\right)-\exp\left(-|x|^{d+1} \right)\right|=\mathrm{o}(1),
\end{equation*}
where $\mathfrak{b}$ is defined in (\ref{eq_defnbfrak}). 
\item[(2).] When $d>1$ and $\phi^{-1}<\mathsf{s}_3,$ we have that for all $x \in \mathbb{R}$
%there exists some constant $\mathsf{c}_2>0$, the following holds with probability at least $1-\mathrm{O}(n^{-\mathsf{c}_2})$
\begin{equation}\label{eq_cltresult1}
\left|\mathbb{P} \left( \sqrt{n \mathsf{v}^{-1}} (\lambda_{1}-L_+) \leq x \big|X \right)-\Phi(x)\right|=\mathrm{o}(1),
\end{equation}
 where $\Phi(x)$ is the CDF of a real standard Gaussian random variable. 
\item[(3).]  When $-1<d \leq 1$, on the one hand, if $\mathsf{v} \gtrsim n^{-1/3+\iota}$, for some sufficiently small constant $\iota>0,$ then (\ref{eq_cltresult1}) holds. On the other hand, if $\mathsf{v}=\mathrm{O}(n^{-2/3-\iota}),$ we have that for all $x \in \mathbb{R}$ 
\begin{equation}\label{eq_cltresult12222}
\left|\mathbb{P} \left( \gamma n^{2/3}  (\lambda_{1}-L_+) \leq x \big|X \right)-\mathbf{1}(\gamma_0 n^{2/3}(\widehat{\lambda}_1-E_+) \leq x)\right|=\mathrm{o}(1),
\end{equation}
where $\gamma_0, E_+$ are defined in (\ref{eq_twresultoriginal}), $\gamma$ is defined in (\ref{eq_twllllaaalllaaa}) and $\mathbf{1}(\cdot)$ is the indicator function.   
\end{enumerate}
\end{theorem}

\begin{remark}\label{remark_important}
Several remarks are in order.  First, Theorem \ref{thm_main_bounded_conditional} shows that, for the conditional bootstrap, 
the largest eigenvalue also exhibits several phase transitions, similar to those in the 
unconditional bootstrap. Moreover, compared with the unconditional bootstrap results in 
Theorem \ref{thm_main_bounded}, we find that in parts (1) and (2), as well as in part of 
part (3) when $\mathsf{v} \gtrsim n^{-1/3+\iota}$, the unconditional and conditional 
bootstrap yield similar results.

Second, as summarized from Theorem \ref{thm_main_bounded_conditional}, in general the conditional 
bootstrap cannot be directly and fully applied to recover the TW law in 
(\ref{eq_twresultoriginal}). However, the results of Theorem 
\ref{thm_main_bounded_conditional}, especially part (3), can still be used to conduct certain 
inference tasks related to (\ref{eq_twresultoriginal}) under slightly stronger assumptions. 

More specifically, in addition to Assumption \ref{assumption_techincial}, we further suppose that the empirical spectral distribution of $\Sigma$ converges to a nonrandom 
probability distribution $H$ and 
\begin{equation}\label{eq_ratioassumption}
\lim_{n \rightarrow \infty}\frac{p}{n}=c \in (0, \infty). 
\end{equation}
Then the limiting system of equations in 
\eqref{eq_systemequationsm1m2intergrate} reduces to
\begin{align}\label{eq_systemequations_limiting}
    m_{1,c}(z)=c\int & \frac{t}{-z(1+tm_{2,c}(z))}\mathrm{d}H(t),\quad 
    m_{2,c}(z)=\int_0^l \frac{s}{-z(1+s m_{1,c}(z))} \mathrm{d} F(s), \nonumber\\
   & m_{0,c}(z)=\int\frac{1}{-z(1+tm_{2,c}(z))}\mathrm{d}H(t).
\end{align}
Furthermore, if we choose $\xi^2$ so that $\mathbb{E}\xi^2=1$ and $\operatorname{Var}(\xi^2)=\mathrm{o}(1),$ the conditions of Theorem 1.1 in \cite{bai2008large} will be satisfied so that (\ref{eq_systemequations_limiting})  collapses to a single equation, as in 
\eqref{eq_systemequation_S_integralform} of the supplement. In particular, we have $L_+=E_+.$

Based on the above discussion, we see that Theorem \ref{thm_main_bounded_conditional} opens the door to using 
the conditional bootstrap to construct confidence intervals for $E_{+}$ in 
(\ref{eq_twresultoriginal}), as (\ref{eq_cltresult1}) reduces to
\begin{equation*}
\left|\mathbb{P} \left( \sqrt{n \mathsf{v}^{-1}} (\lambda_{1}-E_+) \leq x 
\mid X \right)-\Phi(x)\right|=\mathrm{o}(1).
\end{equation*}
Moreover, according to the discussion in 
Remark \ref{rmk_varaincecontrol} of the supplement, $\mathsf{v}\asymp \operatorname{Var}(\xi^2).$ This motivates us to choose a bounded multiplier satisfying condition (3) of Theorem \ref{thm_main_bounded_conditional} with $n^{-1/3} \ll \mathsf{v} \ll 1$, which can be directly translated into the design of the multipliers. With such a choice, one can avoid directly estimating the population covariance 
matrix $\Sigma$ while still constructing a confidence interval for $E_{+}$. 
The details are provided in the next section.

% Then, for  and each $k=1,\dots,n$,
%\begin{align*}
%\mathbb{E}\big|(\sqrt{n}\Sigma^{1/2}\mathbf{x}_k\xi_k)^*B(\sqrt{n}\Sigma^{1/2}\mathbf{x}_k\xi_k)
%-\operatorname{tr}(B\Sigma)\big|^2
%=\operatorname{Var}(\xi^2)\big(\operatorname{tr}(B\Sigma)\big)^2+\mathrm{o}(n^2),
%\end{align*}
%for any deterministic matrix $B$ with bounded operator norm, 
%provided that ., 
%and hence the condition holds if we have $\mathsf{v}=\mathrm{o}(1)$.
%
%
%
%
%
%then  whenever 
%$\mathsf{v}=\mathrm{o}(1)$. Consequently, 
%which further.

%the corresponding edge solutions $L_{+}$ and $E_{+}$ of \eqref{eq_systemequations_limiting} and \eqref{eq_systemequation_S}, respectively, coincide.

Finally, as discussed in Remark \ref{rmk_mainresults_unbounded}, we can generalize the results of Theorem \ref{thm_main_bounded_conditional} to the joint distribution of $k$ largest eigenvalues for any fixed $k.$ In particular, Theorem \ref{thm_main_bounded_conditional} also implies that any fixed number of leading eigenvalues share the same
multiplier-induced Gaussian edge shift in the Gaussian regimes, while
their mutual gaps remain Tracy-Widom scale, which is summarized in the following corollary.
%{\color{blue} 1. $\operatorname{Var}(\xi^2) \asymp \mathsf{v}.$ 2. discussions that whenever $\mathsf{v}$ vanishes. Two equations collapses to one. We can essentially write 
%\begin{equation*}
%L_+=E_+
%\end{equation*}
%under relatively strongly assumption that \begin{equation}\label{eq_ratioassumption}
%\lim_{n \rightarrow \infty}\frac{p}{n}=c \in (0, \infty).
%\end{equation}
%}
\end{remark}

\begin{corollary}\label{cor_joint_gaussian_edge}
Suppose that either part
$(2)$ or the Gaussian regime in part $(3)$ of Theorem \ref{thm_main_bounded_conditional} applies. For any fixed
integer $k\geq1$, there exists a standard Gaussian random variable $\mathtt{Z}_n:=\sqrt{n\mathsf{v}^{-1}}(\widehat{L}_+-L_+)$ where $\widehat{L}_+$ is the random right edge associated with the empirical multiplier distribution, defined in \eqref{eq_conditionaledgedefinition} of the supplement. Then, for any sufficiently small constant $\epsilon>0$,
\begin{equation}\label{eq_joint_gaussian_representation}
\max_{1\leq i\leq k}
\left|
\lambda_i-L_+
-\sqrt{\frac{\mathsf{v}}{n}}\mathtt{Z}_n
\right|
=
\rO_{\mathbb{P}(\cdot|X)}(n^{-2/3+\epsilon}).
\end{equation}

Consequently, for any fixed $(x_1,\ldots,x_k)\in\mathbb{R}^k$, it holds with probability at least $1-\mathrm{o}(1)$ that
\begin{equation}\label{eq_joint_gaussian_conditional}
\left|
\mathbb{P}\left(
\sqrt{n\mathsf{v}^{-1}}(\lambda_i-L_+)\leq x_i,\ 
1\leq i\leq k
\Big|X
\right)
-
\Phi\left(\min_{1\leq i\leq k}x_i\right)
\right|
\rightarrow0.
\end{equation}
In particular, when $k\geq2$, the leading eigenvalue gaps satisfy
\begin{equation}\label{eq_eigengap_rate_gaussian}
\max_{1\leq i\leq k-1}
(\lambda_i-\lambda_{i+1})
=
\rO_{\mathbb{P}(\cdot|X)}(n^{-2/3+\epsilon}).
\end{equation}
\end{corollary}

\section{Multiplier bootstrap meets the spiked covariance matrix model}\label{sec_boostrapeffect_spike}

In this section, we derive the asymptotic distributions of the largest eigenvalues of both the unconditional and conditional bootstrapped sample covariance matrices (\ref{eq_samplecov}) when the population covariance matrix $\Sigma$ exhibits a prominent spiked structure. Under this setting, the extreme eigenvalues of the sample covariance matrices $\widetilde{S}$ in Table \ref{table_notations} follow a Gaussian distribution which can be summarized as follows. Recall that the eigenvalues of $\widetilde{S}$ are denoted as $\widehat{\mu}_1 \geq \widehat{\mu}_2 \geq \cdots.$ Denote
\begin{align}\label{eq_keynotations}
\mathsf{M}^\mathtt{S}_i=1+\frac{1}{n}\sum_{k=r+1}^p \frac{\sigma_k}{\widetilde{\sigma}_i}, \ \     \mathsf{V}_i^{\mathtt{S}}=2+\sum_{k=1}^pv_{ki}^4(\mathfrak{m}_4-3),
\end{align}
where $\mathbf{v}_i=(v_{1i},\dots,v_{pi})^* \in \mathbb{R}^p$ is the $i$-th eigenvector of $\widetilde{\Sigma}$ and $\mathfrak{m}_4=\mathbb{E}(\sqrt{n}x_{11})^4.$

\begin{theorem}\label{thm_sample} Suppose  Assumptions \ref{assum_model},  \ref{assumption_techincial} and \ref{assum_additional_techinical} with $D=I$ hold. For the spiked population covariance matrix $\widetilde{\Sigma}$ in (\ref{eq_truemodelspiked}),  assume that $\widetilde{\sigma}_i \gg n^{1/4}$ for all $1 \leq i \leq r,$ when $n$ is sufficiently large, we have that for all $x \in \mathbb{R}$
\begin{align*}
\left|\mathbb{P}\left(\sqrt{\frac{n}{\mathsf{V}_i^{\mathtt{S}}}}
\left(\frac{\widehat{\mu}_i}{\widetilde{\sigma}_i}-\mathsf{M}_i^{\mathtt{S}}\right)\leq x\right)-\Phi(x) \right|=\mathrm{o}(1),
\end{align*}
where $\Phi(x)$ is the CDF of a real standard Gaussian random variable. 
\end{theorem}

In what follows, we demonstrate that the bootstrap method—particularly the conditional bootstrap—is generally valid for recovering \eqref{eq_keynotations} in Theorem \ref{thm_sample}, provided that the multipliers are properly chosen. Specifically, this validity holds when the population spikes exceed a certain threshold, which can be as low as $\mathrm{O}(n^{1/4+\epsilon})$, for some sufficiently small constant $\epsilon>0,$ substantially smaller than the $\mathrm{O}(n^{1/2+\epsilon})$ threshold established in \cite{karoui2016bootstrap}. 

% Building on this, in Section \ref{sec_spike_themodification}, we propose a novel bias correction procedure to enhance the efficiency and accuracy of the multiplier bootstrap.}

%In Section \ref{sec_spike_thegood},

%\begin{remark}
%{\color{blue} here}
%\end{remark}
%
%
%\subsection{The good: multiplier bootstrap can be useful for the spikes}\label{sec_spike_thegood}
%In this section, we demonstrate that the multiplier bootstrap can be a valuable tool for analyzing the asymptotics of large eigenvalues in sample covariance matrices with spiked structures. Specifically, when suitably chosen multipliers are applied, the leading eigenvalues of the bootstrapped sample covariance matrix corresponding to the population spikes retain Gaussian distributions, making it possible to study the asymptotics of the largest eigenvalues of the sample covariance matrices. 

To formalize the results, we introduce the following threshold for various multipliers according to Assumption \ref{assum_D}
\begin{equation}\label{eq_bootstrap_threshold}
\mathsf{T}:= 
\begin{cases}
n^{1/\alpha} \log n, & \text{if (\ref{ass3.1}) holds}; \\
\log ^{1/\beta} n,  & \text{if (\ref{ass3.2}) holds}; \\
l, &\text{if (\ref{ass3.4}) holds}.
\end{cases}    
\end{equation}
In fact, $\mathsf{T}$ serves as a reference point for identifying suitable multipliers, determined by the spike strength of $\widetilde{\Sigma}$ as in (\ref{eq_truemodelspiked}). 
%{\color{blue}(Remove to supp) For each $1\le i\le r$, we denote a deterministic quantity $\theta_i$ which is the unique solution of the equation,
%\begin{equation*}
%    \frac{\theta_i}{\widetilde{\sigma}_i}=\left(1-\frac{1}{n\theta_i}\sum_{j=r+1}^p\frac{\sigma_j}{1-\widetilde{\sigma}_i^{-1}\sigma_j}\right)^{-1},
%\end{equation*}
%with restriction $\theta_i\in[\widetilde{\sigma}_i,2\widetilde{\sigma}_i]$.} 
Recall (\ref{eq_keynotations}).  We denote
\begin{equation}
    \mathsf{M}_i^\mathtt{Con2}:=1+\frac{\operatorname{Var}(\xi^2)}{(\mathbb{E} \xi^2)^2}\left( \mathsf{M}_i^\mathtt{S}-1 \right),\quad \mathsf{M}_i^\mathtt{Unc}:=\mathsf{M}_i^\mathtt{Con2}+\mathsf{M}_i^\mathtt{S}-1,\label{eq_biasedmean_spike1}
\end{equation}
 and
 \begin{equation}   
    \mathsf{V}_i^\mathtt{Unc}=\frac{\mathbb{E}\xi^4}{(\mathbb{E}\xi^2)^2}\left(\mathsf{V}_i^{\mathtt{S}}+1\right)-1. \label{eq_biasedmean_spike2}
\end{equation}

Recall from Table \ref{table_notations} that $\{\mu_i\}$ are the eigenvalues of the bootstrapped sample covariance matrix $\widetilde{Q}.$ Our first result establishes the theoretical guarantee for the unconditional version of the multiplier bootstrap.
{
\begin{theorem}[Unconditional bootstrap]\label{thm_main_spike_unconditional}Suppose  Assumptions \ref{assum_model}, \ref{assum_D} and  \ref{assumption_techincial} hold.  For the spikes in (\ref{eq_truemodelspiked}) and (\ref{eq_separation}) and $\mathsf{T}$ in (\ref{eq_bootstrap_threshold}), we assume that for all $1 \leq i \leq r$
\begin{equation}\label{spiked_assumption}
\widetilde{\sigma}_i \gg \max\{\mathsf{T}, n^{1/4}\}.
\end{equation}
Then we have for $1 \leq i \leq r$ and $x \in \mathbb{R},$ when $n$ is sufficiently large
\begin{equation}\label{eq_resultuncondtional}
\left|\mathbb{P} \left( \sqrt{\frac{n}{\mathsf{V}_i^\mathtt{Unc}}}\left( \frac{1}{\mathbb{E}\xi^2}\frac{\mu_i}{\widetilde{\sigma}_i}-\mathsf{M}_i^\mathtt{Unc}\right) \leq x \right)-\Phi(x)\right|=\mathrm{o}(1), 
\end{equation}
where $\mathsf{M}^\mathtt{Unc}$ and $\mathsf{V}^\mathtt{Unc}$ are defined in (\ref{eq_biasedmean_spike1}) and (\ref{eq_biasedmean_spike2}), respectively, and $\Phi(x)$ is the CDF of the standard Gaussian random variable. 
\end{theorem}
}

\begin{remark}\label{rem_asymGaussian_spike}
Two remarks are in order. First, the condition \eqref{spiked_assumption} can be verified in practice. Recall from Table \ref{table_notations} that $\{\widehat{\mu}_i\}$ are the eigenvalues of the spiked sample covariance matrices $\widetilde{S}.$ Specifically, according to \cite{CHP}, for $1 \leq i \leq r$, it holds that $$\frac{\widehat{\mu}_i}{\widetilde{\sigma}_i} = 1 + \ro_{\mathbb{P}}(1).$$ Consequently, $\widehat{\mu}_i$ can serve as a proxy for $\widetilde{\sigma}_i$, offering valuable guidance for selecting suitable multipliers based on $\mathsf{T}$. 

Second, the results in Theorem \ref{thm_main_spike_unconditional} suggest that, even though one can obtain multiple copies of $\widetilde{Q}$ to apply (\ref{eq_resultuncondtional}), this approach is not directly useful for understanding Theorem \ref{thm_sample} in terms of estimating (\ref{eq_keynotations}). In fact,
% in order to have $\mathsf{V}_i^{\mathtt{Unc}}=\mathsf{V}_i^{\mathtt{S}},$ the multipliers would need to satisfy $\mathbb{E}\xi^4=(\mathbb{E}\xi^2)^2,$ which is impossible for continuous random variables. Similarly, 
 if we want $\mathsf{M}_i^{\mathtt{Unc}}=\mathsf{M}_i^{\mathtt{S}},$ a direct calculation shows that this would require $\operatorname{Var}(\xi^2)/(\mathbb{E}\xi^2)^2=0,$ which is impossible for continuous random variables. 

%Second, as noted in Table \ref{table_notations}, the non-spiked eigenvalues of $\widetilde{Q}$ closely follow those of the non-spiked matrix $Q$. Specifically, for any fixed integer $k$, we can show that 
%\begin{equation}\label{eq_sticking} \left| \mu_{r+i} - \lambda_i \right| = \rO_{\mathbb{P}} \left( n^{-1/2 + 2\epsilon} d_1 \right), \quad 1 \leq i \leq k, \end{equation}
%where $d_1$ is a slowly divergent constant defined in \eqref{eq_firstddefinition} of our supplement. By combining \eqref{eq_sticking} with Remark \ref{rmk_mainresults_unbounded}, we conclude that $\mu_{r+i}$ $(1 \leq i \leq k)$ follows either a Fréchet or Gumbel distribution, depending on the tail behavior of the multiplier. In contrast, Theorem \ref{thm_main_spike} shows that the spiked eigenvalues are always Gaussian. This distinction highlights the differing distributions of the spiked and non-spiked eigenvalues in the bootstrapped sample covariance matrix, providing a theoretical foundation for detecting spikes. This aspect will be explored further in Section \ref{sec_application_factorselection}.
\end{remark}

Our second result establishes the theoretical guarantee for the
conditional version of the multiplier bootstrap. For the purpose of applicable modification, we provide two versions of conditional bootstrap as follows. Recall (\ref{eq_biasedmean_spike1}). Denote
\begin{equation}
   \mathsf{M}_i^{\mathtt{Con1}}:=\left(\mathbf{v}_i^{*}XX^{*}\mathbf{v}_i-1\right)+\mathsf{M}^{\mathtt{Unc}}_i, \label{eq_biasedmean_conditional} 
    \end{equation}
 and
 \begin{equation}   
     \mathsf{V}^{\mathtt{Con}}_i:=\frac{\operatorname{Var}(\xi^2)}{(\mathbb{E}\xi^2)^2} \left(\mathsf{V}_i^\mathtt{S}+1\right) \label{eq_biasedvar_conditional}.
\end{equation}
To streamline the presentation of the results below, we consider multipliers satisfying $\operatorname{Var}(\xi^2) \asymp 1$.
%Then, we have the following results. {\color{purple} merge Theorem 4.5 and 4.7 into one theorem}
\begin{theorem}[Conditional bootstrap]\label{thm_main_spike_conditional}
Suppose the assumptions of Theorem \ref{thm_main_spike_unconditional} hold. 
Then we have for $1 \leq i \leq r$ and $x \in \mathbb{R},$  the following statements hold with probability at least $1-\ro(1):$ 
\begin{enumerate}
\item For $\mathsf{M}^\mathtt{Con1}$ and $\mathsf{V}^\mathtt{Con}$ defined in \eqref{eq_biasedmean_conditional} and \eqref{eq_biasedvar_conditional},
\begin{equation}\label{eq_main_spike_conditional_theta}
\left| \mathbb{P} \left( \sqrt{\frac{n}{\mathsf{V}^\mathtt{Con}_i}}\left( \frac{1}{\mathbb{E}\xi^2 }\frac{\mu_i}{ \widetilde{\sigma}_i}-\mathsf{M}_i^\mathtt{Con1}\right) \leq x |X \right)-\Phi(x)\right|=\mathrm{o}(1),
\end{equation}
where $\Phi(x)$ is the CDF of the standard Gaussian random variable. 
\item For $\mathsf{M}^\mathtt{Con2}$ and $\mathsf{V}^\mathtt{Con}$ defined in \eqref{eq_biasedmean_spike1} and \eqref{eq_biasedvar_conditional},
\begin{equation}\label{eq_main_spike_conditional_hatmu}
      \left|\mathbb{P}\Big(\sqrt{\frac{n}{\mathsf{V}^\mathtt{Con}_i}}\Big( \frac{1}{\mathbb{E}\xi^2 }\frac{\mu_i}{\widehat{\mu}_i}-\mathsf{M}_i^\mathtt{Con2}\Big)\leq x|X\Big)-\Phi(x)\right|=\mathrm{o}(1),
    \end{equation}
where $\Phi(x)$ is the CDF of the standard Gaussian random variable. 
\end{enumerate}
\end{theorem} 

\begin{remark}\label{rem_conditionalspike}
Theorem \ref{thm_main_spike_conditional} provides two versions of the conditional bootstrap. 
However, since $\mathsf{M}_i^\mathtt{Con1}$ in (\ref{eq_main_spike_conditional_theta}) generally involves the unobserved quantity $\mathbf{v}_i^*XX^*\mathbf{v}_i$, it is not directly applicable in practice.

The more interesting and practically useful result is given in (\ref{eq_main_spike_conditional_hatmu}). 
Indeed, a straightforward calculation shows that
\begin{equation*}
\mathsf{M}_i^\mathtt{Con2}=\mathsf{M}_i^\mathtt{S}, 
\qquad 
\mathsf{V}_i^{\mathtt{Con}}=\mathsf{V}_i^{\mathtt{S}}+1,
\end{equation*}
hold simultaneously as long as the multipliers satisfy
\begin{equation}\label{eq_xicon}
\operatorname{Var}(\xi^2)=\mathbb{E}(\xi^2)^2.
\end{equation}
Furthermore, if we additionally impose $\mathbb{E}\xi^2=1$, then (\ref{eq_main_spike_conditional_hatmu}) becomes
\begin{equation}\label{eq_main_spike_conditional_hatmu11}
\left|
\mathbb{P}\left(
\sqrt{\frac{n}{\mathsf{V}_i^\mathtt{S}+1}}
\left(
\frac{\mu_i}{\widehat{\mu}_i}-\mathsf{M}_i^\mathtt{S}
\right)
\leq x \,\middle|\, X
\right)
-\Phi(x)
\right|
=\mathrm{o}(1),
\end{equation}
which is directly applicable together with Theorem \ref{thm_sample}, for example, for bootstrapping a confidence interval for the true spikes $\widetilde{\sigma}_i$. More specifically, one can conduct Monte Carlo simulations using (\ref{eq_main_spike_conditional_hatmu11}) with multipliers satisfying (\ref{eq_xicon}) to obtain sufficiently efficient estimators of $\mathsf{M}_i^\mathtt{S}$ and $\mathsf{V}_i^\mathtt{S}$.

%The CLT result in Theorem \ref{thm_main_spike_conditional} highlights that while the multiplier bootstrap methodology is possibly useful for spiked sample covariance matrices, it generally introduces bias parts, as reflected in  $\mathsf{M}_i$ and $\mathsf{V}_i$. This sheds light on why the bootstrap technique may exhibit instability in approximating the distribution of the top eigenvalues across almost all types of multipliers, as observed in \cite{karoui2016bootstrap}.
%
%In practice,  both $\mathsf{M}_i$ and  $\mathsf{V}_i$ are often hard to be estimated. In the literature, several technical assumptions are commonly employed to ease this difficulty, such as assuming the population covariance is isotropic \cite{bai2002determining,bai2008large,fan2022estimating,stock2016dynamic} or requiring the spike strength to exceed the threshold of $n^{1/2}$ \cite{yao2023rates,yu2024testing}. However, leveraging the intrinsic nature of the bootstrap methodology, one can approximate $\mathsf{M}_i$ and $\mathsf{V}_i$ directly through asymptotic normality by performing the multiplier bootstrap procedure multiple times.  The details of this idea will be presented in the next subsection. 
\end{remark}

\section{Numerical simulations}\label{sec_simulations}
In this section, we conduct several numerical simulations to examine the accuracy of our theoretical results and to illustrate their practical usefulness. For simplicity and due to space constraints, we present only the results for $\Sigma = I$. We have also performed additional simulations for other forms of $\Sigma$, where the eigenvalues are uniformly sampled from a compact interval. The conclusions are similar: all simulations support the accuracy of our theoretical findings. Moreover, the simulations illustrate the usefulness of the conditional bootstrap results in Remark \ref{remark_important} and part (2) of Theorem~\ref{thm_main_spike_conditional} for practical statistical inference tasks.

First, Table \ref{table_section31} presents the simulation results related to Theorems \ref{thm_main_unbounded} and \ref{thm_main_unbounded_conditional} concerning the non-spiked model with unbounded multipliers, based on a comparison between the percentiles of the limiting CDF and those of the empirical CDF. To save space, we report only the results for multipliers satisfying (\ref{ass3.1}), which serve to validate (\ref{eq_mainresultunboundedfrechet}) and (\ref{eq_mainresultunboundedfrechet_conditional}). The results indicate that our theoretical findings are reasonably accurate when $p$ is reasonably large, with accuracy improving as $p$ increases. Similar conclusions can be drawn when the multipliers satisfy \eqref{ass3.2}. 

%\begin{table}[ht]
%\centering
%
%\begin{minipage}{0.48\textwidth}
%\centering
%\begin{tabular}{cccc}
%\hline
%Quantiles & Percentiles & $100\times50$ & $400\times50$ \\
%\hline
%$-3.73$ & $0.01$ & $0.004$ & $0.007$ \\
%$-3.20$ & $0.05$ & $0.033$ & $0.041$ \\
%$-2.90$ & $0.10$ & $0.072$ & $0.089$  \\
%$-2.27$ & $0.30$ & $0.269$ & $0.292$  \\
%$-1.81$ & $0.50$ & $0.479$ & $0.497$  \\
%$-1.33$ & $0.70$ & $0.691$ & $0.702$  \\
%$-0.60$ & $0.90$ & $0.901$ & $0.908$  \\
%$-0.23$ & $0.95$ & $0.953$ & $0.956$  \\
%$0.48$  & $0.99$ & $0.991$ & $0.992$  \\
%\hline
%\end{tabular}
%\caption*{(a) Case 1}
%\end{minipage}
%\hfill
%\begin{minipage}{0.48\textwidth}
%\centering
%\begin{tabular}{cccc}
%\hline
%Quantiles & Percentiles & $100\times50$ & $400\times50$ \\
%\hline
%$-3.73$ & $0.01$ & $0.004$ & $0.007$ \\
%$-3.20$ & $0.05$ & $0.033$ & $0.041$ \\
%$-2.90$ & $0.10$ & $0.072$ & $0.089$  \\
%$-2.27$ & $0.30$ & $0.269$ & $0.292$  \\
%$-1.81$ & $0.50$ & $0.479$ & $0.497$  \\
%$-1.33$ & $0.70$ & $0.691$ & $0.702$  \\
%$-0.60$ & $0.90$ & $0.901$ & $0.908$  \\
%$-0.23$ & $0.95$ & $0.953$ & $0.956$  \\
%$0.48$  & $0.99$ & $0.991$ & $0.992$  \\
%\hline
%\end{tabular}
%\caption*{(b) Case 2}
%\end{minipage}
%
%\caption{Simulation results.}
%\end{table}

\begin{table}[ht]
\centering

\begin{subtable}{0.48\textwidth}
\centering
\begin{tabular}{cccc}
\hline
Quantiles & Percentiles & $p=800$ & $p=1500$ \\
\hline
$0.659$ & $0.10$ & 0.03 & 0.06 \\
$0.911$ & $0.30$ & 0.21 & 0.24 \\
$1.201$ & $0.50$ & 0.43 & 0.46 \\
$1.674$ & $0.70$ & 0.66 & 0.67 \\
$2.117$ & $0.80$ & 0.78 & 0.78  \\
$2.481$ & $0.85$ & 0.84 & 0.85  \\
$3.081$ & $0.90$ & 0.89 & 0.90 \\
$4.415$ & $0.95$ & 0.94 & 0.95 \\
$9.975$ & $0.99$ & 0.99 & 0.99 \\
\hline
\end{tabular}
\caption{Unconditional bootstrap: Theorem \ref{thm_main_unbounded}.}
\end{subtable}
\hfill
\begin{subtable}{0.48\textwidth}
\centering
\begin{tabular}{cccc}
\hline
Quantiles & Percentiles & $p=800$ & $p=1500$ \\
\hline
$0.659$ & $0.10$ & 0.04 & 0.05 \\
$0.911$ & $0.30$ & 0.20 & 0.24 \\
$1.201$ & $0.50$ & 0.43 & 0.45 \\
$1.674$ & $0.70$ & 0.67 & 0.68 \\
$2.117$ & $0.80$ & 0.78 & 0.79 \\
$2.481$ & $0.85$ & 0.84 & 0.85 \\
$3.081$ & $0.90$ & 0.89 & 0.90 \\
$4.415$ & $0.95$ & 0.94 & 0.95 \\
$9.975$ & $0.99$ & 0.99 & 0.99 \\
\hline
\end{tabular}
\caption{Conditional bootstrap: Theorem \ref{thm_main_unbounded_conditional}.}
\end{subtable}

\caption{Simulation results for Section \ref{sec_nonspike_thebad}: non-spiked model with unbounded multipliers. The theoretical quantiles correspond to the distribution with CDF $F(x)=\exp(-x^{-2})$ with various percentiles. The empirical percentiles are calculated using these theoretical quantitles with the simulated data. In the simulations, we set $p=0.5n$, $\Sigma=I$, and the entries of $X$ are i.i.d.\ $\mathcal{N}(0,n^{-1})$. The multiplier satisfies $\xi^2\sim Y/2$, where $Y$ is a standard Pareto random variable with tail index $2$. We compare the percentiles of the empirical distribution of $\lambda_1/(\varphi b_n)$ with those of $F(x)$ for the unconditional and conditional bootstrap procedures described in Theorems \ref{thm_main_unbounded} and \ref{thm_main_unbounded_conditional}, respectively. The results are based on 2,000 simulation repetitions.}\label{table_section31}
\end{table}

Second, Table \ref{table_section32} presents the simulation results related to Theorems \ref{thm_main_bounded} and \ref{thm_main_bounded_conditional} concerning the non-spiked model with bounded multipliers. To save space, we report only the Gaussian case within part (3) of these results, as it is most relevant for subsequent statistical applications. For this purpose, we consider multipliers constructed as 
\begin{align}\label{eq_xiconstruction}
\xi^2=1+\mathbb{E}U-U,\quad U\sim\mathrm{Beta}(d+1,T^{-1/2}),
\end{align}
with $d=0$ and $T=20$, so that numerically $n^{-1/3} \ll \mathsf{v} \ll 1$. The results indicate that our theoretical findings are reasonably accurate when $p$ is reasonably large, with the accuracy improving as $p$ increases. Similar conclusions can be drawn for the other settings covered by Theorems \ref{thm_main_bounded} and \ref{thm_main_bounded_conditional}.

\begin{table}[!ht] 
\centering

\begin{subtable}{0.48\textwidth}
\centering
\begin{tabular}{cccc}
\hline
Quantiles & Percentiles & $p=800$ & $p=1500$ \\
\hline
$-2.326$ & $0.01$ &0.012  &0.012  \\
$-1.645$ & $0.05$ &0.059  &0.053  \\
$-1.282$ & $0.10$ &0.101  &0.111  \\
$-0.674$ & $0.25$ &0.238  &0.266  \\
$0$ & $0.50$ &0.486  &0.511  \\
$0.674$ & $0.75$ &0.750  &0.762  \\
$1.282$ & $0.90$ &0.894  &0.905  \\
$1.645$ & $0.95$ &0.949  &0.956  \\
$2.326$ & $0.99$ &0.982  &0.993  \\
\hline
\end{tabular}
\caption{Unconditional bootstrap: Theorem \ref{thm_main_bounded}.}
\end{subtable}
\hfill
\begin{subtable}{0.48\textwidth}
\centering
\begin{tabular}{cccc}
\hline
Quantiles & Percentiles & $p=800$ & $p=1500$ \\
\hline
$-2.326$ & $0.01$ &0.014  &0.012  \\
$-1.645$ & $0.05$ &0.057  &0.052  \\
$-1.282$ & $0.10$ &0.104  &0.102  \\
$-0.674$ & $0.25$ &0.252  &0.245  \\
$0$ & $0.50$ &0.510  &0.506  \\
$0.674$ & $0.75$ &0.756  &0.749  \\
$1.282$ & $0.90$ &0.905  &0.899  \\
$1.645$ & $0.95$ &0.954  &0.949  \\
$2.326$ & $0.99$ &0.991  &0.987  \\
\hline
\end{tabular}
\caption{Conditional bootstrap: Theorem \ref{thm_main_bounded_conditional}.}
\end{subtable}

\caption{Simulation results for Section \ref{sec_nonspike_thegood}:  non-spiked model with bounded multipliers. The theoretical quantiles correspond to the standard Gaussian distribution. We set $p=0.5n$, $\Sigma=I$, and the entries of $X$ are i.i.d.\ $\mathcal{N}(0,n^{-1})$. The multiplier satisfies \eqref{eq_xiconstruction} with $T=20$ and $d=0$. We compare the percentiles of the empirical distribution of $\sqrt{n\mathsf{v}^{-1}}(\lambda_1-L_{+})$  with those of standard Gaussian distribution for the unconditional and conditional bootstrap procedures described Theorem \ref{thm_main_bounded} and Theorem \ref{thm_main_bounded_conditional}, respectively. The results are based on 2,000 simulation repetitions.
}
\label{table_section32}
\end{table}

Third, Figure \ref{fig_11} illustrates Corollary \ref{cor_joint_gaussian_edge} for the joint leading-eigenvalue behavior in the non-spiked model. We take $\Sigma=I$ and multipliers satisfying \eqref{eq_xiconstruction}. For replicate $b$, define the multiplier-induced edge shift $\mathcal{X}_b=\frac{1}{n}\sum_{j=1}^n(h(\xi_{b,j}^2)-\mathbb{E}h(\xi^2))$ with $h(s)=s/(1+s m_{1,c}(L_+))$. The joint edge expansion 
\begin{align*}
    \lambda_{b,i}-L_{+}=\mathcal X_b+\mathrm{o}_{\mathbb P}(n^{-1/2}),\quad 1\leq i\leq5.
\end{align*}
predicts that $\mathcal X_b$ is the common shift of the leading eigenvalues.  We order 2000 repetitions by $\mathcal X_b$ and divide them into 25 bins of size 80. Within each bin, we average $\mathcal X_b$ and $\Delta_{b,i}:=\lambda_{b,i}-(2000)^{-1}\sum_{b=1}^{2000}\lambda_{b,i}$.
and plot the averages on a common empirical scale. In both the unconditional and conditional settings, all five curves closely follow the identity line, confirming the common multiplier-induced shift.

%Similar conclusions can be drawn for the other settings covered by Theorems \ref{thm_main_spike_unconditional} and \ref{thm_main_spike_conditional}.

\begin{table}[ht] 
\centering

\begin{subtable}{0.48\textwidth}
\centering
\begin{tabular}{cccc}
\hline
Quantiles & Percentiles & $p=300$ & $p=600$ \\
\hline
$-2.326$ & $0.01$ &0.008  &0.007  \\
$-1.645$ & $0.05$ &0.051  &0.043  \\
$-1.282$ & $0.10$ &0.091  &0.099  \\
$-0.674$ & $0.25$ &0.253  &0.254  \\
$0$ & $0.50$ &0.496  &0.486  \\
$0.674$ & $0.75$ &0.736  &0.740  \\
$1.282$ & $0.90$ &0.880  &0.881  \\
$1.645$ & $0.95$ &0.934  &0.931  \\
$2.326$ & $0.99$ &0.973  &0.977  \\
\hline
\end{tabular}
\caption{Unconditional bootstrap: Theorem \ref{thm_main_spike_unconditional}.}
\end{subtable}
\hfill
\begin{subtable}{0.48\textwidth}
\centering
\begin{tabular}{cccc}
\hline
Quantiles & Percentiles & $p=300$ & $p=600$ \\
\hline
$-2.326$ & $0.01$ &0.007  &0.007  \\
$-1.645$ & $0.05$ &0.045  &0.044  \\
$-1.282$ & $0.10$ &0.096  &0.094  \\
$-0.674$ & $0.25$ &0.254  &0.246  \\
$0$ & $0.50$ &0.508  &0.494  \\
$0.674$ & $0.75$ &0.746  &0.740  \\
$1.282$ & $0.90$ &0.886  &0.887  \\
$1.645$ & $0.95$ &0.937  &0.938  \\
$2.326$ & $0.99$ &0.982  &0.984  \\
\hline
\end{tabular}
\caption{Conditional bootstrap: Theorem \ref{thm_main_spike_conditional}.}
\end{subtable}

\caption{Simulation results for Section \ref{sec_boostrapeffect_spike}:  spiked model with unbounded multiplier. The theoretical quantiles correspond to the standard Gaussian distribution. We set $p=0.5n$, $\widetilde{\Sigma}=\operatorname{diag}(20,1,\dots,1)$, and the entries of $X$ are i.i.d.\ $\mathcal{N}(0,n^{-1})$. The multiplier is chosen as $\xi^2\sim \exp(1)$.  For the unconditional bootstrap, we compare the percentiles of the empirical distribution of $\sqrt{n(\mathsf{V}_1^{\mathtt{Unc}})^{-1}}(\mu_1/\widetilde{\sigma}_1-\mathsf{M}_1^{\mathtt{Unc}})$  with those of standard Gaussian distribution as in Theorem \ref{thm_main_spike_unconditional}, while for the  conditional bootstrap, we compare $\sqrt{n(\mathsf{V}_1^{\mathtt{Con}})^{-1}}(\mu_1/\widehat{\mu}_1-\mathsf{M}_1^{\mathtt{Con2}})$ with those of standard Gaussian distribution as in Theorem \ref{thm_main_spike_conditional}. The results are based on 2,000 simulation repetitions.
}
\label{table_section4}
\end{table}

\begin{figure}[ht]
    \centering
    \includegraphics[width=0.9\textwidth]{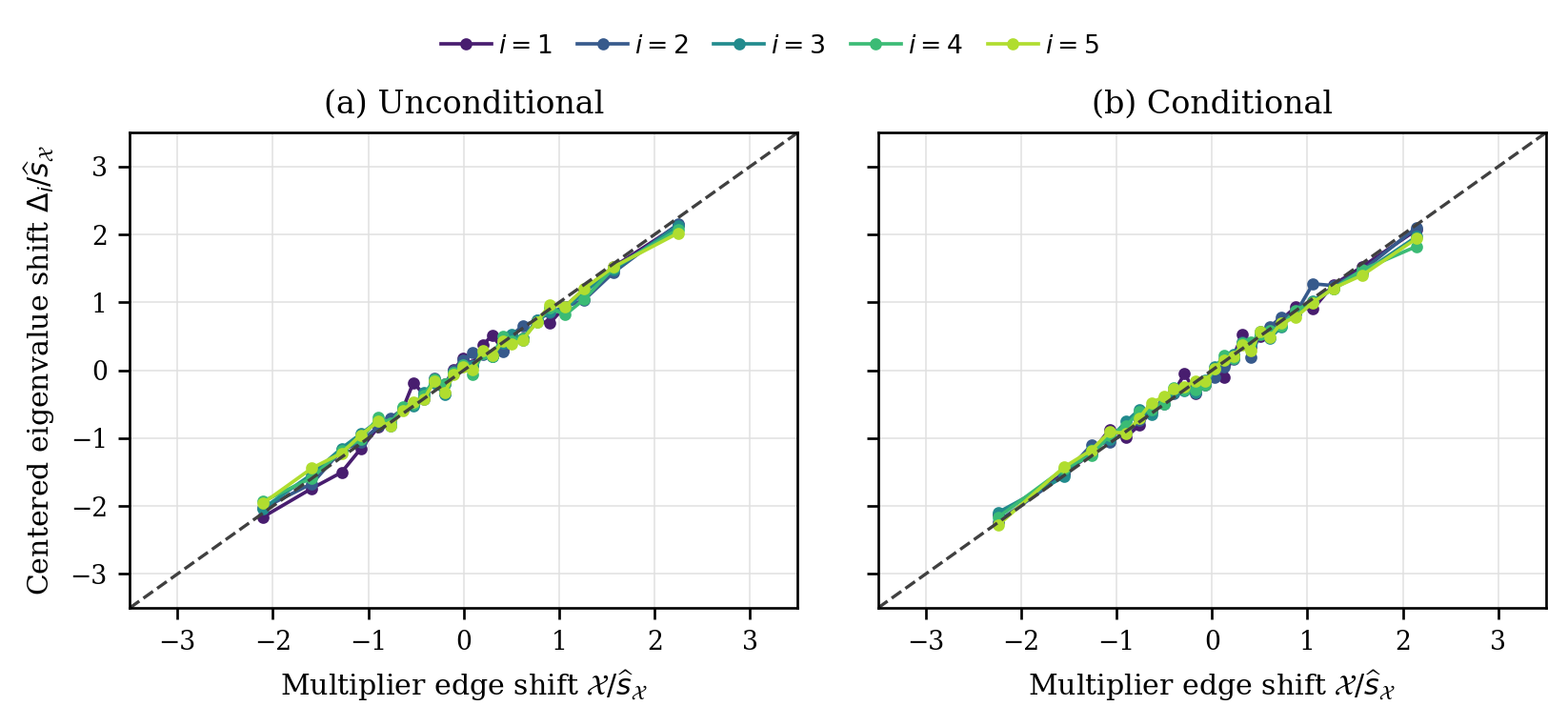}
    \caption{Simulation results for Corollary \ref{cor_joint_gaussian_edge} with $p=400$, $n=800$, $\Sigma=I$, and multipliers satisfying \eqref{eq_xiconstruction} with $d=0$ and $T=20$. For the five leading eigenvalues, the standardized centred shifts $\Delta_{b,i}/\widehat{s}_{\mathcal X}$ are plotted against the standardized multiplier-induced edge shift $\mathcal X_b/\widehat{s}_{\mathcal X}$ under the unconditional and conditional laws, where $\widehat{s}_{\mathcal X}$ is the empirical standard deviation of $\mathcal{X}_b$. Connected points are within-bin averages, and the dashed line is the identity line. Their close alignment illustrates the common edge shift. Results are based on 2,000 repetitions.}\label{fig_11}
    \par\smallskip\noindent Alt text: Two panels show unconditional and conditional results. In each panel, five coloured curves for the five leading eigenvalues lie close to the dashed identity line over standardized shifts from about minus two to two, indicating a common multiplier-induced edge displacement.
\end{figure}

Finally, we present simulation results related to the applications discussed in Remark \ref{remark_important} for the non-spiked model and in Remark \ref{rem_conditionalspike} for the spiked model, both concerning the use of the conditional bootstrap. In Figure \ref{fig_1}, we conduct numerical simulations to illustrate the accuracy and practical usefulness for some properly chosen multipliers. In particular, we use the multipliers constructed in \eqref{eq_xiconstruction}, to build confidence intervals for the edge $E_+$ in \eqref{eq_twresultoriginal} under the non-spiked model. Then, Figure \ref{fig_2} illustrates the accuracy and usefulness of the results discussed in Remark \ref{rem_conditionalspike} for constructing confidence intervals for the population spikes for the spiked model.

\begin{figure}[ht]
    \centering
    \includegraphics[width=0.48\textwidth]{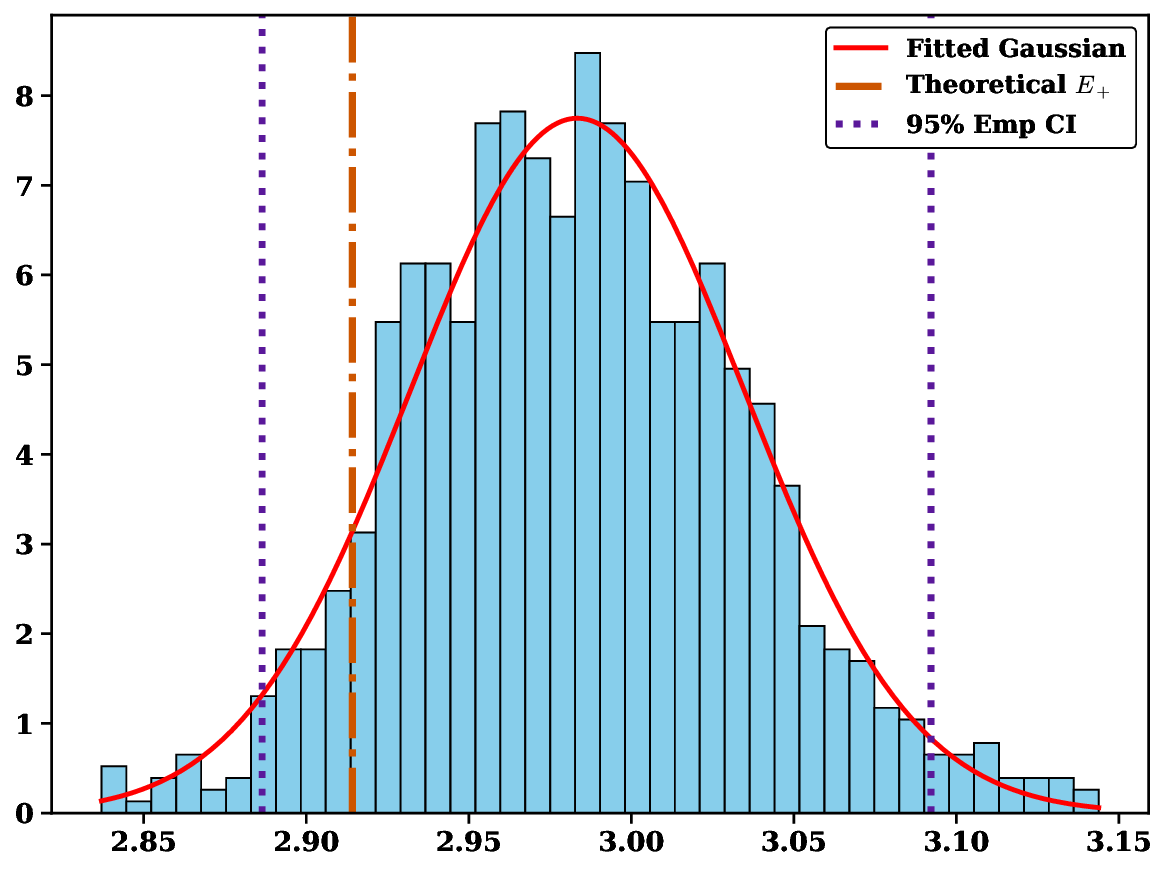}
    \includegraphics[width=0.48\textwidth]{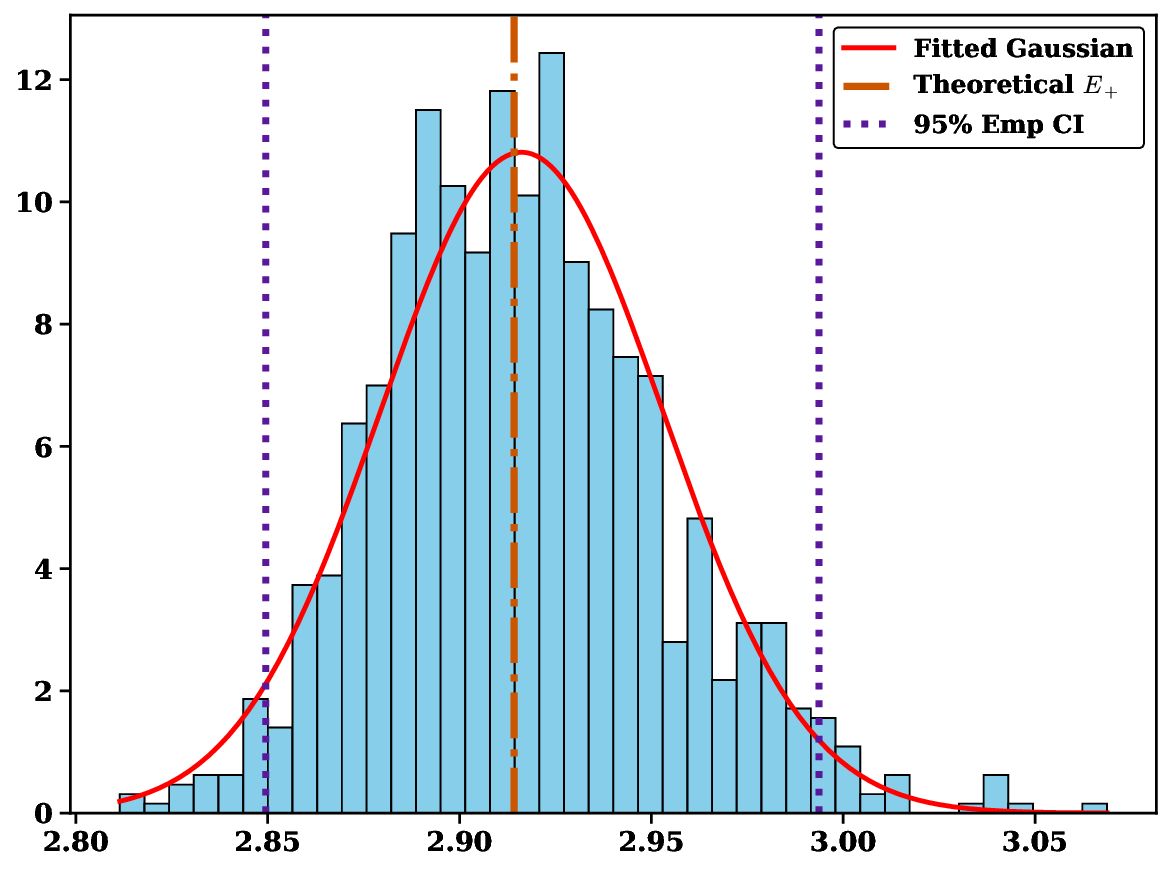}
    \caption{Construction of  confidence interval for $E_{+}$. We set $p=300$, $n=600$, $\Sigma=I$ and the entries of $X$ are i.i.d.\ $\mathcal{N}(0,n^{-1})$. The multiplier satisfies \eqref{eq_xiconstruction} with $T=50$ and the simulation is conducted under $d=0$ (left panel) and $d=1$ (right panel), respectively. In the figures, we plot the $95\%$ quantiles and the theoretical $E_{+}$. The results are based on 2,000 simulation repetitions. }\label{fig_1}
\end{figure}

\begin{figure}[!ht]
    \centering
    \includegraphics[width=0.48\textwidth]{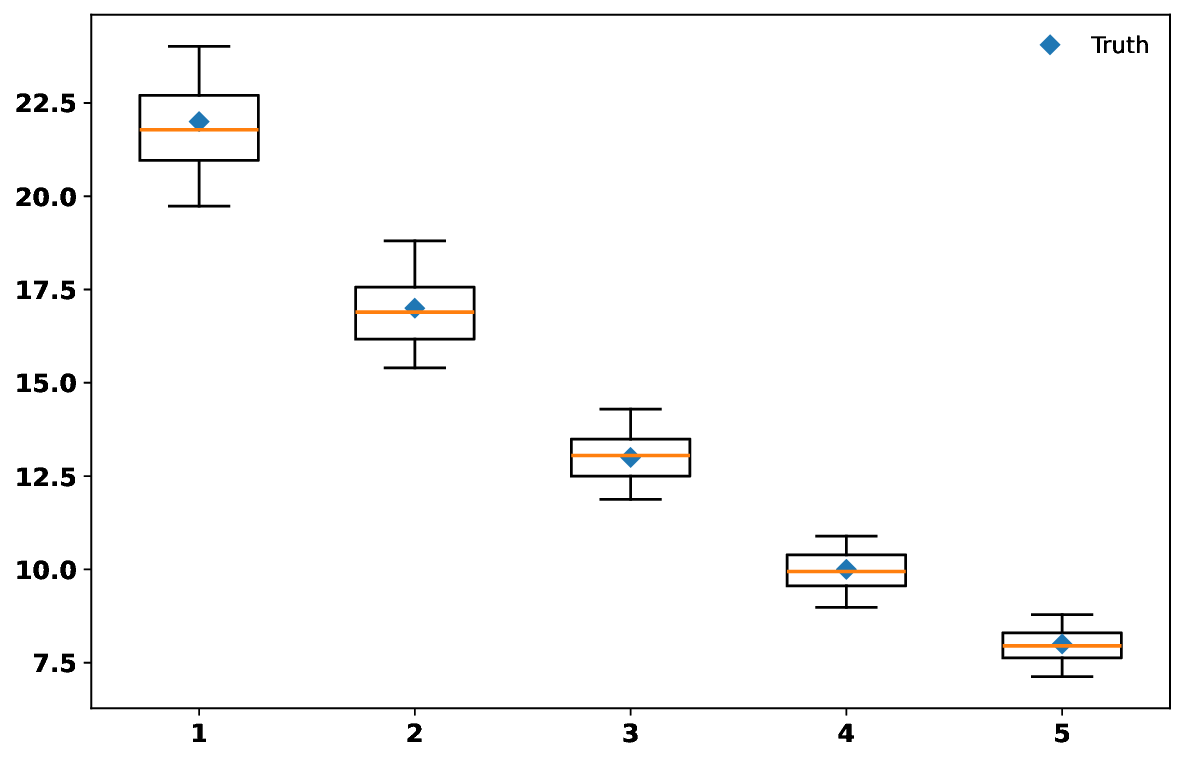}
    \includegraphics[width=0.48\textwidth]{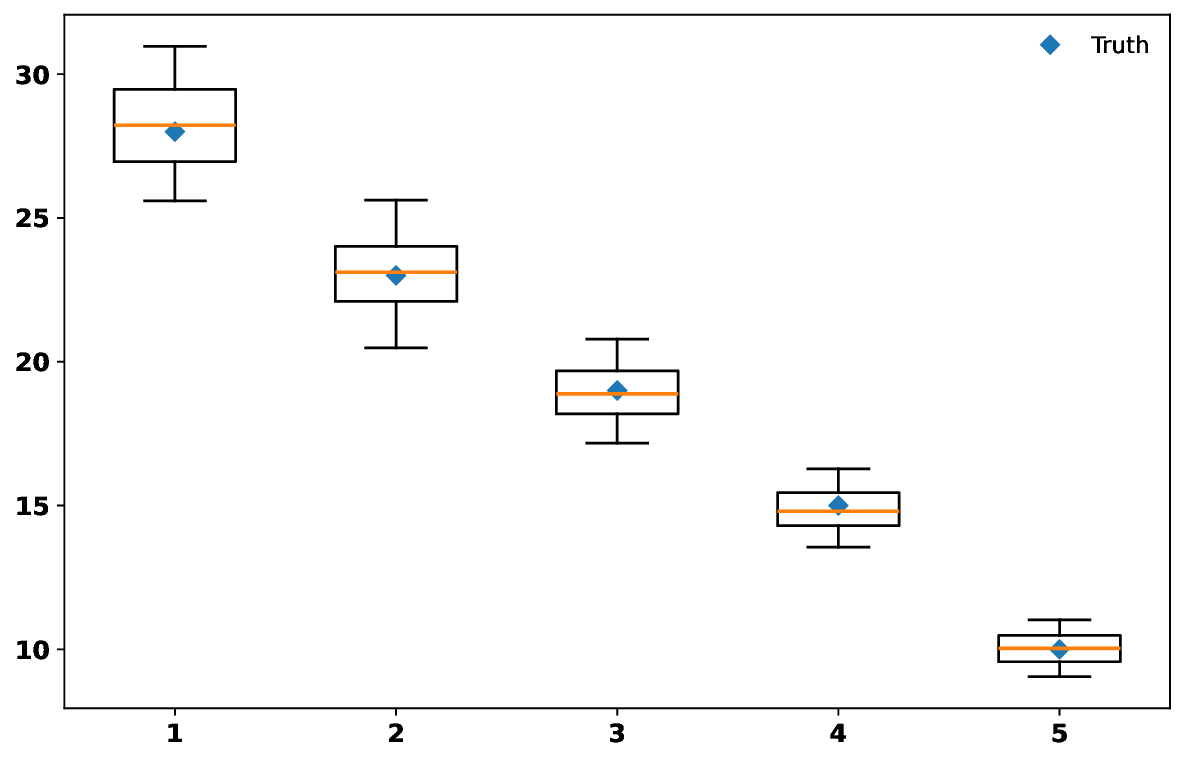}
    \caption{Construction of confidence intervals for the spikes. We set $p=300$, $n=600$ and the entries of $X$ are i.i.d. $\mathcal{N}(0,n^{-1})$. We consider two settings of population spikes and multiplies. In the first setting, we set $\widetilde{\Sigma}=\{22,17,13,10,8,1,\dots,1\}$ with unbounded multiplier $\xi^2\sim\exp(1)$ (left panel), while we use $\widetilde{\Sigma}=\{28,23,19,15,10,1,\dots,1\}$ with bounded multiplier $\xi^2\sim \operatorname{U}(4/3,28/3)$  in the second setting (right panel) to adapt \eqref{spiked_assumption}. Both of the multipliers are chosen to satisfy \eqref{eq_xicon}. The results are based on $2,000$ simulation repetitions.}\label{fig_2}
\end{figure}

%\begin{figure}[!ht]
%\subfigure{\includegraphics[width=4.5cm,height=5.1cm]{sigmaI_d0_bootstrap_lambda1.eps}}
%\hspace*{0.2cm}
%\subfigure{\includegraphics[width=4.5cm,height=5.1cm]{sigmaI_d1_bootstrap_lambda1.eps}}
%\caption{}
%\label{fig_app1}
%\end{figure}

%Third, Table \ref{} records the simulation results concerning Corollary \ref{coro_conditionalboostrap} and the applications discussed in Remark \ref{rmk_confidenceinterval}. For better illustration, we use Figure...
%
%
%. Moreover, for better illustration of the application for conditional bootstrap as discussed in Remark \ref{rem_conditionalspike}, in Figure... 

\section{Strategies of the proof}\label{sec_proofstrategy}
In this section, we present a concise overview of the proof strategies for the main results in Sections \ref{sec_boostrapeffect_nonspike} and \ref{sec_boostrapeffect_spike}. To precisely analyze the bootstrap effect on the fluctuation of the largest eigenvalue of the sample covariance matrix, we leverage and further develop techniques from random matrix theory to quantify the interaction between the randomness introduced by the multipliers and the observed data samples. At a high level, guided by concepts from free probability theory \cite{mingo2017free}, this interaction can be interpreted as the convolution between the multiplier matrix $D$ and the data matrix $X$. To clarify the technical details, we begin by introducing several fundamental notations and concepts from the random matrix theory literature. 

%in Section \ref{sec_resolvent&asymptoticlaws}, followed by an explanation of our proof strategies in Section \ref{sec_sketchofproof}.

%\subsection{Asymptotic laws}\label{sec_resolvent&asymptoticlaws}

Recall the bootstrapped sample covariance matrices $Q$ in (\ref{eq_samplecov}) with its $n \times n$ companion $\mathcal Q:=DX^*\Sigma XD$. The empirical spectral distributions (ESD) of $Q$ and $\mathcal{Q}$ are denoted as 
\begin{equation}\label{eq_measures}
\mu_{Q}:=\frac{1}{p}\sum_{i=1}^p \delta_{\lambda_i(Q)},\quad \mu_{\mathcal{Q}}:=\frac{1}{n}\sum_{j=1}^n \delta_{\lambda_j(\mathcal{Q})}.
\end{equation}
It is well-known that the ESDs can be best described via its the Stieltjes transforms as follows %{\color{red}[we may need to change this a a little bit]}
\begin{equation}\label{eq_mq}
m_{Q}:=\int\frac{1}{x-z}\mu_{Q},\quad m_{\mathcal{Q}}:=\int\frac{1}{x-z}\mu_{\mathcal{Q}}, \ z \in \mathbb{C}_+.
\end{equation}

%
%
% Define the Green functions for sample covariance matrices $Q,\mathcal{Q}$ as,

%Since $Q$ and $\mathcal{Q}$ share the same non-trivial eigenvalues, it suffices to study $\mu_Q$ and $m_Q.$ The limit of $\mu_Q$ can be described by a  system of equations \cite{couillet2014analysis, ding2021spiked,Karoui2009, hu2019central,paul2009no, Zhanggeneral}. To avoid repetitions, we summarize these equations in the following definition. 
%\begin{definition}[Systems of consistent equations]\label{defn_couplesystem} For $z \in \mathbb{C}_+,$ we define the triplets $(m_{1n}, m_{2n}, m_n) \in \mathbb{C}^3_+,$  via the following systems of equations.
%\begin{align}\label{eq_systemequationsm1m2}
%    m_{1n}(z)=\frac{1}{n}\sum_{i=1}^p & \frac{\sigma_i}{-z(1+\sigma_im_{2n}(z))},\quad m_{2n}(z)=\frac{1}{n}\sum_{i=1}^n\frac{\xi_i^2}{-z(1+\xi^2_im_{1n}(z))},\\
%   & m_{n}(z)=\frac{1}{p}\sum_{i=1}^p\frac{1}{-z(1+\sigma_im_{2n}(z))}. \nonumber
%\end{align}
%\end{definition}

\quad For sufficiently large $n,$ we find that $\mu_Q$ has a nonrandom deterministic equivalent and can be uniquely characterized by the following systems of equations 
\begin{align}\label{eq_syssyssys}
    m_{1n,c}(z)=\frac{1}{n}\sum_{i=1}^p & \frac{\sigma_i}{-z(1+\sigma_im_{2n,c}(z))},\quad m_{2n,c}(z)=\int \frac{s}{-z(1+s m_{1n,c}(z))} \mathrm{d} F(s), \nonumber\\
   & m_{n,c}(z)=\frac{1}{p}\sum_{i=1}^p\frac{1}{-z(1+\sigma_im_{2n,c}(z))}. 
\end{align} 
The result can be formally summarized by the following theorem. Its proof can be obtained by following lines of the arguments  of \cite[Theorem 2]{Karoui2009} and \cite[Theorem 1]{paul2009no} verbatim and we omit it.

\begin{theorem}\label{lem_solutionsystem1}
Suppose Assumptions \ref{assum_model}, \ref{assum_D} and \ref{assumption_techincial} hold. Then 
%conditional on some event $\Omega \equiv \Omega_n$ that $\mathbb{P}(\Omega)=1-\ro(1),$ 
for any $z \in \mathbb{C}_+,$ when $n$ is sufficiently large, there exists a unique solution $(m_{1n,c}(z), m_{2n,c}(z), m_{n,c}(z)) \in \mathbb{C}_+^3$ to the systems of equations in  (\ref{eq_syssyssys}). Moreover, $m_{n,c}(z)$ is the Stieltjes transform of some probability density function $\widetilde{\rho} \equiv \widetilde{\rho}_n$ defined on $\mathbb{R}$ which can be obtained using the inversion formula.
\end{theorem}

We point out that,  when $\xi^2$ has bounded support as in Case (ii) of Assumption \ref{assum_D}, we can actually obtain stronger results as in \cite{paul2009no} that the support of the associated probability density function  is bounded and denoted as in (\ref{edge_edge_edge}). 

%{\color{blue}
%\begin{remark}\label{rmk_systemequationsremark}
%
%
%
%
%%Third, the proofs of Theorem \ref{lem_solutionsystem} . We omit the details. 
%
%%{\color{red} remarks: 1. the probability event. and can be moved when it is bounded.  2. some key issues in the proof and literature. }
%\end{remark}
%}

%\subsection{Sketch of the proof strategies}\label{sec_sketchofproof}

With the above preparation, we now proceed to outline our proof strategies.  We start with the non-spiked model (Section \ref{sec_boostrapeffect_nonspike}). In this setting, the arguments differ for bounded and unbounded multipliers $D$. On the one hand, when $D$ has unbounded support (Theorems \ref{thm_main_unbounded} and \ref{thm_main_unbounded_conditional}), regardless of unconditional or conditional, the eigenvalues will be divergent. In this setting, we utilize a perturbation argument. However, as in this case, the associated $\widetilde{\rho}$ in Theorem \ref{lem_solutionsystem1} may also be unbounded, the perturbation approach developed in \cite{bloemendal2016principal,ding2021spiked,knowles2013isotropic} cannot be applied directly. Instead, we modify the perturbation arguments by isolating $\mathbf{y}_i$ corresponding to the largest multiplier $\xi_{(1)}^2$ from the bootstrapped matrix $Y$ as in (\ref{eq_samplecov}). More explicitly, for the proof of the unconditional bootstrap in Theorem \ref{thm_main_unbounded}, with $m_{1n}(z)$ in (\ref{eq_systemequationsm1m2}), the key is to introduce a real auxiliary quantity $\vartheta_1>0$ to be the largest solution of 
\begin{equation}\label{eq: def of vartheta_1}
    1+(\xi^2_{(1)}+d_1)m_{1n}(\vartheta_{1})=0,
\end{equation}
where $d_1=\ro_{\mathbb{P}}(\xi_{(1)}^2)$ (cf. (\ref{eq_firstddefinition}) of our supplement) is introduced for some technical reasons. First, as will be seen in our proofs (cf. (\ref{eq:F(m,z)}) and (\ref{eq_mu1part}) of our supplement), (\ref{eq: def of vartheta_1}) provides a natural way to connect $\vartheta_1$ and $\xi_{(1)}^2$ in the sense that $\vartheta_1/\xi_{(1)}^2=\varphi+\ro_{\mathbb{P}}(1).$ Secondly, it establishes a connection between $\vartheta_1$ and $\lambda_1$ through a modified perturbation argument based on \cite{bloemendal2016principal, ding2021spiked11, ding2021spiked}. Specifically, $\lambda_1$ can be uniquely characterized by the equation $M(\lambda_1) = 0$ (cf. (\ref{eq_masterequation}) of our supplement), where $M(\cdot)$ (cf. (\ref{eq_defnmlambda}) of our supplement) is a random quantity isolating the column in (\ref{eq_datamatrix}) associated with $\xi_{(1)}^2$. Using our newly established local laws (cf. Theorem \ref{thm_unboundedcaselocallaw} of our supplement), we can further demonstrate that $1 + (\xi_{(1)}^2 + d_1)m_{1n}(\lambda_1) \approx 0$. Subsequently, a detailed continuity and stability analysis shows that $\lambda_1 / \vartheta_1 = 1 + \ro_{\mathbb{P}}(1)$, thereby completing the proof of Theorem \ref{thm_main_unbounded}. For the conditional bootstrap result in Theorem \ref{thm_main_unbounded_conditional}, the proof follows similar strategies to those used in the unconditional case in Theorem \ref{thm_main_unbounded}. The main difference is that the argument is carried out on a high-probability event $\Omega_X$ determined solely by $X$ (cf. Definition \ref{def_OmegaX} in the supplement). Once we restrict to this event, the analysis proceeds in essentially the same way as in the unconditional setting.

On the other hand, when the multiplier $D$ has bounded support (Theorems \ref{thm_main_bounded} and \ref{thm_main_bounded_conditional}), our arguments are non-perturbative and extend the approach introduced in \cite{Kwak2021,lee2016extremal}. This generalization is non-trivial and requires a dedicated analysis of the relationship between multipliers and the largest eigenvalues. We start with the discussions for the unconditional bootstrap. In this case, we need a sophisticated understanding of the systems of equations, such as those in (\ref{eq_systemequationsm1m2}), on local scales. A key input is the distinct local behaviors of the asymptotic law (cf. $\widetilde{\rho}(x)$ in (\ref{edge_edge_edge})) near the edge under varying settings (cf. Lemma \ref{localestimate2} of our supplement). More precisely, we find that  $\widetilde{\rho}(x) \sim \sqrt{L_+ - x}$ under the setup of (2) and (3) of Theorem \ref{thm_main_bounded}, and $\widetilde{\rho}(x) \sim (L_+ - x)^d$ under (1) of Theorem \ref{thm_main_bounded}. In the actual proof for Theorem \ref{thm_main_bounded}, we decompose $\lambda_1$ into two components: $\lambda_1 - \widehat{L}_{+}$ and $\widehat{L}_{+} - L_+$, where the quantity $\widehat{L}_+$ is defined as the edge of the conditional density function $\rho$ in Theorem \ref{lem_solutionsystem} by fixing a realization of the multipliers $\{\xi_i^2 \}$ on some high-probability event {$\Omega_D$} (cf. Definition \ref{defn_probset} of the supplement). For cases (2) and (3) of Theorem \ref{thm_main_bounded}, where the square root behavior emerges, $n^{2/3}(\lambda_1 - \widehat{L}_+)$ asymptotically follows the Tracy–Widom (TW) law with constant-order variance. Regarding the unconditional distribution, the fluctuation of $\widehat{L}_+$, due to the i.i.d. assumption of $D^2$, is asymptotically Gaussian by the Central Limit Theorem (CLT), with a variance that may decay to zero. Consequently, the overall distribution can be expressed as a sum of the TW law and a Gaussian component (potentially with vanishing variance). For case (1) of Theorem \ref{thm_main_bounded}, the key observation is that $\widehat{L}_+$  can be represented as the solution of %the equation that (cf. (\ref{eq_conditionaledgedefinition}))
\begin{equation}\label{eq_masterequationtwo}
m_{1n}(\widehat{L}_+)=-l^{-1}. 
\end{equation}
Moreover, for $z \in \mathbb{C}_+$ in a small neighborhood of $\widehat{L}_+,$ $m_{1n}(z)$ can be expanded linearly as $m_{1n}(z)-m_{1n}(\widehat{L}_+)=\beta (z-\widehat{L}_+)+\ro(n^{-1/(d+1)})$ for some constant $\beta$ (cf. Lemma \ref{localestimate2} of our supplement). Then, it allows us to locate $\lambda_1$ neighboring around $\widehat{L}_{+}$ as
\begin{gather}\label{eq: def of gamma}
    \operatorname{Re}m_{1n}(\lambda_1+\ri\eta_0)\approx-\xi^{-2}_{(1)}, \quad \eta_0=n^{-1/2-\epsilon_{\mathsf{d}}},
\end{gather}
where $\epsilon_{\mathsf{d}}$ is introduced in  \eqref{eq_spectraldomaintwo} of the supplement. Connecting \eqref{eq: def of gamma} with $l-\xi^2_{(1)}$ and \eqref{eq_masterequationtwo}, we can establish a representation for $\widehat{L}_{+}$. Moreover, it follows from Lemma \ref{localestimate2} of our supplement that $L_+=\widehat{L}_++\rO_{\mathbb{P}}(n^{-1/2+\delta}),$ for some small constant $\delta>0$. Since $d>1$, we can conclude that $\lambda_1$ is only influenced by $\xi_{(1)}^2$ and asymptotically Weibull.

For the conditional bootstrap result in Theorem \ref{thm_main_bounded_conditional}, the proofs of statements (1) and (2) follow from ideas similar to those used for their unconditional counterparts in parts (1) and (2) of Theorem \ref{thm_main_bounded}, except that the argument is carried out on the high-probability event $\Omega_X$ determined by $X$. More specifically, on this event, the local behavior of $\widetilde{\rho}$ remains unchanged. Consequently, the asymptotic distribution of $\lambda_1$ is still governed by $\xi^2_{(1)}$ and by the average of ${\xi_i^2}$, respectively, exactly as in the unconditional setting.

The main difference arises in the proof of part (3). When $\mathsf{v}$ (equivalently $\operatorname{Var}(\xi^2)$) remains large, we employ the same decomposition as before, $\lambda_1-\widehat{L}_+ + \widehat{L}_+ - L_+ . $ The first term fluctuates at order $\mathrm{O}(n^{-2/3})$, whereas the second term, conditional on $X$, is asymptotically Gaussian by the central limit theorem. Hence, whenever $\sqrt{\mathsf{v}/n} \gg n^{-2/3}$, the Gaussian fluctuation dominates, leading to asymptotic normality. This agrees with the corresponding result for the unconditional bootstrap.

However, the conditional bootstrap behaves very differently from the unconditional bootstrap when $\operatorname{Var}(\xi^2)$ vanishes too rapidly. To capture this regime, we further decompose $\lambda_1-L_{+}$ into three terms: $\widehat{\lambda}_1-E_{+}$, $\lambda_1-\widehat{\lambda}_1$, and $E_{+}-L_{+}$, where $\widehat{\lambda}_1$ denotes the largest eigenvalue of the original sample covariance matrix $S$, and $E_{+}$ is the right edge of its limiting spectral distribution. Conditional on $X$, the leading term $\widehat{\lambda}_1-E_{+}$ fluctuates on the Tracy--Widom scale, which is at least of order $n^{-2/3-\epsilon}$ for some small $\epsilon>0$. By contrast, when $\mathsf{v}\ll n^{-2/3-\epsilon}$, the remaining terms $\lambda_1-\widehat{\lambda}_1$ and $E_{+}-L_{+}$ are of smaller order than $\widehat{\lambda}_1-E_{+}$. Consequently, the fluctuation of $\lambda_1-L_{+}$ is dominated by $\widehat{\lambda}_1-E_{+}$ at the scale of $n^{-2/3}$. As a result, after normalization by the usual Tracy--Widom scaling, the conditional distribution of $\lambda_1-L_{+}$ collapses to a point mass given $X$.
 
%  For this part, we consider the decomposition of $\lambda_1-L_{+}$ into the two terms $\lambda_1-\widehat{L}_{+}$ and $\widehat{L}_{+}-L_{+}$. Conditional on $X$, the fluctuation $\lambda_1-\widehat{L}_{+}$ has deterministic scale at most $n^{-2/3}$, whereas $\widehat{L}_{+}-L_{+}$ i
% This proves statement (3) of Theorem \ref{thm_main_bounded_conditional}.

We then discuss the proof of the spiked model (Section \ref{sec_boostrapeffect_spike}). Given the spiked structure in (\ref{eq_truemodelspiked}), our proof primarily relies on a perturbative argument tailored for divergent spikes, as developed in \cite{CHP}. However, the inclusion of multipliers introduces additional complexity, requiring a  generalization of the above perturbative techniques. We start with the unconditional bootstrap in Theorem \ref{thm_main_spike_unconditional}. A key element of our proof is the introduction of a random term, $\pi_1$ (cf. \eqref{eq_def_zeta} of our supplement), which captures the randomness of $\mu_1/\theta_1$ solely through the randomness of $\{\xi_i\}$ with $\theta_1$ being some intermediate term. The asymptotic Gaussianity arises from the asymptotic normality of $\pi_1$, which can be established via a standard central limit theorem (CLT) argument. To quantify the bias term  in (\ref{eq_biasedmean_spike1}), we carefully analyze $\pi_1 - \theta_1 / \widetilde{\sigma}_1$. This requires a more refined analysis of the local scales of the systems defined in Definition \ref{defn_couplesystem}, leveraging our newly established local laws in Theorems \ref{thm_unboundedcaselocallaw} and \ref{thm_boundedcaselocallaw} of our supplement.

 For the conditional bootstrap result in Theorem \ref{thm_main_spike_conditional}, by restricting to the high-probability event $\Omega_X$ determined by $X$, we obtain the same limiting representation for $\pi_i-\theta_i/\widetilde{\sigma}_i$ as in the unconditional case. However, as in part (1) of the result, when conditional on $X$, an additional deterministic term $\mathbf{v}_i^*XX^*\mathbf{v}_i-1$ arises. As a consequence, compared with the unconditional setting, the first statement contains an additional bias term in the asymptotic mean of $\mu_i/\widetilde{\sigma}_i$. To remove this bias, in part (2) of the result,  we instead consider the statistic $\mu_i/\widehat{\mu}_i$. Combining this with the limiting representation of $\widehat{\mu}_i$ as in the proof of Theorem \ref{thm_sample}, we find that the unobservable quantity $\mathbf{v}_i^*XX^*\mathbf{v}_i-1$ cancels out, and the asymptotic mean takes a form closer to that in Theorem \ref{thm_sample}.

\begin{acks}[Acknowledgments]
The author is particularly grateful to Xiucai Ding, Long Yu and Wang Zhou for their valuable guidance, careful suggestions and helpful feedback throughout the development of this work. The author also thanks Zhigang Bao, Miles Lopes and Fan Yang for many helpful discussions.
\end{acks}

\clearpage
\appendix
\setcounter{figure}{0}
\renewcommand{\thefigure}{\Alph{figure}}

\phantomsection
\addcontentsline{toc}{section}{Supplementary Material}
\begin{center}
{\large\bfseries Supplement to ``Multiplier Bootstrap and Edge Phase Transitions of High-Dimensional Covariance Matrices''}
\end{center}
\medskip

\input{supplement_body_arxiv}

\bibliographystyle{imsart-number}
\bibliography{references_arxiv}

\end{document}

%% file: supplement_body_arxiv.tex
In this supplement, we provide the proofs for the main results and some auxiliary lemmas. 
%{As indicated in Section \ref{sec_proofstrategy} that our actual proof relies on two technical inputs. One is the detailed analysis of the Stieltjes transforms of the limiting ESD on local scales in some carefully chosen spectral domains. The other one is a finer control of the randomness of the quantities associated with ESD.
In Section \ref{appendix_preliminary}, we provide some preliminary results and some technique results including averaged local laws.  In Section \ref{appendxi_locallawproof}, we prove the averaged local laws near the edges. In Section \ref{sec_locationofeigenvalues}, we provide the asymptotic locations of the edge eigenvalues and prove the main results. Finally, building on the above preliminary results, the proofs of the main results in Sections 3 and 4 of the main paper are given in Sections \ref{sec_proof_nonspike} and \ref{sec_proof_spike}, while some auxiliary lemmas are collected in Section \ref{appendix_last}.

\section{Some preliminary results}\label{appendix_preliminary}

In this section, we introduce some preliminary results. First, in Section \ref{sec_append_notations}, we summarize some notations and definitions. Second, in Sections \ref{sec_asymptoticlocalaveragedlocal} and \ref{secaveragelocalllca}, we provide the properties of the asymptotic local laws and establish the averaged local laws. Third, in Section \ref{appendix_goodconfiguration}, we examine the properties of the entries of $D^2$ and construct some probability events to which our arguments will be restricted. Then, in Section \ref{sec_summaryofextremevalue}, we provide some useful lemmas and a short review of the extreme value theory. Finally, in Section \ref{sec_characterization_OmegaX}, we construct some high-probability events in terms of $X$, which will be used in the study of the conditional bootstrap.

\subsection{Notations and definitions}\label{sec_append_notations}

For technical convenience, in addition to the limiting system of equations in (3.9), we introduce the following random discretized system of equations.
\begin{definition}[Systems of consistent equations]\label{defn_couplesystem} For $z \in \mathbb{C}_+,$ we define the triplets $(m_{1n}, m_{2n}, m_n) \in \mathbb{C}^3_+,$  via the following systems of equations.
\begin{align}\label{eq_systemequationsm1m2}
    m_{1n}(z)=\frac{1}{n}\sum_{i=1}^p & \frac{\sigma_i}{-z(1+\sigma_im_{2n}(z))},\quad m_{2n}(z)=\frac{1}{n}\sum_{i=1}^n\frac{\xi_i^2}{-z(1+\xi^2_im_{1n}(z))},\\
   & m_{n}(z)=\frac{1}{p}\sum_{i=1}^p\frac{1}{-z(1+\sigma_im_{2n}(z))}. \nonumber
\end{align}
\end{definition}

\quad Note that in general the quantities in (\ref{eq_systemequationsm1m2}) are random. The following theorem is a counterpart of Theorem 6.1 whose proof can be also obtained by following lines of the arguments  of \cite[Theorem 2]{Karoui2009} and \cite[Theorem 1]{paul2009no} verbatim.

\begin{theorem}[Asymptotic laws]\label{lem_solutionsystem}
Suppose Assumptions 2.1, 2.2 and 2.4 hold. Then conditional on some event {$\Omega_D \equiv \Omega_{n,D}$ that $\mathbb{P}(\Omega_D)=1-\ro(1),$} for any $z \in \mathbb{C}_+,$ when $n$ is sufficiently large, there exists a unique solution $(m_{1n}(z), m_{2n}(z), m_{n}(z)) \in \mathbb{C}_+^3$ to the systems of equations in  (\ref{eq_systemequationsm1m2}). Moreover, $m_n(z)$ is the Stieltjes transform of some probability density function $\rho \equiv \rho_n$ defined on $\mathbb{R}$ which can be obtained using the inversion formula.
\end{theorem}

\begin{remark}
Several remarks on Theorem \ref{lem_solutionsystem} are in order. First, the probability event {$\Omega_D$} can be constructed explicitly as in Definition \ref{defn_probset} and Lemma \ref{lem_probabilitycontrol}. Second, Theorem 6.1 is  the counterpart for Theorem \ref{lem_solutionsystem} by integrating out the randomness of multiplier $\xi^2.$ 
%Recall $F(x)$ is the CDF of $\xi^2.$ We can define the counterpart for (\ref{eq_systemequationsm1m2}) as follows 
\end{remark}

%We start with introducing the properties of the asymptotic local laws as in Definition \ref{defn_couplesystem}. 
Recall the definitions of the Stieltjes transforms of ESDs in 6.2. In practice, it is convenient to define the Green functions of $Q$ and $\mathcal{Q}$ for $z=E+\ri \eta \in \mathbb{C}_+,$ 
\begin{equation}\label{eq_resolvents}
G(z)=(Q-zI)^{-1} \in \mathbb{R}^{p \times p},\quad \mathcal{G}(z)=(\mathcal{Q}-zI)^{-1} \in \mathbb{R}^{n \times n}.
\end{equation}
Then (6.2) can be rewritten as 
 $   m_Q=\frac{1}{p}\operatorname{tr}G(z),\ m_{\mathcal{Q}}=\frac{1}{n}\operatorname{tr}\mathcal{G}(z). $

Before we proceed ahead, we revisit the definition of the systems of equations in Definition \ref{defn_couplesystem}. Thanks to Theorem \ref{lem_solutionsystem}, it is easy to see that the study of the systems of equations in (\ref{eq_systemequationsm1m2}) can be reduced to the analysis of 
\begin{equation}\label{eq_functionFequal}
F_n(m_{1n}(z),z)=0, \quad z\in\mathbb{C}_{+},
\end{equation}
where $F_n(\cdot, \cdot)$ is defined as follows
\begin{equation}\label{eq:F(m,z)}
    F_n(m_{1n}(z),z)
    =\frac{1}{n}\sum_{i=1}^p\frac{\sigma_i}{-z+\frac{\sigma_i}{n}\sum_{j=1}^n\frac{\xi^2_j}{1+\xi^2_j m_{1n}(z)}}-m_{1n}(z).
\end{equation}
The conditional and unconditional version in \eqref{eq_systemequationsm1m2}, (3.9) are both useful in their own aspects. To be more specific, the conditional version is more powerful when $\xi^2$ has unbounded support whereas the unconditional version is more convenient when $\xi^2$ has bounded support.

%Moreover, to avoid the singularity in the definitions of the systems of equations, we introduce some additional assumption on $\Sigma$ which will be used when the multiplier $\xi^2$ has bounded support in the sense of (ii) of Assumption \ref{assum_D}. Such an assumption has been frequently used in the random matrix theory literature, for example, see \cite{Bao2015,Ding&Yang2018,ding2021spiked,9779233,el2007tracy,knowles2017anisotropic,lee2016tracy}. Recall the notations $m_{2n,c}$ and $L_+$ in Remark \ref{rmk_systemequationsremark}. 

%\begin{assumption}\label{assum_additional_techinical} When (ii) of Assumption \ref{assum_D} holds, for $\Sigma$ satisfying Assumption \ref{assumption_techincial}, we assume that for some constant $\tau>0$ 
%\begin{equation*}
%\min_{1 \leq i \leq p} |1+\sigma_i m_{2n,c}(L_+)| \geq \tau. 
%\end{equation*}  
%\end{assumption}

We now define the sets of spectral parameters as follows.  For $\xi^2$ with unbounded support, as in Case (i) of Assumption (2.2), {let $\vartheta_1$ be defined by (\ref{eq: def of vartheta_1}), and define $d_1$ by}
\begin{equation}\label{eq_firstddefinition}
d_1:=
\begin{cases}
n^{1/\alpha}\log^{-\epsilon_{\mathsf{p}}}n, & \text{if (\ref{ass3.1}) holds}; \\
\log^{1/\beta-1-\epsilon_{\mathsf{e}}}n, & \text{if (\ref{ass3.2}) holds},
\end{cases}
\end{equation} 
for some sufficiently small constants $\epsilon_{\mathsf{p}}>c_{\mathsf{p},0}$ and $\epsilon_{\mathsf{e}}>c_{\mathsf{e},0}$ where $c_{\mathsf{p},0},c_{\mathsf{e},0}$ are some constants introduced in Definition \ref{defn_probset} below.
We denote for sufficiently large constant $\mathtt C>0$ that
\begin{equation}\label{eq_spectraldomainone}
\mathbf{D}_{u} \equiv \mathbf{D}_{u}(\mathtt C):=\left\{z=E+\ri\eta \in \mathbb{C}_+: |E-\vartheta_1| \leq \mathtt C d_{1}, \ n^{-2/3}\leq\eta\leq \mathtt C \vartheta_1\right\}.
\end{equation}
For $\xi^2$ with bounded support as in Case (ii) of Assumption \ref{assum_D}, for some sufficiently small constants $\mathtt{c}>0$ and $0<\epsilon_{\mathsf{d}}<(1/2-1/(d+1))/\mathtt{C}_d,$ with $\mathtt{C}_d>0$ being a large constant, we denote (recall $L_+$ in (\ref{edge_edge_edge}))
\begin{equation}\label{eq_spectraldomaintwo}
\mathbf{D}_{b} \equiv \mathbf{D}_{b}(\mathtt c):=\left\{z=E+\ri\eta \in \mathbb{C}_+: L_{+}-\mathtt c \leq E\leq L_{+}+\mathtt c, \ n^{-1/2-\epsilon_{\mathsf{d}}}\leq\eta\leq n^{-1/(d+1)+\epsilon_{\mathsf{d}}} \right\}.
\end{equation}

Throughout the paper, we will frequently use the minors of a matrix.  For the data matrix $Y$ in (\ref{eq_datamatrix}), denote the index set ${\mathcal I}=\{1,\dots,n\}$. Given an index set $\mathcal{T}\subset{\mathcal I}$, we introduce
	the notation $Y^{(\mathcal{T})}$ to denote the $p \times(n-|\mathcal{T}|)$ minor of $Y$
	obtained from removing all the $i$th columns of $Y$ for $i\in \mathcal{T}$ and keep the original indices of $Y$. In particular, $Y^{(\emptyset)}=Y$. For convenience, we briefly write $(\{i\})$, $(\{i,j\})$ and {$(\{i,j\}\cup \mathcal{T})$} as $(i)$, $(i,j)$
	and $(ij\mathcal{T})$ respectively. Correspondingly, we denote their sample covariance matrices and resolvents as 
	\begin{equation}\label{eq_defnminor}
	Q^{(\mathcal{T})}=(Y^{(\mathcal{T})}) (Y^{(\mathcal{T})})^*,\ {\mathcal Q}^{(\mathcal{T})}=(Y^{(\mathcal{T})})^* (Y^{(\mathcal{T})}), 
	\end{equation}
	\begin{equation}\label{eq_defnminorG}
	G^{(\mathcal{T})}(z)=(Q^{(\mathcal{T})}-zI)^{-1}, \quad{\mathcal G}^{(\mathcal{T})}(z)=({\mathcal Q}^{(\mathcal{T})}-zI)^{-1}.
	\end{equation}
Similar to (\ref{eq_mq}) and Definition \ref{defn_couplesystem}, we can define 	
$m_Q^{(\mathcal{T})}(z)$, $m_{\mathcal{Q}}^{(\mathcal{T})}(z),$ $m_{1n}^{(\mathcal{T})}(z)$, $m_{2n}^{(\mathcal{T})}(z)$ and  $m_n^{(\mathcal{T})}(z)$ by removing $\mathbf{y}_i, i \in \mathcal{T}$ or $\xi_i^2, i \in \mathcal{T}.$

\subsection{Properties of asymptotic laws}\label{sec_asymptoticlocalaveragedlocal}
\quad We begin with the summary of the results when $\xi^2$ has unbounded support as in Case (i) of Assumption \ref{assum_D}. The proofs will be deferred to Section \ref{sec_proofpreliminarystieltjes}. Conditional on some probability event, we provide some useful deterministic estimates for $m_{1n}, m_{2n}$ and $m_n(z)$ on the above concerned spectral domain (\ref{eq_spectraldomainone}).  Denote the control parameter $e$ as follows
\begin{equation}\label{e2_definition}
e:=
\begin{cases}
\frac{\log^{\epsilon_{\mathsf{p}}} n}{n^{1/\alpha}}, & \ \text{if (\ref{ass3.1}) holds}; \\
\frac{1}{\log^{1/\beta} n}, & \ \text{if (\ref{ass3.2}) holds}.  
\end{cases}
\end{equation} 
\begin{lemma}\label{lem: basic bounds}
Suppose Assumptions \ref{assum_model}, \ref{assumption_techincial} and (i) of Assumption \ref{assum_D} hold. For any fixed realization $\{\xi_i^2\} \in \Omega_D$ where $\Omega_D \equiv \Omega_{n,D}$ is some  probability event that $\mathbb{P}(\Omega_D)=1-\ro(1)$, we have  
\begin{enumerate}
\item For $z \in \mathbf{D}_{u},$ we have that for some constants $C_1, C_2>0$ 
\begin{equation*}
\operatorname{Re} m_{1n}(z) \asymp -E^{-1}, \ C_1 \eta E^{-2} \leq \operatorname{Im} m_{1n}(z) \leq C_2 \eta E^{-1}.  
\end{equation*}
\item When $|E-\vartheta_1| \leq \mathtt{C} d_1$ for  some sufficiently large constant $\mathtt{C}>0,$ let $m_{1n}(E)=\lim_{\eta \downarrow 0} m_{1n}(E+\ri \eta),$ then we have that 
\begin{equation*}
    m_{1n}(E) \asymp -E^{-1}.
\end{equation*}
\item For $z \in \mathbf{D}_{u}$ and $e$ defined in (\ref{e2_definition}), we have that 
\begin{gather*}
    |m_{2n}(z)|=\rO(e),\quad |m_n(z)|=\rO(E^{-1}),\\
    \operatorname{Im}m_{2n}(z)=\rO(\eta E^{-1}),\quad \operatorname{Im}m_n(z)=\rO(\eta E^{-2}).
\end{gather*}
\end{enumerate} 
%{\color{red}[NEED TO SEPERATE THE CASE WHEN IT IS EXPONENTIAL DECAY]}
\end{lemma}

\begin{remark}\label{rem_keyrem}
The above lemma provides some controls for the Stieltjes transforms.  First, the construction of the probability event $\Omega_D$ will be given in Section \ref{appendix_goodconfiguration}. Second, by a discussion similar to (\ref{eq_mu1part}), conditional on $\Omega_D$, we can replace $\mu_1$ with $\varphi \xi^2_{(1)}.$  Third, The above results hold when we replace $m_{1n}, m_{2n}$ and $m_n$ with $m_{1n}^{(\mathcal{T})}, m_{2n}^{\mathcal{(T)}}$ and $m_n^{\mathcal{(T)}}$ for any finite $\mathcal{T}.$ Fourth, Lemma \ref{lem: basic bounds} also implies the existence of $\vartheta_1$ defined in \eqref{eq_defnvarphi}.
\end{remark}

\quad Then we state the results when $\xi^2$ has bounded support as in Case (ii) of Assumption \ref{assum_D}. 
%As mentioned in Remark \ref{rmk_systemequationsremark}, 
For the bounded support case, it will be more convenient to use both the conditional and unconditional systems. For the conditional setting, when restricted to $\Omega_D,$ we denote the rightmost edge of $\rho$ as $\widehat{L}_+.$ Moreover, parallel to (\ref{eq_phasetransition}), we introduce the following quantities
\begin{align}\label{eq_finitesample123}
 &  \widehat{\mathsf{s}}_1:=
   \frac{1}{n}\sum_{j=1}^n\frac{l^2\xi^4_j}{(l-\xi^2_j)^2},
   \quad 
   \widehat{\mathsf{s}}_2:=
   \frac{1}{n}\sum_{j=1}^n\frac{l\xi^2_j}{l-\xi^2_j},\quad 
   \widehat{\mathsf{s}}_3:=
   \frac{1}{p}\sum_{i=1}^p \frac{\sigma_i^2\widehat{\mathsf{s}}_1}{(\widehat{L}_+-\sigma_i\widehat{\mathsf{s}}_2)^2},\\
 &  {\widehat{\mathsf{s}}_4:=
   \frac{1}{n}\sum_{i=1}^p \frac{\sigma_i}{(\widehat{L}_+-\sigma_i \widehat{\mathsf{s}}_2)^2}.}
\end{align}
Furthermore, we need the following spectral parameter set 
\begin{equation}\label{eq_spectralparameterprime}
\mathbf{D}_b^\prime=\Big\{z\in\mathbf{D}_b:|1+\xi^2_jm_{1n,c}(z)|>\frac{1}{2}n^{-1/(d+1)-\epsilon_{\mathsf{d}}},\; \text{for all} \ 2 \leq j \leq n\Big\}.
\end{equation}

We point out that for $\eta_0$ defined in (\ref{eq: def of gamma}), as will be seen from (\ref{eq_secondconnect}) and (\ref{def4}) that, with probability $1-\ro_{\mathbb{P}}(1),$ $\lambda_1+\ri \eta_0 \in \mathbf{D}_b'.$ Next, the properties of the asymptotic laws in Theorem \ref{lem_solutionsystem} can be summarized as follows.

\begin{lemma}\label{localestimate2}
Suppose Assumptions \ref{assum_model}, \ref{assumption_techincial}, \ref{assum_additional_techinical} and (ii) of Assumption \ref{assum_D} hold. Then for any fixed realization $\{\xi_i^2\} \in \Omega_D$ where $\Omega_D \equiv \Omega_{n,D}$ is some  probability event that $\mathbb{P}(\Omega_D)=1-\ro(1),$  for sufficiently large $n,$ we have that 
\begin{enumerate}
\item[(a).] If $d>1$ and $\phi^{-1}>\widehat{\mathsf{s}}_3,$ $\widehat{L}_+$ can be expressed explicitly by the following equation 
\begin{equation}\label{eq_conditionaledgedefinition}
    1=\frac{1}{n}\sum_i\frac{-l\sigma_i}{(-\widehat{L}_{+}+\sigma_i\widehat{\mathsf{s}}_2)}.
\end{equation}
Moreover, for any $0\leq\kappa\leq \widehat{L}_{+}$,
\begin{gather}\label{eq: concave decay of rho_Q}
    \rho(\widehat{L}_{+}-\kappa) \asymp \kappa^d,
\end{gather}
Moreover, for some sufficiently small constant $\epsilon>0$
  \begin{equation}\label{eq_closenessequation}
 \mathsf{s}_k=\widehat{\mathsf{s}}_k+\rO(n^{-1/2+\epsilon_{\mathsf{b}}}), k=1,2,3,4;\ L_+=\widehat{L}_++\rO(n^{-1/2+\epsilon_{\mathsf{b}}}),   
\end{equation}
where $\epsilon_{\mathsf{b}}$ is defined in Definition \ref{defn_probset} below.
%{\color{red} something is missing here, how the assumption from unconditional can be transferred to conditional setting.}
 In addition, let $z=\widehat{L}_{+}-\kappa+\ri\eta\in\mathbf{D}_b$, then 
 \begin{gather}\label{eq_expansionlinear}
     m_{1n}(\widehat{L}_{+})-m_{1n}(z)=\frac{\widehat{\mathsf{s}}_4}{(1-\phi\widehat{\mathsf{s}}_3)}\left(\widehat{L}_{+}-z\right)+\rO((\log n)(\kappa+\eta)^{\min\{d,2\}}).
 \end{gather}
 Similarly, for any $z,z^{\prime}\in\mathbf{D}_b$, we have
\begin{gather}\label{eq_a11coro}
    m_{1n}(z)-m_{1n}(z^{\prime})=\frac{\widehat{\mathsf{s}}_4}{(1-\phi \widehat{\mathsf{s}}_3)}(z-z^{\prime})+\rO((\log n)(n^{-1/(d+1)})^{\min\{d-1,1\}}|z-z^{\prime}|).
\end{gather}
Finally, for $z \in \mathbf{D}_b^{\prime}$ in (\ref{eq_spectralparameterprime}), we have that 
\begin{equation}\label{eq_zopointrate11}
\operatorname{Im} m_{1n} (z) =\rO\left( \max \left\{ \eta, \frac{1}{n \eta}\right\} \right), \ \ \operatorname{Im} m_{n} (z) =\rO\left( \max \left\{ \eta, \frac{1}{n \eta}\right\} \right).
\end{equation}
Moreover, for $z_0=E_0+\mathrm{i}\eta_0$ defined in (\ref{eq: def of E0}), we have that 
\begin{equation}\label{eq_zopointrate}
\operatorname{Im} m_{1n}(z_0) \asymp n^{-1/2}, \ \operatorname{Im} m_n(z_0) \asymp n^{-1/2}.
\end{equation}
and for $z=E+\ri \eta_0 \in \mathbf{D}_b^\prime$ in (\ref{eq_spectralparameterprime}), we have that 
\begin{equation}\label{eq_oneregimeedgecontrol}
\operatorname{Im} m_{1n}(z) \asymp \eta_0, \ \operatorname{Im} m_n(z) \asymp \eta_0,
\end{equation}
if $|z-z_0| \geq C n^{-1/2+3\epsilon_{\mathsf{d}}}$ for some constant $C>0.$ 
%{\color{red}[add other results on the control of the Stieltjes transforms.]}
\item[(b).] If $-1<d \leq 1$ or $d>1$ and $\phi^{-1}<\widehat{\mathsf{s}}_3,$  we have that for some fixed constant $\tau>0$
\begin{equation}\label{eq_boundedfrombelowimportant}
\left|1+\xi_{i}^2 m_{1n}(\widehat{L}_+) \right| \geq \tau, \ \ 1 \leq i \leq n.
\end{equation}
Moreover, we have that  for any $0\leq\kappa\leq \widehat{L}_{+}$,
\begin{gather}\label{thm_main_squared_root_bounded}
    \rho(\widehat{L}_{+}-\kappa) \asymp \kappa^{1/2}.
\end{gather}
Equivalently, for some constant $\gamma>0,$ we have for $\kappa \downarrow 0$ 
\begin{equation}\label{eq_gammadefinition}
\rho(\widehat{L}_+-\kappa)=\frac{1}{\pi} \gamma^{3/2} \sqrt{\kappa}+\rO(\kappa). 
\end{equation}

\end{enumerate}

Finally,  the results of (a) and (b) still hold unconditionally when $m_{1n}$ is replaced by $m_{1n,c},$ $\rho$ is replaced by $\widetilde{\rho}$ and $\widehat{L}_+$ is replaced by $L_+$ as in (\ref{edge_edge_edge})  where $\widehat{\mathsf{s}}_3$ and $\widehat{\mathsf{s}}_4$ should be replaced by $\mathsf{s}_3$ and $\mathsf{s}_4$ as in (\ref{eq_phasetransition}). 

\end{lemma}

\begin{remark}\label{rmk_boundedatleastonesolution}
Using discussions similar to the paragraphs around equation (5.1) of \cite{Kwak2021}, by continuity of $m_{1n}(z)$, together with (\ref{eq_expansionlinear}), (\ref{def4}), and the fact $m_{1n}(\widehat{L}_+)=-l^{-1},$ equation (\ref{eq: def of gamma}) admits at least one solution. 
\end{remark}

\subsection{Averaged local laws}\label{secaveragelocalllca}

\quad In this section, we provide the results of the averaged local laws. 
Throughout the paper, we will consistently use the notion of \emph{stochastic domination} to systematize the statements of the form ``$\xi$ is bounded by $\zeta$ with high probability up to a small power of $n$."
\begin{definition}[Stochastic domination]

(i) Let
\[\xi=\left(\xi^{(n)}(u):n\in\mathbb{N}, u\in U^{(n)}\right),\hskip 10pt \zeta=\left(\zeta^{(n)}(u):n\in\mathbb{N}, u\in U^{(n)}\right),\]
be two families of nonnegative random variables, where $U^{(n)}$ is a possibly $n$-dependent parameter set. We say $\xi$ is stochastically dominated by $\zeta$, uniformly in $u$, if for any fixed (small) $\epsilon>0$ and (large) $D>0$, 
\[\sup_{u\in U^{(n)}}\mathbb{P}\left(\xi^{(n)}(u)>n^\epsilon\zeta^{(n)}(u)\right)\leq n^{-D}\]
for large enough $n \geq n_0(\epsilon, D)$, and we shall use the notation $\xi\prec\zeta$. Throughout this paper, the stochastic domination will always be uniform in all parameters that are not explicitly fixed, such as the matrix indices and the spectral parameter $z$.  If for some complex family $\xi$ we have $|\xi|\prec\zeta$, then we will also write $\xi \prec \zeta$ or $\xi=\mathrm{O}_\prec(\zeta)$.
%\item[(ii)] 

\vspace{5pt}

\noindent (ii) We say an event $\Xi$ holds with high probability if for any constant $D>0$, $\mathbb P(\Xi)\geq 1- n^{-D}$ for large enough $n$.
\end{definition} 
\quad Similar to  \cite{couillet2014analysis,ding2021spiked, paul2009no, yang2019edge}, instead of working directly with $m_Q$ and $m_{\mathcal{Q}}$ in (\ref{eq_mq}), it is more convenient to study the following quantities
\begin{equation}\label{eq_m1m2}
m_1(z)=\frac{1}{n}{\rm tr}\left(G(z)\Sigma\right) ,
\quad  m_2(z)=
\frac{1}{n}\sum_{i=1}^n\xi^2_i\mathcal G_{ii}(z).
\end{equation}
Analogously, using the minors in (\ref{eq_defnminor}), we can define $m_1^{(\mathcal{T})}(z)$ and $m_2^{(\mathcal{T})}(z).$ In what follows, we summarize the averaged local laws that play a central role in the proofs, both in the unconditional setting and in the conditional setting on $X$. In each setting, we treat separately the cases where the multiplier $\xi^2$ has unbounded support and bounded support. The next two results, Theorems \ref{thm_unboundedcaselocallaw} and \ref{thm_boundedcaselocallaw}, concern the local laws in the unconditional setting.

%{\color{green} explain here}

\begin{theorem}[Averaged local laws for unbounded support $\xi^2:$ unconditional on $X$] \label{thm_unboundedcaselocallaw} Suppose Assumptions \ref{assum_model}, \ref{assumption_techincial} and (i) of Assumption \ref{assum_D} hold. For any fixed realization $\{\xi_i^2\} \in \Omega_D$ where $\Omega_D \equiv \Omega_{n,D}$ is introduced in Lemma \ref{lem: basic bounds}, let $m_1^{(1)}(z)$ and $m_{1n}^{(1)} (z)$ be defined by removing the column or entries associated with $\xi_{(1)}^2.$ We have that the following statements hold true uniformly for $z \in \mathbf{D}_{u}$ in (\ref{eq_spectraldomainone})
\begin{enumerate}
\item If {\normalfont Case (a)} of (i) of  Assumption \ref{assum_D} holds, we have that 
\begin{equation*}
%m_{1}(z)=m_{1n}(z)+\rO_{\prec}\left(n^{-1/2-2/\alpha}\right), \ 
m_{1}^{(1)}(z)=m_{1n}^{(1)}(z)+\rO_{\prec}\left(n^{-1/2-2/\alpha}\right). 
\end{equation*}
\item If {\normalfont Case (b)} of (i) of Assumption \ref{assum_D} holds, we have that 
\begin{equation*}
%m_{1}(z)=m_{1n}(z)+\rO_{\prec}\left(n^{-1/2}\right), \ 
m_{1}^{(1)}(z)=m_{1n}^{(1)}(z)+\rO_{\prec}\left(n^{-1/2}\right).
\end{equation*}
\end{enumerate}  
\end{theorem}

\begin{theorem}[Averaged local laws for bounded support $\xi^2:$ unconditional on $X$] \label{thm_boundedcaselocallaw} Suppose Assumptions \ref{assum_model}, \ref{assumption_techincial}, \ref{assum_additional_techinical} and (ii) of Assumption \ref{assum_D} hold. When $d>1$ and $\phi^{-1}>\mathsf{s}_3,$ for any fixed realization $\{\xi_i^2\} \in \Omega_D$ where $\Omega_D \equiv \Omega_{n,D}$ is introduced in Lemma \ref{localestimate2}, for $\eta_0=n^{-1/2-\epsilon_{\mathsf{d}}}$ defined in (\ref{eq: def of gamma}),  we have that the following statements hold true uniformly for $z \in \mathbf{D}_b^\prime$ in (\ref{eq_spectralparameterprime}) 
\begin{equation*}
\left| m_{1n}(z)-m_{1n,c}(z) \right| \leq n^{-1/2+\epsilon_{\mathsf{d}}}, \ \left| m_1(z)-m_{1n}(z) \right|=\rO_{\prec} \left((n \eta_0)^{-1} \right),
\end{equation*}
and 
\begin{equation*}
\left| m_{n}(z)-m_{n,c}(z) \right| \leq n^{-1/2+\epsilon_{\mathsf{d}}} , \ \left| m_Q(z)-m_{n}(z) \right|=\rO_{\prec} \left((n \eta_0)^{-1} \right).
\end{equation*}
\end{theorem}

\quad Next we present the averaged local laws conditional on the data $X$ in Theorems \ref{thm_averagedlocallaw_unboundedmultiplier_conditionalonX} and \ref{thm_averagedlocallaw_boundedmultiplier_conditionalonX}. 
More specifically, we restrict our analysis to a probability event $\Omega_X$, 
depending only on $X$, which will be introduced in 
Section \ref{sec_characterization_OmegaX} after the necessary notation is 
established. This event satisfies $\mathbb{P}(\Omega_X)=1-\mathrm{o}(1)$.

\begin{theorem}[Averaged local laws for unbounded multiplier $\xi^2:$ conditional on $X$]\label{thm_averagedlocallaw_unboundedmultiplier_conditionalonX}
%    Suppose Assumption \ref{assumption_techincial} and $(i)$ of Assumption \ref{assum_D} hold. Under some high probability event $\Omega_X$ which is specified in Definition \ref{} below, for any fixed realization $\{\xi_i^2\} \in \Omega_D$, we have the followings hold true for all $z\in\mathbf{D}_u$ in \eqref{eq_spectraldomainone}.
Suppose the assumptions of Theorem \ref{thm_unboundedcaselocallaw} hold. 
Then, restricted on a probability event $\Omega_X$ defined in 
Definition \ref{def_OmegaX} below, which satisfies $\mathbb{P}(\Omega_X)=1-\mathrm{o}(1)$, we have
    \begin{itemize}
        \item [1.] If Case $(a)$ of $(i)$ of Assumption \ref{assum_D} holds, we have that 
        \begin{align*}
            m_{1}^{(1)}(z)=m_{1n}^{(1)}(z)+\mathrm{O}(n^{-1/2-2/\alpha}n^{\epsilon_1}),
        \end{align*}
        for any small constant $\epsilon_1>0$.
        \item [2.] If Case $(b)$ of $(i)$ of Assumption \ref{assum_D} holds, we have that 
        \begin{align*}
            m_{1}^{(1)}(z)=m_{1n}^{(1)}(z)+\mathrm{O}(n^{-1/2}n^{\epsilon_2}),
        \end{align*}
        for any small constant $\epsilon_2>0$.
    \end{itemize}
\end{theorem}

\begin{theorem}[Averaged local laws for bounded multiplier $\xi^2$: conditional on $X$]\label{thm_averagedlocallaw_boundedmultiplier_conditionalonX}
%    Suppose Assumption \ref{assumption_techincial} and $(ii)$ of Assumption \ref{assum_D} hold. Under some high probability event $\Omega_X$ which is specified in Definition \ref{} below, for any fixed realization $\{\xi_i^2\} \in \Omega_D$, when $d>1$ and $\phi^{-1}>\mathsf{s}_3$, for $\eta_0=n^{-1/2-\epsilon_{\mathsf{d}}}$ defined in \eqref{eq: def of gamma}, we have that the followings hold true for all $z\in\mathbf{D}_b^{\prime}$ in \eqref{eq_spectralparameterprime}
Suppose the assumptions of Theorem \ref{thm_boundedcaselocallaw} hold. 
Then, restricted on a probability event $\Omega_X$ defined in 
Definition \ref{def_OmegaX} below, which satisfies $\mathbb{P}(\Omega_X)=1-\mathrm{o}(1)$, we have
    \begin{align*}
        \big|m_{1n}(z)-m_{1n,c}(z)\big|\leq n^{-1/2+\epsilon_{\mathsf{d}}},\,\big|m_{1}(z)-m_{1n}(z)\big|=\mathrm{O}((n\eta_0)^{-1}n^{\epsilon_3}),
    \end{align*}
    and
    \begin{align*}
        \big|m_{n}(z)-m_{n,c}(z)\big|\leq n^{-1/2+\epsilon_{\mathsf{d}}},\,\big|m_{Q}(z)-m_{n}(z)\big|=\mathrm{O}((n\eta_0)^{-1}n^{\epsilon_4}),
    \end{align*}
    for any small constants $\epsilon_3,\epsilon_4>0$.
\end{theorem}

Theorem \ref{thm_averagedlocallaw_unboundedmultiplier_conditionalonX} and 
Theorem \ref{thm_averagedlocallaw_boundedmultiplier_conditionalonX} yield results 
analogous to those in Theorem \ref{thm_unboundedcaselocallaw} and 
Theorem \ref{thm_boundedcaselocallaw}, except that the control errors hold 
deterministically rather than in the stochastic domination sense. 
Indeed, when conditioning on a realization from the high-probability event 
$\Omega_X$, the sources of randomness from $X$ in $m_{1}(z)$, $m_{2}(z)$, 
and $m_{Q}(z)$ become fully deterministic. The construction of the event 
$\Omega_X$ will be given in Definition \ref{def_OmegaX}.
 %  The proofs of these theorems will be put into Section \ref{appendxi_locallawproof}.

\subsection{Characterization of $\Omega_D:$ good configurations for multipliers}\label{appendix_goodconfiguration}

In this subsection, independent of Sections \ref{sec_asymptoticlocalaveragedlocal} and \ref{secaveragelocalllca}, we define some probability events which are some "good configurations" for the first few largest eigenvalues of $D^2.$ Our proofs will be restricted on these probability events. In fact, as will be seen in Lemma \ref{lem_probabilitycontrol}, under Assumption \ref{assum_D},  these probability events hold with high probability when $n$ is sufficiently large. 

Recall Assumption \ref{assum_D} and $D^2=\operatorname{diag}\left\{ \xi_1^2, \cdots, \xi_n^2 \right\}.$
Moreover, we define the order statistics of $\{\xi^2_i\}$ as $\xi_{(1)}^2 \geq \xi_{(2)}^2 \geq \cdots \geq \xi_{(n)}^2. $ In what follows, we define these probability events according to the various assumptions of $\{\xi_i^2\}$ in (\ref{ass3.1})--(\ref{ass3.4}).

In what follows, we characterize the multiplier event $\Omega_D$. 
To distinguish among the different multiplier assumptions in Assumption \ref{assum_D}, 
we decompose $\Omega_D$ into three events, denoted by $\Omega_{\mathsf{p}}$, 
$\Omega_{\mathsf{e}}$, and $\Omega_{\mathsf{b}}$, corresponding respectively to the cases 
of polynomially decaying tails in (\ref{ass3.1}), exponentially decaying tails in 
(\ref{ass3.2}), and bounded support in (\ref{ass3.4}). For several carefully chosen 
constants introduced below, we will likewise use the subscripts $\mathsf{p}$, 
$\mathsf{e}$, and $\mathsf{b}$ to indicate their dependence on the corresponding 
regime in Assumption \ref{assum_D}.

%{\color{green} explain! $\mathsf{p}, \mathsf{e}, \mathsf{b}$}

\begin{definition}\label{defn_probset}
Denote {$\Omega_D \equiv \Omega_{n,D}$, $\Omega_D=\Omega_{\mathsf{p}}\cup\Omega_{\mathsf{e}}\cup\Omega_{\mathsf{b}}$} be the event on $\{\xi_i^2\}$ so that the following conditions hold:
\begin{itemize}
\item[] {\bf (a). Unbounded support with polynomial decay.}  When $\{\xi^2_i\}$ has unbounded support with polynomial decay tail as in (\ref{ass3.1}), we assume that for all $\epsilon_{\mathsf{p},1}\in(0,1/\alpha)$, $b_{\mathsf{p}}\in(1/2,1]$  and some constants $C_{\mathsf{p},0}>0,C_{\mathsf{p},1}>1,C_{\mathsf{p},2}>0$, $C_{\mathsf{p},3}>C_{\mathsf{p},1}$ and $c_{\mathsf{p},0}>1,$ the following holds on  $\Omega_{\mathsf{p}}$
\begin{equation}\label{def1}
\begin{aligned}
&  \xi^2_{(1)}-\xi^2_{(2)}\geq C_{\mathsf{p},0}^{-1}n^{1/\alpha}\log^{-c_{\mathsf{p},0}}n, \\
& C_{\mathsf{p},1}^{-1}n^{1/\alpha}\log^{-1}n\leq\xi^2_{(1)}\leq C_{\mathsf{p},1}n^{1/\alpha}\log n,  \\
&  \xi^2_{(i)}-\xi^2_{(i+1)}\geq C_{\mathsf{p},2}^{-1}n^{\epsilon_{\mathsf{p},1}}\log^{-1}n, \ 1 \leq i \leq \lfloor \log n\rfloor,\\
&  \xi^2_{(1)}-\xi^2_{(\lceil n^{b_{\mathsf{p}}} \rceil)}\geq C_{\mathsf{p},3}^{-1}n^{1/\alpha}\log^{-1} n, \\
&  \frac{1}{n}\sum_{i=1}^n\xi^2_i< \infty. 
%C\log^{1/\alpha} n.
\end{aligned}
\end{equation}
\item[] {\bf (b). Unbounded support with exponential decay.} When $\{\xi^2_i\}$ has unbounded support with  exponential decay tail as in (\ref{ass3.2}), we assume that for some constant $C_{\mathsf{e},0}>0$, $C_{\mathsf{e},1}>\max\{(\max(t,t^{-1}))^{1/\beta},1\}$, $C_{\mathsf{e},2}>C_{\mathsf{e},1}$ , $c_{\mathsf{e},0}>0$ and $0<c_{\mathsf{e},1}<\min\{t(C_{\mathsf{e},1}^{-1}-C_{\mathsf{e},2}^{-1})^{\beta},1\},$ 
%and sequence of constants $\mathsf c_n$ that $\mathsf c_n=\ro (1),$ 
the following holds on $\Omega_{\mathsf{e}}$
\begin{equation}\label{def3}
\begin{aligned}
&  \xi^2_{(1)}-\xi^2_{(2)}\geq C_{\mathsf{e},0}^{-1}\log^{1/\beta-1-c_{\mathsf{e},0}} n,\\
&  C_{\mathsf{e},1}^{-1}\log^{1/\beta} n\leq\xi^2_{(1)}\leq C_{\mathsf{e},1}\log^{1/\beta} n, \\
& \xi^2_{(1)}-\xi^2_{(\lceil n^{1-c_{\mathsf{e},1}}\rceil)}\geq C_{\mathsf{e},2}^{-1}\log^{1/\beta}n,\\
&  \frac{1}{n}\sum_{i=1}^n\xi^2_i<\infty. 
%\mathsf c_{n} \log n.
\end{aligned}
\end{equation}
\item[] {\bf (c). Bounded support with $d>1.$}  When $\{\xi^2_i\}$ has bounded support satisfying (\ref{ass3.4}) with $d>1,$ we assume that for some sufficiently small constants $0<\epsilon_{\mathsf{b}},\epsilon_{\mathsf{d}}<(1/2-1/(d+1))/\mathtt{C}_d$ ($\mathtt{C}_d>0$ is a large constant), $0<b \leq 1$, $0<c_{\mathsf{b},0}, c_{\mathsf{b},1}<1$ and $0<C_{\mathsf{b},0}<l,$ the following holds on $\Omega_{\mathsf{b}}$
\begin{equation}\label{def4}
\begin{aligned}
& n^{-1/(d+1)-\epsilon_{\mathsf{d}}}<l-\xi^2_{(1)}<n^{-1/(d+1)}\log n, \\
& \xi^2_{(1)}-\xi^2_{(2)}> n^{-1/(d+1)-\epsilon_{\mathsf{d}}}, \\
& l-\xi^2_{(\lfloor bn \rfloor)}>C_{\mathsf{b},0}, \\
& \frac{1}{n}\sum_{i=1}^n\xi^2_{i}\leq l, \\
&  \left|\frac{1}{n}\sum_{i=1}^n\frac{\xi^2_i}{1+\xi^2_{i}m_{1 n,c}(z)}-\int\frac{t}{1+tm_{1 n,c}(z)}\mathrm{d} F(t)\right|\le\frac{Cn^{\epsilon_{\mathsf{b}}}}{\sqrt{n}}, \ \text{for} \  z \in \mathbf{D}_b,\\
&\left|\sum_{i=1}^n(\xi_i^2-\mathbb{E}\xi^2)\right|\leq C\sqrt{(n\log^{c_{\mathsf{b},0}}n)\operatorname{Var}\xi^2 }, \\
&{\frac{1}{n}\sum_{i=1}^n(\xi_i^2-\mathbb{E}\xi^2)^2}\leq C (\log^{c_{\mathrm{b},1}}n)\operatorname{Var}\xi^2.
\end{aligned}
\end{equation}
where we recall that  $F(t)$ is the distribution of $\xi^2$ and $C>0$ is some generic constant. 
\end{itemize}
\end{definition}

\begin{remark}
Three remarks are in order. First, on the event $\Omega_D,$ for the unbounded support case, according to (a) and (b), we see that the first few largest $\xi_i^2$ are divergent and well separated from each other. Second, for the bounded support case, we only provide the results for $d>1$ in (c). Nevertheless, it is easy to see that similar results can be obtained for $-1<d \leq 1.$ Third, $\Omega_D$ lies in the sigma-algebra generated by the randomness in $D$, 
that is, $\Omega_D \in \mathcal{F}_D = \sigma(D)$. Last, since only one regime in Assumption \ref{assum_D} is active at a time, $\Omega_D$ should be understood as the corresponding good event for that regime; we write it as $\Omega_{\mathsf{p}}\cup\Omega_{\mathsf{e}}\cup\Omega_{\mathsf{b}}$ for notational convenience.
\end{remark}

\quad The following lemma shows that under Assumption \ref{assum_D}, the event $\Omega_D$ occurs with high probability in all the four settings. The proof will be given in Section \ref{sec_appendxi_goodevent}. 

\begin{lemma}\label{lem_probabilitycontrol}
Let $\Omega_D$ be the events defined in Definition \ref{defn_probset}, suppose Assumption \ref{assum_D} holds, we then have that when $n$ is sufficiently large  
\begin{equation*}
\mathbb{P}(\Omega_D)=1-\rO(\log^{-D}n),
\end{equation*}
for some constant $D>0$.
\end{lemma}

\begin{remark}\label{rmk_probability}
Essentially, for each multiplier regime in Assumption \ref{assum_D}, we have 
$\mathbb{P}(\Omega_{*})=1-\mathrm{o}(1)$, where $*\in\{\mathsf{p},\mathsf{e},\mathsf{b}\}$ 
is specified in Definition \ref{defn_probset}. Precise probabilistic error bounds 
for each case are provided in the proof of Lemma \ref{lem_probabilitycontrol} in Section \ref{sec_appendxi_goodevent}. Although the 
probability of each event $\Omega_{*}$ can be controlled separately, for simplicity 
of presentation and to streamline the proofs, throughout the paper we work on the 
combined event 
\[
\Omega_D=\Omega_{\mathsf{p}}\cup\Omega_{\mathsf{e}}\cup\Omega_{\mathsf{b}},
\]
which is required to hold with probability tending to one. Consequently, in 
Lemma \ref{lem_probabilitycontrol}, it suffices to choose a constant $D>0$ such 
that the events $\Omega_{\mathsf{p}}$, $\Omega_{\mathsf{e}}$, and $\Omega_{\mathsf{b}}$ 
hold uniformly.
\end{remark}

\subsection{Some auxiliary lemmas}\label{sec_summaryofextremevalue}
In this subsection, we first provide some technical lemmas which will be used in our proof. The following resolvent identities play an important role in our proof. Recall the resolvents defined in (\ref{eq_resolvents}) and the minors defined in (\ref{eq_defnminor}). 
\begin{lemma}[Firs resolvent identities]\label{lem: first resolvent}
For any $z_1,z_2\in\mathbb{C}_{+}$, we have
\begin{align*}
    G(z_1)-G(z_2)=(z_1-z_2)G(z_1)G(z_2),\quad \mathcal{G}(z_1)-\mathcal{G}(z_2)=(z_1-z_2)\mathcal{G}(z_1)\mathcal{G}(z_2).
\end{align*}
\end{lemma}
\begin{lemma}[Second resolvent identities]\label{lem: Resolvent}
Let $\{\mathbf{y}_i\} \subset \mathbb{R}^p$ be the columns of $Y$ as in (\ref{eq_datamatrix}), then we have
\begin{eqnarray*}
			\mathcal G_{ii}(z) & = & -\frac{1}{z+z\mathbf{y}_{i}^{*}{ G}^{(i)}(z)\mathbf{y}_{i}}; \ \
			\mathcal G_{ij}(z)  =  z \mathcal{G}_{ii}(z)\mathcal{G}_{jj}^{(i)}(z)\mathbf{y}_{i}^{*}{G}^{(ij)}(z)\mathbf{y}_{j}\quad i\ne j,\\
			\mathcal G_{ij}(z) & = & \mathcal G_{ij}^{(k)}(z)+\frac{\mathcal G_{ik}(z)\mathcal G_{kj}(z)}{\mathcal G_{kk}(z)}\quad i,j\ne k.
		\end{eqnarray*}
%The results also hold for $G^{\mathcal{T}}$.
\end{lemma}

\begin{proof}
The proof is straightforward using Schur's complement formula; for example see \cite[Lemma 2.3]{PillaiandYin2014}. 
\end{proof}

\begin{lemma}[Matrix identities] For any finite subset $\mathcal{T}\subset\{1,\dots,n\}$, we have that 
\begin{equation}\label{lem:Wald}
\left\|G^{(\mathcal{T})}\Sigma^{1/2} \right \|^2_F=\eta^{-1}\operatorname{Im}\operatorname{Tr} \left(G^{(\mathcal{T})} \Sigma \right). 
\end{equation}
Moreover, we have that 
	\begin{gather}
			\left|{\rm Tr}(\mathcal{G}^{(i)}-\mathcal{G}) \right|  \leq  \eta^{-1}, \nonumber \\
			\left|{\rm Tr}({ G}^{(i)}-{ G}) \right| \leq  |z|^{-1}+\eta^{-1}, \label{lem:trace_difference} \\
			\left|\operatorname{Im}{\rm Tr}({G}^{(i)}-{\mathcal G})\right|  \leq  \eta|z|^{-2}+\eta^{-1}. \nonumber
		\end{gather}
\end{lemma}
\begin{proof}
Due to similarity, we focus our discussion on the separable covariance i.i.d. data, i.e., Case (2) of Assumption \ref{assum_model}. In fact, it is easier to handle Case (1) since $\Sigma$ can  be always assumed to be  diagonal. 

We start with the proof of (\ref{lem:Wald}). Recall (\ref{eq_defnminorG}). We can write
\begin{equation*}
G^{(\mathcal{T})}=\left( \Sigma^{1/2} X^{(\mathcal{T})} D^2 X^{(\mathcal{T})} \Sigma^{1/2}-z   \right)^{-1}.
\end{equation*}
Let the spectral decomposition $\Sigma=U\Lambda U^*.$ Observe that 
\begin{align*}
\| & G^{(\mathcal{T})} \Sigma^{1/2} \|_F^2  =\operatorname{Tr} \left( \left( \Sigma^{1/2} X^{(\mathcal{T})} D^2 X^{(\mathcal{T})} \Sigma^{1/2}-z   \right)^{-1} U\Lambda U^* \left( \Sigma^{1/2} X^{(\mathcal{T})} D^2 X^{(\mathcal{T})} \Sigma^{1/2}-\bar{z}   \right)^{-1}\right) \\
&=\operatorname{Tr} \left( U \left( \Lambda^{1/2} U^* X^{(\mathcal{T})} D^2 X^{(\mathcal{T})} U \Lambda^{1/2}-z   \right)^{-1} U^* U\Lambda U^* U \left( \Lambda^{1/2} U^* X^{(\mathcal{T})} D^2 X^{(\mathcal{T})} U \Lambda^{1/2}-\bar{z}   \right)^{-1} U^*\right) \\
%&= \left\| \left( \Lambda^{1/2} U^* X^{(\mathcal{T})} D^2 X^{(\mathcal{T})} U \Lambda^{1/2}-z   \right)^{-1} \Lambda^{1/2} \right\|_F^2 \\
&=\eta^{-1} \operatorname{Im} \operatorname{Tr} \left[ \left( \Lambda^{1/2} U^* X^{(\mathcal{T})} D^2 X^{(\mathcal{T})} U \Lambda^{1/2}-z   \right)^{-1} \Lambda \right] , \\
&=\eta^{-1} \operatorname{Im} \operatorname{Tr} \left[ U^* U\left( \Lambda^{1/2} U^* X^{(\mathcal{T})} D^2 X^{(\mathcal{T})} U \Lambda^{1/2}-z   \right)^{-1} U^* U \Lambda \right]=\eta^{-1} \operatorname{Im} \operatorname{Tr} \left( G^{(\mathcal{T})} \Sigma \right),
\end{align*}
where in the fourth step we used Ward's identity (see the equation below (4.42) of \cite{ding2023local}).

Second, the proof of (\ref{lem:trace_difference}) follows from the definitions of the resolvents; see \cite[Lemma A.4]{Ding&Yang2018} and the proof of \cite[Lemma 4.6]{Bao2015} for more detail.   
 
\end{proof}

%{\color{red} [revise here]}
%{\color{red} ALSO need to list the i.i.d. case}
\begin{lemma}[Large deviation bounds]\label{lem:large deviation}
Let $\mathbf{u}=(u_1, u_2, \cdots, u_p)^*, \widetilde{\mathbf{u}}=(\widetilde{u}_1, \widetilde{u}_2, \cdots, \widetilde{u}_p)^* \in \mathbb{R}^p$  be two real independent random vectors. Moreover, let $A$ be a $p \times p$ matrix independent of the above vectors. Suppose  the entries of the random vectors are centered i.i.d. random variables with variance $n^{-1}$ and  $\mathbb{E}|\sqrt{n} v_{i}|^k \leq C_k,$ where $v_i=u_i, \widetilde{u}_i, 1 \leq i \leq p,$  then we have 
		\begin{align*}
		|\widetilde{\mathbf{u}}^{*}\mathbf{u}|  \prec  \sqrt{\frac{\|\mathbf{u}\|^{2}}{n},} \ \ \left|\mathbf{u}^{*}A\tilde{\mathbf{u}}\right| \prec  \frac{1}{n}\|A\|_{F}, \ \
		|\mathbf{u}^{*}A\mathbf{u}-\frac{1}{n}{\rm Tr}A|  \prec  \frac{1}{n}\|A\|_{F}. 		
		\end{align*}
\end{lemma}
\begin{proof}
The proof can be found in Lemma 3.4 of \cite{PillaiandYin2014} or Lemma 5.6 of \cite{yang2019edge}. 
\end{proof}

\quad In what follows, we provide a mini-review of the extreme value theory for a sequence of i.i.d. random variables following \cite{hansen2020three}. For more systematic treatments, we refer to the monographs \cite{beirlant2004statistics,coles2001introduction,fang1990,resnick2008extreme}.  

\begin{lemma}[Fisher-Tippett-Gnedenko Theorem]\label{lem_summaryevt} Let $\{x_i^2\}$ be a sequence of i.i.d. random variables and denote $M_n:=x_{(1)}^2$ as the largest order statistic. 
\begin{enumerate}
\item If there exist some constants $\alpha_n>0$ and $\beta_n \in \mathbb{R}$ and some non-degenerate {\normalfont cdf} G such that $\alpha_n^{-1}(M_n-\beta_n)$ converges in distribution to G, then $G$ belongs to the type of one of the following three {\normalfont cdfs:}
\begin{equation*}
\begin{split}
&\text{Gumbel}: \ G_0(x)=\exp(-e^{-x}), \ x \in \mathbb{R}, \\
&\text{Fr{\'e}chet}: \ G_{1, \alpha}(x)=\exp(-x^{-\alpha}), \ x \geq 0, \ \alpha>0, \\
& \text{Weibull}: \ G_{2, \alpha}(x)=\exp(-|x|^{\alpha}), \ x \leq 0, \ \alpha>0.
\end{split}
\end{equation*}
\item Recall (\ref{eq_defnbn}). First, if $\{x_i\}$ satisfies (\ref{ass3.1}), we have that 
\begin{equation*}
\frac{M_n}{b_n} \overset{d}{\Rightarrow}  G_{1, \alpha},
\end{equation*}
{for $b_n=(L n)^{1/\alpha}.$} Second, if $\{x_i\}$ satisfies (\ref{ass3.2}) and (\ref{assum_gg}), we have that
\begin{equation*}
\mathsf{g}'(b_n)(M_n-b_n) \overset{d}{\Rightarrow}  G_{0}.
\end{equation*} 
Finally, if (\ref{ass3.4}) holds, recall $\mathfrak{b}$ in (\ref{eq_defnbfrak}), we have that 
\begin{equation*}
(\mathfrak{b}n)^{1/(d+1)}(M_n-l) \overset{d}{\Rightarrow} G_{2, d+1}. 
\end{equation*}
 \end{enumerate}
\end{lemma}

\begin{proof}
The proof can be found in the standard textbook or review article regarding extreme value theory. For example,  see \cite{hansen2020three} and \cite{beirlant2004statistics}. 
\end{proof}

In the following, we summarize several notations for the sample covariance matrix $S$ in Table \ref{table_notations} for (\ref{eq_twresultoriginal}). For $z=E+\mathrm{i}\eta\in\mathbb{C}_{+}$, define the resolvent of $S$ and the Stieltjes transform of its empirical spectral distribution by
\begin{equation*}
G^{\mathtt{S}}(z)=(S-zI)^{-1}, \qquad m_S(z)=p^{-1}\operatorname{tr}G^{\mathtt{S}}(z).
\end{equation*} 

Under Assumptions \ref{assum_model} and \ref{assumption_techincial}, we conclude from \cite[Theorem 3.1]{knowles2017anisotropic} that $m_S(z)$ converges to a deterministic limit $m_n^{\mathtt{S}}(z)$ satisfying
\begin{align*}
F_{n,\mathtt{S}}(m_n^{\mathtt{S}}(z),z)=0,\qquad z\in\mathbb{C}_{+},
\end{align*}
where
\begin{align}\label{eq_systemequation_S}
F_{n,\mathtt{S}}(m_n^{\mathtt{S}}(z),z)=\frac{1}{-z+\frac{1}{n}\sum_{i=1}^p\frac{\sigma_i}{1+\sigma_i m_n^{\mathtt{S}}(z)}}-m_n^{\mathtt{S}}(z).
\end{align}

Essentially, $m_{n}^{\mathtt{S}}(z)$ is the Stieltjes transform of the limiting spectral distribution of $S$, namely $\rho^{\mathtt{S}}$. As in (\ref{eq_twresultoriginal}), the rightmost edge of $\rho^{\mathtt{S}},$ $E_{+},$ can be characterized via
\begin{align}\label{eq_def_Eplus}
m_n^{\mathtt{S}}(E_{+})&=\frac{1}{-E_{+}+\frac{1}{n}\sum_{i=1}^p\frac{\sigma_i}{1+\sigma_i m_n^{\mathtt{S}}(E_{+})}}, \quad
1=\frac{\frac{1}{n}\sum_{i=1}^p\frac{\sigma_i^2}{|1+\sigma_i m_n^{\mathtt{S}}(E_{+})|^2}}{\left|E_{+}-\frac{1}{n}\sum_{i=1}^p\frac{\sigma_i}{1+\sigma_i m_n^{\mathtt{S}}(E_{+})}\right|^2}.
\end{align}
Moreover, for $\gamma_0$ in (\ref{eq_twresultoriginal}), we have 
\begin{align}\label{eq_def_gamma0}
\gamma_0=\Big(\frac{|f_{\mathtt{S}}^{\prime\prime}(m_n^{\mathtt{S}}(E_{+}))|}{2}\Big)^{-1/3},\qquad
f_{\mathtt{S}}(x):=-\frac{1}{x}+\frac{1}{n}\sum_{i=1}^p\frac{\sigma_i}{1+\sigma_i x},
\end{align}
which served as the natural scaling factor for the Tracy–Widom fluctuation; see \cite{knowles2017anisotropic}.

We point out that under the stronger assumptions that the empirical spectral distribution of $\Sigma$ converges to a nonrandom probability measure $H$ and that $\lim_{n\to\infty} p/n = c\in(0,\infty)$, the Stieltjes transform of the limiting spectral distribution of $S$ satisfies the following equation \cite{bai2009spectral}
\begin{align}\label{eq_systemequation_S_integralform}
\frac{1}{m_0^{\mathtt{S}}(z)}=-z+c\int\frac{t}{1+t m_0^{\mathtt{S}}(z)}\mathrm{d}H(t).
\end{align}
In this setting, (\ref{eq_def_Eplus}) reduces to
\begin{align*}
    m_0^{\mathtt{S}}(E_{+})&=\frac{1}{-E_{+}+c\int\frac{t}{1+tm_0^{\mathtt{S}}(E_{+})}\mathrm{d}H(t)}, \quad
1=\frac{c\int\frac{t^2}{|1+tm_0^{\mathtt{S}}(E_{+})|^2}\mathrm{d}H(t)}{\left|E_{+}-c\int\frac{t}{1+t m_0^{\mathtt{S}}(E_{+})}\mathrm{d}H(t)\right|^2}.
\end{align*}
A similar expression for $\gamma_0$ can be obtained by replacing 
$m_n^{\mathtt{S}}$ with $m_0^{\mathtt{S}}$ in (\ref{eq_def_gamma0}).

\subsection{Characterization of $\Omega_X$: good configurations for the data matrix $X$}\label{sec_characterization_OmegaX}

In this subsection, independent of Sections \ref{sec_asymptoticlocalaveragedlocal} 
and \ref{secaveragelocalllca}, we define several probability events that represent 
``good configurations'' of $X$. In particular, when studying the largest eigenvalues 
of $Q$ (or $\widetilde{Q}$) (see Table \ref{table_notations} for the definitions) 
conditional on $X$, we will work on a high-probability event that is measurable 
with respect to the $\sigma$-algebra generated by $X$ under Assumption 
\ref{assum_model}. In this subsection, we characterize this event, denoted by 
$\Omega_X$.  To this end, let $\mathcal{F}_X=\sigma(X)$, and let $\Omega_X \in \mathcal{F}_X$ 
be an event depending only on $X$ such that $\mathbb{P}(\Omega_X)\to 1$.

As will be seen below, $\Omega_X$ is defined as the intersection of several events 
encoding the necessary probabilistic properties of the randomness in $X$ that 
follow from Assumption \ref{assum_model}, each of which holds with probability at 
least $1-\mathrm{O}(n^{-D})$ for some $D>0$.

\begin{definition}\label{def_OmegaX}
Denote $\Omega_X$ as the event on $X$ so that the following conditions hold:
\begin{enumerate}
    \item [(1)] For any deterministic $p\times p$ matrix $A$, we have all $1\leq i,j\leq n$, the columns of $X$ satisfy
    \begin{align*}
    |\mathbf{x}_i^*\mathbf{x}_j|=\mathrm{O}\left(n^{\varepsilon_1}\sqrt{\frac{\|\mathbf{x}_i\|^2}{n}}\right), \, |\mathbf{x}_i^*A\mathbf{x}_j|=\mathrm{O}\left(n^{\varepsilon_1}\frac{\|A\|_F}{n}\right),\,|\mathbf{x}_i^*A\mathbf{x}_i-\frac{1}{n}\operatorname{Tr}A|=\mathrm{O}\left(n^{\varepsilon_1}\frac{\|A\|_F}{n}\right),
\end{align*}
    for an arbitrarily small constant $\varepsilon_1>0$.
    \item [(2)] $\|XX^*\|\leq C_{X,0}$ for some constant $C_{X,0}>0$.
    \item [(3)] The eigenvalue rigidity and level repulsion estimates for the non-spiked sample covariance matrix $S$ take the form
    \begin{align*}
        &|\widehat{\lambda}_i-\gamma_i|\leq C_{X,1}(i\wedge n+1-i)^{-1/3}n^{-2/3+\varepsilon_2},\quad \text{where}\; n\int_{\gamma_i}^{\infty}\mathrm{d}\rho^{\mathtt{S}}=i-\frac{1}{2},\quad i=1,\dots,n;\\
        &|\widehat{\lambda}_k-\widehat{\lambda}_{k+1}|\geq C_{X,2}k^{-1/3}n^{-2/3-\varepsilon_3}, \quad k=1,\dots, \lfloor \omega p \rfloor,\quad |\widehat{\lambda}_1-E_{+}|\geq C_{X,3} n^{-2/3-\varepsilon_3},
    \end{align*}
   for some values $0<\omega<1$, $C_{X,1},C_{X,2},C_{X,3}>0$ and arbitrarily small constants $\varepsilon_2,\varepsilon_3>0$.
    \item [(4)]  The complete delocalization of the non-degenerated eigenvectors $\{\nu_k\}$ of $S$ holds as
    \begin{align*}
        |\langle \nu_k, \mathbf{u}\rangle|^2\leq C_{X,4}n^{-1+\varepsilon_4},
    \end{align*}
    uniformly for $1\leq k\leq p$ and all unit  $\mathbf{u}\in\mathbb{R}^p$, where $\varepsilon_4>0$ is an arbitrarily small constant and $C_{X,4}>0$.
    \item [(5)] Let $\mathcal{S}:=X^*\Sigma X$ be the companion matrix of $S$, the complete delocalization of the non-degenerated eigenvectors $\{w_j\}$ of $\mathcal{S}$ holds as
    \begin{align*}
        |\langle w_j, \mathbf{v}\rangle|^2\leq C_{X,5}n^{-1+\varepsilon_5},
    \end{align*}
    uniformly for $1\leq j\leq n$ and all unit  $\mathbf{v}\in\mathbb{R}^n$, where $\varepsilon_5>0$ is an arbitrarily small constant and $C_{X,5}>0$.
    \item [(6)] For any fixed deterministic unit vector { $\mathbf{k}\in\mathbb{R}^p$, it holds $|\mathbf{k}^*XX^*\mathbf{k}-1|=\mathrm{O}(n^{-1/2+\varepsilon_6})$,} where $\varepsilon_6>0$ is an arbitrarily small constant. Moreover, for the columns of $X$, it holds that 
    \begin{align*}
        {\sum_{i=1}^n(\mathbf{k}^*\mathbf{x}_i)^4=n^{-1}(3+\sum_{j=1}^p\mathsf{k}_{j}^4(\mathbb{E}(\sqrt{n}x_{11})^4-3))+\mathrm{o}(n^{-1}),\quad \sum_{i=1}^n(\mathbf{k}^*\mathbf{x}_i)^8=\mathrm{O}(n^{-3}\log^{\varepsilon_7}n),}
    \end{align*}
    for any arbitrarily small constant $\varepsilon_7>0$, where {$\mathsf{k}_j$ is the $j$-th element of $\mathbf{k}$.}
\end{enumerate}
\end{definition}

The following lemma shows that under some regularity conditions, the probability event $\Omega_X$ holds with high probability. 
\begin{lemma}\label{lem_Xgoodevents}
    Let $\Omega_X$ be the event defined in Definition \ref{def_OmegaX}, under Assumptions \ref{assum_model}, \ref{assumption_techincial}, and  \ref{assum_additional_techinical}, we then have that when $n$ is sufficiently large,
    \begin{align*}
        \mathbb{P}(\Omega_X)=1-\mathrm{O}(n^{-D}),
    \end{align*}
    for some large constant $D>0$.
\end{lemma}
\begin{proof}
The results in Definition \ref{def_OmegaX} have been established in \cite{Bao2015,Alex2014,ding2024eigenvector,Ding&Yang2018,knowles2013isotropic,knowles2017anisotropic}.
\end{proof}

\begin{remark}\label{rmk_notationconventions} In this remark, we provide several probabilistic clarifications and conventions, as there are two sources of randomness: one arising from the data matrix $X$, and the other from the multipliers in $D$.

Since the multiplier matrix $D$ is assumed to be independent of the data matrix $X$, we can realize them on a common product probability space. More specifically, let
\begin{align*}
    (\widetilde{\Omega}_X,\mathcal{F}_X,\mathbb{P}_X)
    \qquad \text{and} \qquad
    (\widetilde{\Omega}_D,\mathcal{F}_D,\mathbb{P}_D),
\end{align*}
be probability spaces on which $X$ and $D$ are defined, respectively. We then consider the product probability space
\[
(\widetilde{\Omega},\mathcal{F},\mathbb{P})
:=
(\widetilde{\Omega}_X \times \widetilde{\Omega}_D,\ \mathcal{F}_X \otimes \mathcal{F}_D,\ \mathbb{P}_X \otimes \mathbb{P}_D),
\]
where $\widetilde{\Omega}_X \times \widetilde{\Omega}_D$ denotes the Cartesian product of the two sample spaces, $\mathcal{F}_X \otimes \mathcal{F}_D$ denotes the product $\sigma$-algebra, and $\mathbb{P}_X \otimes \mathbb{P}_D$ denotes the product probability measure. We regard $X$ and $D$ as coordinate random elements on $\widetilde{\Omega}$, namely,
\begin{align*}
    X(\omega_X,\omega_D) := X(\omega_X),
    \qquad
    D(\omega_X,\omega_D) := D(\omega_D).
\end{align*}
In particular, $X$ and $D$ are independent by construction.

Building on the above product probability space, all random quantities considered below— including $Q=\Sigma^{1/2}XD^2X^*\Sigma^{1/2}$, $\lambda_1(Q)$, and $\xi_{(1)}^2$—are understood as random variables defined on $(\widetilde{\Omega},\mathcal{F},\mathbb{P})$. By convention, $\mathbb{P}$ denotes the joint probability measure on this product space.
If an event depends only on $X$ (respectively, only on $D$), then its probability under $\mathbb{P}$ coincides with the corresponding marginal probability under $\mathbb{P}_X$ (respectively, $\mathbb{P}_D$). Moreover, $\mathbb{P}(\,\cdot\,\mid X)$ denotes the conditional probability with respect to $\sigma(X)$, as introduced above. 

Based on the above discussion, let $\Omega_D \in \sigma(D)$ and $\Omega_X \in \sigma(X)$ denote the “good” events introduced in Definition~\ref{defn_probset} and Definition~\ref{def_OmegaX}, respectively. Since both are events in the common $\sigma$-algebra $\mathcal{F}$, their intersection $\Omega_X \cap \Omega_D$ is well defined. Moreover, by independence,
\begin{align*}
    \mathbb{P}(\Omega_X \cap \Omega_D)
    =
    \mathbb{P}(\Omega_X)\,\mathbb{P}(\Omega_D).
\end{align*}

\end{remark}

\section{Proof of averaged local laws}\label{appendxi_locallawproof}

{
In this section, we prove the local laws stated in Theorems 
\ref{thm_unboundedcaselocallaw}--\ref{thm_averagedlocallaw_boundedmultiplier_conditionalonX}. }

\subsection{Unbounded support and unconditional setting: proof of Theorem \ref{thm_unboundedcaselocallaw}}\label{appendix_sec_prooflocallawunbounded}

In this section, we will prove Theorem \ref{thm_unboundedcaselocallaw}. Due to similarity, we focus on the proof of part 1 and briefly discuss that of part 2. The proof contains two steps. In the first step, we establish the local laws for $Q$ outside the bulk of the spectrum on the domain $\widetilde{\mathbf{D}}_{u}$ denoted as follows   
\begin{equation}\label{eq_spectraldomainonereduced}
\widetilde{\mathbf{D}}_{u  } \equiv \widetilde{\mathbf{D}}_{u }(\mathtt C):=\left\{z=E+\ri\eta: 0<E-\vartheta_1 \leq \mathtt C d_{1}, \ n^{-2/3}\leq\eta\leq \mathtt C \vartheta_1 \right\},
\end{equation}
where the parameters $\mathtt{C}$, $d_1$ and $\vartheta_1$ are same as in \eqref{eq_spectraldomainone}.
To be specific, we will establish the following proposition. 
\begin{proposition} \label{eq_propooursidebulk} Under the assumptions of Theorem \ref{thm_unboundedcaselocallaw}, the following statements holds uniformly on $\widetilde{\mathbf{D}}_{u}$ in (\ref{eq_spectraldomainonereduced}) when conditional on the event $\Omega_D$ in Definition \ref{defn_probset}. 
\begin{enumerate}
\item[(1).] If {\normalfont Case (a)} of (i) of  Assumption \ref{assum_D} holds, we have that 
\begin{align}\label{eq_entrywiselocallaw}
\mathcal{G}_{ij}(z)=-\frac{\delta_{ij}}{z\big(1+m_{1n}(z)\xi_i^2\big)}+\rO_{\prec} \left( n^{-1/2-1/\alpha} \right),
\end{align}
where $\delta_{ij}$ is the Dirac delta function so that $\delta_{ij}=1$ when $i=j$ and $\delta_{ij}=0$ when $i \neq j.$ Moreover, we have that 
\begin{equation*}
m_1(z)=m_{1n}(z)+\rO_{\prec}\left( n^{-1/2-2/\alpha}\right), \  m_2(z)=m_{2n}(z)+\rO_{\prec}\left( n^{-1/2-1/\alpha}\right), 
\end{equation*}
and 
\begin{equation}\label{eq_averagedlawused}
m_{Q}(z)=m_{n}(z)+\rO_{\prec} \left( n^{-1/2-2/\alpha} \right).
\end{equation}
\item[(2).] If {\normalfont Case (b)} of (i) of Assumption \ref{assum_D} holds, we have that the results in part (1) hold by setting $\alpha=\infty.$
\end{enumerate}  
\end{proposition}
\quad  Once Proposition \ref{eq_propooursidebulk} is proved, we can quantify the rough locations of the eigenvalues of $Q$ as summarized in the following lemma. 
 
% {\color{red}[NEED TO BE MORE CAREFUL ON THE CASE $\alpha=\infty$. Need to go over the proof below again.]}
\begin{lemma}\label{lem: upper bound for eigenvalues}
Suppose Assumptions \ref{assum_model}, \ref{assumption_techincial} and (i) of Assumption \ref{assum_D} hold. For some sufficiently large constant $C>0,$ with high probability, for any fixed realization $\{\xi_i^2\} \in \Omega_D$ in Definition \ref{defn_probset}, for all $1 \leq i \leq \min\{p,n\},$ we have that 
\begin{equation}\label{eq_boundestimateone}
\lambda_i(Q) \notin (\vartheta_1, C n^{1/\alpha} \log n), \ \text{if \normalfont{Case (i)-a} of Assumption \ref{assum_D} holds},  
\end{equation}
and 
\begin{equation}\label{eq_boundestimatetwo}
\lambda_i(Q) \notin (\vartheta_1, C  \log^{1/\beta} n), \ \text{if {\normalfont Case (i)-b of Assumption \ref{assum_D} holds}}.  
\end{equation}
%
%\ref{ass1},\ref{ass2},\ref{ass3.1},\ref{ass3.2} and events $\Omega_n$ hold. With high probability there is no eigenvalue of $\mathcal{W}$ in $(\lambda_{(1)},Cn^{2/\alpha}\log n)$ for $\alpha\in(0,+\infty)$ and no eigenvalues in $(\lambda_{(1)},C\log^{K}n)$ for $\alpha=+\infty$.
\end{lemma}
\begin{proof}
Due to similarity, we focus our arguments on (\ref{eq_boundestimateone}). We prove the results by contradiction. 
%We use contradiction to prove this lemma and we take the case $\alpha\in(0,+\infty)$ for example. 
Assume that there is an eigenvalue of $Q$ in the interval as in (\ref{eq_boundestimateone}), denoted as $\widehat{\lambda}$. Let $z=\widehat{\lambda}+\ri n^{-2/3}.$ Since $z\in\widetilde{\mathbf{D}}_u \subset \mathbf{D}_u$ as in (\ref{eq_spectraldomainone}), by Lemma \ref{lem: basic bounds}, we obtain $\operatorname{Im}m_n(z)=\eta \widehat{\lambda}^{-2}$. According to (\ref{def1}) and \eqref{eq_mu1part}, we have on event $\Omega_D$
\begin{equation}\label{eq_lowerboundmu1}
\vartheta_1 \gtrsim n^{1/\alpha} \log^{-1} n.
\end{equation} 
Then, according to (\ref{eq_averagedlawused}), together with \eqref{eq_lowerboundmu1} and assumption of $\widehat{\lambda}$, we observe that 
\begin{equation}\label{eq: value of m_W outside spectrum}
    \begin{split}
        \operatorname{Im}m_{Q}(z)&=\operatorname{Im}m_n(z)+\operatorname{Im}(m_{Q}(z)-m_n(z))\\
        &\prec n^{-2/\alpha-2/3}+n^{-1/2-2/\alpha} \prec n^{-1/2-2/\alpha}.
    \end{split}
\end{equation}
On the other hand, we have
\begin{equation*}
\operatorname{Im}m_{Q}(z)=\frac{1}{n}\sum_i\frac{\eta}{(\lambda_i-\widehat{\lambda})^2+\eta^2}\geq \frac{1}{n\eta}=n^{-1/3},
\end{equation*}
which contradicts \eqref{eq: value of m_W outside spectrum}. Therefore, there is no eigenvalue in this interval. Similarly, we can prove (\ref{eq_boundestimatetwo}). The only difference is that (\ref{eq_lowerboundmu1}) should be replaced by $\vartheta_1 \gtrsim  \log^{1/\beta} n$ according to (\ref{def3}) so that the error rate in (\ref{eq: value of m_W outside spectrum}) should be updated to $n^{-1/2}.$ This completes the proof. 
\end{proof}

\quad Armed with the above lemma, we can proceed to the second step to conclude the proof of Theorem \ref{thm_unboundedcaselocallaw}.  In what follows, we first provide the proof of Proposition \ref{eq_propooursidebulk} in Section \ref{prop_subsubsubsubsubsub}. After that, we prove Theorem \ref{thm_unboundedcaselocallaw} in Section \ref{sec_proofa2}.

\subsubsection{Proof of Proposition \ref{eq_propooursidebulk}}\label{prop_subsubsubsubsubsub}

We first prepare two lemmas. The first is to establish Proposition \ref{eq_propooursidebulk} on a large scale of $\eta.$

\begin{lemma}[Average local law for large $\eta$]\label{lem: average local law for large eta} Proposition \ref{eq_propooursidebulk} holds when $\eta=\mathtt{C} \vartheta_1.$
\end{lemma}

\begin{proof}
 In the sequel, without loss of generality, we assume that $\xi_1^2 \geq \xi_2^2 \geq \cdots \geq \xi_n^2.$ According to (\ref{def1}), (\ref{eq_mu1part}) and the definition of $d_1$ in (\ref{eq_firstddefinition}), we have for $z\in\widetilde{\mathbf{D}}_u$ with $\eta=\mathtt{C} \vartheta_1$, 
\begin{equation}\label{eq_Econtrol}
{
E \asymp \vartheta_1\gtrsim\xi_1^2,\quad \max\left\{\|G^{(\mathcal{T})}\|,\|\mathcal{G}^{(\mathcal{T})}\|\right\}\leq\eta^{-1}=\mathtt{C}^{-1} \vartheta_1^{-1}}
\end{equation}
for any finite $\mathcal{T}\subset\{1,\dots, n\}$. With the above preparation, in the sequel, we establish the system equations of $m_1$ and $m_2$ using Lemma \ref{lem: Resolvent}. We start with $m_2.$ By Lemma \ref{lem: Resolvent} and the definition of $m_2$ in (\ref{eq_m1m2}), we have
\begin{gather}\label{eq: decomp m_2}
    {
    m_2=\frac{1}{n}\sum_{i=1}^n\frac{\xi^2_i}{-z-z\mathbf{y}^{*}_i G^{(i)}\mathbf{y}_i}=\frac{1}{n}\sum_{i=1}^n\frac{\xi^2_i}{-z(1+\xi^2_in^{-1}{\rm tr}G^{(i)}\Sigma+Z_i)},}\\
    Z_i=\mathbf{y}^{*}_i G^{(i)}\mathbf{y}_i-\xi^2_in^{-1}{\rm tr}G^{(i)}\Sigma. \nonumber
\end{gather}
As $\mathbf{y}_i$ is independent of $G^{(i)}$, by (1) of  Lemma \ref{lem:large deviation}, we see that 
\begin{equation}\label{eq: bound for Z}
    Z_i\prec\frac{\xi^2_i}{n}\|G^{(i)}\Sigma\|_F\leq\frac{\xi^2_i}{n}\|G^{(i)}\|\|\Sigma\|_F\prec\frac{\xi^2_i}{\sqrt{n}\eta}.
\end{equation}
Moreover, using the definition of $m_1$ in (\ref{eq_m1m2}) and resolvent identity, we obtain that for some constant $C>0$ 
\begin{equation}\label{eq_m1approximate}
    \frac{1}{n}{\rm tr}(G^{(i)}\Sigma)-m_1(z)=\frac{1}{n}\mathbf{y}_i^{*}G \Sigma G^{(i)}\mathbf{y}_i\leq C\frac{\xi^2_i}{n \eta^2}.
\end{equation}
On the other hand, by (\ref{eq_Econtrol}) and Definition \ref{defn_probset}, there exists some constant $C>C_{\mathsf{p},1}$ such that
\begin{equation*}
|1+\xi^2_i m_1|\geq 1-C \mathtt{C}^{-1}>0, 
\end{equation*} 
when $\mathtt{C}>0$ is chosen large. 
Then, on the event $\Omega_D$, (\ref{eq: decomp m_2}) reads that
\begin{equation}\label{eq: m_2 by m_1}
    m_2=\frac{1}{n}\sum_{i=1}^n\frac{\xi^2_i}{-z(1+\xi^2_im_1+\rO_{\prec}(\frac{\xi^2_i}{\sqrt{n}\eta}))}=\frac{1}{n}\sum_{i=1}^n\frac{\xi^2_i}{-z(1+\xi^2_im_1)}+\rO_{\prec}(n^{-1/2-1/\alpha}).
\end{equation}

Next, we consider $m_1.$ Decompose that 
\[
Q-zI=\sum_{i=1}^n\mathbf{y}_i\mathbf{y}_i^{*}+zm_2(z)\Sigma-z(I+m_2(z)\Sigma).
\]
Applying the resolvent expansion to the first order, we obtain
\[
G=-z^{-1}(I+m_2(z)\Sigma)^{-1}+z^{-1}G\left(\sum_{i=1}^n\mathbf{y}_i\mathbf{y}_i^{*}+zm_2(z)\Sigma\right)(I+m_2(z)\Sigma)^{-1}.
\]
Using the Shernman-Morrison formula, we have
\begin{equation}\label{eq_usefulformula}
G\mathbf{y}_i=\frac{G^{(i)}\mathbf{y}_i}{1+\mathbf{y}_i^{*} G^{(i)}\mathbf{y}_i}.
\end{equation}
Combining the above two identities and using the Lemma \ref{lem: Resolvent}, it follows that
\begin{equation}\label{eq: decomp of mathcal_G}
\begin{split}
    G&=-z^{-1}(I+m_2(z)\Sigma)^{-1}+\left[z^{-1}\sum_{i=1}^n\frac{G^{(i)}(\mathbf{y}_i\mathbf{y}_i^{*}-n^{-1}\xi^2_i\Sigma)}{1+\mathbf{y}_i^{*}G^{(i)}\mathbf{y}_i}(I+m_2(z)\Sigma)^{-1} \right]\\
    &+\left[z^{-1}\frac{1}{n}\sum_{i=1}^n\frac{(G^{(i)}-G)\xi^2_i\Sigma}{1+\mathbf{y}_i^{*}G^{(i)}\mathbf{y}_i}(I+m_2(z)\Sigma)^{-1} \right]\\
    &:=-z^{-1}(I+m_2(z)\Sigma)^{-1}+R_1+R_2.
\end{split}
\end{equation}
In what follows, we control the two error terms $R_1,R_2.$ For $R_1$, we notice that 
\begin{equation}\label{eq: decomp R_1}
    \begin{split}
    &\frac{z}{n}{\rm tr}(R_1\Sigma)=\frac{1}{n}\sum_i{\rm tr}\left(\frac{G^{(i)}(\mathbf{y}_i\mathbf{y}_i^{*}-n^{-1}\xi^2_i\Sigma)}{1+\mathbf{y}_i^{*} G^{(i)}\mathbf{y}_i}(I+m_2^{(i)}\Sigma)^{-1}\Sigma\right)\\
    &+\frac{1}{n}\sum_i{\rm tr}\left(\frac{ G^{(i)}(\mathbf{y}_i\mathbf{y}_i^{*}-n^{-1}\xi^2_i\Sigma)}{1+\mathbf{y}_i^{*} G^{(i)}\mathbf{y}_i}(I+m_2\Sigma)^{-1}(m_2^{(i)}-m_2)\Sigma(I+m^{(i)}_2\Sigma)^{-1}\Sigma\right)\\
    &:=\mathtt R_{11}+\mathtt R_{12}.
\end{split}
\end{equation}
Since $\|\mathcal G^{(i)}\| \leq\eta^{-1}$, using (\ref{eq_Econtrol}), on the event $\Omega_D$, we have that for some constant $C>0,$ 
\begin{equation}\label{eq_eqcontrolcontrolcontrolcontrol}
|m_2^{(i)}(z)|\leq\frac{1}{n}\sum_{j\neq i}\xi^2_j| \mathcal G^{(i)}_{jj}|\leq C \frac{\log^2 n}{n^{1/\alpha}}. 
\end{equation}
Moreover, according to (\ref{eq: bound for Z}), on the event $\Omega_D$, when $n$ is sufficiently large, we have for some constant $C>C_{\mathsf{p},1}$
\begin{equation}\label{eq_controlower}
|1+\mathbf{y}_i^{*} G^{(i)}\mathbf{y}_i| \asymp  |1+\xi_i^2 n^{-1} {\rm tr} G^{(i)} \Sigma | \geq 1- C \mathtt{C} ^{-1}>0,
\end{equation}
whenever $\mathtt{C}$ is chosen sufficiently large. Both above two estimations hold with high probability. Consequently, for all $i=1,\dots,n$, we have
\begin{align}\label{eq_controncontrolcontrol} 
  &{\rm tr}\Big(\frac{ G^{(i)}(\mathbf{y}_i\mathbf{y}_i^{*}-n^{-1}\xi^2_i\Sigma)}{1+\mathbf{y}_i^{*}  G^{(i)}\mathbf{y}_i}(I+m_2^{(i)}\Sigma)^{-1}\Sigma\Big)   \asymp {\rm tr} \left(\xi^2_i G^{(i)}(\mathbf{x}_i\mathbf{x}_i^{*}-n^{-1}I) (I+m_2^{(i)}\Sigma)^{-1} \Sigma^2 \right) \\
   &= \xi_i^2 \left( \mathbf{x}_i^* G^{(i)} (I+m_2^{(i)}\Sigma)^{-1} \Sigma^2 \mathbf{x}_i -n^{-1} {\rm tr} \left( G^{(i)} (I+m_2^{(i)}\Sigma)^{-1} \Sigma^2 \right) \right) \prec \xi_i^2 \frac{1}{\eta\sqrt{n}}\nonumber, 
%   &\le C_1  \xi^2_i\eta^{-1}\|(\mathbf{u}_i\mathbf{u}_i^{*}-n^{-1}I)\|_2\prec \xi^2_i\eta^{-1}n^{-1/2},
\end{align}
where in the third step we used (1) of Lemma \ref{lem:large deviation}.  Then, on the event $\Omega_D$ and using (\ref{eq_Econtrol}), we have $\mathtt{R}_{11} \prec n^{-1/2-1/\alpha}. $ For $\mathtt R_{12}$, using the definition in (\ref{eq_m1m2}), Lemma \ref{lem: Resolvent} and the definition of $\mathcal{G}^{(i)}$ (see (\ref{eq_trivialcontrol}) below), we find that  
\begin{equation}\label{eq_decompositionleavoneout}
m_2(z)-m^{(i)}_2(z)=\frac{1}{n}\sum_{j=1}^n \xi^2_j(\mathcal G_{jj}-\mathcal G^{(i)}_{jj})=\frac{1}{n}\sum_{j\neq i}\xi^2_j\frac{ \mathcal G_{ji} \mathcal G_{ij}}{\mathcal G_{ii}}+\frac{\xi_i^2(\mathcal{G}_{ii}-|z|^{-1})}{n}.
\end{equation}
Using Lemma \ref{lem: Resolvent} and Lemma \ref{lem:large deviation}, we observe that on the event $\Omega_D$, 
\begin{equation*}
  \frac{1}{\mathcal G_{ii}(z)}=-z-z\mathbf{y}_i^{*} G^{(i)}\mathbf{y}_i\prec |z|.
\end{equation*}
Moreover, Lemmas \ref{lem: Resolvent} and \ref{lem:large deviation} yield that 
\begin{equation*}  
  \mathcal G_{ij}(z)=z\mathcal G_{ii}(z)\mathcal G^{(i)}_{jj}(z)\mathbf{y}_i^{*} G^{(ij)}\mathbf{y}_j\prec |z|\eta^{-2} |\xi_i\xi_j| n^{-1}\|\mathcal{G}^{(ij)}\|_F\prec n^{-1/2}|z|\eta^{-3}|\xi_i\xi_j| ,\quad i\neq j.  
\end{equation*}
Combining the above bounds with (\ref{def1}), we see that
\begin{equation*}
 m_2(z)-m^{(i)}_2(z) \prec n^{-1-1/\alpha}.
\end{equation*}

Together with (\ref{eq_controlower}) and (\ref{eq_controncontrolcontrol}), we arrive at 
$\mathtt R_{12} \prec 1/(\eta^2 n^{3/2}).$
Summarizing the above bounds, we see that $(zn^{-1}){\rm tr}(R_1\Sigma)\prec n^{-1/2-1/\alpha}.$

For $R_2$, applying the Sherman–Morrison formula to $((G^{(i)})^{-1}+\mathbf{y}_i \mathbf{y}_i^*)^{-1},$ we obtain that 
\begin{gather}\label{eq_simimimimimi}
\begin{split}
    \frac{1}{n}\left|{\rm tr}\Big(\frac{(G^{(i)}-G)\Sigma(I+m_2\Sigma)^{-1}\Sigma}{1+\mathbf{y}_i^{*} G^{(i)}\mathbf{y}_i}\Big) \right|&=\frac{1}{n}\left|\frac{\mathbf{y}_i^{*} G^{(i)}\Sigma(I+m_2\Sigma)^{-1}\Sigma G \mathbf{y}_i}{1+\mathbf{y}_i^{*}G^{(i)}\mathbf{y}_i}\right|
 %   &\le \frac{\sigma_1^2\xi^2_i}{n\log n}\|\mathcal{G}\|^2_2\\
    \prec \frac{\xi_i^2}{n \eta^2},
\end{split}
\end{gather}
where in the second step we used Lemma \ref{lem:large deviation} and (\ref{eq_controlower}) and a discussion similar to (\ref{eq_eqcontrolcontrolcontrolcontrol}). Together with the definition of $R_2$ in (\ref{eq: decomp of mathcal_G}), on the event $\Omega_D$, we find that
\[
\begin{split}
  \frac{z}{n} \left|{\rm tr}(R_2\Sigma) \right|
  %&=|\frac{1}{n^2}\sum_i\xi^2_i {\rm tr}\Big(\frac{(\mathcal{G}^{(i)}-\mathcal{G})\Sigma(I+m_2\Sigma)^{-1}\Sigma}{(1+\mathbf{y}_i^{*}\mathcal{G}^{(i)}\mathbf{y}_i)}\Big)|\\
  \leq \frac{1}{n^2}\sum_i\frac{\xi^2_i}{\eta^2 }\prec n^{-1-2/\alpha}.
\end{split}
\]
As a result, in light of the definition $m_1$ in (\ref{eq_m1m2}),
we have
\begin{equation}\label{eq: m_1 by m_2}
\begin{split}
    m_1&=\frac{1}{n}{\rm tr}(G(z)\Sigma)=-z^{-1}\frac{1}{n}{\rm tr}((I+m_2(z)\Sigma)^{-1}\Sigma)+\rO_{\prec}(n^{-1/2-2/\alpha})\\
    &=-\frac{1}{n}\sum_{i=1}^p\frac{\sigma_i}{z(1+m_2\sigma_i)}+\rO_{\prec}(n^{-1/2-2/\alpha}).
\end{split}
\end{equation}

Now, we control $m_2(z)-m_{2n}(z).$ Recall the definition of $m_{2n}(z)$ in (\ref{eq_systemequationsm1m2}). Combining \eqref{eq: m_2 by m_1} and \eqref{eq: m_1 by m_2},  we observe
\begin{align}\label{eq_closenessm2m2n}
%\begin{split}
& m_2(z)-m_{2n}(z)  \nonumber\\
   & =\frac{1}{n}\sum_{i=1}^n \Big(\frac{\xi^2_i}{-z(1+\xi^2_im_1)}+\frac{\xi^2_i}{z(1+\xi^2_im_{1n})} \Big)+\rO_{\prec}(n^{-1/2-1/\alpha})\\
 %  & =\frac{1}{n}\sum_{i=1}^n\frac{\xi^4_i(m_1-m_{1n})}{z(1+\xi^2_im_1)(1+\xi^2_im_{1n})}+\rO_{\prec}(n^{-1/2-1/\alpha}) \nonumber \\
   & =\Big(\frac{1}{n}\sum_{i=1}^n\frac{\xi^4_i}{z(1+\xi^2_im_1)(1+\xi^2_im_{1n})}\Big)\Big(\frac{1}{n}\sum_{i=1}^n\frac{\sigma_i^2(m_2-m_{2n})}{z(1+\sigma_im_2)(1+\sigma_im_{2n})}\Big)+\rO_{\prec}(n^{-1/2-1/\alpha}). \nonumber
%\end{split}
\end{align}
By a discussion similar to (\ref{eq_eqcontrolcontrolcontrolcontrol}) and (\ref{eq_controlower}), with high probability, on the event $\Omega_D$, we have 
%Since $|z|\sim \xi^2_{s_1}$ with $|1+\xi^2_im_1|>0$,$|1+\xi^2_im_{1n}|>0$,$|1+\sigma_im_1|>0$,$|1+\sigma_im_{1n}|>0$, thereafter we have
\[
\begin{split}
  |m_2-m_{2n}|&=\rO\Big(\frac{1}{n}\sum_{i=1}^n\frac{\xi^4_i}{z^2}|m_2-m_{2n}| \Big)+\rO_{\prec}(n^{-1/2-1/\alpha})\\
  &=\rO(n^{-2/\alpha}|m_2-m_{2n}|)+\rO_{\prec}(n^{-1/2-1/\alpha}),
\end{split}
\]
where in the second step we used (\ref{def1}). Then we can conclude that $m_2(z)-m_{2n}(z)\prec n^{-1/2-1/\alpha}$. Parallel, applying the above argument to $m_1(z)-m_{1n}(z)$, we also have $m_1(z)-m_{1n}(z)\prec n^{-1/2-2/\alpha}$.

Armed with the estimates of $m_1-m_{1n}$ and $m_2-m_{2n}$, we proceed to finish the rest of the proof. Recall $m_Q$ in (\ref{eq_mq}). Using (\ref{eq: decomp of mathcal_G}) and a discussion similar to \eqref{eq: m_1 by m_2}, one can see that 
\begin{align}\label{eq_verbalttt}
%\begin{split}
    m_{\mathcal{Q}}&=\frac{1}{p}{\rm tr}(G(z))=-\frac{1}{p}\sum_{i=1}^p\frac{1}{z(1+m_2\sigma_i)}+\rO_{\prec}(n^{-1/2-2/\alpha})\nonumber\\
    &=-\frac{1}{p}\sum_{i=1}^p\frac{1}{z(1+m_{2n}\sigma_i)}+\frac{1}{p}\sum_{i=1}^p\frac{(m_2-m_{2n})\sigma_i}{z(1+m_2\sigma_i)(1+m_{2n}\sigma_i)}+\rO_{\prec}(n^{-1/2-2/\alpha}) \nonumber \\
    &=m_{n}+\rO_{\prec}(n^{-1/2-2/\alpha}),
%\end{split}
\end{align}
where in the last we used the definition of $m_n(z)$ in (\ref{eq_systemequationsm1m2}). Finally, for the control of the matrix $\mathcal G,$ for the diagonal entries, by Lemma \ref{lem: Resolvent} and a discussion similar to  \eqref{eq: bound for Z} and (\ref{eq_m1approximate}),  we have
\[
\begin{split}
    \mathcal G_{ii}&=-\frac{1}{z(1+\mathbf{y}^{*}_i G^{(i)}\mathbf{y}_i)}=-\frac{1}{z(1+\xi^2_in^{-1}{\rm tr} G^{(i)}\Sigma+\rO_{\prec}(\frac{\xi^2_i}{\sqrt{n}\eta}))}\\
    &=-\frac{1}{z(1+\xi^2_im_1+\rO_{\prec}(\frac{\xi^2_i}{\sqrt{n}\eta}))}=-\frac{1}{z(1+\xi^2_im_{1n})}+\rO_{\prec}(n^{-1/2-1/\alpha}).
\end{split}
\]
For off-diagonal entries, using Lemmas \ref{lem: Resolvent} and \ref{lem:large deviation}, we have
\[
\begin{split}
    |\mathcal G_{ij}|&\leq|z||\mathcal G_{ii}||\mathcal G^{(i)}_{ii}||\mathbf{y}^{*}_i G^{(ij)}\mathbf{y}_j| \prec n^{-1/2-2/\alpha}. 
%    \\
%    &\le \frac{\log^4n}{p}\|\mathcal{G}^{(ij)}\|_F\prec n^{-1/2-2/\alpha}.
\end{split}
\]
This completes the proof when (\ref{ass3.1}) holds. For the case (\ref{ass3.2}), the main difference is to use the estimates of (\ref{def3}) instead of (\ref{def1}) whenever it is needed. We omit further details. This completes our proof.

\end{proof}
\quad  The second part is to prove Proposition \ref{eq_propooursidebulk} under a prior control of the resolvent as shown in the following lemma.  

\begin{lemma} \label{lem: self improvement}
Proposition \ref{eq_propooursidebulk} holds if (\ref{eq_entrywiselocallaw}) holds uniformly for $z \in \widetilde{\mathbf{D}}_{u}.$ 
\end{lemma}
\begin{proof}
According to the prior control (\ref{eq_entrywiselocallaw}), we have that for $1 \leq i \neq j \leq n$ 
\begin{gather}\label{eq_expansionone}
\mathcal G_{ii}=\frac{1}{z(1+\xi^2_im_{1n}(z))}+\rO_{\prec}(n^{-1/2-1/\alpha}), \quad \mathcal G_{ij} =\rO_{\prec}(n^{-1/2-1/\alpha}).
\end{gather}
For the diagonal entries, when $i=1,$ using (\ref{eq: def of vartheta_1}), we observe that
\begin{align}\label{eq_g11control}
 %   \begin{split}
        \mathcal G_{11}&=-\frac{1}{z(1+\xi^2_{1}m_{1n}(z))}+\rO_{\prec}(n^{-1/2-1/\alpha}) \nonumber\\
        &=-\frac{1}{z(1+\xi^2_{1}m_{1n}(\vartheta_1))}+\frac{z\xi_{1}^2(m_{1n}(z)-m_{1n}(\vartheta_1))}{(z(1+\xi^2_{1}m_{1n}(\vartheta_1)))(z(1+\xi^2_{1}m_{1n}(z)))}+\rO_{\prec}(n^{-1/2-1/\alpha})\\
        &=\frac{1}{zd_1m_{1n}(\vartheta_1)}-\frac{z\xi_{1}^2(m_{1n}(z)-m_{1n}(\vartheta_1))}{zd_1m_{1n}(\vartheta_1)}\big(\mathcal G_{11}+\rO_{\prec}(n^{-1/2-2/\alpha})\big)+\rO_{\prec}(n^{-1/2-1/\alpha}) \nonumber\\
        &=\frac{1}{zd_1m_{1n}(\vartheta_1)}-\frac{z\xi_{1}^2}{zd_1m_{1n}(\vartheta_1)}(\mathcal G_{11}+\rO_{\prec}(n^{-1/2-1/\alpha}))\times \rO_{\prec}(n^{-1/\alpha})+\rO_{\prec}(n^{-1/2-1/\alpha}), \nonumber
 %   \end{split}
\end{align}
where in the fourth step we used Lemma \ref{lem: basic bounds}. Also, by Lemma \ref{lem: basic bounds}, (\ref{eq_mu1part}) and (\ref{def1}),  it yields that on the event $\Omega_D$, there exists some constant $C>0$ such that
\begin{gather}\label{eq_g11}
\begin{split}
    |\mathcal G_{11}|&=\frac{1}{|zd_1m_{1n}(\vartheta_1)|}+\rO_{\prec}(n^{-1/2-1/\alpha})=\frac{C}{d_1}+\rO_{\prec}(n^{-1/2-1/\alpha}).
\end{split}
\end{gather}
For $2 \leq i \leq n,$ by (\ref{def1}) and the definition of $d_1,$ using Lemma \ref{lem: basic bounds}, we can also find that on the event $\Omega_D$,
\begin{equation}\label{eq_expansiontwo}
\mathcal{G}_{ii}=\rO_{\prec}(n^{-1/\alpha}). 
\end{equation}

In the sequel, we provide some basic controls for the matrix $\mathcal G^{(i)}$ for all $1 \leq i \leq n.$ By an elementary calculation, it is not hard to see that
\begin{equation}\label{eq_trivialcontrol}
\mathcal G_{ii}^{(i)}=-z^{-1}; \ \mathcal{G}_{i k}^{(i)}=0, \  1 \leq k \neq i \leq n. 
\end{equation}  
Moreover, using (\ref{eq_expansionone}), (\ref{eq_expansiontwo}) and the third identity of Lemma \ref{lem: Resolvent}, for $1 \leq i \leq n,$ 
\begin{equation}\label{eq_nontrivialcontrol}
\mathcal G_{kk}^{(i)}=\rO_{\prec}(n^{-1/\alpha}),  \ k \neq i; \ \mathcal G_{kl}^{(i)}=\rO_{\prec}(n^{-1/2-1/\alpha}), \ k,l \neq i.
\end{equation}

With the above preparation, we now proceed to the control of $Z_i$ in (\ref{eq: decomp m_2}). Since $Q$ and $\mathcal{Q}$ have the same non-zero eigenvalues, instead of the strategy in (\ref{eq: bound for Z}), we bound it as follows 
\begin{equation}\label{eq_standarddiscussion}
\begin{split}
     Z_i&\prec\frac{\xi^2_i}{n}\| G^{(i)}\Sigma\|_F\prec\frac{\xi^2_i}{n}\|G^{(i)}\|_F=\frac{\xi^2_i}{n}(\operatorname{tr}((G^{(i)})^2))^{1/2}\leq\frac{\xi^2_i}{n}\operatorname{tr}((\mathcal G^{(i)})^2)^{1/2}+\frac{\xi_i^2}{n}\frac{\sqrt{|n-p|}}{|z|}\\
    & \asymp \frac{\xi^2_i}{n}\|\mathcal G^{(i)}\|_F+\frac{\xi^2_i}{n}\frac{n^{1/2}}{n^{1/\alpha}}=\frac{\xi^2_i}{n} \Big((\mathcal G_{ii}^{(i)})^2+\sum_{j\neq i}(\mathcal G^{(i)}_{jj})^2+\sum_{j\neq k\neq i}(\mathcal G^{(i)}_{jk})^2 \Big)^{1/2}+\frac{\xi^2_i}{n^{1/2+1/\alpha}}  \\
& \prec \frac{\xi_1^2}{n}\left( |z|^{-2}+n n^{-2/\alpha}+n^2 n^{-1-2/\alpha}\right)^{1/2}+\frac{\xi^2_i}{n^{1/2+1/\alpha}}\prec\frac{\xi_i^2}{n^{1/2+1/\alpha}},
\end{split}
\end{equation}
where in the third and fourth steps we used (\ref{eq_trivialcontrol}) and (\ref{eq_nontrivialcontrol}), respectively. 

Besides, by a discussion similar to (\ref{eq_standarddiscussion}),  we now have from Lemma \ref{lem:large deviation} that 
\begin{equation}\label{eq_cccccc11111}
\begin{split}
   T_i:=\frac{1}{n}{\rm tr} G^{(i)}\Sigma-m_1(z)&=\frac{1}{n}\mathbf{y}_i^{*}G \Sigma G^{(i)}\mathbf{y}_i=\frac{1}{n}\frac{\mathbf{y}^{*}_i G^{(i)}\Sigma G^{(i)}\mathbf{y}_i}{1+\mathbf{y}^{*}_i G^{(i)}\mathbf{y}_i}\\
  %  &\prec\frac{\xi^2_i}{n}\frac{n^{-1}\| G^{(i)}\|^2_F}{|1+\mathbf{y}^{*}_i G^{(i)}\mathbf{y}_i|}\\
    &\prec\frac{\xi^2_i}{n^2}|z||\mathcal G_{ii}|\left(\|\mathcal G^{(i)}\|^2_F+\frac{n}{n^{2/\alpha}} \right) \prec \frac{\xi_i^2}{n^{1+2/\alpha}}. 
\end{split}
\end{equation}
where in the second step we used the relation (\ref{eq_usefulformula}), in the third step we used Lemma \ref{lem: Resolvent} and in the last two steps we used a discussion similar to (\ref{eq_standarddiscussion}). 

With above estimations, we now use an idea similar to the proof of  Lemma \ref{lem: average local law for large eta} to conclude the proof. The key ingredient is the relation of $m_1$ and $m_2.$ We start with $m_2.$ Using Lemma \ref{lem: Resolvent}, we find that  
\begin{equation}\label{eq_eeeone}
\frac{1}{-z(1+\xi_i^2 m_1(z))}=\frac{1}{\mathcal{G}_{ii}^{-1}+z(Z_i+T_i)}.
\end{equation}
Consequently, by (\ref{eq_g11}), (\ref{eq_expansiontwo}), (\ref{eq_standarddiscussion}) and (\ref{eq_cccccc11111}), we have
\begin{equation}\label{eq_eeetwo}
\frac{1}{-z(1+\xi_i^2 m_1(z))} \prec n^{-1/\alpha}. 
\end{equation}
Then using the decomposition as in (\ref{eq: decomp m_2}), we find
\begin{gather}\label{eq_modeltemp}
    \begin{split}
        m_2&=\frac{1}{n}\frac{\xi^2_1}{-z(1+\xi^2_1n^{-1}\operatorname{tr} G^{(1)}\Sigma+Z_1)}+\frac{1}{n}\sum_{i=2}^n\frac{\xi^2_i}{-z(1+\xi^2_in^{-1}\operatorname{tr} G^{(i)}\Sigma+Z_i)}\\
        &=\frac{1}{n}\frac{\xi^2_1}{-z(1+\xi^2_1m_1(z)+\xi_1^2 n^{-1-2/\alpha}+Z_1)}+\frac{1}{n}\sum_{i=2}^n\frac{\xi^2_i}{-z(1+\xi^2_im_1(z)+\xi_i^2 n^{-1-2/\alpha}+Z_i)}\\
        &=\frac{1}{n}\sum_{i=1}^n\frac{\xi^2_i}{-z(1+\xi^2_im_1(z))}+\rO_{\prec}\Big(n^{-3/2}+\frac{1}{n}\sum_{i=2}^n\frac{\xi^4_i}{|z|n^{1/2+1/\alpha}}\Big)\\
 %       &=\frac{1}{n}\sum_{i=1}^n\frac{\xi^2_i}{-z(1+\xi^2_im_1(z))}+O_{\prec}(\frac{1}{np}n^{-2/\alpha+\epsilon}+\frac{1}{n^{3/2}}+n^{-1/2-2/\alpha})\\
        &=\frac{1}{n}\sum_{i=1}^n\frac{\xi^2_i}{-z(1+\xi^2_im_1(z))}+\rO_{\prec}(n^{-1/2-1/\alpha}),
    \end{split}
\end{gather}
where in the second step we used (\ref{eq_cccccc11111}), in the third step we used a discussion similar to (\ref{eq_g11control}) and in the last step we used (\ref{def1}). Next, recalling the definition of $m_1$ in \eqref{eq_m1m2}, we can also obtain the expression for $m_1$ using a parallel argument between (\ref{eq: decomp of mathcal_G}) and (\ref{eq: m_1 by m_2}). We provide only some key ingredients in the sequel. First, for $\mathtt{R}_{11}$ in (\ref{eq: decomp R_1}), note that by the definition of $m_2^{(i)}$ and (\ref{eq_trivialcontrol}), we have that   
\begin{equation*}
\begin{split}
   \big |m^{(i)}_2 \big |&\leq\frac{1}{n} \big(\sum_{j\neq i}\xi^2_j| \mathcal G^{(i)}_{jj}|+|z|^{-1} \big)\\
    &\leq\frac{1}{n} \Big[\sum_{j\neq 1}\xi^2_j\Big( | \mathcal G_{jj}|+\frac{|\mathcal G_{j1}||\mathcal G_{1j}|}{|\mathcal G_{11}|} \Big)+|z|^{-1} \Big] \prec n^{-1/\alpha},
    %\frac{1}{n}\sum_{j\neq1}\frac{\xi^2_j}{n^{2/\alpha}}\rightarrow0,
\end{split}
\end{equation*}
where in the second step we used Lemma \ref{lem: Resolvent} and in the last step we used (\ref{eq_expansionone}), (\ref{eq_expansiontwo}) and (\ref{def1}). Moreover, by Lemma \ref{lem: Resolvent} and (\ref{eq_expansiontwo}), we find that $(1+\mathbf{y}_i^* G^{(i)} \mathbf{y}_i)^{-1} \prec 1.$ Therefore, we conclude that for all $1 \leq i \leq n,$
\begin{equation*}
\begin{split}
  &   {\rm tr}\Big(\frac{ G^{(i)}(\mathbf{y}_i\mathbf{y}_i^{*}-n^{-1}\xi^2_i\Sigma)}{1+\mathbf{y}_i^{*} G^{(i)}\mathbf{y}_i}(I+m_2^{(i)}\Sigma)^{-1}\Sigma\Big)  \asymp {\rm tr} \left(\xi^2_i G^{(i)}(\mathbf{u}_i\mathbf{u}_i^{*}-n^{-1}I) (I+m_2^{(i)}\Sigma)^{-1} \Sigma^2 \right)\\
   &= \xi_i^2 \left( \mathbf{u}_i^* G^{(i)} (I+m_2^{(i)}\Sigma)^{-1} \Sigma^2 \mathbf{u}_i -n^{-1} {\rm tr} \left( G^{(i)} (I+m_2^{(i)}\Sigma)^{-1} \Sigma^2 \right) \right)\\
 & \prec \frac{\xi_i^2}{n}\|G^{(i)}\|_F\prec\frac{\xi_i^2}{n^{1/2+1/\alpha}},
\end{split}
\end{equation*}
where in the last step we used a discussion similar to (\ref{eq_standarddiscussion}). Together with (\ref{def1}), we can conclude that $\mathtt R_{11} \prec n^{-1/2-1/\alpha}.$ For $\mathtt R_{12},$ using (\ref{eq_decompositionleavoneout}), (\ref{eq_expansionone}) and (\ref{eq_expansiontwo}) 
\begin{equation*}
m_2-m^{(i)}_2 \prec\frac{1}{n}\sum_{j\neq i}\frac{\xi^2_in^{-1-2/\alpha}}{n^{-1/\alpha}} + \frac{1}{n}\prec n^{-1}.
\end{equation*}
Then by an argument similar to (\ref{eq_simimimimimi}), we can conclude that $\mathtt R_{12} \prec n^{-1-1/\alpha}.$ Similarly, for $R_2,$ we have that
\[
\begin{split}
    &\frac{1}{n} \left|{\rm tr}\left(\frac{(G^{(i)}-G)\Sigma(I+m_2\Sigma)^{-1}\Sigma}{1+\mathbf{y}_i^{*} G^{(i)}\mathbf{y}_i}\right) \right|=\frac{1}{n}\left|\frac{\mathbf{y}_i^{*} G^{(i)}\Sigma(I+m_2\Sigma)^{-1}\Sigma G\mathbf{y}_i}{1+\mathbf{y}_i^{*} G^{(i)}\mathbf{y}_i}\right|\\
    &\asymp\frac{1}{n}\left|\mathbf{y}_i^{*} G^{(i)}\Sigma(I+m_2\Sigma)^{-1}\Sigma G^{(i)}\mathbf{y}_i \right | \prec\frac{\xi_i^2}{n^2}\| G^{(i)}\|_F \| G \|_F \prec\frac{\xi_i^2}{n^{1+2/\alpha}}.
\end{split}
\]
Consequently, we have that 
\begin{equation*}
\frac{z}{n} \left|{\rm tr}(R_2\Sigma) \right|\prec\frac{1}{n}\sum_i\frac{\xi_i^2}{n^{1+1/\alpha}}\prec n^{-1-1/\alpha}.
\end{equation*}

Combining all the above arguments, we find that \eqref{eq: m_1 by m_2} still holds. Then, together with \eqref{eq_modeltemp}, we have already obtained that 
\begin{align*}
    m_2=\frac{1}{n}\sum_{i=1}^n\frac{\xi^2_i}{-z(1+\xi^2_im_1(z))}+\mathrm{O}_{\prec}(n^{-1/2-1/\alpha}),\\
    m_1=\frac{1}{n}\sum_{i=1}^p\frac{\sigma_i}{-z(1+m_2\sigma_i)}+\mathrm{O}_{\prec}(n^{-1/2-2/\alpha}).
\end{align*}
Then, using an argument similar to the discussions between (\ref{eq_closenessm2m2n}) and (\ref{eq_verbalttt}), we can conclude the proof. 
\end{proof}

\quad By Lemma \ref{lem: average local law for large eta} and Lemma \ref{lem: self improvement}, we now proceed to the proof of Proposition \ref{eq_propooursidebulk}.  We will use a continuity argument as in \cite[Lemma A.12]{Ding&Yang2018} or \cite[Section 4.1]{Alex2014}. In fact, our discussion is easier since the real part in the spectral domain $\widetilde{\mathbf{D}}_{u}$ is divergent so that the rate is independent of $\eta$. Due to similarity, we focus on explaining the key ingredients. 

\begin{proof}[\bf Proof of Proposition \ref{eq_propooursidebulk}]
 For each $z=E+\ri\eta\in\widetilde{\mathbf{D}}_{u}$, we fix the real part and construct a sequence $\{\eta_j\}$ by setting $\eta_j=\mathtt C \vartheta_1-j n^{-3}$. Then it is clear that $\eta$ falls in an interval $[\eta_{j-1},\eta_j]$ for some $0\leq j\leq Cn^{1/\alpha+3}\log n, C>0$ is some large constant.

In Lemma \ref{lem: average local law for large eta}, we have proved that the results hold  for $\eta_0$. Now we assume (\ref{eq_entrywiselocallaw}) holds for some $\eta_j$. Then according to Lemma \ref{lem: self improvement}, we have that
\begin{equation*}
|m_1(z_j)-m_{1n}(z_j)|+|m_Q(z_j)-m_n(z_j)| \prec n^{-1/2-2/\alpha}, \ |m_2(z_j)-m_{2n}(z_j)| \prec n^{-1/2-1/\alpha}. 
\end{equation*}
%{\color{red} here}
%{\color{blue} Now assume $\max_{ij}(G+z^{-1}(I+m_{1n}(z)D^2))\prec n^{-1/2-2/\alpha}$ holds } 
For any $\eta^{\prime}$ lying in the interval $[\eta_{k-1},\eta_k]$, denote $z^{\prime}=E+\ri\eta^{\prime}$ and $z_j=E+\ri\eta_j$. According to the first resolvent identity in  Lemma \ref{lem: first resolvent}, we have that 
 \begin{equation} \label{eq: Lip for mathcal G}
 \| \mathcal G(z')-\mathcal{G}(z_j) \| \leq n^{-3} \| \mathcal{G}(z') \| \| \mathcal{G}(z_j) \| \prec n^{-11/6-1/\alpha},
\end{equation}  
where in the second step we used the basic bound $\|\mathcal G(z^{\prime})\|\leq n^{2/3},$ (\ref{eq_expansionone}), (\ref{eq_expansiontwo}) and Gershgorin circle theorem.

On the one hand, according to the definitions in (\ref{eq_m1m2}), using the first resolvent identity, we have that  
\begin{equation*}
\begin{split}
    m_1(z^{\prime})-m_1(z_j)&=\frac{1}{n}{\rm tr}[(G(z^{\prime})-G(z_j))\Sigma]=\frac{1}{n^4}{\rm tr}(G(z^{\prime})G(z_j)\Sigma) \prec\frac{1}{n^4 \eta_j}\|G(z_j)\|_F \prec n^{-17/6-1/\alpha},
\end{split}
\end{equation*}
where in the last step we used a discussion similar to (\ref{eq_standarddiscussion}). Similarly, combining (\ref{eq: Lip for mathcal G}) and (\ref{def1}),  we have that  
\begin{equation*}
    m_2(z^{\prime})-m_2(z_j)=\frac{1}{n}\sum_{i=1}^n\xi^2_i(\mathcal G_{ii}(z^{\prime})-\mathcal G_{ii}(z_j))
\prec n^{-11/6-1/\alpha},
\end{equation*}
and by a discussion similar to (\ref{eq_standarddiscussion})
\begin{align*}
|m_Q(z')-m_Q(z_j)|& =\frac{1}{p} \left| \operatorname{tr} \left( G(z')-G(z_j) \right) \right|  \leq n^{-4} \|G(z') \|_F \| G(z_j) \|_F \\
& \prec n^{-4} (n^{1-2/\alpha})^{1/2} n^{1/2+2/3}=n^{-7/3-1/\alpha}. 
\end{align*}
 On the other hand, using the definitions in (\ref{eq_systemequationsm1m2}), we decompose that 
\[
\begin{split}
    m_{1n}(z^{\prime})-m_{1n}(z_j)
    &=\frac{1}{n}\sum_i\Big(\frac{\sigma_i}{-z^{\prime}(1+\sigma_im_{2n}(z^{\prime}))}-\frac{\sigma_i}{-z^{\prime}(1+\sigma_im_{2n}(z_j))}\Big)\\
    &+\frac{1}{n}\sum_i\Big(\frac{\sigma_i}{-z^{\prime}(1+\sigma_im_{2n}(z_j))}-\frac{\sigma_i}{-z_j(1+\sigma_im_{2n}(z_j))}\Big)\\
    &:=\mathcal{M}_{11}+\mathcal{M}_{12}.
\end{split}
\]
For $\mathcal{M}_{11}$, according to Lemma \ref{lem: basic bounds} and (\ref{eq_systemequationsm1m2}), we readily obtain that
\[
\begin{split}
    \mathcal{M}_{11}&=\frac{1}{n}\sum_{i}\frac{\sigma^2_i}{-z^{\prime}(1+\sigma_im_{2n}(z^{\prime}))(1+\sigma_im_{2n}(z_j))}\big(m_{2n}(z_j)-m_{2n}(z^{\prime})\big)\\
    &=\rO(|z'|^{-1})\times\frac{1}{n}\sum_i\Big(\frac{\xi^2_i}{-z_j(1+\xi^2_im_{1n}(z_j))}-\frac{\xi^2_i}{-z_j(1+\xi^2_im_{1n}(z^{\prime}))}\Big)\\
    &+\rO(|z'|^{-1})\times\frac{1}{n}\sum_i\Big(\frac{\xi^2_i}{-z_j(1+\xi^2_im_{1n}(z^{\prime}))}-\frac{\xi^2_i}{-z^{\prime}(1+\xi^2_im_{1n}(z^{\prime}))}\Big)\\
    &=\rO(|z'|^{-1}) \times \left( \mathtt{M}_{11,1}+ \mathtt{M}_{11,2} \right). 
\end{split}
\]

For $\mathtt{M}_{11,1},$ by a discussion similar to (\ref{eq_g11}) and (\ref{eq_expansiontwo}), we find that 
\[
\begin{split}
    \mathtt{M}_{11,1}&=\frac{1}{n}\sum_{i=1}^n\frac{\xi_i^4}{-z_j (1+\xi^2_im_{1n}(z^{\prime}))(1+\xi^2_im_{1n}(z_j))}(m_{1n}(z')-m_{1n}(z_j))\\
    &\prec\Big(\frac{\xi_1^4}{-nz_j d^2_1 |m_{1n}(z_j)||m_{1n}(z')|}+\frac{1}{n}\sum_{i\ge2}\frac{\xi_i^4}{|z^{\prime}(1+\xi^2_1m_{1n}(z^{\prime}))(1+\xi^2_im_{1n}(z_j))|}\Big)\times|m_{1n}(z_j)-m_{1n}(z^{\prime})|\\
    &\prec \ro(1) \times |m_{1n}(z_j)-m_{1n}(z^{\prime})|.
\end{split}
\]
Similarly, for $\mathtt{M}_{11,2},$ we have that 
\[
\begin{split}
    \mathtt{M}_{11,2}=\frac{1}{n}\sum_{i}\frac{\xi_i^2 (z_j-z^{\prime})}{z^{\prime}z_j(1+\xi_i^2 m_{1n}(z'))} \prec \frac{z^{\prime}-z_j}{z^{\prime}z_j} \prec n^{-3-2/\alpha}.
\end{split}
\]
Analogously, we can prove that $\mathcal{M}_{12} \prec n^{-3-4/\alpha}$. 
Therefore, combining the above bounds with (\ref{def1}), we see that $|m_{1n}(z^{\prime})-m_{1n}(z_j)|\prec  n^{-3-2/\alpha}.$

\quad By similar procedures and arguments, we can also prove that 
\begin{equation*}
    |m_{2n}(z^{\prime})-m_{2n}(z_j)|\prec n^{-3-2/\alpha}, \  |m_{n}(z^{\prime})-m_{n}(z_j)|\prec n^{-3-2/\alpha}, 
\end{equation*}
\begin{gather*}
    \left\|(z^{\prime})^{-1}(I+m_{1n}(z^{\prime})D^2)^{-1}-(z_j)^{-1}(I+m_{1n}(z_j)D^2)^{-1}\right\|\prec n^{-3-2/\alpha}.
\end{gather*}

Therefore, combining all the above bounds with triangle inequality, we see that the results of part 1 of Theorem \ref{thm_unboundedcaselocallaw} hold for $z'.$ Using an induction procedure and a standard lattice argument (for example, see \cite{Alex2014,Ding&Yang2018}), we find that the results hold for all $z \in \widetilde{\mathbf{D}}_{u}$ and conclude the proof of Proposition \ref{eq_propooursidebulk}. 
\end{proof}

\subsubsection{Proof of Theorem \ref{thm_unboundedcaselocallaw}} \label{sec_proofa2}

Once Proposition \ref{eq_propooursidebulk} is proved, we can roughly locate the edge eigenvalues of $Q$ according to Lemma \ref{lem: upper bound for eigenvalues}, so that we can expand the spectral domain from $\widetilde{\mathbf{D}}_{u}$ to $\mathbf{D}_{u}$ for $Q^{(1)}$ and conclude the proof of Theorem \ref{thm_unboundedcaselocallaw}.

\quad Recall  the definitions of {\color{blue} $\vartheta_2$} and $\lambda_1^{(1)}$ around (\ref{eq_keylocationdefinition}) and in Figure \ref{fig_locationtizubounded}. By Lemma \ref{lem: upper bound for eigenvalues} and an analogous argument, as well as  
Weyl's inequality, we find that conditional on the event $\Omega_D,$ with high probability, 
\begin{equation}\label{eq_keyused}
    \vartheta_1>\lambda_1>\vartheta_2>\lambda_1^{(1)}.
\end{equation}
By (\ref{eq_mu1part}) and a similar argument, we see that $\vartheta_k \asymp \xi_k^2, k=1,2.$ Together with (\ref{def1}), we have that $\vartheta_1-\vartheta_2\geq C_1 n^{1/\alpha}\log^{-1}n$ for some constant $C_1>0$ on the event $\Omega_D.$ This implies for some constant $C>0,$ for $z \in \mathbf{D}_{u},$ 
\begin{gather}\label{eq: eigen gap}
    |\lambda_1^{(1)}-z|\geq Cn^{1/\alpha}\log^{-1}n,
\end{gather}
Now we proceed to the proof of Theorem \ref{thm_unboundedcaselocallaw}. Recall (\ref{eq_defnminor}) and (\ref{eq_defnminorG}). 

\begin{proof}[\bf Proof of Theorem \ref{thm_unboundedcaselocallaw}]  
Observe  by \eqref{eq: eigen gap} that it holds uniformly for $z\in\mathbf{D}_{u}$ and $\mathcal{T}\subset\{2,\dots,n\}$, for some constant $C_1>0$
\begin{gather}\label{eq: est of G^(1)}
    \|G^{(1\mathcal{T})}\|\leq C_1 n^{-1/\alpha} \log n.
\end{gather}
By the definition of $m_2^{(1)}$ and a decomposition similar to (\ref{eq: decomp m_2}),  we have that 
\begin{gather*}
        m_2^{(1)}=\frac{1}{n}\sum_{i=2}^n\frac{\xi^2_i}{-z-z\mathbf{y}_i^{*}G^{(1i)}\mathbf{y}_i}=\frac{1}{n}\sum_{i=2}^n\frac{\xi^2_i}{-z(1+\xi^2_in^{-1}\operatorname{tr}G^{(1i)}\Sigma+Z_i^{(1)})},\\
        Z_i^{(1)}=\mathbf{y}_i^{*}G^{(1i)}\mathbf{y}_i-\xi^2_in^{-1}\operatorname{tr}G^{(1i)}\Sigma.
\end{gather*}
By arguments similar to (\ref{eq: bound for Z}) and (\ref{eq_m1approximate}) but with \eqref{eq: est of G^(1)}, we obtain that 
\begin{gather*}
    Z_i^{(1)}\prec\frac{\xi^2_i}{n}\|G^{(1i)}\Sigma\|_F\leq\frac{\xi^2_i}{n}\|G^{(1i)}\|\|\Sigma\|_F\prec\frac{\xi^2_i}{n^{1/2}}n^{-1/\alpha},\\
    \frac{1}{n}\operatorname{tr}(G^{(1i)}\Sigma)-m_1^{(1)}(z)=\frac{1}{n}\mathbf{y}_i^{*}G^{(1)}\Sigma G^{(1i)}\mathbf{y}_i\prec \frac{\xi^2_i}{n}n^{-2/\alpha}.
\end{gather*}
In addition, using (\ref{eq: est of G^(1)}) and a discussion similar to (\ref{eq_eeeone})--(\ref{eq_modeltemp}), we readily see that 
\begin{gather*}
    m^{(1)}_2=\frac{1}{n}\sum_{i=2}^n\frac{\xi^2_i}{-z(1+\xi^2_im^{(1)}_1)}+\rO_{\prec}(n^{-1/2-1/\alpha}).
\end{gather*}
Using the decomposition 
\begin{gather*}
    Q^{(1)}-zI=\sum_{i=2}^n\mathbf{y}_i\mathbf{y}_i^{*}+zm^{(1)}_2(z)\Sigma-z(I+m_2^{(1)}(z)\Sigma),
\end{gather*}
by arguments similar to (\ref{eq: decomp of mathcal_G})--(\ref{eq: m_1 by m_2}) with $\|\mathcal{G}^{(1\mathcal{T})}\|=\|G^{(1 \mathcal{T})}\|\prec n^{-1/\alpha}$, we conclude that 
\begin{gather*}
\begin{split}
    m_1^{(1)}&=-z^{-1}\frac{1}{n}\operatorname{tr}((I+m_2^{(1)}(z)\Sigma)^{-1}\Sigma)+\rO_{\prec}(n^{-1/2-2/\alpha})\\
    &=-\frac{1}{n}\sum_{i=1}^p\frac{\sigma_i}{z(1+m_2^{(1)}\sigma_i)}+\rO_{\prec}(n^{-1/2-2/\alpha}).
\end{split}
\end{gather*}
Combining with the definitions in (\ref{eq_systemequationsm1m2}), we see that
\begin{gather*}
    \begin{split}
      &  m_1^{(1)}(z)-m_{1n}^{(1)}(z)=-\frac{1}{n}\sum_{i=1}^p\frac{\sigma_i}{z(1+\sigma_im_2^{(1)}(z))}+\frac{1}{n}\sum_{i=1}^p\frac{\sigma_i}{z(1+\sigma_im_{2n}^{(1)}(z))}+\rO_{\prec}(n^{-1/2-2/\alpha})\\
        &=\frac{1}{n}\sum_{i=1}^p\frac{\sigma_i^2(m_2^{(1)}(z)-m_{2n}^{(1)}(z))}{z(1+\sigma_im_{2n}^{(1)}(z))(1+\sigma_im_2^{(1)}(z))}+\rO_{\prec}(n^{-1/2-2/\alpha})\\
        &=\Big(\frac{1}{n}\sum_{i=1}^p\frac{\sigma_i^2}{z(1+\sigma_im_{2n}^{(1)}(z))(1+\sigma_im_2^{(1)}(z))}\Big)\Big(\frac{1}{n}\sum_{i=2}^n\frac{\xi^4_i(m_1^{(1)}(z)-m_{1n}^{(1)}(z))}{z(1+\xi^2_im_1^{(1)}(z))(1+\xi^2_im_{1n}^{(1)}(z))} \Big)+\rO_{\prec}(n^{-1/2-2/\alpha})\\
        &=\ro(1)(m_1^{(1)}(z)-m_{1n}^{(1)}(z))+\rO_{\prec}(n^{-1/2-2/\alpha}),
    \end{split}
\end{gather*} 
where in the third step we used a discussion similar to (\ref{eq_eeetwo}) and (\ref{def1}). 
\end{proof}

\subsection{Unbounded support and conditional on $X$ setting: proof of Theorem \ref{thm_averagedlocallaw_unboundedmultiplier_conditionalonX}}

In this section, we prove Theorem \ref{thm_averagedlocallaw_unboundedmultiplier_conditionalonX} 
by an argument similar to that used in the proof of 
Theorem \ref{thm_unboundedcaselocallaw}. Recall that in the proof of 
Theorem \ref{thm_unboundedcaselocallaw}, we first establish the local laws for $Q$ 
outside the bulk of the spectrum on the domain $\widetilde{\mathbf{D}}_u$, which 
allows us to roughly locate the edge eigenvalues of $Q$. In the second step, we 
extend the spectral domain from $\widetilde{\mathbf{D}}_u$ to $\mathbf{D}_u$ for 
$Q^{(1)}$, thereby completing the proof. The proof of Theorem \ref{thm_averagedlocallaw_unboundedmultiplier_conditionalonX} 
follows the same strategy, except that the error bounds become deterministic when 
working on the event $\Omega_X$ and conditioning on a realization of 
$\{\xi_i^2\}$.

To this end, the following result is parallel to 
Proposition \ref{eq_propooursidebulk}.

\begin{proposition}\label{lem_ousidebulk_conditional}
    Under the assumptions of Theorem \ref{thm_averagedlocallaw_unboundedmultiplier_conditionalonX}, the following results hold uniformly on the spectral domain $\widetilde{\mathbf{D}}_u$ when restricted on $\Omega_X$ and $\Omega_D$,
    \begin{enumerate}
    \item[(1).] If {\normalfont Case (a)} of (i) of  Assumption \ref{assum_D} holds, we have that 
    \begin{align}
        \mathcal{G}_{ij}(z)=-\frac{\delta_{ij}}{z\big(1+m_{1n}(z)\xi_i^2\big)}+\mathrm{O}( n^{\epsilon_1}n^{-1/2-1/\alpha} ),
    \end{align}
    where $\delta_{ij}$ is the Dirac delta function and $\epsilon_1>\varepsilon_1$ in Definition \ref{def_OmegaX} is some small constant. Moreover, we have that 
    \begin{equation*}
        m_1(z)=m_{1n}(z)+\mathrm{O}( n^{\epsilon_2}n^{-1/2-2/\alpha}), \  m_2(z)=m_{2n}(z)+\mathrm{O}( n^{\epsilon_2}n^{-1/2-1/\alpha}), 
    \end{equation*}
    and 
    \begin{equation}
        m_{Q}(z)=m_{n}(z)+\mathrm{O}( n^{\epsilon_2}n^{-1/2-2/\alpha} ),
    \end{equation}
    for some small constant $\epsilon_2>\varepsilon_1$.
    \item[(2).] If {\normalfont Case (b)} of (i) of Assumption \ref{assum_D} holds, we have that the results in part (1) hold by setting $\alpha=\infty.$
\end{enumerate}  
\end{proposition}
\begin{proof}
The proof is similar to that of Proposition \ref{eq_propooursidebulk}, except that we need to restrict our discussion to $\Omega_X$ as defined in Definition \ref{def_OmegaX}. We therefore focus only on the key differences. We first establish Proposition \ref{lem_ousidebulk_conditional} for large $\eta=\mathtt{C}\vartheta_1$, as in the proof of Lemma \ref{lem: average local law for large eta}, where $\mathtt{C}>0$ is a sufficiently large constant. On the event $\Omega_X$ in Definition \ref{def_OmegaX}, we have
    \begin{align*}
        Z_i=\mathbf{y}_i^*G^{(i)}\mathbf{y}_i-\xi^2_in^{-1}\operatorname{tr}G^{(i)}\Sigma=\mathrm{O}(n^{\varepsilon_1}\frac{\xi^2_i}{\sqrt{n}\eta}).
    \end{align*}
%    for any sufficiently small constant $\varepsilon_1>0$ in Definition \ref{def_probmeasure_conditionalX}. 
    In the sequel, to simplify notation, we use the same symbol $\epsilon\geq \varepsilon_1$ to denote a generic small positive constant, whose value may vary from line to line. It then follows that
    \begin{align*}
        m_2=\frac{1}{n}\sum_{i=1}^n\frac{\xi^2}{-z(1+\xi^2_im_1)}+\mathrm{O}(n^{\epsilon}n^{-1/2-1/\alpha}).
    \end{align*}
    On the other hand, we decompose $G$ as
    \begin{align*}
        G=z^{-1}(I+m_2(z)\Sigma)^{-1}+R_1+R_2,
    \end{align*}
   with estimates parallel to those in the proof of Lemma \ref{lem: average local law for large eta} that
    \begin{align*}
        \frac{z}{n}|\operatorname{tr}(R_1\Sigma)|\leq C_1n^{\epsilon}n^{-1/2-1/\alpha},\quad \frac{z}{n}|\operatorname{tr}(R_2\Sigma)|\leq C_2n^{\epsilon}n^{-1-2/\alpha},
    \end{align*}
    for some constants $C_1,C_2>0$. Hence,
    \begin{align*}
        m_1(z)=-\frac{1}{n}\sum_{i=1}^p\frac{\sigma_i}{z(1+m_2(z)\sigma_i)}+\mathrm{O}(n^{\epsilon} n^{-1/2-2/\alpha}).
    \end{align*}
    Consequently, for $z\in\widetilde{\mathbf{D}}_u$ with $\eta=\mathtt{C}\vartheta_1$, we obtain
    \begin{gather*}
        |m_1(z)-m_{1n}(z)|=\mathrm{O}(n^{\epsilon}n^{-1/2-2/\alpha}),\quad |m_2(z)-m_{2n}(z)|=\mathrm{O}(n^{\epsilon}n^{-1/2-1/\alpha}),\\
        |m_{\mathcal{Q}}(z)-m_n(z)|=\mathrm{O}(n^{\epsilon}n^{-1/2-2/\alpha}).
    \end{gather*}
    Moreover, at $\eta=\mathtt{C}\vartheta_1$, we also have
    \begin{align}\label{eq_priorinput_averagelocallaw_conditional}
        \left|\mathcal{G}_{ii}(z)+\frac{1}{z(1+\xi^2_im_{1n}(z))}\right|=\mathrm{O}(n^{\epsilon}n^{-1/2-1/\alpha}),\quad |\mathcal{G}_{ij}(z)|=\mathrm{O}(n^{\epsilon}n^{-1/2-2/\alpha}).
    \end{align}
   This proves Proposition \ref{lem_ousidebulk_conditional} for $\eta=\mathtt{C}\vartheta_1$.

  Next, suppose that \eqref{eq_priorinput_averagelocallaw_conditional} holds uniformly for $z\in\widetilde{\mathbf{D}}_u$. Then, for some constant $C>0$,
\begin{align*}
    |\mathcal{G}_{11}(z)|=\frac{1}{|zd_1m_{1n}(\vartheta_1)|}+\mathrm{O}(n^{\epsilon}n^{-1/2-1/\alpha})=\frac{C}{d_1}+\mathrm{O}(n^{\epsilon}n^{-1/2-1/\alpha})=\mathrm{O}(n^{\epsilon}n^{-1/\alpha}).
\end{align*}
By the same argument, we also have $|\mathcal{G}_{ii}(z)|=\mathrm{O}(n^{\epsilon}n^{-1/\alpha})$. In addition, a direct calculation yields, for $1\leq i\leq n$,
\begin{align*}
    |\mathcal{G}_{kk}^{(i)}|=\mathrm{O}(n^{\epsilon}n^{-1/\alpha}),\;k\neq i;\quad |\mathcal{G}_{kl}^{(i)}|=\mathrm{O}(n^{\epsilon}n^{-1/2-1/\alpha}), \; k,l\neq i.
\end{align*}
Taking these Green function bounds as prior inputs and repeating the same argument as in the large $\eta$ case above, we conclude that the estimates in Proposition \ref{lem_ousidebulk_conditional} hold uniformly for $z\in\widetilde{\mathbf{D}}_u$.

Finally, we apply the same continuity argument on the lattice $\eta_j=\mathtt{C}\vartheta_1-jn^{-3}, 0\leq j\leq Cn^{1/\alpha+3}\log n$ as in the proof of Proposition \ref{eq_propooursidebulk}. Especially, for $z^{\prime}=E+\mathrm{i}\eta^{\prime}$ with $\eta^{\prime}\in[\eta_{j-1},\eta_j]$, we have
\begin{gather*}
    |m_{1n}(z^{\prime})-m_{1n}(z_j)|=\mathrm{O}(n^{\epsilon}n^{-3-2/\alpha}),\quad |m_{2n}(z^{\prime})-m_{2n}(z_j)|=\mathrm{O}(n^{\epsilon}n^{-3-2/\alpha})\\
    |m_{n}(z^{\prime})-m_{n}(z_j)|=\mathrm{O}(n^{\epsilon}n^{-3-2/\alpha}),
\end{gather*}
and 
\begin{align*}
    \|(z^{\prime})^{-1}(I+m_{1n}(z^{\prime})D^2)^{-1}-(z_j)^{-1}(I+m_{1n}(z_j)D^2)^{-1}\|=\mathrm{O}(n^{\epsilon}n^{-3-2/\alpha}).
\end{align*}
%Combining an induction argument with the standard lattice continuity argument, we conclude the proof of the lemma.
\end{proof}

The following lemma is the conditional counterpart Lemma \ref{lem: upper bound for eigenvalues}. 
\begin{lemma}\label{lem_upperbound_unboundedeigenvalue_conditional}
Suppose Assumptions \ref{assum_model}, \ref{assumption_techincial} and (i) of Assumption \ref{assum_D} hold. Under $\Omega_X$, for some sufficiently large constant $C>0,$ fix any realization $\{\xi_i^2\} \in \Omega_D$, we have that for all $1 \leq i \leq \min\{p,n\},$
\begin{equation}
\lambda_i(Q) \notin (\vartheta_1, C n^{1/\alpha} \log n), \ \text{if \normalfont{Case (i)-a} of Assumption \ref{assum_D} holds},  
\end{equation}
and 
\begin{equation}
\lambda_i(Q) \notin (\vartheta_1, C  \log^{1/\beta} n), \ \text{if {\normalfont Case (i)-b of Assumption \ref{assum_D} holds}}.  
\end{equation}
%
%\ref{ass1},\ref{ass2},\ref{ass3.1},\ref{ass3.2} and events $\Omega_n$ hold. With high probability there is no eigenvalue of $\mathcal{W}$ in $(\lambda_{(1)},Cn^{2/\alpha}\log n)$ for $\alpha\in(0,+\infty)$ and no eigenvalues in $(\lambda_{(1)},C\log^{K}n)$ for $\alpha=+\infty$.
\end{lemma}
\begin{proof}
   The proof again proceeds by contradiction, following the same argument as in the proof of Lemma \ref{lem: upper bound for eigenvalues}. The only difference is that, if we assume that $\widehat{\lambda}$ is an eigenvalue of $Q$ lying in the interval $(\vartheta_1, C n^{1/\alpha}\log n)$, then Proposition \ref{lem_ousidebulk_conditional} yields
    \begin{align*}
        \operatorname{Im}m_Q(z)=\operatorname{Im}m_n(z)+\operatorname{Im}(m_Q(z)-m_n(z))=\mathrm{O}(n^{\epsilon}n^{-1/2-2/\alpha}),
    \end{align*}
    for $z=\widehat{\lambda}+\mathrm{i}n^{-2/3}\in\widetilde{\mathbf{D}}_u$. The remainder of the proof is similar to that of Lemma \ref{lem: upper bound for eigenvalues}, and we therefore omit the details.
\end{proof}

With the above preparation, we proceed to conclude the proof of Theorem \ref{thm_averagedlocallaw_unboundedmultiplier_conditionalonX}
following the arguments in Section \ref{sec_proofa2}.

\begin{proof}[\bf Proof of Theorem \ref{thm_averagedlocallaw_unboundedmultiplier_conditionalonX}]
By Proposition \ref{lem_ousidebulk_conditional} and Lemma \ref{lem_upperbound_unboundedeigenvalue_conditional}, we can enlarge the spectral domain for $Q^{(1)}$ from $\widetilde{\mathbf{D}}_u$ to $\mathbf{D}_u$. We first note that, on the events $\Omega_X$ and $\Omega_D$,
\begin{align*}
    \vartheta_1>\lambda_1>\vartheta_2>\lambda_1^{(1)}.
\end{align*}
Since $\vartheta_1-\vartheta_2\geq C_1 n^{1/\alpha}\log^{-1}n$ for some constant $C_1>0$, it follows that for any $z\in\mathbf{D}_u$, $|\lambda_1^{(1)}-z|\geq C_2 n^{1/\alpha}\log^{-1}n$ for some constant $C_2>0$. This spectral separation yields the a priori bound 
\begin{align*}
    \|G^{(1\mathcal{T})}(z)\|\leq C_3n^{-1/\alpha}\log n,
\end{align*}
uniformly for $z\in\mathbf{D}_u$ and $\mathcal{T}\subset\{2,\dots,n\}$, with some constant $C_3>0$. Next, by the definition of $m_2^{(1)}$ and a decomposition analogous to \eqref{eq: decomp m_2}, we have
\begin{gather*}
        m_2^{(1)}=\frac{1}{n}\sum_{i=2}^n\frac{\xi^2_i}{-z-z\mathbf{y}_i^{*}G^{(1i)}\mathbf{y}_i}=\frac{1}{n}\sum_{i=2}^n\frac{\xi^2_i}{-z(1+\xi^2_in^{-1}\operatorname{tr}G^{(1i)}\Sigma+Z_i^{(1)})},\\
        Z_i^{(1)}=\mathbf{y}_i^{*}G^{(1i)}\mathbf{y}_i-\xi^2_in^{-1}\operatorname{tr}G^{(1i)}\Sigma.
\end{gather*}
On the event $\Omega_X$, using the above a priori bound on $|G^{(1i)}(z)|$, we obtain
\begin{gather*}
    |Z_i^{(1)}|\leq C n^{\epsilon}\frac{\xi^2_i}{n}\|G^{(1i)}\Sigma\|_F=\mathrm{O}\big(n^{\epsilon}\frac{\xi^2_i}{n^{1/2}}n^{-1/\alpha}\big),\\
    |\frac{1}{n}\operatorname{tr}(G^{(1i)}\Sigma)-m_1^{(1)}(z)|=|\frac{1}{n}\mathbf{y}_i^{*}G^{(1)}\Sigma G^{(1i)}\mathbf{y}_i|=\mathrm{O}\big(n^{\epsilon}\frac{\xi^2_i}{n}n^{-2/\alpha}\big).
\end{gather*}
Therefore
\begin{align*}
    m_2^{(1)}=\frac{1}{n}\sum_{i=2}^n\frac{\xi^2_i}{-z(1+\xi^2_im_1^{(1)})}+\mathrm{O}(n^{\epsilon}n^{-1/2-1/\alpha}).
\end{align*}
By the same argument, we also obtain
\begin{align*}
    m_1^{(1)}=-\frac{1}{n}\sum_{i=1}^p\frac{\sigma_i}{z(1+m_2^{(1)}\sigma_i)}+\mathrm{O}(n^{\epsilon}n^{-1/2-2/\alpha}).
\end{align*}
Finally, we observe that
\begin{align*}
    &m_1^{(1)}(z)-m_{1n}^{(1)}(z)=-\frac{1}{n}\sum_{i=1}^p\frac{\sigma_i}{z(1+m_2^{(1)}\sigma_i)}+\frac{1}{n}\sum_{i=1}^p\frac{\sigma_i}{z(1+m_{2n}^{(1)}\sigma_i)}+\mathrm{O}(n^{\epsilon}n^{-1/2-2/\alpha})\\
    &=\mathrm{o}(1)(m_1^{(1)}(z)-m_{1n}^{(1)}(z))+\mathrm{O}(n^{\epsilon}n^{-1/2-2/\alpha}).
\end{align*}
This completes the proof.
\end{proof}

\subsection{Bounded support and unconditional setting: proof of Theorem \ref{thm_boundedcaselocallaw}}\label{sec_proof_bounded}
In this section, we will prove Theorem \ref{thm_boundedcaselocallaw}. In Section \ref{sec_proofpartiboundedlocallaw}, we study $m_{1n,c}(z)-m_{1n}(z)$ and $m_{n,c}-m_n,$ which is a counterpart of Lemma 4.4 of \cite{Kwak2021}. Then in Section \ref{sec_proofpartiiboundedlocallaw2}, we study $m_{1n}(z)-m_{1}(z)$ and $m_Q-m_n,$ which is a counterpart of Proposition 5.1 of \cite{Kwak2021}. 

\subsubsection{Control of $m_{1n,c}(z)-m_{1n}(z)$ and $m_{n,c}(z)-m_n(z)$}\label{sec_proofpartiboundedlocallaw}

Due to similarity, we focus on $|m_{1n,c}-m_{1n}|$ and briefly discuss $|m_{n,c}-m_n|$ in the end. The proof ideas follow  Lemma 4.5 of \cite{lee2016extremal} or Lemma 4.4 of \cite{Kwak2021}. We focus on the parts that differ the most from the argument in \cite{lee2016extremal,Kwak2021}.

\begin{proof}
According to the definitions of $m_{1n,c}$ and $m_{1n}$ in (\ref{eq_systemequationsm1m2intergrate}) and (\ref{eq_systemequationsm1m2}), we observe that 
\begin{align}\label{eq_p1plusp2}
    &|m_{1n,c}(z)-m_{1n}(z)|\\
    & \leq \Big|\frac{1}{n}\sum_{i=1}^p\frac{\sigma_i}{-z+\sigma_i\int\frac{s}{1+sm_{1n,c}(z)}\mathrm{d}F(s)}-\frac{1}{n}\sum_{i=1}^p\frac{\sigma_i}{-z+\frac{\sigma_i}{n}\sum_{j=1}^n\frac{\xi^2_j}{1+\xi^2_jm_{1n,c}(z)}}\Big| \nonumber \\
    &+|m_{1n,c}(z)-m_{1n}(z)|\Big|\frac{1}{n}\sum_{i=1}^p\frac{\frac{\sigma_i^2}{n}\sum_{j=1}^n\frac{\xi^4_j}{(1+\xi^2_jm_{1n}(z))(1+\xi^2_jm_{1n,c}(z))}}{(-z+\frac{\sigma_i}{n}\sum_{j=1}^n\frac{\xi^2_j}{1+\xi^2_jm_{1n}(z)})(-z+\frac{\sigma_i}{n}\sum_{j=1}^n\frac{\xi^2_j}{1+\xi^2_jm_{1n,c}(z)})}\Big| \\ \nonumber
    &:=  \mathsf{P}_1+\mathsf{P}_2. 
\end{align}

On the one hand, for $\mathsf{P}_1, $ we have that 
\begin{equation*}
\mathsf{P}_1=\frac{1}{n}\sum_{i=1}^p\frac{\sigma_i^2|n^{-1}\sum_j\frac{\xi^2_j}{1+\xi^2_jm_{1n,c}(z)}-\int\frac{s}{1+sm_{1n,c}(z)}\mathrm{d}F(s)|}{|(-z+\frac{\sigma_i}{n}\sum_{j=1}^n\frac{\xi^2_j}{1+\xi^2_jm_{1n,c}(z)})(-z+\sigma_i\int\frac{s}{1+sm_{1n,c}(z)}\mathrm{d}F(s))|}.
\end{equation*}
Since $z \in \mathbf{D}_b^\prime \subset \mathbf{D}_b,$ according to Assumption \ref{assum_additional_techinical} and the continuity of $m_{1n,c}$, 
we conclude that $|-z+\sigma_i\int\frac{s}{1+sm_{1n,c}(z)}\mathrm{d}F(s))| \geq c$ for some constant $c>0.$ Moreover, by definition of $\mathbf{D}_b^\prime$, together with (\ref{def4}), we can show that $|-z+\frac{\sigma_i}{n}\sum_{j=1}^n\frac{\xi^2_j}{1+\xi^2_jm_{1n,c}(z)}| \geq c'$ for some $c'>0$ when $n$ is sufficiently large. Using (\ref{def4}) again, we conclude, on the event $\Omega_D,$ for some small constant $\epsilon>0$ and some constant $C>0$ 
\begin{equation}\label{eq_defnp1}
\mathsf{P}_1 \leq  C n^{-1/2+\epsilon}. 
\end{equation} 

\quad On the other hand, for $\mathsf{P}_2,$ for notional convenience, we further write it as $\mathsf{P}_2=|m_{1n,c}(z)-m_{1n}(z)| \times |\mathsf{T}|.$ For $\mathsf{T},$ by Cauchy-Schwarz inequality, we have that
\begin{equation}\label{eq_tintot1t2}
|\mathsf{T}| \leq \mathsf{E}_1 \mathsf{E}_2,
\end{equation}
where $\mathsf{E}_k, k=1,2,$ are defined as 
\begin{gather}
    \begin{split}
        &\mathsf{E}_1:= \Big(\frac{1}{n}\sum_{i=1}^p\frac{\frac{\sigma^2_i}{n}\sum_{j=1}^n\frac{\xi^4_j}{(1+\xi^2_jm_{1n}(z))^2}}{|-z+\frac{\sigma_i}{n}\sum_{j=1}^n\frac{\xi^2_j}{1+\xi^2_jm_{1n}(z)}|^2}\Big)^{1/2},\\
        &\mathsf{E}_2:=\Big(\frac{1}{n}\sum_{i=1}^p\frac{\frac{\sigma^2_i}{n}\sum_{j=1}^n\frac{\xi^4_j}{(1+\xi^2_jm_{1n,c}(z))^2}}{|-z+\frac{\sigma_i}{n}\sum_{j=1}^n\frac{\xi^2_j}{1+\xi^2_jm_{1n,c}(z)}|^2}\Big)^{1/2}.
    \end{split}
\end{gather}
Together with the identity (\ref{eq_expansionusefullessorequaltoone}) below and the fact $m_{1n}(z) \asymp 1$ for $z \in \mathbf{D}_b^\prime$, we find that $\mathsf{E}_1  \leq 1.$ For the term $\mathsf{E}_2,$ we first consider a closely related quantity $\mathsf{W}(z)$ defined as 
\begin{equation*}
\mathsf{W}(z):=\frac{1}{n} \sum_{i=1}^p \frac{\sigma_i^2 \int \frac{s^2}{|1+s m_{1n,c}(z)|^2} \mathrm{d} F(s) }{|-z+\sigma_i \int \frac{s}{1+s m_{1n,c}(z)} \mathrm{d} F(s)|^2}=1-\eta \frac{|m_{1n,c}(z)|^2}{\operatorname{Im} m_{1n,c}(z)}.
\end{equation*}
By assumption that $\phi^{-1}>\mathsf{s}_3$ (recall (\ref{eq_phasetransition})) and  $m_{1n,c}(L_+)=-l^{-1}$ , we see that
\begin{equation}\label{eq_edgeresults}
\mathsf{W}(L_+)<1.
\end{equation}
Armed with (\ref{eq_edgeresults}), using  (\ref{def4}) and Assumption \ref{assum_additional_techinical}, we can apply  an argument similar to Lemma A.6 of \cite{lee2016extremal} or Lemma A.7 of \cite{Kwak2021}  to conclude that  when $n$ is sufficiently large, for $z \in \mathbf{D}_b$ and  some constant $0<\mathfrak{c}'<1$
\begin{equation}\label{eq_defnmathsfW}
\mathsf{E}^2_2=\mathsf{W}(L_+)+\ro(1)<\mathfrak{c}'.
\end{equation}

Consequently, we find that when $n$ is sufficiently large, $\mathsf{E}_2<1.$ Together with (\ref{eq_tintot1t2}), we can conclude that $|\mathsf{T}|<1$. This yields that $\mathsf{P}_2 = \mathfrak{c} |m_{1n,c}-m_{1n}|,$ for some constant $0<\mathfrak{c}<1.$ 

\quad Inserting the above controls back into (\ref{eq_p1plusp2}), using (\ref{eq_defnp1}), we can conclude our proof
\begin{equation}\label{eq_deterministiclose}
m_{1n,c}=m_{1n}(z)+\rO(n^{-1/2+\epsilon}). 
\end{equation} 
By the result of $|m_{1n,c}-m_{1n}|$, the proof of $m_{n,c}-m_n$ follows from an argument similar to (\ref{eq_p1plusp2}), using (\ref{eq_systemequationsm1m2}) and the equivalent expressions
\begin{equation*}
m_n(z)=\frac{1}{p} \sum_{i=1}^p \frac{1}{-z+\sigma_i n^{-1} \sum_{j=1}^n \frac{\xi_j^2}{1+\xi_j^2 m_{1n}}}, \ m_{n,c}(z)=\frac{1}{p} \sum_{i=1}^p \frac{1}{-z+\sigma_i \int_0^l \frac{s}{1+s m_{1n,c}(z)} \mathrm{d} F(s)}.
\end{equation*}
We omit the details. 
\end{proof}

\subsubsection{Control of $m_{1n}(z)-m_1(z)$ and $m_n(z)-m_Q(z)$}\label{sec_proofpartiiboundedlocallaw2}
As the proof of the remaining part is analogous to Section \ref{sec_proofpartiboundedlocallaw}, we focus on $|m_{1n,c}-m_{1n}|$ and will briefly discuss $|m_Q-m_n|$ from line to line. 
The proof ideas follow  Proposition 5.1 of \cite{lee2016extremal} or Proposition 5.1 of \cite{Kwak2021}. We focus on explaining the parts deviates the most.  The proof relies on the following two lemmas. 

\begin{lemma}\label{lem: first bound for m_1n-m_1}
Conditional on the event $\Omega_D$ in Theorem \ref{thm_boundedcaselocallaw}, for all $z=E+\ri\eta\in\mathbf{D}_b^{\prime}$ with $n^{-1/2+\epsilon_{\mathsf{d}}}\leq\eta\leq n^{-1/(d+1)+\epsilon_{\mathsf{d}}}$, we have 
\begin{gather*}
    |m_{1n}(z)-m_1(z)|\prec\frac{1}{n\eta_0}, \   |m_{n}(z)-m_Q(z)|\prec\frac{1}{n\eta_0}.
\end{gather*}
\end{lemma}
%{\color{red}[from here]}
\begin{lemma}\label{lem: second bound for m_1n-m_1}
 Assuming that $|m_{1n}(z)-m_1(z)|\prec n^{\epsilon_{\mathsf{d}}}(n\eta_0)^{-1}$, then conditional on the event $\Omega_D,$ we have that for all $z \in \mathbf{D}_b^\prime$ 
\begin{gather*}
    |m_{1n}(z)-m_1(z)|\prec\frac{1}{n\eta_0}, \   |m_{n}(z)-m_Q(z)|\prec\frac{1}{n\eta_0}.
\end{gather*}
\end{lemma}

\quad Armed with the above two lemmas, we now proceed to the control of $m_{1n}(z)-m_1(z).$ 

\begin{proof}[\bf Proof: control of $m_{1n}(z)-m_1(z)$ and $m_n(z)-m_Q(z)$] Due to similarity, we only prove $m_{1n}(z)-m_1(z).$  We prove this by mathematical induction as that of Proposition 5.1 of \cite{lee2016extremal}. Fix $E$ such that $z=E+\ri\eta_0\in\mathbf{D}_b^{\prime}$, we consider a sequence $(\eta_j)$ defined by $\eta_j=\eta_0+jn^{-2}$. Let $K$ be the smallest positive integer such that $\eta_K\geq n^{-1/2+\epsilon_{\mathsf{d}}}$. Note that for $j=K$, by Lemma \ref{lem: first bound for m_1n-m_1}, we have that  $|m_{1n}(z_j)-m_1(z_j)|\prec\frac{1}{n\eta_0}. $
Then for any $z=E+\ri\eta$ with $\eta_{j-1}\le\eta\leq\eta_j$, we have that for some constant $C>0$
\begin{align*}
 |m_1(z_j)-m(z)|=\frac{1}{n} \operatorname{tr}\left[(G(z_j)-G(z)) \Sigma\right]=\frac{|z_j-z|}{n} \operatorname{tr}(G(z_j)G(z) \Sigma) \leq C \frac{|z_j-z|}{\eta_{j-1}^2}\leq C\frac{n^{2\epsilon_{\mathsf{d}}}}{n},
\end{align*}
where we used the first resolvent identity and the trivial bound $|G(z)| \leq \eta^{-1}$, and similarly 
\begin{align*}
|m_{1n}(z_j)-m_{1n}(z)|=\left| \int \left[\frac{1}{x-z_j}-\frac{1}{x-z} \right] \rho(x) \mathrm{d} x   \right|\leq\frac{|z_j-z|}{\eta_{j-1}^2}\leq\frac{n^{2\epsilon_{\mathsf{d}}}}{n}. 
\end{align*}
Thus we find that if $|m_{1n}(z_j)-m_1(z_j)|\prec (n\eta_0)^{-1}$, then by Lemma \ref{lem: second bound for m_1n-m_1}, for some constant $C'>0$
\begin{gather}\label{eq_latticeargumentbelow}
    |m_{1n}(z)-m_1(z)|\leq|m_{1n}(z_j)-m_1(z_j)|+\frac{C' n^{2\epsilon_{\mathsf{d}}}}{n}\prec\frac{n^{\epsilon_{\mathsf{d}}}}{n\eta_0}.
\end{gather}
This gives the result that $|m_{1n}(z)-m_1(z)|\prec(n\eta_0)^{-1}$ for $z=E+\ri\eta$ with $\eta_{j-1}\leq\eta\leq\eta_j$. The proof for each $z$ can be completed by an induction on  $j.$ Finally, using an induction procedure and a standard lattice argument (for example, see \cite{Alex2014,Ding&Yang2018}), we find that the results hold for all $z \in \mathbf{D}_{b}^\prime$. More specifically, we construct a lattice $\mathcal{L}$ from $z^{\prime}=E^{\prime}+\ri\eta_0\in\mathbf{D}_b^{\prime}$ with $|z-z^{\prime}|\leq n^{-3}$. It is obvious that the bound holds uniformly on $\mathcal{L}$. For any $z=E+\ri\eta_0\notin\mathcal{L}$, we find a $z^{\prime}\in\mathcal{L}$ and then $|z-z^{\prime}|\leq n^{-3}$. Moreover, using resolvent identity, we can conclude that  $|m_1(z)-m_1(z^{\prime})|\leq\eta^{-2}_0|z-z^{\prime}| \ll (n \eta_0)^{-1}$. Therefore, we conclude the proof.

\end{proof}

\quad In what follows, we prove lemmas \ref{lem: first bound for m_1n-m_1} and \ref{lem: second bound for m_1n-m_1}. The proofs are similar to those of Lemmas 5.6 and 5.7 of \cite{lee2016extremal}, except that we will need a weak local law as follows.

\begin{proposition}[Weak averaged local law]  \label{proposition_boundedweaklocallaw}
Suppose the assumptions of Theorem \ref{thm_boundedcaselocallaw} hold. We have that for $z \in \mathbf{D}_b^\prime$ 
\begin{equation*}
|m_Q(z)-m_n(z)|+|m_1(z)-m_{1n}(z)|+|m_2(z)-m_{2n}(z)|=\rO_{\prec}\left( (n \eta)^{-1/4} \right).
\end{equation*}
\end{proposition} 
\begin{proof}
The proof of Proposition \ref{proposition_boundedweaklocallaw} is relatively standard in the random matrix literature, for example, see Section 4.1 of \cite{Alex2014} or Section 3.6 of \cite{erdHos2013spectral} or Appendix A.2 of \cite{Ding&Yang2018} or Section 5.2 of \cite{yang2019edge}. Due to similarity, as in Lemma 5.12 of \cite{yang2019edge},  we only provide the key ingredients. Define the $z$-dependent parameter
\begin{gather}\label{eq_defnpsi}
    \Psi(z):=\sqrt{\frac{\operatorname{Im}m_{1}(z)}{n\eta}}+\frac{1}{n \eta}.
\end{gather}      
Recall (\ref{eq: decomp m_2}). By Lemma \ref{lem:large deviation} and (\ref{lem:Wald}), we find that 
\begin{gather}\label{eq_zibound}
    Z_i\prec\frac{\xi^2_i}{n}\|G^{(i)}\Sigma^{1/2}\|_F \leq l\sqrt{\frac{\operatorname{Im}m^{(i)}_{1}(z)}{n\eta}}\asymp \Psi,
\end{gather}
where in the last step we used (\ref{lem:trace_difference}).  Together with (\ref{lem:trace_difference}) and the first equation of (\ref{eq: decomp m_2}), we conclude that 
\begin{gather}\label{eq_m2uboundeddecomposition}
    m_2=\frac{1}{n}\sum_{i=1}^n\frac{\xi^2_i}{-z(1+\xi^2_im_1(z)+\rO_{\prec}(\Psi))}.
\end{gather} 

\quad For $m_1(z),$ recall (\ref{eq: decomp of mathcal_G}). According to the definition of $m_1(z)$ in (\ref{eq_m1m2}), we have that 
\begin{gather}\label{eq_m1decompositionfinafinalfinal}
    m_1(z)=\frac{1}{n}\operatorname{tr}(G(z)\Sigma)=-\frac{1}{n}\sum_{i=1}^p\frac{\sigma_i}{z(1+m_2(z)\sigma_i)}+\frac{1}{n}\operatorname{tr}(R_1\Sigma)+\frac{1}{n}\operatorname{tr}(R_2\Sigma).
\end{gather}
Similarly, for $m_Q(z)$ in (\ref{eq_mq}), we have that
\begin{equation}\label{eq_mqdecomposition}
  m_{Q}(z)=\frac{1}{p}\operatorname{tr}(G(z))=-\frac{1}{p}\sum_{i=1}^p\frac{1}{z(1+m_2(z)\sigma_i)}+\frac{1}{p}\operatorname{tr}(R_1)+\frac{1}{p}\operatorname{tr}(R_2).
\end{equation}

On the one hand, when $\eta \asymp 1,$ by a discussion similar to (5.45) of \cite{yang2019edge} , we find that $\| (I+m_2^{(i)} \Sigma)^{-1} \|<\infty.$ Then using (\ref{eq_zibound}), by a discussion similar to the equations between (\ref{eq: decomp R_1}) and (\ref{eq: m_1 by m_2}), we find that 
\begin{gather}\label{eq: est of m_1 and m_2}
m_1(z)=-\frac{1}{n}\sum_{i=1}^p\frac{\sigma_i}{z(1+m_2(z)\sigma_i)}+\rO_{\prec}\Big(\frac{1}{n}\sum_i\frac{\xi^2_i\Psi}{z(1+\xi^2_im_1(z)+\rO_{\prec}(\Psi))}\Big).
\end{gather}
Similarly, we have 
\begin{equation*}
  m_{Q}(z)=-\frac{1}{p}\sum_{i=1}^p\frac{1}{z(1+m_2(z)\sigma_i)}+\rO_{\prec}\Big(\frac{1}{p}\sum_i\frac{\xi^2_i\Psi}{z(1+\xi^2_im_1(z)+\rO_{\prec}(\Psi))}\Big).
\end{equation*}

On the other hand, denote $\Xi:=\{|\mathcal{G}_{ij}(z)+\delta_{ij}(z(1+m_{1n}(z) \xi_i^2))^{-1}|+|m_2(z)-m_{2n}(z)| \leq (\log n)^{-1} \}.$ When restricted on $\Xi,$ by Assumption \ref{assum_additional_techinical}, we also have that $\| (I+m_2^{(i)} \Sigma)^{-1} \|<\infty.$ By an analogous argument, we find that (\ref{eq: est of m_1 and m_2}) also holds true. By an argument similar to (\ref{eq_nontrivialcontrol}) using Lemma \ref{lem: Resolvent}, we have that for $i \neq j,$ $\mathbf{1}(\eta \geq 1) \mathcal{G}_{ij} \prec \Psi, \ \mathbf{1}(\Xi) \mathcal{G}_{ij} \prec \Psi. $

The above arguments show that the counterparts of Lemmas 5.9 and 5.10 of \cite{yang2019edge} hold. Therefore, by the same arguments as in Lemma 5.12 of \cite{yang2019edge}, we can conclude the proof.  

\end{proof} 

\quad Next we provide some useful controls whose proofs and results will be used in the proof of Lemmas \ref{lem: first bound for m_1n-m_1} and \ref{lem: second bound for m_1n-m_1}. The results and arguments are analogous to Lemma 5.4 of \cite{lee2016extremal}. We only point out the key ingredients in our proof and refer the readers to \cite{lee2016extremal} for more details. 

\begin{lemma}\label{lem: est for Im m_1}
Suppose the assumptions of Theorem \ref{thm_boundedcaselocallaw} hold. Then we have that on the event $\Omega_D$ and for all $z=E+\ri\eta_0\in\mathbf{D}_b^{\prime}$
\begin{gather*}
    \operatorname{Im}m_1(z)\prec\frac{1}{n\eta_0}, \    \  \operatorname{Im}m_Q(z)\prec\frac{1}{n\eta_0}.
\end{gather*}
\end{lemma}
\begin{proof}
Due to similarity, we focus our proof on $\operatorname{Im} m_Q(z).$ We prove by contradiction. Given some $\epsilon>0,$ conditional on $\Omega_D,$ for some sufficiently small constants $0<c_1, c_2<1,$ we first introduce a probability event $\Xi_1 \equiv \Xi_1(\epsilon)$ so that the followings holds:
\begin{enumerate}
\item For $z \in \mathbf{D}_b^\prime,$
\begin{equation*}
|m_Q(z)-m_n(z)|+|m_1(z)-m_{1n}(z)|+|m_2(z)-m_{2n}(z)| \leq  (n \eta)^{-1/4+c_1\epsilon}.
\end{equation*}
\item For $z \in \mathbf{D}_b^\prime$ and $Z_i$ in (\ref{eq: decomp m_2}),
\begin{equation*}
\max_i Z_i \leq n^{c_2 \epsilon}\Psi.
\end{equation*}
\end{enumerate}
According to Proposition \ref{proposition_boundedweaklocallaw} and (\ref{eq_zibound}), we find that there exists some large constant $D>0$ so that $\mathbb{P}(\Xi_1)=1-n^{-D}.$ In what follows, we restrict ourselves on $\Xi_1$ so that the discussions are purely deterministic.  

\quad  Assuming that  $\operatorname{Im}m_1(z)>n^{\epsilon}\frac{1}{n\eta_0}. $ According to the definition of $\Psi$ in (\ref{eq_defnpsi}), we conclude
\begin{equation}\label{eq_Psicontrol}  
\Psi=\ro(\operatorname{Im}m_{1}(z)).
\end{equation}
 Moreover, by (\ref{eq_zopointrate}) and (\ref{eq_oneregimeedgecontrol}), we readily see that    $\operatorname{Im}m_{1n}(z)\ll \operatorname{Im}m_1(z) $. This implies that for some constant $C>0$  
 \begin{gather}\label{eq_tobeusedasacontradiction}
    |m_{1n}(z)-m_1(z)|\geq |\operatorname{Im}m_{1n}-\operatorname{Im}m_1|>Cn^{\epsilon}\frac{1}{n\eta_0}.
\end{gather}

\quad On the other hand, by Proposition \ref{proposition_boundedweaklocallaw} and Assumption \ref{assum_additional_techinical}, (\ref{eq: est of m_1 and m_2}) still holds. Together with $m_{1n}$ in (\ref{eq_systemequationsm1m2}), using (\ref{eq: decomp m_2}), we find that 
\begin{align}
m_{1n}-m_1 & =\frac{1}{n}\sum_{i=1}^p\frac{\frac{\sigma_i^2}{n}\sum_j\frac{\xi^4_j(m_{1n}-m_1)+\xi^2_j\rO(n^{c_2 \epsilon}\Psi)}{(1+\xi^2_jm_{1n})(1+\xi^2_jm_1+\rO(n^{c_2 \epsilon}\Psi))}}{(z-\frac{\sigma_i}{n}\sum_j\frac{\xi^2_j}{1+\xi^2_jm_{1n}})(z-\frac{\sigma_i}{n}\sum_j\frac{\xi^2_j}{1+\xi^2_jm_{1}+\rO(n^{c_2\epsilon}\Psi)})}+\rO(n^{c_2 \epsilon}\Psi) \nonumber \\
&=\mathsf{C}_1(m_{1n}-m_1)+\mathsf{C}_2+\rO(n^{c_2 \epsilon} \Psi), \label{eq_decompositiondeomdeomdeodemdem}
\end{align}
where $\mathsf{C}_1, \mathsf{C}_2$ are defined as 
\begin{align*}
& \mathsf{C}_1:=\frac{1}{n}\sum_{i=1}^p\frac{\frac{\sigma_i^2}{n}\sum_j\frac{\xi^4_j}{(1+\xi^2_jm_{1n})(1+\xi^2_jm_1+\rO(n^{c_2 \epsilon}\Psi))}}{(z-\frac{\sigma_i}{n}\sum_j\frac{\xi^2_j}{1+\xi^2_jm_{1n}})(z-\frac{\sigma_i}{n}\sum_j\frac{\xi^2_j}{1+\xi^2_jm_{1}+\rO(n^{c_2\epsilon}\Psi)})}, \\
& \mathsf{C}_2:=\frac{1}{n}\sum_{i=1}^p\frac{\frac{\sigma_i^2}{n}\sum_j\frac{\xi^2_j\rO(n^{c_2 \epsilon}\Psi)}{(1+\xi^2_jm_{1n})(1+\xi^2_jm_1+\rO(n^{c_2 \epsilon}\Psi))}}{(z-\frac{\sigma_i}{n}\sum_j\frac{\xi^2_j}{1+\xi^2_jm_{1n}})(z-\frac{\sigma_i}{n}\sum_j\frac{\xi^2_j}{1+\xi^2_jm_{1}+\rO(n^{c_2\epsilon}\Psi)})}.
\end{align*}
We first control $\mathsf{C}_2.$ It is easy to see that $m_1\asymp 1$ by contradiction. If $m_1\ll 1$, one can observe from (\ref{eq_m2uboundeddecomposition}) that $m_2\asymp 1$ which yields $m_1 \asymp 1$ by (\ref{eq: est of m_1 and m_2}). If $m_1\gg 1$, we have $m_2\ll 1$ from (\ref{eq_m2uboundeddecomposition}), and then it gives that $m_1 \asymp 1$ by (\ref{eq: est of m_1 and m_2}). Similarly, we can show that $m_2 \asymp 1.$ Together with Proposition \ref{proposition_boundedweaklocallaw}, we find that $m_{1n}, m_{2n} \asymp 1.$ Since $z \asymp 1,$ using the definition of $m_{2n}$ in (\ref{eq_systemequationsm1m2}) and Proposition \ref{proposition_boundedweaklocallaw}, we find that 
\begin{equation*}
\frac{1}{n}\sum_j\frac{\xi^2_j\rO(n^{c_2 \epsilon}\Psi)}{(1+\xi^2_jm_{1n})(1+\xi^2_jm_1+\rO(n^{c_2 \epsilon}\Psi))}=\rO(n^{c_2 \epsilon} \Psi ).  
\end{equation*}
Moreover, by Proposition \ref{proposition_boundedweaklocallaw}, (\ref{eq_deterministiclose}) and Assumption \ref{assum_additional_techinical}, we find that 
\begin{equation*}
\frac{1}{n}\sum_i \frac{1}{(z-\frac{\sigma_i}{n}\sum_j\frac{\xi^2_j}{1+\xi^2_jm_{1n}})(z-\frac{\sigma_i}{n}\sum_j\frac{\xi^2_j}{1+\xi^2_jm_{1}+\rO(n^{c_2\epsilon}\Psi)})} \asymp 1. 
\end{equation*}  
This yields that 
\begin{equation}\label{eq_mathsfc2control}
\mathsf{C}_2=\rO(n^{c_2 \epsilon} \Psi). 
\end{equation}

\quad For $\mathsf{C}_1,$ using Proposition \ref{proposition_boundedweaklocallaw} and Assumption \ref{assum_additional_techinical}, by an argument similar to (\ref{eq_tintot1t2}), we can conclude that when $n$ is sufficiently large, for some constant $0<\mathfrak{c}<1,$ 
\begin{equation}\label{eq_mathsfc1control}
\mathsf{C}_1  \leq \mathfrak{c}.
\end{equation}
Combining (\ref{eq_decompositiondeomdeomdeodemdem}), (\ref{eq_mathsfc2control}) and (\ref{eq_mathsfc1control}), we conclude that $|m_{1n}-m_1|=\rO(n^{c_2 \epsilon}\Psi),$
which contradicts (\ref{eq_tobeusedasacontradiction}) since $c_2<1$ is sufficiently small. This completes our proof for each fixed $z.$ For uniformity in $z,$ we can follow a standard lattice argument as discussed below (\ref{eq_latticeargumentbelow}). This finishes the proof. The discussion for $m_Q$ follows from an analogous discussion with the help of (\ref{eq_m2uboundeddecomposition}) and (\ref{eq: est of m_1 and m_2}). 
\end{proof}

\begin{remark}\label{remk_zibound}
Two remarks are in order. First, it is easy to see that repeating the proof of Lemma \ref{lem: est for Im m_1}, we can prove the results for all $\eta$ as specified in (\ref{eq_spectraldomaintwo}). Second, we remark that combining (\ref{eq_zibound}) and Lemma \ref{lem: est for Im m_1}, when $z=E+\ri \eta_0 \in \mathbf{D}_b^\prime,$ we have that conditional on $\Omega_D,$ $Z_i \prec(n \eta_0)^{-1}.$
\end{remark}

\quad Armed with the above discussions and results, following the strategies of Lemmas 5.6 and 5.7 of \cite{lee2016extremal} or \cite{Kwak2021},  we prove Lemmas \ref{lem: first bound for m_1n-m_1} and \ref{lem: second bound for m_1n-m_1} using similar arguments as in Lemma \ref{lem: est for Im m_1}. Due to similarity, we only provide the key ingredients. 

\begin{proof}[\bf Proof of Lemmas \ref{lem: first bound for m_1n-m_1} and \ref{lem: second bound for m_1n-m_1}] Due to similarity, we focus our proof on Lemma \ref{lem: first bound for m_1n-m_1} and briefly mention that of Lemma \ref{lem: second bound for m_1n-m_1} in the end. Due to similarity, we only explain $|m_{1n}(z)-m_1(z)|.$

\quad The proof is similar to that of Lemma \ref{lem: est for Im m_1} and we prove by contradiction. We also restrict ourselves on the event $\Xi_1$ in Lemma \ref{lem: est for Im m_1}. We assume that $|m_{1n}(z)-m_1(z)|>n^{\epsilon}(n\eta_0)^{-1}$. To see a contraction,  in addition to the the arguments of Lemma \ref{lem: est for Im m_1}, we need to provide a finer control for $\Psi$ since in the current proof it depends on $\eta$ instead of $\eta_0.$ Note that
\begin{gather}\label{eq_modificationkeykey}
    \begin{split}
       n^{c_2 \epsilon} \Psi&=n^{c_2 \epsilon}\sqrt{\frac{|\operatorname{Im}m_1-\operatorname{Im}m_{1n}+\operatorname{Im}m_{1n}|}{n\eta}}+n^{c_2 \epsilon} \frac{1}{n\eta}\\
        &\leq   n^{c_2 \epsilon} \sqrt{\frac{|\operatorname{Im}m_1-\operatorname{Im}m_{1n}|}{n\eta}} +  n^{c_2 \epsilon} \sqrt{\frac{\operatorname{Im}m_{1n}}{n\eta}}+n^{c_2 \epsilon} \frac{1}{n\eta}\\
        &=\ro(|m_1-m_{1n}|),
    \end{split}
\end{gather}
where in the last step we used (\ref{eq_zopointrate11}) and the assumption  $|m_{1n}(z)-m_1(z)|>n^{\epsilon}(n\eta_0)^{-1} \gg (n \eta)^{-1}$ when $n^{-1/2+\epsilon_{\mathsf{d}}} \leq \eta \leq n^{-1/(d+1)+\epsilon_{\mathsf{d}}}.$ Replacing $  n^{c_2 \epsilon} \Psi$ with $\ro(|m_1-m_{1n}|)$ in the arguments between (\ref{eq_decompositiondeomdeomdeodemdem}) and (\ref{eq_mathsfc1control}), we find that $|m_{1n}-m_1|=\ro(|m_{1n}-m_1|)$ which is a contraction. This proves the result for each fixed $z.$ The uniformity follows from the same lattice argument as mentioned in the end of the proof of Lemma \ref{lem: est for Im m_1}. 

The proof of Lemma \ref{lem: second bound for m_1n-m_1} is similar. We also prove by contradiction and assume that $n^{\epsilon}(n\eta_0)^{-1}<|m_1(z)-m_{1n}(z)|\leq n^{\epsilon+\epsilon_{\mathsf{d}}}(n\eta_0)^{-1}.$  Under this assumption, together with (\ref{eq_zopointrate}) and (\ref{eq_oneregimeedgecontrol}), we find that (\ref{eq_modificationkeykey}) still holds. The rest of the arguments are similar and we omit the details.

\end{proof}

\subsection{Bounded support and conditional on $X$ setting: proof of Theorem \ref{thm_averagedlocallaw_boundedmultiplier_conditionalonX}}

In this section, we prove Theorem \ref{thm_averagedlocallaw_boundedmultiplier_conditionalonX} 
by an argument similar to that used in the proof of 
Theorem \ref{thm_boundedcaselocallaw}.

\begin{proof}[\bf Proof of Theorem \ref{thm_averagedlocallaw_boundedmultiplier_conditionalonX}]
We begin by recalling the definitions of $m_{1n}(z)$ and $m_{1n,c}(z)$ in \eqref{eq_systemequationsm1m2} and \eqref{eq_systemequationsm1m2intergrate}. By construction, both $m_{1n}(z)$ and $m_{1n,c}(z)$ depend only on the population covariance $\Sigma$ and the multiplier matrix $D$. Then, on the event $\Omega_D$, the proof of
\begin{align*}
    |m_{1n}(z)-m_{1n,c}(z)|\leq n^{-1/2+\epsilon_{\mathsf{d}}},\quad |m_n(z)-m_{n,c}(z)|\leq n^{-1/2+\epsilon_{\mathsf{d}}},\quad z\in\mathbf{D}_b^{\prime},
\end{align*}
follows exactly the same argument as in the proof of Theorem \ref{thm_boundedcaselocallaw}.

We next estimate $|m_1(z)-m_{1n}(z)|$ and $|m_Q(z)-m_n(z)|$ on the events $\Omega_X$ and $\Omega_D$. Recall \eqref{eq: decomp m_2}.
\begin{align*}
    m_2(z)=\frac{1}{n}\sum_{i=1}^n&\frac{\xi^2_i}{-z-z\mathbf{y}^{*}_i G^{(i)}\mathbf{y}_i}=\frac{1}{n}\sum_{i=1}^n\frac{\xi^2_i}{-z(1+\xi^2_in^{-1}{\rm tr}G^{(i)}\Sigma+Z_i)},\\
    &Z_i=\mathbf{y}^{*}_i G^{(i)}\mathbf{y}_i-\xi^2_in^{-1}{\rm tr}G^{(i)}\Sigma.
\end{align*}
Restricting to the events $\Omega_D$ and $\Omega_X$, we obtain
\begin{align*}
    Z_i=\mathrm{O}\left(n^{\varepsilon_1}\frac{\xi^2_i\|G^{(i)}\Sigma^{1/2}\|_F}{n}\right)=\mathrm{O}\left(n^{\varepsilon_1}\sqrt{\frac{\operatorname{Im}m_{1}^{(i)}(z)}{n\eta}}\right)=\mathrm{O}(n^{\varepsilon_1}\Psi),
\end{align*}
where $\varepsilon_1>0$ is defined in Definition \ref{def_OmegaX} and $\Psi(z)$ is
defined in \eqref{eq_defnpsi}. Consequently,
\begin{align*}
    m_{2}=\frac{1}{n}\sum_{i=1}^n&\frac{\xi^2_i}{-z-z\mathbf{y}^{*}_i G^{(i)}\mathbf{y}_i}=\frac{1}{n}\sum_{i=1}^n\frac{\xi^2_i}{-z(1+\xi^2_in^{-1}{\rm tr}G^{(i)}\Sigma+\mathrm{O}(n^{\varepsilon_1}\Psi)}.
\end{align*}

On the other hand, for $m_1(z)$, using the same decomposition as in \eqref{eq: decomp of mathcal_G}, we obtain
\begin{align*}
    m_{1}(z)=\frac{1}{n}\operatorname{tr}(G(z)\Sigma)=-\frac{1}{n}\sum_{i=1}^p\frac{\sigma_i}{z(1+m_{2}(z)\sigma_i)}+\frac{1}{n}\operatorname{tr}(R_{1}\Sigma)+\frac{1}{n}\operatorname{tr}(R_{2}\Sigma),
\end{align*}
and similarly 
\begin{align*}
    m_{Q}(z)=\frac{1}{p}\operatorname{tr}(G(z))=-\frac{1}{p}\sum_{i=1}^p\frac{1}{z(1+m_{2}(z)\sigma_i)}+\frac{1}{p}\operatorname{tr}(R_{1})+\frac{1}{p}\operatorname{tr}(R_{2}),
\end{align*}
where 
\begin{align*}
    &R_{1}=z^{-1}\sum_{i=1}^n\frac{G^{(i)}(\mathbf{y}_i\mathbf{y}_i^*-n^{-1}\xi^2_i\Sigma)}{1+\mathbf{y}_i^*G^{(i)}\mathbf{y}_i}(I+m_{2}(z)\Sigma)^{-1},\\
    &R_{2}=z^{-1}\frac{1}{n}\sum_{i=1}^n\frac{(G^{(i)}-G)\xi^2_i\Sigma}{1+\mathbf{y}_i^*G^{(i)}\mathbf{y}_i}(I+m_{2}(z)\Sigma)^{-1}.
\end{align*}
On $\Omega_X$, when $\eta\asymp 1$, we have 
\begin{align*}
    m_{1}(z)=-\frac{1}{n}\sum_{i=1}^p\frac{\sigma_i}{z(1+m_2(z))\sigma_i}+\mathrm{O}\Big(\frac{1}{n}\sum_i\frac{n^{\varepsilon_1}\xi^2_i\Psi}{z(1+\xi^2_im_{1}(z))+\mathrm{O}(n^{\varepsilon_1}\Psi)}\Big).
\end{align*}
Under $\Omega_X$ and $\Omega_D$, using the above estimates as initial input and following a similar argument as in the proof of Proposition \ref{proposition_boundedweaklocallaw}, we obtain the following weak averaged local law.
\begin{lemma}[Weak averaged local law conditional on sample]\label{lem_weakaveragedlocallaw_conditional}
    Suppose the assumptions of Theorem \ref{thm_averagedlocallaw_boundedmultiplier_conditionalonX}, under events $\Omega_D$ and $\Omega_X$, we have that for $z\in\mathbf{D}_b^{\prime}$,
    \begin{align*}
        |m_{Q}(z)-m_n(z)|+|m_{1}(z)-m_{1n}(z)|+|m_{2}-m_{2n}(z)|=\mathrm{O}(n^{\epsilon_1}(n\eta)^{-1/4}),
    \end{align*}
    for some small constant $\epsilon_1>\varepsilon_1$.
\end{lemma}

Lemma \ref{lem_weakaveragedlocallaw_conditional} implies that, the following two statements hold deterministically for some small constant $\epsilon_2>\varepsilon_1$:
\begin{itemize}
    \item [1.] For $z\in\mathbf{D}_b^{\prime}$,
    \begin{align*}
        |m_{Q}(z)-m_{n}(z)|+|m_{1}(z)-m_{1n}(z)|+|m_{2}(z)-m_{2n}(z)|\leq n^{\epsilon_2}(n\eta)^{-1/4}.
    \end{align*}
    \item [2.] For $z\in\mathbf{D}_b^{\prime}$,
    \begin{align*}
        \max_iZ_i\leq n^{\epsilon_2}\Psi.
    \end{align*}
\end{itemize}
Let $\Xi_1$ denote the event that these two statements hold. Then, by Lemma \ref{lem: est for Im m_1}, on $\Omega_D$ and $\Omega_X$,  $\operatorname{Im}m_{1}(z)=\mathrm{O}(\frac{n^{\epsilon_3}}{n\eta_0}),\quad\operatorname{Im}m_{Q}(z)=\mathrm{O}(\frac{n^{\epsilon_3}}{n\eta_0}),$ for $\eta_0=n^{-1/2-\epsilon_{\mathsf{d}}}$ and some small constant $\epsilon_3>\varepsilon_1$. These bounds serve as the key input for a contradiction argument analogous to those in the proofs of Lemmas \ref{lem: first bound for m_1n-m_1} and \ref{lem: second bound for m_1n-m_1}. We can  obtain the following result. 
\begin{lemma}
    When restricted on the events $\Omega_D$ and $\Omega_X$, for all $z=E+\mathrm{i}\eta\in\mathbf{D}_b^{\prime}$ with $n^{-1/2+\epsilon_{\mathsf{d}}}\leq\eta\leq n^{-1/(d+1)+\epsilon_{\mathsf{d}}}$, we have
    \begin{align*}
        |m_{1n}(z)-m_{1}(z)|=\mathrm{O}(\frac{n^{\epsilon_4}}{n\eta_0}),\quad |m_{n}(z)-m_{Q}(z)|=\mathrm{O}(\frac{n^{\epsilon_4}}{n\eta_0}).
    \end{align*}    
    for some $\epsilon_4>\varepsilon_1$.
\end{lemma}

With the above estimates, we can complete the proof of Theorem \ref{thm_averagedlocallaw_boundedmultiplier_conditionalonX} by following arguments analogous to those used in the proof of Theorem \ref{thm_boundedcaselocallaw}. This concludes the proof.
%
%We further have the following iterative improvement.
%\begin{lemma}
%        Assuming that $|m_{1n}(z)-m_{1}(z)|=\mathrm{O}(n^{\epsilon_4}(n\eta_0)^{-1})$, then under the events $\Omega_D$ and $\Omega_X$, we have for all $z\in\mathbf{D}_b^{\prime}$,
%        \begin{align*}
%            |m_{1n}(z)-m_{1}(z)|=\mathrm{O}(\frac{n^{\epsilon_5}}{n\eta_0}),\quad |m_n(z)-m_{Q}(z)|=\mathrm{O}(\frac{n^{\epsilon_5}}{n\eta_0}),
%        \end{align*}
%        for small constant $\epsilon_5>\varepsilon_1$.
%\end{lemma}
% Finally, a standard lattice argument yields that, for all $z\in\mathbf{D}_b^{\prime}$, under $\Omega_D$ and $\Omega_X$,
%\begin{align*}
%    |m_{1n}(z)-m_{1}(z)|=\mathrm{O}(\frac{n^{\epsilon_5}}{n\eta_0}),\quad |m_n(z)-m_{Q}(z)|=\mathrm{O}(\frac{n^{\epsilon_5}}{n\eta_0}).
%\end{align*}
%This concludes Theorem \ref{thm_averagedlocallaw_boundedmultiplier_conditionalonX}. 
\end{proof}
%Observe that 
%\begin{align*}
%    \mathbb{P}\big(\Omega|\mathcal{F}_X\big)=\mathbb{P}(\Omega)=1-\mathrm{O}(\log^{-D}n),
%\end{align*}
%and $\{|m_{1n}(z)-m_{1}(z)|\ge n^{\epsilon_7}/(n\eta_0)|\mathcal{F}_X\}\Rightarrow\{\Omega^c\}$, we have
%\begin{align*}
%    &\mathbb{P}\big(|m_{1n}(z)-m_{1}(z)|\ge \frac{n^{\epsilon_7}}{n\eta_0}|\mathcal{F}_X\big)=\mathbb{E}\big[\mathbb{P}\big(|m_{1n}(z)-m_{1}(z)|\ge \frac{n^{\epsilon_7}}{n\eta_0}|\mathcal{F}_X,\mathcal{F}_D\big)|\mathcal{F}_X\big]\\
%    &\le\mathbb{E}\big[\mathbb{P}(\Omega^c|\mathcal{F}_X,\mathcal{F}_D)|\mathcal{F}_X\big]=\mathrm{O}(\log^{-D}n).
%\end{align*}

\section{Locations for extreme bootstrapped eigenvalues and proof of the main results}\label{sec_locationofeigenvalues}
In this section, we study the limiting behavior of the largest eigenvalue of Q, namely $\lambda_1(Q)$. In what follows, we focus on multipliers with polynomially decaying tails; the case of exponentially decaying tails can be treated by a parallel argument.
% In Section \ref{sec_sub_unbounded1st}, we investigate the case when $\{\xi_i^2\}$ have unbounded support as in (i) of Assumption \ref{assum_D}. In Section \ref{sec_sub_bounded1st}, we study the bounded support case as in (ii) of Assumption \ref{assum_D}.  

\subsection{The unbounded support and unconditional case}\label{sec_sub_unbounded1st}

In order to quantify the location of $\lambda_1 \equiv \lambda_1(Q),$ we need to introduce several auxiliary quantities. Recall $\vartheta_1$ defined in (\ref{eq: def of vartheta_1}). Similarly, we denote $\vartheta_2$ by replacing $\xi^2_{(1)}$ with $\xi^2_{(2)}$ in (\ref{eq: def of vartheta_1}).  Moreover,  for $d_1$ and the sufficiently small constant $\epsilon>0$ in (\ref{eq_firstddefinition}), we denote 
\begin{equation}\label{eq_keylocationdefinition}
\vartheta_1^{\pm}:=\vartheta_1 \pm n^{-1/2+2 \epsilon} d_1 ,
%\left( \log n \right)^{\alpha \delta(\alpha-4)-1} n^{\epsilon},
\end{equation}
and recall that 
\begin{equation}\label{eq_qremoveonecolumn}
Q^{(1)}:=Q-\mathbf{y}_{(1)}\mathbf{y}^{*}_{(1)},
\end{equation}
where $\mathbf{y}_{(1)}$ is the column of $Y$ in (\ref{eq_datamatrix}) associated with $\xi^2_{(1)}$. Accordingly, we denote the largest eigenvalue of $Q^{(1)}$ as $\lambda_1^{(1)} \equiv \lambda_1(Q^{(1)}).$ Throughout this section, we shall prove Figure \ref{fig_locationtizubounded} so that the location of $\lambda_1$ can be quantified with high probability on the event $\Omega_D$. 

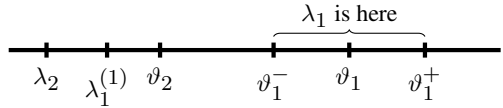
\begin{figure}[ht]
\centering
\begin{tikzpicture}
\foreach \x/\y in {0/\lambda_2, 0.8/\lambda_1^{(1)}, 1.5/\vartheta_2, 3/\vartheta_1^-, 4/\vartheta_1, 5/\vartheta_1^+}
 \draw[ultra thick] (\x,0.1) -- (\x,-0.1) node[below]{$\y$};
\draw[ultra thick] (-0.5,0)--(6,0);
\foreach \mycoord in {(3,0)}
    \draw [mybrace] \mycoord -- node[above, yshift=2mm]{$\lambda_1$ is here} ++(2,0); 
\end{tikzpicture}
\caption{Location of the largest eigenvalue of $Q$.} \label{fig_locationtizubounded} 
\end{figure}

More formally, the main result is summarized in Proposition \ref{lem: eigenvalue rigidity} below. 
\begin{proposition}\label{lem: eigenvalue rigidity}
Suppose Assumptions \ref{assum_model}, \ref{assumption_techincial} and (i) of Assumption \ref{assum_D} hold. For some sufficiently small constant $\epsilon>0$ and $\vartheta_1^{\pm}$ defined in (\ref{eq_keylocationdefinition}), restricted  on the probability event $\Omega_D$ in Lemma \ref{lem_probabilitycontrol}, with high probability, we have that 
\begin{equation*}
\lambda_1 \in [\vartheta_1^-, \vartheta_1^+]. 
\end{equation*}
\end{proposition}

We now proceed to the proof of Proposition \ref{lem: eigenvalue rigidity} following  Figure \ref{fig_locationtizubounded}. 
\begin{proof}[\bf Proof of Proposition \ref{lem: eigenvalue rigidity}]
%Due to similarity, we focus on the setting $2\leq \alpha<\infty$ {\color{red} [revised] and will only discuss the case $\alpha=\infty$ briefly in the end. Almost mention how we condition on the probability event}. 
In what follows, we restrict the discussion on the probability event $\Omega_D$ in Lemma \ref{lem_probabilitycontrol}.
By Weyl's inequality, we have that $\lambda_2 \leq \lambda_1^{(1)}.$ Moreover, by (\ref{eq_keyused}), we see that with high probability $\lambda_1^{(1)}<\vartheta_2.$ The rest of the proof leaves to prove that the following two claims: 
\begin{equation}\label{eq_mu1mu2bound}
\vartheta_1-\vartheta_2 \geq n^{1/\alpha} \log^{-1}n,
\end{equation}
and for $Q^{(1)}$ in (\ref{eq_qremoveonecolumn}) and 
\begin{equation}\label{eq_defnmlambda}
M(\lambda)=1+\mathbf{y}^{*}_{(1)}G_1^{(1)}(\lambda)\mathbf{y}_{(1)}, \ G^{(1)}_1(\lambda):=(Q^{(1)}-\lambda I)^{-1},
\end{equation}
$M(\lambda)$ changes sign with high probability at $\vartheta_1^-$ and $\vartheta_1^+.$  In fact, for $\lambda_1,$ it should satisfy the following equation with high probability
\begin{equation}\label{eq_masterequation}
{\rm det}(\lambda_1I-\mathbf{y}_{(1)}\mathbf{y}^{*}_{(1)}-Q^{(1)})=0\Rightarrow M(\lambda_1)=0,
\end{equation}
as long as $\lambda_1>\lambda_1^{(1)}.$ On the other hand, if $M(\lambda)$ changes sign at $\vartheta_1^{\pm},$ by continuity,  there must at least be an eigenvalue of $Q$ in the interval $[\vartheta_1^-, \vartheta_1^+].$ If (\ref{eq_mu1mu2bound}) holds, combining the above arguments, we see that the only possibility is $\lambda_1$ and it is also true that $\lambda_1>\lambda_1^{(1)}.$

We first justify (\ref{eq_mu1mu2bound}). Recall that $\vartheta_1$ is defined in (\ref{eq: def of vartheta_1}) according to $1+(\xi^2_{(1)}+d_{1})m_{1n}(\vartheta_1)=0. $ Together with (\ref{eq_functionFequal}) and (\ref{eq:F(m,z)}), we readily obtain that 
\begin{equation}\label{eq: 2.2 equa}
1=\frac{1}{n}\sum_{i=1}^p\frac{\sigma_i}{\frac{\vartheta_1}{\xi^2_{(1)}+d_{1}}-\frac{\sigma_i}{n}\sum_{j=1}^n\frac{\xi_{(j)}^2}{\xi^2_{(1)}+d_{1}-\xi^2_{(j)}}}.
\end{equation}
Recall (\ref{e2_definition}). Using the definition of $d_1$ in \eqref{eq_firstddefinition} and \eqref{def1}, we see that on $\Omega_D,$ for some constant $C>0$

\begin{equation}\label{eq: 2.2 beta=2}
\begin{split}
\frac{1}{n}\sum_{j=1}^n\frac{\xi_{(j)}^2}{\xi^2_{(1)}+d_{1}-\xi^2_{(j)}}&=\frac{1}{n}\frac{\xi_{(1)}^2}{d_{1}}+\frac{1}{n}\sum_{j=2}^{n}\frac{\xi_{(j)}^2}{\xi^2_{(1)}+d_{1}-\xi^2_{(j)}}\\
   &\leq \frac{C\log^{1+\epsilon_{\mathsf{p}}} n}{n}+\frac{1}{n}\sum_{j=2}^{n}\frac{\xi_{(j)}^2}{\xi^2_{(1)}+d_{1}-\xi^2_{(2)}}\\
   &\leq\frac{C\log^{1+\epsilon_{\mathsf{p}}}}{n}+\frac{C \log^{\epsilon_{\mathsf{p}}}n}{n^{1/\alpha}}=C e.
\end{split}
\end{equation}

Recall the definition of $\varphi$ in Theorem \ref{thm_main_unbounded}, 
the above arguments imply that on $\Omega_D$
\begin{equation}\label{eq_mu1part}
\frac{\vartheta_1}{\xi_{(1)}^2+d_1}=\varphi +\rO(e).  
\end{equation}
Moreover, by Taylor's expansion, we can further calculate the high order error of 
\begin{align*}
    \frac{\vartheta_1}{\xi^2_{(1)}+d_1}-\varphi=\frac{\xi^2_{(1)}+d_1}{\vartheta_1}\times\frac{1}{n}\sum_{i=1}^n\sigma_i^2\times\frac{1}{n}\sum_{j=2}^n\frac{\xi^2_{(j)}}{\xi^2_{(1)}+d_1-\xi^2_{(j)}}+\mathrm{O}(e^2).
\end{align*}
We have the replacement on the ride-hand-side of the above equation,
\begin{align*}
    \frac{1}{n}\sum_{j=2}^n\frac{\xi^2_{(j)}}{\xi^2_{(1)}+d_1-\xi^2_{(j)}}=\frac{\mathbb{E}\xi^2}{\xi^2_{(1)}+d_1}+\mathrm{O}\big(\frac{1}{n}\sum_{j=2}^n\frac{\xi^4_{(j)}}{(\xi^2_{(1)}+d_1)(\xi^2_{(1)}+d_1-\xi^2_{(j)})}+\frac{1}{n}\big),
\end{align*}
 Then, we have 
\begin{align*}
   \frac{1}{n}\sum_{j=2}^n\frac{\xi^4_{(j)}}{(\xi^2_{(1)}+d_1)(\xi^2_{(1)}+d_1-\xi^2_{(j)})}=\mathrm{O}(\frac{e}{\xi^2_{(1)}+d_1}),
\end{align*}
where we used the estimates in \eqref{def1}.

Then by Assumption \ref{assum_D}, we can obtain that 
\begin{equation*}
	\frac{1}{n}\sum_{j=2}^n\frac{\xi^2_{(j)}}{\xi^2_{(1)}+d_1-\xi^2_{(j)}}=\frac{\mathbb{E}\xi^2}{\xi^2_{(1)}}+\rO(\frac{e}{\xi^2_{(1)}+d_1}).
\end{equation*}  As a result, we have
\begin{align}\label{eq_expansion_vartheta1}
    \frac{\vartheta_1}{\xi^2_{(1)}+d_1}-\varphi=\frac{\xi^2_{(1)}+d_1}{\vartheta_1}\times\frac{1}{n}\sum_{i=1}^n\sigma_i^{2}\times\frac{\mathbb{E}\xi^2}{\xi^2_{(1)}}+\mathrm{O}(e^{2}).
\end{align}
Therefore, we conclude an enhanced form for \eqref{eq_mu1part} that
\begin{equation}\label{eq_mu1linear}
	\vartheta_1=\varphi\times(\xi^2_{(1)}+d_1)+\mathbb{E}\xi^2\times\frac{1}{n\varphi}\sum_{i=1}^n\sigma^2_i+\rO(e\log^{\epsilon_{\mathsf{p}}}n).
\end{equation}

By an analogous argument, we have that for some constants $C_k>0, k=1,2,3,$ 
%
%
%we conclude that 
%\[
%\frac{\mu_2}{\xi^2_{(2)}+d_{2}}=C\phi\bar{\sigma}.
%\]
%Moreover, one has
\[
\begin{split}
\frac{1}{n}\Big(\sum_{j=1}^n \frac{\xi^2_{(j)}}{\xi^2_{(1)}+d_{1}-\xi^2_{(j)}}-\sum_{j=2}^n\frac{\xi^2_{(j)}}{\xi^2_{(2)}+d_{1}-\xi^2_{(j)}}\Big)
%&=\frac{1}{n}\Big(\frac{\xi^2_{(1)}}{d_{1}}+\sum_{j=2}^n\frac{\xi^2_{(j)}}{\xi^2_{(1)}+d_{1}-\xi^2_{(j)}}-\sum_{j=2}^n \frac{\xi^2_{(j)}}{\xi^2_{(2)}+d_{1}-\xi^2_{(j)}}\Big)\\
%&=C_1e-\frac{1}{n}\sum_{j=2}^n\frac{\xi^2_{(j)}(\xi^2_{(1)}-\xi^2_{(2)})}{(\xi^2_{(1)}+d_{1}-\xi^2_{(j)})(\xi^2_{(2)}+d_{1}-\xi^2_{(j)})}\\
\geq C_1 e-\frac{1}{n}\sum_{j=2}^n \frac{\xi^2_{(j)}}{\xi^2_{(2)}+d_{1}-\xi^2_{(j)}}\geq C_2 e,
\end{split}
\]
and on the other hand
\[
\begin{split}
    &\frac{1}{n}\Big(\sum_{j=1}^n \frac{\xi^2_{(j)}}{\xi^2_{(1)}+d_{1}-\xi^2_{(j)}}-\sum_{j=2}^n \frac{\xi^2_{(j)}}{\xi^2_{(2)}+d_{1}-\xi^2_{(j)}}\Big)\\
    &\leq \frac{1}{n}\Big(\frac{\xi^2_{(1)}}{d_{1}}+\sum_{j=2}\frac{\xi^2_{(j)}}{\xi^2_{(1)}+d_{1}-\xi^2_{(j)}}+\sum_{j=2}\frac{\xi^2_{(j)}}{\xi^2_{(2)}+d_{1}-\xi^2_{(j)}}\Big) \leq C_3e.
\end{split}
\]
Using the above control, the definition of $\vartheta_2$ and a discussion similar to (\ref{eq_mu1part}), we can prove 
\begin{equation}\label{eq_mu2part}
\frac{\vartheta_2}{\xi_{(2)}^2+d_1}=\varphi+\rO(e).  
\end{equation}

Combining (\ref{eq_mu1part}) and (\ref{eq_mu2part}), we immediately see that  
\begin{equation}\label{eq_bbbbb}
\frac{\vartheta_1}{\xi^2_{(1)}+d_{1}}-\frac{\vartheta_2}{\xi^2_{(2)}+d_{1}}=\rO(e).
\end{equation}
This implies that 
\[
\begin{split}
    \vartheta_1-\vartheta_2&=(\xi^2_{(1)}+d_{1})\Big(\frac{\vartheta_1}{\xi^2_{(1)}+d_{1}}-\frac{\vartheta_2}{\xi^2_{(2)}+d_{1}}\Big)+\vartheta_2\Big(\frac{\xi^2_{(1)}+d_{1}}{\xi^2_{(2)}+d_{1}}-1\Big)\\
    &=(\xi^2_{(1)}+d_{1})\Big(\frac{\vartheta_1}{\xi^2_{(1)}+d_{1}}-\frac{\vartheta_2}{\xi^2_{(2)}+d_{1}}\Big)+\frac{\vartheta_2}{\xi^2_{(2)}+d_1}(\xi^2_{(1)}-\xi^2_{(2)}) \geq n^{2/\alpha}\log^{-c_{\mathsf{p},0}}n,
\end{split}
\]
where in the third step we used (\ref{eq_bbbbb}), (\ref{def1}) and the definition of $e$ in (\ref{e2_definition}). This completes the proof of (\ref{eq_mu1mu2bound}).

Next, we will show that  $M(\vartheta_1^-)<0,  \ M(\vartheta_1^+)>0. $ Due to similarity, in what follows, we focus on the first inequality. Note that 
\begin{equation*}
\mathbf{y}_1^* G_1^{(1)}(\vartheta_1^-) \mathbf{y}_1= \xi_{(1)}^2  \ub_1^*\Sigma^{1/2} G_1^{(1)}(\vartheta_1^-) \Sigma^{1/2} \ub_1. 
\end{equation*}
Moreover, recall that $m_1^{(1)}(\vartheta_1^-)=n^{-1} \operatorname{tr}(\Sigma^{1/2}G_1^{(1)}(\vartheta_1^-)\Sigma^{1/2}).$ Then according to Lemma \ref{lem:large deviation}, we have that 
\begin{align}\label{eq_conconconone}
\mathbf{y}_1^* G_1^{(1)}(\vartheta_1^-) \mathbf{y}_1& =\xi^2_{(1)} m_1^{(1)}(\vartheta_1^-)+\rO_{\prec} \Big( \frac{\xi^2_{(1)}}{n}\| G_1^{(1)}(\vartheta_1^-)\|_F \Big)  \nonumber \\
& =\xi^2_{(1)} m_1^{(1)}(\vartheta_1^-)+\rO_{\prec} \Big( \frac{\xi^2_{(1)}}{n^{1/2+1/\alpha}}\Big), 
\end{align}
where in the second step we used (\ref{eq_mu1mu2bound}) and the fact $\vartheta_2>\lambda_1^{(1)}$.  Moreover, for some sufficiently small constant $\epsilon_0>0$ and $z_0=\vartheta_1^-+\mathrm{i} n^{-1/2-\epsilon_0},$ we can decompose that  

\begin{align}\label{eq_bigexpansion}
m_1^{(1)}(\vartheta_1^-)& =\left[m_1^{(1)}(\vartheta_1^-)-m_1^{(1)}(z_0) \right]+\left[m_1^{(1)}(z_0)-m_{1n}^{(1)}(z_0)\right]+\left[m_{1n}^{(1)}(z_0)-m_{1n}^{(1)}(\vartheta_1^-)\right]+m_{1n}^{(1)}(\vartheta_1^-) \nonumber \\
&=\mathsf{P}_1+\mathsf{P}_2+\mathsf{P}_3+m_{1n}^{(1)}(\vartheta_1^-).
\end{align}
First, by Theorem \ref{thm_unboundedcaselocallaw}, we have that $\mathsf{P}_2 \prec n^{-1/2-2/\alpha}$. Second, let $\{\mathbf{v}^{(1)}_i\}$ be the eigenvectors of $Q^{(1)}$ associated with the eigenvalues $\{\lambda_i^{(1)}\},$ then we have that
%
%
%
%Suppose $\lambda^{(1)}_i$ be the ordered eigenvalues of $\mathcal{Q}^{(1)}$ with the corresponding eigenvectors $\mathbf{v}^{(1)}_i$, we may rewrite the first distinction as
\[
\begin{split}
|\mathsf{P}_1|&\leq\frac{1}{n}\sum_{i=1}^p|T\mathbf{v}^{(1)}_i|^2\left|\frac{1}{\lambda^{(1)}_i-\vartheta_1^-}-\frac{1}{\lambda^{(1)}_i-z_0}\right|\\
  &=\frac{1}{n}\sum_{i=1}^p|T\mathbf{v}^{(1)}_i|^2\left|\frac{\mathrm{i}n^{-1/2-\epsilon_0}}{(\lambda^{(1)}_i-\vartheta_1^-)(\lambda^{(1)}_i-z_0)}\right|\\
  &\leq\frac{1}{n}\sum_{i=1}^p|T\mathbf{v}^{(1)}_i|^2\Big|\frac{\mathrm{i}n^{-1/2-\epsilon_0}+\rO_{\prec}(n^{-2/\alpha-1/2}\log^2n)}{|\lambda^{(1)}_i-z_0|^2}\Big|\\
  &\prec\operatorname{Im}(m_1^{(1)}(z_0))\times \rO_{\prec}(n^{-1/2-\epsilon_0})
%  &\prec\operatorname{Im}(m_{1n}^{(1)}(z_0))\times O_{\prec}(n^{-1/2-\epsilon})+O_{\prec}(n^{-1-4/\alpha})\\
\prec n^{-1-1/\alpha},
\end{split}
\]
where in the third step we used (\ref{eq_mu1mu2bound}) and the fact $\vartheta_2>\lambda_1^{(1)}$ and in the last step we used Lemma \ref{lem: basic bounds}, (\ref{eq_mu1part}) and (\ref{def1}). Third, according to Lemma \ref{lem_solutionsystem},  we can decompose that 
\[
\begin{split}
   \mathsf{P}_3 &=\frac{1}{n}\sum_{i=1}^p\frac{\sigma_i}{-z_0(1+\frac{\sigma_i}{n}\sum_{j=2}^n\frac{\xi^2_{(j)}}{-z_0(1+m^{(1)}_{1n}(z_0)\xi^2_{(j)})})}-\frac{1}{n}\sum_{i=1}^p\frac{\sigma_i}{-\vartheta_1^-(1+\frac{\sigma_i}{n}\sum_{j=2}^n \frac{\xi^2_{(j)}}{-\vartheta_1^-(1+m^{(1)}_{1n}(\vartheta_1^-)\xi^2_{(j)})})}\\
    &=\big[\frac{1}{n}\sum_{i=1}^p\frac{\sigma_i}{-z_0(1+\frac{\sigma_i}{n}\sum_{j=2}^n \frac{\xi^2_{(j)}}{-z_0(1+m^{(1)}_{1n}(z_0)\xi^2_{(j)})})}-\frac{1}{n}\sum_{i=1}^p\frac{\sigma_i}{-z_0(1+\frac{\sigma_i}{n}\sum_{j=2}^n\frac{\xi^2_{(j)}}{-\vartheta_1^-(1+m^{(1)}_{1n}(\vartheta_1^-)\xi^2_{(j)})})} \big]\\
    &+\big[\frac{1}{n}\sum_{i=1}^p\frac{\sigma_i}{-z_0(1+\frac{\sigma_i}{n}\sum_{j=2}^n \frac{\xi^2_{(j)}}{-\vartheta_1^-(1+m^{(1)}_{1n}(\vartheta_1^-)\xi^2_{(j)})})}-\frac{1}{n}\sum_{i=1}^p\frac{\sigma_i}{-\vartheta_1^-(1+\frac{\sigma_i}{n}\sum_{j=2}^n\frac{\xi^2_{(j)}}{-\vartheta_1^-(1+m^{(1)}_{1n}(\vartheta_1^-)\xi^2_{(j)})})} \big]\\
    &:=\mathcal{M}^{(1)}_{31}+\mathcal{M}^{(1)}_{32}.
\end{split}
\]
Note that according to (\ref{eq_mu1part}), (\ref{eq_mu1mu2bound}), (\ref{def1}) and Lemma \ref{lem: basic bounds}, we conclude that with high probability $|1+\sigma_i m_{2n}^{(1)}(z_0)|, |1+\sigma_i m_{2n}^{(1)}(\vartheta_1^-)|  \asymp 1.$ For $\mathcal{M}^{(1)}_{31}$, using the definitions in  (\ref{eq_systemequationsm1m2}) and the above bounds,  we have that with high probability
\begin{align} \label{eq_m11control}
  &  \mathcal{M}^{(1)}_{31}=\frac{1}{n}\sum_{i=1}^p \frac{\sigma^2_i}{-z_0(1+\sigma_im^{(1)}_{2n}(z))(1+\sigma_im^{(1)}_{2n}(\vartheta_1^-))}\big(m^{(1)}_{2n}(\vartheta_1^-)-m^{(1)}_{2n}(z)\big) \nonumber  \\
    &=\rO(|z_0|^{-1})\times\frac{1}{n}\sum_{j=2}^n \Big(\frac{\xi^2_{(j)}}{-\vartheta_1^-(1+\xi^2_{(j)}m_{1n}^{(1)}(\vartheta_1^-))}-\frac{\xi^2_{(j)}}{-z(1+\xi^2_{(j)}m_{1n}^{(1)}(z_0))}\Big) \nonumber \\
    &=\rO(|z_0|^{-1})\times\frac{1}{n}\sum_{j=2}^n \Big(\frac{\xi^2_{(j)}}{-\vartheta_1^-(1+\xi^2_{(j)}m^{(1)}_{1n}(\vartheta_1^-))}-\frac{\xi^2_{(j)}}{-\vartheta_1^-(1+\xi^2_{(j)}m^{(1)}_{1n}(z_0))}\Big) \nonumber \\
    &+\rO(|z_0|^{-1})\times\frac{1}{n}\sum_{j=2}^n \Big(\frac{\xi^2_{(j)}}{-\vartheta_1^-(1+\xi^2_{(j)}m^{(1)}_{1n}(z))}-\frac{\xi^2_{(j)}}{-z(1+\xi^2_{(j)}m_{1n}^{(1)}(z_0))}\Big) \nonumber \\ 
    &=\rO(|z_0|^{-1})\times\frac{1}{n}\sum_{j=2}^n \frac{\xi^4_{(j)}(m^{(1)}_{1n}(z_0)-m_{1n}^{(1)}(\vartheta_1^-))}{-\vartheta_1^-(1+\xi^2_{(j)}m^{(1)}_{1n}(\vartheta_1^-))(1+\xi^2_{(j)}m^{(1)}_{1n}(z_0))}+\rO(|z_0|^{-1})\times\frac{1}{n}\sum_{j=2}^n \frac{\xi^2_{(j)}(\vartheta_1^--z_0)}{z\vartheta_1^-(1+\xi^2_{(j)}m^{(1)}_{1n}(z_0))} \nonumber \\
    &=\ro(1)\times(m^{(1)}_{1n}(z_0)-m^{(1)}_{1n}(\vartheta_1^-))+\rO(n^{-2/\alpha-1/2-\epsilon_0}). 
\end{align}
where in the second to last step we used Lemma \ref{lem: basic bounds} and in the last step we used Lemma \ref{lem: basic bounds}  and (\ref{def1}). This implies with high probability 
\begin{equation*}\label{control_m11}
\mathcal{M}_{31}^{(1)}=\rO \left( n^{-2/\alpha-1/2-\epsilon_0} \right).  
\end{equation*}

Similarly, for $\mathcal{M}^{(1)}_{32}$, we have that 
\begin{equation}\label{eq_controlm12}
\begin{split}
    \mathcal{M}^{(1)}_{32}=\frac{1}{n}\sum_{i=1}^p \frac{\sigma_i(z_0-\vartheta_1^-)}{z_0\vartheta_1^-(1+\sigma_im^{(1)}_{2n}(\vartheta_1^-))} =\rO(n^{-2/\alpha-1/2-\epsilon_0}).
\end{split}
\end{equation}
Combining the above arguments, we have that $\mathsf{P}_3=\rO\left( n^{-2/\alpha-1/2-\epsilon_0} \right). $

Inserting the bounds of $\mathsf{P}_k, 1 \leq k \leq 3$ into (\ref{eq_bigexpansion}), we conclude that 
\[
|m^{(1)}_1(\vartheta_1^-)-m^{(1)}_{1n}(\vartheta_1^-)|\prec n^{-1/\alpha-1/2-\epsilon_0}.
\]
Together with (\ref{eq_conconconone}) and (\ref{def1}), it yields that  
\begin{equation}\label{eq_M1reducedcase}
  M(\vartheta_1^-)=1+(\xi^2_{(1)}+d_1)m^{(1)}_{1n}(\vartheta_1^-)+\rO_{\prec}(n^{-1/2-\epsilon_0}).
\end{equation}
%{\color{red}[from here]}where we use the fact $m^{(1)}_{1n}(\lambda)\prec \|(\mathcal{Q}^{(1)}-\lambda I)^{-1}\|_2\prec n^{-2/\alpha}$.

In what follows, we study $1+(\xi^2_{(1)}+d_1)m^{(1)}_{1n}(\vartheta_1^-)$. We rewrite that
\begin{equation}\label{eq_reducedcontrol}
\begin{split}
    1+(\xi^2_{(1)}+d_1)m^{(1)}_{1n}(\vartheta_1^-)&=1+(\xi^2_{(1)}+d_1)m^{(1)}_{1n}(\vartheta_1)-(\xi^2_{(1)}+d_1)\big(m^{(1)}_{1n}(\vartheta_1)-m^{(1)}_{1n}(\vartheta_1^-)\big).
\end{split}
\end{equation}
Using the definition for $\vartheta_1$ that $1+(\xi^2_{(1)}+d_1)m_{1n}(\vartheta_1)=0$ and the definitions in (\ref{eq_systemequationsm1m2}), by a discussion similar to (\ref{eq_m11control}),  we have that for some constant $C>0$
\begin{equation}\label{eq: m^s_1_1n-m_1n}
\begin{split}
    &1+(\xi^2_{(1)}+d_1)m^{(1)}_{1n}(\vartheta_1)\\
    &=1+(\xi^2_{(1)}+d_1)m_{1n}(\vartheta_1)+(\xi^2_{(1)}+d_1)\big(m_{1n}^{(1)}(\vartheta_1)-m_{1n}(\vartheta_1)\big)\\
    &=(\xi^2_{(1)}+d_1)\frac{1}{n}\sum_{i=1}^p\Big(\frac{\sigma_i}{-\vartheta_1(1+\sigma_i m_{2n}^{(1)}(\vartheta_1))}-\frac{\sigma_i}{-\vartheta_1(1+\sigma_i m_{2n}(\vartheta_1))}\Big)\\
%    &=(\xi^2_{(1)}+d_1) \left[\frac{1}{n}\sum_{i=1}^p\frac{\sigma^2_i}{-\vartheta_1(1+\sigma_i m_{2n}^{(1)}(\vartheta_1))(1+\sigma_i m_{2n}(\vartheta_1))}\right]\big(m_{2n}(\vartheta_1)-m_{2n}^{(1)}(\vartheta_1)\big)\\
    & \leq C \frac{(\xi^2_{(1)}+d_1)}{\vartheta_1}\times \Big|\frac{1}{n}\sum_{j=1}^n\frac{\xi^2_{(j)}}{-\vartheta_1(1+\xi^2_{(j)}m_{1n}(\vartheta_1))}-\frac{1}{n}\sum_{j=2}^n\frac{\xi^2_{(j)}}{-\vartheta_1(1+\xi^2_{(j)}m^{(1)}_{1n}(\vartheta_1))}\Big|\\
    &\leq C \frac{(\xi^2_{(1)}+d_1)}{\vartheta_1}\times\xi^2_{(2)}|m_{1n}^{(1)}(\vartheta_1)-m_{1n}(\vartheta_1)|+n^{-1}.
\end{split}
\end{equation}
where in the last step we again used (\ref{def1}). This yields that 
\begin{equation*}
|(\xi^2_{(1)}+d_1)\big(m_{1n}^{(1)}(\vartheta_1)-m_{1n}(\vartheta_1)\big)| \leq C \frac{(\xi^2_{(1)}+d_1)}{\vartheta_1}\times\xi^2_{(2)}|m_{1n}^{(1)}(\vartheta_1)-m_{1n}(\vartheta_1)|+n^{-1}, 
\end{equation*}
which implies $(\xi^2_{(1)}+d_1)\big(m^{(1)}_{1n}(\vartheta_1)-m_{1n}(\vartheta_1)\big)=\rO(n^{-1})$. Together with (\ref{eq_reducedcontrol}), we have 
\begin{equation}\label{eq_finalpartone}
1+(\xi^2_{(1)}+d_1)m^{(1)}_{1n}(\vartheta_1^-)=-(\xi^2_{(1)}+d_1)\big(m^{(1)}_{1n}(\vartheta_1)-m^{(1)}_{1n}(\vartheta_1^-)\big)+\rO(n^{-1}).
\end{equation}

Recall that we have proved the facts that $\vartheta_1, \vartheta_1^->\lambda_1^{(1)},$ by Theorem \ref{thm_boundedcaselocallaw} and the monotonicity of $m_1^{(1)}$ outside the bulk,  the first term on the right-hand side of (\ref{eq_finalpartone}) is negative. In order to show $M(\vartheta_1^-)<0,$ in light of (\ref{eq_M1reducedcase}), it suffices to show that its magnitude is much larger than $\rO(n^{-1/2-\epsilon_0}).$  To see this, we decompose that
\[
\begin{split}
    &m^{(1)}_{1n}(\vartheta_1)-m^{(1)}_{1n}(\vartheta_1^-)\\
  %  &=\frac{1}{n}\sum_{i=1}^p\frac{\sigma_i}{-\vartheta_1(1+\sigma_im^{(1)}_{2n}(\vartheta_1))}-\frac{1}{n}\sum_{i=1}^p\frac{\sigma_i}{-\vartheta_1^-(1+\sigma_im^{(1)}_{2n}(\vartheta_1^-))}\\
    &=\Big[\frac{1}{n}\sum_{i=1}^p\frac{\sigma_i}{-\vartheta_1(1+\frac{\sigma_i}{n}\sum_{j=2}^n\frac{\xi^2_{(j)}}{-\vartheta_1(1+\xi^2_{(j)}m_{1n}^{(1)}(\vartheta_1)})}-\frac{1}{n}\sum_{i=1}^p\frac{\sigma_i}{-\vartheta_1(1+\frac{\sigma_i}{n}\sum_{j=2}^n\frac{\xi^2_{(j)}}{-\vartheta_1(1+\xi^2_{(j)}m_{1n}^{(1)}(\vartheta_1^-))})} \Big]\\
    &+\Big[\frac{1}{n}\sum_{i=1}^p\frac{\sigma_i}{-\vartheta_1(1+\frac{\sigma_i}{n}\sum_{j=2}^n\frac{\xi^2_{(j)}}{-\vartheta_1(1+\xi^2_{(j)}m_{1n}^{(1)}(\vartheta_1^-))})}-\frac{1}{n}\sum_{i=1}^p\frac{\sigma_i}{-\vartheta_1^-(1+\frac{\sigma_i}{n}\sum_{j=2}^n\frac{\xi^2_{(j)}}{-\vartheta_1^-(1+\xi^2_{(j)}m_{1n}^{(1)}(\vartheta_1^-))})} \Big]\\
    &:=\tilde{\mathcal{M}}_{11}^{(1)}+\tilde{\mathcal{M}}_{12}^{(1)}.
\end{split}
\]
Similar to the discussion of (\ref{eq_m11control}), we have that $\tilde{\mathcal{M}}_{11}^{(1)}$,
\[
\begin{split}
    \tilde{\mathcal{M}}_{11}^{(1)}&=\frac{1}{n}\sum_{i=1}^p\frac{\sigma^2_i\Big(\frac{1}{n}\sum_{j=2}^n\frac{\xi^2_{(j)}}{-\vartheta_1(1+\xi^2_{(j)}m_{1n}^{(1)}(\lambda))}-\frac{1}{n}\sum_{j=2}^n\frac{\xi^2_{(j)}}{-\vartheta_1(1+\xi^2_{(j)}m_{1n}^{(1)}(\vartheta_1))}\Big)}{-\vartheta_1(1+\frac{\sigma_i}{n}\sum_{j=2}^n\frac{\xi^2_{(j)}}{-\vartheta_1(1+\xi^2_{(j)}m_{1n}^{(1)}(\vartheta_1))})(1+\frac{\sigma_i}{n}\sum_{j=2}^n\frac{\xi^2_{(j)}}{-\vartheta_1(1+\xi^2_{(j)}m_{1n}^{(1)}(\vartheta_1^-))})}\\
    &=\rO(\frac{1}{\vartheta_1})\times\frac{1}{n}\sum_{j=2}^n\frac{\xi^4_{(j)}(m^{(1)}_{1n}(\vartheta_1)-m^{(1)}_{1n}(\vartheta_1^-))}{-\vartheta_1(1+\xi^2_{(j)}m_{1n}^{(1)}(\vartheta_1^-))(1+\xi^2_{(j)}m_{1n}^{(1)}(\vartheta_1))}\\
    &=\ro(1)\times(m^{(1)}_{1n}(\vartheta_1)-m^{(1)}_{1n}(\vartheta_1^-)).
\end{split}
\]
Moreover, similar to (\ref{eq_controlm12}), for $\tilde{\mathcal{M}}_{12}^{(1)}$ we have that with high probability
\[
\begin{split}
    \tilde{\mathcal{M}}_{12}^{(1)}&=\frac{1}{n}\sum_{i=1}^p\frac{\sigma_i(\vartheta_1-\vartheta_1^-)}{\vartheta_1\vartheta_1^-(1+\frac{\sigma_i}{n}\sum_{j=2}^n\frac{\xi^2_{(j)}}{-\vartheta_1(1+\xi^2_{(j)}m_{1n}^{(1)}(\vartheta_1^-))})(1+\frac{\sigma_i}{n}\sum_{j=2}^n\frac{\xi^2_{(j)}}{-\vartheta_1^-(1+\xi^2_{(j)}m_{1n}^{(1)}(\vartheta_1^-))})}\\
    & \asymp \frac{\vartheta_1-\vartheta_1^-}{\vartheta_1^-\vartheta_1}  \asymp n^{-1/2-1/\alpha+\epsilon},
\end{split}
\]
where $\epsilon>0$ is defined in \eqref{eq_keylocationdefinition}. This implies that $m^{(1)}_{1n}(\vartheta_1)-m^{(1)}_{1n}(\lambda) \asymp n^{-1/2-1/\alpha+\epsilon}. $
Together with (\ref{eq_finalpartone}), the definition of {\color{blue} $d_1$} and (\ref{def1}), we readily see that 
\begin{equation}\label{eq: lambda-lambda_(1)}
   1+(\xi^2_{(1)}+{\color{blue} d_1})m^{(1)}_{1n}(\vartheta_1^-) \asymp -n^{-1/2+\epsilon},
\end{equation}
which concludes the proof of $M(\vartheta_1^-)<0$ when $n$ is sufficiently large. Similarly, we can prove that $M(\vartheta_1^+)>0.$ This completes the proof.

\end{proof}

\subsection{The unbounded support and conditional on $X$ case}

In this subsection, we provide the counterpart of the results in Proposition \ref{lem: eigenvalue rigidity} conditional on $X$ via the probability event $\Omega_X$ in Definition \ref{def_OmegaX}. 
%When the multipliers are unbounded, the proofs are 
%
%
%Now, with above preparation, we are going to prove the parallel result of Proposition \ref{lem: eigenvalue rigidity} under $\Omega_X$.
\begin{proposition}\label{lem_largesteigenvaluelocation_conditional}
Suppose Assumptions \ref{assum_model}, \ref{assumption_techincial} and (i) of Assumption \ref{assum_D} hold. For some  small constant $\epsilon>0$ and $\vartheta_1^{\pm}$ defined in (\ref{eq_keylocationdefinition}), restricted on $\Omega_X$ and $\Omega_D$, we have that
\begin{equation*}
\lambda_1 \in [\vartheta_1^-, \vartheta_1^+]. 
\end{equation*}
\end{proposition}
\begin{proof}
The proof of this lemma follows that of Proposition \ref{lem: eigenvalue rigidity}. The main differences mirror those in the modification from Theorem \ref{thm_unboundedcaselocallaw} to Theorem \ref{thm_averagedlocallaw_unboundedmultiplier_conditionalonX}, where the controls in Definition \ref{def_OmegaX} are incorporated. Due to the similarity, we only sketch the main ideas.      
    
     Recall the definition of $\vartheta_1$ and $\vartheta_2$ as
    \begin{align*}
        1+(\xi^2_{(1)}+d_1)m_{1n}(\vartheta_1)=0,\quad 1+(\xi^2_{(2)}+d_1)m_{1n}(\vartheta_2)=0,
    \end{align*}
    where $m_{1n}(z)$ satisfies
    \begin{align*}
        m_{1n}(z)=\frac{1}{n}\sum_{i=1}^p\frac{\sigma_i}{-z+\frac{\sigma_i}{n}\sum_{j=1}^n\frac{\xi^2_{j}}{1+\xi^2_jm_{1n}(z)}}.
    \end{align*}
    Then, based on event $\Omega_D$, we have
    \begin{align*}
        \vartheta_1-\vartheta_2\geq n^{1/\alpha}\log^{-1}n.
    \end{align*}
    On the other hand, for $Q^{(1)}$ and
    \begin{align*}
        M(\lambda)=1+\mathbf{y}_{(1)}^*G_1^{(1)}(\lambda)\mathbf{y}_{(1)},
    \end{align*}
    where we recall that $G_1^{(1)}(\lambda)=(Q^{(1)}-\lambda I)^{-1}$. Under $\Omega_X$ and $\Omega_D$, we have
    \begin{align*}
        \mathbf{y}_{(1)}^*G_1^{(1)}(\vartheta_1^{-})\mathbf{y}_{(1)}=\xi^2_{(1)}m_{1}^{(1)}(\vartheta_1^{-1})+\mathrm{O}(n^{\varepsilon_1}\frac{\xi^2_{(1)}}{n^{1/2+1/\alpha}}),
    \end{align*}
    for small $\varepsilon_1>0$ in Definition \ref{def_OmegaX}. Moreover, it establishes that 
    \begin{align*}
        |m_{1}^{(1)}(\vartheta_1^{-1})-m_{1n}^{(1)}(\vartheta_1^{-1})|\leq n^{\epsilon_1}n^{-1/\alpha-1/2-\epsilon_0},
    \end{align*}
    for small $\epsilon_1>\varepsilon_1$ and {$\epsilon_0>0$ defined from $z_0=\vartheta_1^{-}+\mathrm{i}n^{-1/2-\epsilon_0}$ around \eqref{eq_bigexpansion}.} Therefore, it yields that 
    \begin{align*}
        M(\vartheta_1^{-1})=1+(\xi^2_{(1)}+d_1)m_{1n}^{(1)}(\vartheta_1^-)+\mathrm{O}(n^{\epsilon_3}n^{-1/2-\epsilon_0}),
    \end{align*}
    for small $\epsilon_3>\varepsilon_1$. On the other hand, we have
    \begin{align*}
        1+(\xi^2_{(1)}+d_1)m_{1n}^{(1)}(\vartheta_1^-)=-(\xi^2_{(1)}+d_1)(m_{1n}^{(1)}(\vartheta_1)-m_{1n}^{(1)}(\vartheta_1^-))+\mathrm{O}(n^{-1})\asymp -n^{-1/2+\epsilon},
    \end{align*}
    for $\epsilon>0$ defined in \eqref{eq_keylocationdefinition}. This concludes that $M(\vartheta_1^{-})<0$ when restricted on $\Omega_X$ and $\Omega_D$. Parallel, it can be shown that $M(\vartheta_1^{+})>0$ under $\Omega_X$ and $\Omega_D$. This implies that $M(\lambda)$ changes sign in $[\vartheta_1^{-},\vartheta_1^{+}]$. By the equivalence
    \begin{align*}
        \operatorname{det}(\lambda_1 I-\mathbf{y}_{(1)}\mathbf{y}_{(1)}^*-Q^{(1)})=0\Leftrightarrow M(\lambda_1)=0,
    \end{align*}
 which concludes the proof. 
    \end{proof}

\subsection{The bounded support case}\label{sec_sub_bounded1st} 
In this section, we study the first order convergence of $\lambda_1$ under the assumptions of part (1) of Theorem \ref{thm_main_bounded}. The other parts will be discussed in Section \ref{sec_proof_nonspike}. The main result of this section can be summarized in the following proposition. 

\begin{proposition}\label{prop_boundedsetting} Suppose the assumptions of part (1) of Theorem \ref{thm_main_bounded} hold, then conditional on the probability event $\Omega_D$ as in Lemma \ref{localestimate2}, we have that 
   \begin{gather*}
       \left |\lambda_{1}-\left(\widehat{L}_{+}-\frac{  1-\phi \widehat{\mathsf{s}}_3}{\widehat{\mathsf{s}}_4}\frac{l-\xi^2_{(1)}}{l\xi^2_{(1)}}\right) \right |=\rO_{\mathbb{P}} \left[\frac{1}{n^{1/(d+1)}} \left(\frac{n^{3 \epsilon_{\mathsf{d}}}}{n^{-1/(d+1)+1/2}}+ \frac{\log n}{n^{(\min\{d-1,1\})/(d+1)}} \right) \right].
    \end{gather*}
\end{proposition}

{
In the sequel, to facilitate the argument, we introduce an auxiliary quantity $E_0$, defined by
\begin{align}\label{eq: def of E0}
    \operatorname{Re}m_{1n}(E_0+\mathrm{i}\eta_0)=-(\xi^2_{(1)})^{-1},\quad \eta_0=n^{-1/2-\epsilon_{\mathsf{d}}}.
\end{align}
}
The proof of Proposition \ref{prop_boundedsetting} relies crucially on the following lemma whose justification will be provided in the end of this section.

\begin{lemma}\label{lem_connection} Suppose the assumptions of Proposition \ref{prop_boundedsetting} hold. Conditional on the probability event $\Omega_D$ as in Lemma \ref{localestimate2},  we have that
\begin{equation}\label{eq_holdspartoneoneone}
\lambda_1=E_0+\rO_{\mathbb{P}} \left( n^{-1/2+3\epsilon_{\mathsf{d}}} \right), 
\end{equation}
and
\begin{gather}\label{eq_secondconnect}
    \operatorname{Re}m_{1n}(\lambda_{1}+\ri\eta_0)=-\frac{1}{\xi^2_{(1)}}+\rO_{\mathbb{P}}(n^{-1/2+3\epsilon_{\mathsf{d}}}).
\end{gather}

\end{lemma}

\quad Armed with Lemma \ref{lem_connection}, we proceed to the proof of Proposition \ref{prop_boundedsetting}.

\begin{proof}[\bf Proof of Proposition \ref{prop_boundedsetting}] Recall (\ref{eq_conditionaledgedefinition}), we have that conditional on $\Omega_D,$ $m_{1n}(\widehat{L}_+)=-l^{-1}.$  Together with (\ref{eq_expansionlinear}), we conclude that
\begin{equation*}
\operatorname{Re}m_{1n}(\widehat{L}_++\ri n^{-1/2-\epsilon_{\mathsf{d}}})=-l^{-1}+\rO(n^{-1/2-\epsilon_{\mathsf{d}}}).
\end{equation*}
Moreover, according to the definition in (\ref{eq: def of E0}), we have that
\begin{equation*}
\operatorname{Re} m_{1n}(E_0+\ri n^{-1/2-\epsilon_{\mathsf{d}}})=-(\xi_{(1)}^2)^{-1}. 
\end{equation*}
By (\ref{def4}), we obtain that conditional on $\Omega_D$
\begin{equation*}
\operatorname{Re}m_{1n}(\widehat{L}_++\ri n^{-1/2-\epsilon_{\mathsf{d}}})=\operatorname{Re} m_{1n}(E_0+\ri n^{-1/2-\epsilon_{\mathsf{d}}})+\rO \left( \frac{\log n}{n^{d+1}} \right),
\end{equation*}
which implies that $\widehat{L}_+=E_0+\rO(\log n/n^{d+1})$ according to the stability of the system equation \eqref{eq:F(m,z)}. Let $\Xi$ be the probability event that (\ref{eq_holdspartoneoneone}) holds. We therefore conclude from (\ref{eq_closenessequation}) that  when restricted on $\Xi$ and $n$ is sufficiently large, $\lambda_1+\ri \eta_0 \in \mathbf{D}_b.$ Consequently, 
%Observe that $\lambda_1 \leq l \sigma_1 \lambda_1(XX^*).$ Since $\lambda_1 (XX^*) \rightarrow (1+\sqrt{\phi})^2$ almost surely \cite{bai2009spectral}, we have that $\lambda_1$ is bounded almost surely.  
by part I of Lemma \ref{localestimate2}, we find the following holds on $\Xi$
 \begin{gather}\label{eq_keykeyequationequation}
     m_{1n}( \widehat{L}_{+})-m_{1n}(\lambda_{1}+\ri\eta_0)=\frac{\widehat{\mathsf{s}}_4}{(1-\phi\widehat{\mathsf{s}}_3)}\left(\widehat{L}_{+}-\lambda_1-\ri \eta_0\right)+\rO\left( (\log n) (n^{-1/(d+1)})^{\min\{d,2\}}\right).
 \end{gather}
Again by  $m_{1n}(\widehat{L}_+)=-l^{-1},$ Together with the second part of Lemma \ref{lem_connection},   considering the real parts of both sides of (\ref{eq_keykeyequationequation}), we obtain that on $\Xi$
\begin{equation*}
-l^{-1}+\xi_{(1)}^{-2}+\rO_{\mathbb{P}}(n^{-1/2+3\epsilon_{\mathsf{d}}})=\frac{\widehat{\mathsf{s}}_4}{(1-\phi\widehat{\mathsf{s}}_3)}\left(\widehat{L}_{+}-\lambda_1\right)+\rO\left( (\log n) (n^{-1/(d+1)})^{\min\{d,2\}}\right).
\end{equation*}
This completes our proof. 

\end{proof}

\quad The rest of this section is left to the proof of Lemma \ref{lem_connection}.  We first prove the following lemma, which will be used in the proof of Lemma \ref{lem_connection}. Essentially, it locates the points in $\mathbf{D}_b^\prime$ for which $\operatorname{Im} m_Q(z) \gg  \eta_0$ is near the edge. It is a counterpart of \cite[Lemmas 5.12, 5.13 and 5.15]{lee2016extremal} and \cite[Lemmas 5.13, 5.14 and 5.16]{Kwak2021}. Due to similarity, we only sketch the key points of the proof. Let {$\tilde{z}_0=E_0+\mathrm{i}\eta_0$ defined in (\ref{eq: def of E0}).}

\begin{lemma}\label{lem_keycomponentend}  Suppose the assumptions of Lemma \ref{lem_connection}, we have that the following statements hold with high probability 
\begin{enumerate}
\item[(1).] For any $z=E+\ri \eta_0 \in \mathbf{D}_b^{\prime}$ satisfying $|z-\tilde{z}_0| \geq n^{-1/2+3\epsilon_{\mathsf{d}}},$ 
\begin{equation}\label{eq_outsideetazero}
\operatorname{Im} m_1(z) \asymp \eta_0, \ \operatorname{Im} m_Q(z) \asymp \eta_0.
\end{equation}
\item[(2).] For $m_1^{(1)}(z)$ and $m_Q^{(1)}(z)$ defined around (\ref{eq_defnminorG}), we have that for all $z=E+\ri \eta_0 \in \mathbf{D}_b^{\prime},$
\begin{equation}\label{eq_outsideetazerominor}
\operatorname{Im} m_1^{(1)}(z) \asymp \eta_0, \ \operatorname{Im} m_Q^{(1)}(z) \asymp \eta_0.
\end{equation}
\item[(3).] There exists some $E_0' \in \mathbb{R}$ such that for $z_0'=E_0'+\ri \eta_0,$ the following statements hold simultaneously 
\begin{equation}\label{eq_insideeta}
|z_0'-\tilde{z}_0| \leq n^{-1/2+3 \epsilon_{\mathsf{d}}}, \ \text{and} \ \operatorname{Im} m(z_0') \gg \eta_0. 
\end{equation}
\end{enumerate}

\end{lemma}

\begin{proof}
Due to similarity, we focus our discussion on $m_1(z)$ and will explain the minor differences for $m_Q(z)$ from line to line. Recall (\ref{eq: decomp m_2}). Our proof relies on the following fluctuation average which provides a stronger control on $n^{-1} \sum_{i=1}^p Z_i$ than the one in Remark \ref{remk_zibound}. They are counterparts of Lemmas 5.8 and 5.9 and Corollary 5.10 of \cite{lee2016extremal}. We deter its proof to the Appendix \ref{sec_FAlemma}. 
\begin{lemma}\label{lem_fa} Suppose the assumptions of Lemma \ref{lem_keycomponentend} hold, we have that the followings hold on $\Omega_D$ 
\begin{enumerate}
\item[(1).] For all $i \neq 1$ and all $z=E+\ri \eta_0 \in \mathbf{D}_b^\prime,$ we have
\begin{equation}\label{eq_c5first}
|m_2-m_2^{(1)}| \prec \frac{1}{n \eta_0}, \ \ |m_2-m_2^{(i)}|+|m_2^{(i)}-m_2^{(i1)}| \prec n^{-1+1/(d+1)+4 \epsilon_{\mathsf{d}}}.  
\end{equation}
\item[(2).] For all $z \in \mathbf{D}_b^\prime,$
\begin{equation*}
\left| \frac{1}{n} \sum_{i=2}^n Z_i \right|+\left| \frac{1}{n} \sum_{i=2}^n Z_i^{(1)} \right| \prec n^{-1/2-\frac{1}{2}(\frac{1}{2}-\frac{1}{d+1})+2 \epsilon_{\mathsf{d}}}.
\end{equation*}
\item[(3).] For all $z \in \mathbf{D}_b^\prime,$
\begin{equation*}
\left| \frac{1}{n} \sum_{i=2}^n \frac{(\xi_i^2+\xi_i^4) Z_i}{(1+\xi_i^2 m_{1n}(z))^2 }  \right| \prec n^{-1/2-\frac{1}{2}(\frac{1}{2}-\frac{1}{d+1})+2 \epsilon_{\mathsf{d}}}. 
\end{equation*}
\end{enumerate}
\end{lemma}

\quad Armed with Lemma \ref{lem_fa}, we proceed to finish our proof.  The proof of part (1) is similar to that of (\ref{eq_oneregimeedgecontrol}) by using the local law Theorem \ref{thm_boundedcaselocallaw}. We only provide the key arguments. Using (\ref{eq_m1decompositionfinafinalfinal}), 
(\ref{eq_zibound}), and (\ref{eq_m2uboundeddecomposition}), we see that 
\begin{gather*}
     m_1=-\frac{1}{n}\sum_{i=1}^p \frac{\sigma_i}{z-\frac{\sigma_i}{n}\sum_{j=1}^n\frac{\xi^2_j}{1+\xi_j^2m_1+Z_j}}+\frac{1}{n} \operatorname{tr}(R_1 \Sigma)+\frac{1}{n} \operatorname{tr}(R_2 \Sigma).
\end{gather*}
According to Theorem \ref{thm_boundedcaselocallaw}, by a discussion similar to (\ref{eq: est of m_1 and m_2}) using Remark \ref{remk_zibound}, we have 
\begin{gather}\label{eq_decomposition}
\begin{split}
    \operatorname{Im}m_1(z)
    &=\frac{1}{n}\sum_{i=1}^p \frac{\sigma_i\eta_0}{|z-\frac{\sigma_i}{n}\sum_{j=1}^n \frac{\xi^2_j}{1+\xi^2_jm_{1n}+	Z_j}|^2}+\frac{1}{n}\sum_{i=1}^p \frac{\frac{\sigma_i^2}{n}\sum_{j=1}^n\frac{\xi^4_j\operatorname{Im}m_{1}+\xi^2_j \operatorname{Im} Z_j}{|1+\xi^2_jm_{1n}+Z_j|^2}}{|z-\frac{\sigma_i}{n}\sum_j\frac{\xi^2_j}{1+\xi^2_jm_{1n}+Z_j}|^2}\\
    &+\rO_{\prec}\left(\frac{1}{n}\sum_{j=1}^n \frac{Z_j}{|1+\xi^2_j m_{1n}+Z_j|^2}+\frac{1}{(n \eta_0)^2}\right) \\
    & :=\mathsf{R}_1+\mathsf{R}_2+\mathsf{R}_3.
\end{split}
\end{gather}
Together with Assumption \ref{assum_additional_techinical}, (\ref{def4}) and Remark \ref{remk_zibound}, we find that for some small constant $c'>0,$ when $n$ is sufficiently large, 
\begin{equation}\label{eq_lowerbound}
\left| z-\frac{\sigma_i}{n} \sum_{j=1}^n \frac{\xi_j^2}{1+\xi_j^2 m_{1n}(z)+Z_j} \right| \geq c'. 
\end{equation}
This implies that 
\begin{equation}\label{eq_R1bound}
\mathsf{R}_1 \asymp \eta_0. 
\end{equation}
For $\mathsf{R}_2,$ on the one hand, by Theorem \ref{thm_boundedcaselocallaw} and (\ref{eq_defnmathsfW}), we can conclude that there exists some constant $0<\mathsf{c}'<1,$
\begin{equation*}
\frac{1}{n}\sum_{i=1}^p \frac{\frac{\sigma_i^2}{n}\sum_{j=1}^n\frac{\xi^4_j}{|1+\xi^2_jm_{1n}+Z_j|^2}}{|z-\frac{\sigma_i}{n}\sum_j\frac{\xi^2_j}{1+\xi^2_jm_{1n}+Z_j}|^2} \leq \mathsf{c}'<1.
\end{equation*} 
Moreover, as $|z-\tilde{z}_0| \geq n^{-1/2+3 \epsilon_{\mathsf{d}}},$ according to (\ref{eq_a11coro}) and Remark \ref{remk_zibound}, we find that $1+\xi_j^2 m_{1n}(z)+Z_j \geq C n^{-1/2+3 \epsilon_{\mathsf{d}}}$ for some constant $C>0.$ Together with (\ref{eq_lowerbound}) and (3) of Lemma \ref{lem_fa}, we find that  with high probability  
\begin{equation*}
\frac{1}{n}\sum_{i=1}^p \frac{\frac{\sigma_i^2}{n}\sum_{j=1}^n\frac{\xi^2_j \operatorname{Im} Z_j}{|1+\xi^2_jm_{1n}+Z_j|^2}}{|z-\frac{\sigma_i}{n}\sum_j\frac{\xi^2_j}{1+\xi^2_jm_{1n}+Z_j}|^2} \ll \eta_0.
\end{equation*}
Consequently, we have that with high probability 
\begin{equation}\label{eq_r2control}
\mathsf{R}_2=\mathsf{c}' \operatorname{Im} m_{1}+\ro(\eta_0). 
\end{equation}
Similarly, we can prove that with high probability $\mathsf{R}_3=\ro(\eta_0).$ Together with (\ref{eq_R1bound}), (\ref{eq_r2control}) and (\ref{eq_decomposition}), we can conclude the prove of $m_1(z).$ The discussion for $m_Q(z)$ is similar except we need to use (\ref{eq_m2uboundeddecomposition}) and (\ref{eq_mqdecomposition}).

The proof of part (2) is similar to that of part (1) using Lemma \ref{lem_fa}, Theorem \ref{thm_boundedcaselocallaw} and Remark \ref{remk_zibound}. The idea is analogous to the proof of Lemma 5.13 of \cite{lee2016extremal} or Lemma 5.14 of \cite{Kwak2021}. We omit further details.  

Finally, for part (3), we find from Lemma \ref{lem:large deviation}, (\ref{eq_outsideetazerominor}) and (\ref{lem:Wald}) that for some small constant $\epsilon'<\epsilon_{\mathsf{d}}/2,$ with high probability, 
\begin{equation}\label{eq_boundboundbound121212121}
|Z_1| \leq n^{-1/2+\epsilon'}.  
\end{equation} 
Without loss of generality, we assume that $\xi_1^2 \geq \xi_2^2 \geq \cdots \geq \xi_n^2.$ On the one hand, by the definition of $m_2(z)$ and (\ref{eq: decomp m_2}), we have 
\begin{equation*}
   m_2=\frac{ \xi_1^2 \mathcal{G}_{11}}{n}+\frac{1}{n}\sum_{i=2}^{n}\frac{\xi^2_i}{-z(1+\xi^2_im_1^{(i)}+Z_i)}.
\end{equation*}
Together with Lemma \ref{lem: Resolvent}, we see that 
\begin{gather}\label{eq_keyrepresentationG11}
        \frac{1}{\xi_1^2 \mathcal{G}_{11}}=-z(1+\xi^2_{1}m_1^{(1)}+Z_{1}). 
\end{gather} 
Denote
\begin{equation*}
 z^{\pm}_{0}=\tilde{z}_0\pm n^{-1/2+3\epsilon_{\mathsf{d}}},\quad \tilde{z}_0=E_0+\mathrm{i}\eta_0. 
\end{equation*}
Recall (\ref{eq: def of E0}) that $1+\xi_1^2  \operatorname{Re} m_{1n}(E_0+\ri \eta_0)=0.$ Using Theorem \ref{thm_boundedcaselocallaw}, (\ref{lem:trace_difference}) and (\ref{eq_a11coro}), together with (\ref{eq_outsideetazerominor}) and Remark \ref{remk_zibound}, we conclude that for some constant $C>0$
\begin{equation*}
\frac{1}{\xi_1^2 \mathcal{G}_{11}(z_0^-)} \geq  Cn^{-1/2+3\epsilon_{\mathsf{d}}}, \ \text{and} \  \ \frac{1}{\xi_1^2 \mathcal{G}_{11}(z_0^+)} \leq -Cn^{-1/2+3\epsilon_{\mathsf{d}}}. 
\end{equation*}
Consequently, by continuity, we find that there exists $z_{1}=E_{1}+\ri\eta_0$ with $E_{1}\in(E_0-n^{-1/2+3\epsilon_{\mathsf{d}}},E_0+n^{-1/2+3\epsilon_{\mathsf{d}}})$ that $\operatorname{Re}\mathcal{G}_{11}(z_{1})=0$. For the choice of $z_1,$ together with (\ref{eq_keyrepresentationG11}), we find that 
\begin{gather}\label{eq_form}
    |\operatorname{Im}(z_1 \xi_1^2 \mathcal G_{11}(z_{1}))|=\frac{1}{|\operatorname{Im}m^{(1)}_1(z_{1})+\operatorname{Im}Z_{1}|}\geq n^{1/2-\epsilon_{\mathsf{d}}/2},
\end{gather}
where we used (\ref{eq_boundboundbound121212121}) with the assumption $\epsilon'<\epsilon_{\mathsf{d}}/2$ and (\ref{eq_outsideetazerominor}).  On the other hand, following lines of the proof of \cite[Lemma 5.15]{lee2016extremal}, by a decomposition similar to (\ref{eq: m_1 by m_2}) and a discussion similar to (\ref{eq_decomposition}), using Lemma \ref{lem_fa}, we find that 
\begin{gather*}
        \operatorname{Im}m_1(z_{1}) \asymp \eta_0+\frac{\operatorname{Im} (\xi_1^2 z_1 \mathcal{G}_{11}(z_1))}{n}.
\end{gather*}
Together with (\ref{eq_form}), we conclude that
\begin{equation*}
\operatorname{Im} m_1(z_1) \gg \eta_0. 
\end{equation*}
The discussion for $m_Q$ is similar and we omit the details. This completes our proof. 

\end{proof}

\quad Finally, armed with Lemma \ref{lem_keycomponentend}, we proceed to the proof of Lemma \ref{lem_connection}. Since the details are similar to those of Proposition 4.6 of \cite{lee2016extremal} or Proposition 4.7 of \cite{Kwak2021}, we only provide the key ingredients.  

\begin{proof}[\bf Proof of Lemma \ref{lem_connection}]
We first prove (\ref{eq_holdspartoneoneone}). Using the spectral decomposition of $Q,$ for $m_Q(z)$ in (\ref{eq_mq}), we find that  
\begin{gather}\label{eq_spectraldecompositionuseful}
    \operatorname{Im}m_Q(E+\ri\eta_0)=\frac{1}{n}\sum_{i=1}^n\frac{\eta_0}{(\lambda_i-E)^2+\eta_0^2}.
\end{gather}
This yields that 
\begin{equation*}
\operatorname{Im}m_Q(\lambda_1+\ri\eta_0)\ge(n\eta_0)^{-1}\gg\eta_0,
\end{equation*}
where we used the definition of $\eta_0$ in (\ref{eq: def of E0}). It is clear that $\lambda_1=\rO_{\mathbb{P}}(1).$ First, if $\lambda_1 \in \mathbf{D}_b^\prime,$ then the proof follows directly from (\ref{eq_outsideetazero}). Second, if $\lambda_1 \notin \mathbf{D}_b^{\prime},$ on the other hand, for the upper bound, 
 %{\color{red} need to add more here, the discussion for the upper bound needs to be revised. we do not know whether it is in $\mathbf{D}_b^\prime$ yet. }
by (3) of Lemma \ref{lem_keycomponentend} and (\ref{eq_zopointrate}), using the definition of $\mathbf{D}_b^{\prime}$ in (\ref{eq_spectralparameterprime}), with an argument similar to Proposition 4.7 of \cite{lee2016extremal}, we have that on $\Omega_D,$  $\lambda_1<E_0+n^{-1/2+3\epsilon_{\mathsf{d}}}$ holds with $1-\ro(1)$ probability. On the other hand, for the lower bound, we prove by contradiction.  We assume that $\lambda_1<E_0-n^{-1/2+3\epsilon_{\mathsf{d}}}$. Then we see from (\ref{eq_spectraldecompositionuseful}) that $\operatorname{Im}m_Q(E+\ri\eta_0)$ is a decreasing function of $E$ on the interval $(E_0-n^{-1/2+3\epsilon_{\mathsf{d}}},E_0+n^{-1/2+3\epsilon_{\mathsf{d}}})$. However, from Lemma \ref{lem_keycomponentend} and its proof (recall that $z_0^-=E_0-n^{-1/2+3\epsilon_{\mathsf{d}}}$), we have seen that, $\operatorname{Im}m_Q(z_0)\gg\eta_0$, $\operatorname{Im}m_Q(z_0^{-})\asymp\eta_0,$ which is a contradiction. It implies that $\lambda_1\geq E_0-n^{-1/2+3\epsilon_{\mathsf{d}}}$ and completes the proof of (\ref{eq_holdspartoneoneone}). 

\quad Then we prove (\ref{eq_secondconnect}). By (\ref{eq_a11coro}), (\ref{eq_zopointrate}) and (\ref{eq_oneregimeedgecontrol}), we find that 
\begin{equation*}
\operatorname{Re} m_{1n}(\lambda_1+\ri \eta_0)=\operatorname{Re} m_{1n}(z_0)+\rO_{\mathbb{P}}(n^{-1/2+3 \epsilon_{\mathsf{d}}})=-\frac{1}{\xi_{(1)}^2}+\rO_{\mathbb{P}}(n^{-1/2+3 \epsilon_{\mathsf{d}}}),
\end{equation*}
where in the last step we used the definition (\ref{eq: def of E0}). This completes our proof. 
\end{proof}

\subsection{The bounded support and conditional on $X$ case}
In this section, we study the convergence of $\lambda_1$ in part (1) of Theorem \ref{thm_main_bounded_conditional}, in parallel with Theorem \ref{thm_main_bounded}. The proof for Theorem \ref{thm_main_bounded_conditional} will be deferred to Section \ref{sec_proof_nonspike}. In the sequel, we provide the key observation to prove the statement (1) of Theorem \ref{thm_main_bounded_conditional}, which is the analogue of Proposition \ref{prop_boundedsetting}.
% {\color{green}[deterministically.]}

\begin{proposition}\label{prop_boundedsetting_conditional} Suppose the assumptions of part (1) of Theorem \ref{thm_main_bounded_conditional} hold, then restricted on $\Omega_X$ and $\Omega_D$, we have
   \begin{gather*}
       \left |\lambda_{1}-\left(\widehat{L}_{+}-\frac{  1-\phi \widehat{\mathsf{s}}_3}{\widehat{\mathsf{s}}_4}\frac{l-\xi^2_{(1)}}{l\xi^2_{(1)}}\right) \right |=\rO \left[\frac{1}{n^{1/(d+1)}} \left(\frac{n^{\mathrm{c} \epsilon_{\mathsf{d}}}}{n^{-1/(d+1)+1/2}}+{ \frac{\log n}{n^{(\min\{d-1,1\})/(d+1)}}} \right) \right],
    \end{gather*}
    {for some constant $\mathrm{c}>3$ and $(\mathrm{c}-3)\epsilon_{\mathrm{d}}>\varepsilon_1$, where $\varepsilon_1$ is introduced in Definition \ref{def_OmegaX}.}
\end{proposition}
\begin{proof}
The proof of Proposition \ref{prop_boundedsetting_conditional} follows similar arguments to that of Proposition \ref{prop_boundedsetting} and relies primarily on the following lemma.
    \begin{lemma}\label{lem_boundedsupport_conditional_lem1}
        Suppose the assumptions of Theorem \ref{thm_main_bounded_conditional} hold, on events $\Omega_X$ and $\Omega_D$, we have
        \begin{align*}
            \lambda_{1}=E_0+\mathrm{O}({ n^{-1/2+\mathrm{c}\epsilon_{\mathsf{d}}}}),
        \end{align*}
        where $E_0$ is defined in Lemma \ref{lem_connection}, and 
        \begin{align*}
            \operatorname{Re}m_{1n}(\lambda_{1}+\mathrm{i}\eta_0)=-\frac{1}{\xi^2_{(1)}}+\mathrm{O}({n^{-1/2+\mathrm{c}\epsilon_{\mathsf{d}}}}),
        \end{align*}
        {for some constant $\mathrm{c}>3$ and $(\mathrm{c}-3)\epsilon_{\mathrm{d}}>\varepsilon_1$.}
    \end{lemma}
%   

%    \begin{align*}
%        \operatorname{Re}m_{1n}(E_0+\mathrm{i}\eta_0)=-\xi^{-2}_{(1)},\quad \eta_0=n^{-1/2-\epsilon_{\mathsf{d}}}.
%    \end{align*}
%    The above lemma follows immediately from the next observation, whose proof is omitted.
%    \begin{lemma}
%    Suppose the assumptions of Lemma \ref{lem_boundedsupport_conditional_lem1}, we have that the followings holds deterministically 
%\begin{enumerate}
%\item[(1).] For any $z=E+\ri \eta_0 \in \mathbf{D}_b^{\prime}$ satisfying that $|z-z_0| \geq n^{-1/2+3\epsilon_{\mathsf{d}}},$ 
%\begin{equation*}
%\operatorname{Im} m_1(z) \asymp \eta_0, \ \operatorname{Im} m_Q(z) \asymp \eta_0.
%\end{equation*}
%\item[(2).] For $m_1^{(1)}(z)$ and $m_Q^{(1)}(z)$ defined around (\ref{eq_defnminorG}), we have that for all $z=E+\ri \eta_0 \in \mathbf{D}_b^{\prime},$
%\begin{equation*}
%\operatorname{Im} m_1^{(1)}(z) \asymp \eta_0, \ \operatorname{Im} m_Q^{(1)}(z) \asymp \eta_0.
%\end{equation*}
%\item[(3).] There exists some $E_0' \in \mathbb{R}$ such that for $z_0'=E_0'+\ri \eta_0,$ the followings holds simultaneously 
%\begin{equation*}
%|z_0'-z_0| \leq n^{-1/2+3 \epsilon_{\mathsf{d}}}, \ \text{and} \ \operatorname{Im} m(z_0') \gg \eta_0. 
%\end{equation*}
%\end{enumerate}
%\end{lemma}
The proof of Lemma \ref{lem_boundedsupport_conditional_lem1} follows from an argument similar to Lemma \ref{lem_connection}. In fact, using the spectral decomposition of $Q$ and the results in Lemma \ref{lem_keycomponentend} (which now holds surely on $\Omega_D$ and $\Omega_X$), it yields that $\operatorname{Im}m_Q(\lambda_1+\mathrm{i}\eta_0)\geq (n\eta_0)^{-1}\gg\eta_0$. On the other hand, on events $\Omega_X$ and $\Omega_D$, we have $\lambda_1<E_0+n^{-1/2+\mathrm{c}\epsilon_{\mathsf{d}}}$. Since $\operatorname{Im}m_Q(E+\mathrm{i}\eta_0)$ is decreasing in $E$ over the interval $(E_0-n^{-1/2+\mathrm{c}\epsilon_{\mathsf{d}}},E_0+n^{-1/2+\mathrm{c}\epsilon_{\mathsf{d}}})$, a contradiction argument shows that  $\lambda_1\geq E_0-n^{-1/2+\mathrm{c}\epsilon_{\mathsf{d}}}$. This proves the first statement of Lemma \ref{lem_boundedsupport_conditional_lem1}. For the second statement, Lemma \ref{localestimate2} yields 
\begin{align*}
    \operatorname{Re}m_{1n}(\lambda_1+\mathrm{i}\eta_0)=\operatorname{Re}m_{1n}(z_0)+\mathrm{O}(n^{-1/2+\mathrm{c}\epsilon_{\mathsf{d}}})=-\frac{1}{\xi^2_{(1)}}+\mathrm{O}(n^{-1/2+\mathrm{c}\epsilon_{\mathsf{d}}}),
\end{align*}
on the event $\Omega_D$.  

Combining the results with the an argument similar to  the proof of Proposition \ref{prop_boundedsetting}, we can prove Proposition \ref{prop_boundedsetting_conditional}. 
\end{proof}

\section{Proof of the results of Section \ref{sec_boostrapeffect_nonspike}}\label{sec_proof_nonspike}

In this section, we prove Theorems \ref{thm_main_unbounded} and \ref{thm_main_bounded} using the results in Sections \ref{sec_sub_unbounded1st} and \ref{sec_sub_bounded1st}. 

\subsection{Proof of the results in Section \ref{sec_nonspike_thebad}}
\subsubsection{Proof of Theorem \ref{thm_main_unbounded}}\label{sec_proofoftheorem31}
 Due to the similarity of the arguments, we focus on proving (\ref{eq_mainresultunboundedfrechet}) for the case where the multipliers have polynomially decaying tails as in \eqref{ass3.1}. Combining (\ref{eq_mu1part}) and  Lemma \ref{lem_probabilitycontrol}, when $n$ is sufficiently large, restricted on event $\Omega_D$, we have that $\vartheta_1=(\xi_{(1)}^2+d_1)\left(\phi \bar{\sigma}_1 +\rO_{\mathbb{P}}(e) \right).$ Together with Proposition \ref{lem: eigenvalue rigidity}, we find that
\begin{equation*}
\lambda_1=(\xi_{(1)}^2+d_1)\left(\phi \bar{\sigma}_1 +\rO_{\mathbb{P}}(e) \right)+\rO_{\mathbb{P}}\left( n^{-1/2+2\epsilon}d_1 \right). 
\end{equation*}
With the above preparation, we now proceed to the proof. For notional simplicity, we denote 
\begin{align}\label{eq_conventionsYnZn}
     Y_n:=\frac{\lambda_1}{\varphi b_n},\quad Z_n:=\frac{\xi^2_{(1)}}{b_n}.
\end{align}
According to the definition of $b_n$ in (\ref{eq_defnbn}), under the assumption of (\ref{ass3.1}), we readily obtain that $b_n \asymp n^{1/\alpha}.$ Therefore, using the definition of $d_1$ in \eqref{eq_firstddefinition}, we obtain on the event $\Omega_D$ that
\begin{align*}
    |Y_n-Z_n|=\mathrm{o}_{\mathbb{P}}(1).
\end{align*}
Equivalently, we have that 
\begin{align}\label{eq_difference}
    |Y_n-Z_n|\mathbf{1}_{\Omega_D}=\mathrm{o}_{\mathbb{P}}(1).
\end{align}
In fact, the above statement still holds with high probability even without restricting on $\Omega_D$ with Lemma \ref{lem_probabilitycontrol}. To see this, we notice that, for any $\varepsilon>0$,
\begin{align*}
    \mathbb{P}(|Y_n-Z_n|>\varepsilon)\leq\mathbb{P}(\Omega_D^c)+\mathbb{P}\big(\{|Y_n-Z_n|>\varepsilon\}\cap\Omega_D\big).
\end{align*}

Together with Lemma \ref{lem_probabilitycontrol} and (\ref{eq_difference}), we obtain that 
\begin{align}\label{eq_ffffuuuuuu}
    |Y_n-Z_n|=\mathrm{o}_{\mathbb{P}}(1).
\end{align}
In addition, by Lemma \ref{lem_summaryevt}, for any $x \geq 0$,
\begin{align*}
 \mathbb{P}(Z_n\leq x)=\exp(-x^{-\alpha})+\mathrm{o}(1).
\end{align*}
We can therefore complete the proof for $Y_n = Y_n - Z_n + Z_n$ using the above results and Slutsky’s theorem.

For (\ref{eq_mainresultunboundedgumbel}) under the assumption that the multipliers have exponentially decaying tails as in (\ref{ass3.2}), the arguments are similar. Under the assumption of (\ref{ass3.2}), we have that $b_n \asymp \log^{1/\beta}n.$ Denote 
\begin{align}\label{eq_mathsfyz}
    \mathsf{Y}_n:=\mathsf{g}^{\prime}(b_n)\big(\varphi^{-1}\lambda_1-(b_n+\mathbb{E}\xi^2\frac{\bar{\sigma}_2}{\varphi\bar{\sigma}_1})\big),\quad \mathsf{Z}_n:=\mathsf{g}^{\prime}(b_n)\big(\xi^2_{(1)}-b_n\big).
\end{align}
By an argument similar to (\ref{eq_ffffuuuuuu}), using Proposition \ref{lem: eigenvalue rigidity} and its proof, we can show that
\begin{align*}
    |\mathsf{Y}_n-\mathsf{Z}_n|=\mathrm{o}_{\mathbb{P}}(1).
\end{align*}
Then the proof follows from Lemma \ref{lem_summaryevt} and Slutsky’s theorem.

\subsubsection{Proof of Theorem \ref{thm_main_unbounded_conditional}}\label{sec_proof33}

The proofs are similar to those presented in Section~\ref{sec_proofoftheorem31}, with the additional need to address arguments involving conditional probability. We focus on the proof of~(\ref{eq_mainresultunboundedfrechet_conditional}) when the multipliers have polynomially decaying tails as in \eqref{ass3.1}. We also recall the probabilistic notation conventions introduced in Remark~\ref{rmk_notationconventions}. When restricted on $\Omega_X$ and $\Omega_D,$ according to \eqref{eq_mu1part}, we have that
    \begin{align*}
        \vartheta_1=(\xi^2_{(1)}+d_1)(\phi\bar{\sigma}_1+\mathrm{O}(e)). 
    \end{align*}
    Together with the results of Proposition \ref{lem_largesteigenvaluelocation_conditional}, we have that when restricted on $\Omega_X$ and $\Omega_D,$
    \begin{align*}
        \lambda_1=(\xi^2_{(1)}+d_1)(\phi\bar{\sigma}_1+\mathrm{O}(e))+\mathrm{O}(n^{-1/2+2\epsilon}d_1).
    \end{align*}
Using the conventions in (\ref{eq_conventionsYnZn}) and an argument similar to (\ref{eq_difference}), we have that for any $\delta>0$, there exists $N_\delta$ such that for all $n \geq N_\delta$,
\begin{align}\label{eq_prf_unboundedconditional_1}
\left|Y_n- Z_n \right|
\leq \delta
\qquad \text{on} \ \Omega_X \cap \Omega_D,
\end{align}
where we refer to Remark \ref{rmk_notationconventions} for the precise definitions of $\Omega_X \cap \Omega_D$.

Now we proceed to the proof. For all $x>0,$ we denote the following events
    \begin{align*}
        A_n(x):=\{Y_n\leq x\},\quad B_n(x):=\{Z_n\leq x\}.
    \end{align*}
Then on the event $\Omega_X \cap \Omega_D$, for $\delta \equiv \delta(x) \in (0,x),$ by (\ref{eq_prf_unboundedconditional_1}), we have that     
    \begin{align*}
        \big|Y_n-Z_n\big|\leq\delta.
    \end{align*}
The above arguments directly imply that   
    \begin{align*}
        B_n(x-\delta)\cap\Omega_X\cap\Omega_D\subset A_n(x),\quad A_n(x)\cap\Omega_X\cap\Omega_D\subset B_n(x+\delta).
    \end{align*} 
This yields that, on $\Omega_X,$    
    \begin{align*}
        B_n(x-\delta)\cap\Omega_D\subset A_n(x),\quad A_n(x)\subset B_n(x+\delta)\cup\Omega_D^c.
    \end{align*}
%Consequently, together with the monotonicity of conditional expectation, on the one hand, we have that 
%\begin{align}
%\mathbf{1}_{\Omega_X}\,1_{B_n(x-\delta)\cap \Omega_D}
%\le
%\mathbf{1}_{\Omega_X}\,1_{A_n(x)} \quad \text{a.s.}
%\end{align}

Applying conditional expectation and using the monotonicity property, we obtain
\begin{align}\label{eq_keyone}
\mathbf{1}_{\Omega_X}\,\mathbb{P}\big(B_n(x-\delta)\cap \Omega_D \mid X\big)
\leq
\mathbf{1}_{\Omega_X}\,\mathbb{P}\big(A_n(x)\mid X\big)
\quad \text{a.s.}
\end{align} 
and      
\begin{align*}
\mathbf{1}_{\Omega_X}\,\mathbb{P}\big(A_n(x)\mid X\big)
\leq
\mathbf{1}_{\Omega_X}\,\mathbb{P}\big(B_n(x+\delta)\mid X\big)
+
\mathbf{1}_{\Omega_X}\,\mathbb{P}\big(\Omega_D^c\mid X\big)
\quad \text{a.s.}
\end{align*}
In addition, note that
\begin{align*}
B_n(x-\delta)\cap \Omega_D
&=
B_n(x-\delta)\setminus \big(B_n(x-\delta)\cap \Omega_D^c\big). 
\end{align*}
Consequently, we have
\begin{align*}
\mathbf{1}_{\Omega_X}\,\mathbb{P}\big(B_n(x-\delta)\cap \Omega_D \mid X\big)
&=
\mathbf{1}_{\Omega_X}\,\mathbb{P}\big(B_n(x-\delta)\mid X\big)
-\mathbf{1}_{\Omega_X}\,\mathbb{P}\big(B_n(x-\delta)\cap \Omega_D^c \mid X\big) \\
&\geq
\mathbf{1}_{\Omega_X}\,\mathbb{P}\big(B_n(x-\delta)\mid X\big)
-\mathbf{1}_{\Omega_X}\,\mathbb{P}\big(\Omega_D^c \mid X\big).
\end{align*}
Combining this with (\ref{eq_keyone}), we obtain
\begin{align}
\mathbf{1}_{\Omega_X}\,\mathbb{P}\big(B_n(x-\delta)\mid X\big)
-\mathbf{1}_{\Omega_X}\,\mathbb{P}\big(\Omega_D^c \mid X\big)
\leq
\mathbf{1}_{\Omega_X}\,\mathbb{P}\big(A_n(x)\mid X\big)
\quad \text{a.s.}
\end{align}

Combining the above arguments, we see that 
\begin{align}\label{eq_conditional_squeeze_unbounded}
\mathbf{1}_{\Omega_X}\Big[
\mathbb{P}\big(B_n(x-\delta)\mid X\big)
-\mathbb{P}\big(\Omega_D^c\mid X\big)
\Big]
\leq
\mathbf{1}_{\Omega_X}\,\mathbb{P}\big(A_n(x)\mid X\big)
\leq
\mathbf{1}_{\Omega_X}\Big[
\mathbb{P}\big(B_n(x+\delta)\mid X\big)
+\mathbb{P}\big(\Omega_D^c\mid X\big)
\Big] \ \text{a.s.}
\end{align}

Since both $B_n(\cdot)$ and $\Omega_D$ are $\sigma(D)$-measurable and $D$ is independent of $X$, we have
\begin{align*}
\mathbb{P}\big(B_n(\cdot)\mid X\big)=\mathbb{P}\big(B_n(\cdot)\big),
\qquad
\mathbb{P}\big(\Omega_D^c\mid X\big)=\mathbb{P}\big(\Omega_D^c\big)
\quad \text{a.s.}
\end{align*}
Thus, on $\Omega_X$, almost surely,
\begin{align*}
\mathbb{P}(B_n(x-\delta))-\mathbb{P}(\Omega_D^c)
\leq
\mathbb{P}(A_n(x)\mid X)
\leq
\mathbb{P}(B_n(x+\delta))+\mathbb{P}(\Omega_D^c).
\end{align*}
Together with Lemmas \ref{lem_summaryevt} and \ref{lem_probabilitycontrol}, we conclude that on $\Omega_X$, almost surely, for any $\delta \in (0,x),$ 
    \begin{align}\label{eq_here}
        \exp \left(-(x-\delta)^{-\alpha}\right)+\mathrm{o}(1)\leq \mathbb{P}\left(\frac{\lambda_1}{\varphi b_n}\leq x|X \right)\leq    \exp \left(-(x+\delta)^{-\alpha}\right)+\mathrm{o(1)}.
    \end{align}
This implies for any fixed $x>0$ by letting $\delta \downarrow 0$, we have  on $\Omega_X$, almost surely,
\begin{equation*}
\Delta_n:=\left| \mathbb{P}\left(\frac{\lambda_1}{\varphi b_n}\leq x|X \right)- \exp \left(-x^{-\alpha}\right) \right|=\mathrm{o}(1). 
\end{equation*} 
Furthermore, to get rid of $\Omega_X,$ we have that for any $\varepsilon>0$,
\begin{align}\label{eq_unboundedconditional_getridofOmegaX}
    \mathbb{P}\big(\Delta_n>\varepsilon\big)\leq
    \mathbb{P}(\Omega_X^c)+
    \mathbb{P}\big(\{\Delta_n>\varepsilon\}\cap\Omega_X\big).
\end{align}  
Combining with Lemma \ref{lem_Xgoodevents}, we readily obtain that 
    \begin{align*}
        \mathbb{P}(\Delta_n>\epsilon)=\mathrm{o}(1).
    \end{align*}  
This completes the proof for all $x>0.$ When $x \downarrow 0,$ using (\ref{eq_here}), by setting $\delta=x/2,$ we have  
    \begin{align*}
        \exp \left(-2^{\alpha} x^{-\alpha}\right)+\mathrm{o}(1)\leq \mathbb{P}\left(\frac{\lambda_1}{\varphi b_n}\leq x|X \right)\leq    \exp \left(-x^{-\alpha}\right)+\mathrm{o(1)}.
    \end{align*} 
    Since $x \downarrow 0,$ we have that   
        \begin{align*}
        \exp \left(-x^{-\alpha}\right)+\mathrm{o}(1)\leq \mathbb{P}\left(\frac{\lambda_1}{\varphi b_n}\leq x|X \right)\leq    \exp \left(-x^{-\alpha}\right)+\mathrm{o(1)}.
    \end{align*}
The remainder of the proof follows by an argument similar to the case $x>0$, which completes the proof.

For (\ref{eq_mainresultunboundedgumbel_conditional}), under the assumption that the multipliers have exponentially decaying tails as in (\ref{ass3.2}), the argument is similar to the discussion at the end of Section \ref{sec_proofoftheorem31}. In particular, one works with $\mathsf{Y}_n$ and $\mathsf{Z}_n$ in (\ref{eq_mathsfyz}) and applies essentially the same verbatim argument as in the proof of (\ref{eq_mainresultunboundedfrechet_conditional}) given above. Due to this similarity, we omit the details.

\subsection{Proof of the results in Section \ref{sec_nonspike_thegood}}\label{sec_prf_nonspike_thegood}
\subsubsection{Proof of Theorem \ref{thm_main_bounded}}\label{sec_prf_main_bounded}
For part (1), it follows from Proposition \ref{prop_boundedsetting} that 
\begin{gather*}
       \left |\lambda_{1}-\left(\widehat{L}_{+}-\frac{  1-\phi \widehat{\mathsf{s}}_3}{\widehat{\mathsf{s}}_4}\frac{l-\xi^2_{(1)}}{l\xi^2_{(1)}}\right) \right |=\rO_{\mathbb{P}} \left[\frac{1}{n^{1/(d+1)}} \left(\frac{n^{3 \epsilon_{\mathsf{d}}}}{n^{-1/(d+1)+1/2}}+{\frac{\log n}{n^{(\min\{d-1,1\})/(d+1)}} }\right) \right].
    \end{gather*}
Changing $\xi^2_{(1)}$ to $l$ and introducing an additional $n^{-1/(d+1)}$ error, we have the following.
\begin{gather*}
       \left |\lambda_{1}-\left(\widehat{L}_{+}-\frac{  1-\phi \widehat{\mathsf{s}}_3}{\widehat{\mathsf{s}}_4}\frac{l-\xi^2_{(1)}}{l^2}\right) \right |=\rO_{\mathbb{P}} \left[\frac{1}{n^{1/(d+1)}} \left(\frac{n^{3 \epsilon_{\mathsf{d}}}}{n^{-1/(d+1)+1/2}}+{\frac{\log n}{n^{(\min\{d-1,1\})/(d+1)}}} \right) \right].
    \end{gather*}
Then, by \eqref{eq_closenessequation} of Lemma \ref{localestimate2}, together with Lemma \ref{lem_summaryevt}, we can conclude the proof using an argument similar to those between (\ref{eq_conventionsYnZn}) and (\ref{eq_ffffuuuuuu}). 

Then we proceed to the proof of parts (2) and (3). Following \cite[Lemma 2.5]{ding2021spiked}, we see that restricted on the event $\Omega_D$ in Lemma \ref{lem_probabilitycontrol}, $(\widehat{L}_+, m_{1n}(\widehat{L}_+))$ should satisfy the following systems of equations 
\begin{gather}\label{eq_edgeequationstwodecide}
    m_{1n}(\widehat{L}_{+})=\frac{1}{n}\sum_{i=1}^p \frac{\sigma_i}{-\widehat{L}_{+}+\frac{\sigma_i}{n}\sum_j\frac{\xi^2_j}{1+\xi^2_jm_{1n}(\widehat{L}_{+})}},\quad 1=\frac{1}{n}\sum_{i=1}^p \frac{\frac{\sigma_i^2}{n}\sum_j\frac{\xi^4_j}{|1+\xi^2_jm_{1n}(\widehat{L}_{+})|^2}}{\left|\widehat{L}_{+}-\frac{\sigma_i}{n}\sum_j\frac{\xi^2_j}{1+\xi^2_jm_{1n}(\widehat{L}_{+})}\right|^2}.
\end{gather}
Similarly, $(L_+, m_{1n,c}(L_+))$ should satisfy the following equations
\begin{gather}\label{eq_edgeequationstwodecide2}
    m_{1n,c}(L_{+})=\frac{1}{n}\sum_i\frac{\sigma_i}{-L_{+}+\sigma_i\int\frac{s}{1+sm_{1n,c}(L_{+})}\mathrm{d}F(s)},\quad 1=\frac{1}{n}\sum_i\frac{\sigma^2_i\int\frac{s^2}{|1+sm_{1n,c}(L_{+})|^2}\mathrm{d}F(s)}{|L_{+}-\sigma_i\int\frac{s}{1+sm_{1n,c}(L_{+})}\mathrm{d}F(s)|^2}.
\end{gather}
Using the definitions of $\mathsf{s}_k$ and $\widehat{\mathsf{s}}_k, 1 \leq k \leq 3,$ by (\ref{def4}) and an argument similar to part (b) of Lemma \ref{localestimate2}, when $n$ is sufficiently large, we see that our assumption $\phi^{-1}<\mathsf{s}_3$ implies $\phi^{-1}<\widehat{\mathsf{s}}_3$ on $\Omega_D.$ This yields that for some constant $\delta>0$
\begin{gather}\label{eq_>1plusdelta}
    \frac{1}{n} \sum_i\frac{\sigma_i^2 \widehat{\mathsf{s}}_1}{(\widehat{L}_{+}-\sigma_i\widehat{\mathsf{s}}_2)^2} >1+\delta,\quad \frac{1}{n}\sum_i\frac{\sigma_i^2\mathsf{s}_1}{(L_{+}-\sigma_i\mathsf{s}_2)^2}>1+\delta.
\end{gather}
where the first inequality is restricted to the event $\Omega_D$. From now on, for notional simplicity, we always restrict ourselves to $\Omega_D$ so that the discussion is purely deterministic. Recall (\ref{eq_phasetransition}) and (\ref{eq_finitesample123}). Combining the second equations in (\ref{eq_edgeequationstwodecide}) and (\ref{eq_edgeequationstwodecide2}) with (\ref{eq_>1plusdelta}), we readily obtain 
\begin{equation}\label{rem1nbound}
m_{1n}(\widehat{L}_{+})>-l^{-1}, \ \ m_{1n,c}(L_{+})>-l^{-1}.
\end{equation}
Together with (\ref{def4}), we  have that for all $1 \leq j \leq n$ and some constant $\delta'>0$ 
\begin{equation}\label{eq_keyboundhastobeenused}
\frac{1}{\left|1+\xi_j^2 m_{1n}(\widehat{L}_+)\right|} \geq \delta', \   \frac{1}{\left|1+s m_{1n,c}(L_+)\right|} \geq \delta' \ \text{for any} \ 0<s \leq l. 
\end{equation}
We now proceed to the proof. The proof consists of two steps. In the first step, we prove the results assuming that
\begin{equation}\label{eq_keyanasztaz}
\left| m_{1n,c}(L_+)-m_{1n}(\widehat{L}_+) \right|=\mathrm{O}_{\mathbb{P}}(n^{-1/2}), \ |L_+-\widehat{L}_+|=\rO_\mathbb{P}( n^{-1/2}). 
\end{equation}
In the second step, we justify (\ref{eq_keyanasztaz}). We start with step one. 

\vspace{3pt}

\noindent{\bf Step one:} Under the assumption \ref{eq_keyanasztaz}, the key component of the proof is the following lemma. Denote
\begin{align*}
&\mathsf{C}_1:=\frac{1}{n} \sum_{i=1}^p \frac{\sigma_i}{ \left( L_+-\sigma_i \int \frac{s}{1+sm_{1n,c}} \mathrm{d} F(s) \right)^2},\quad \mathsf{C}_2:=\frac{1}{n} \sum_{i=1}^p \frac{\sigma_i^2}{ \left( L_+-\sigma_i \int \frac{s}{1+sm_{1n,c}} \mathrm{d} F(s) \right)^2},
\end{align*}
and
\begin{equation*}
 \mathcal{X}:=\frac{1}{n} \sum_{j=1}^n \Big( \frac{\xi_j^2}{1+\xi_j^2 m_{1n,c}(L_+)}-\int \frac{s}{1+sm_{1n,c}(L_+)} \mathrm{d} F(s) \Big).
\end{equation*}
According to Assumption \ref{assum_additional_techinical}, we have
\begin{equation*}
\mathsf{C}_1 \asymp 1,\quad \mathsf{C}_2\asymp 1.
\end{equation*}

\begin{lemma}\label{lem_keyfinalfinalkey} Under the assumptions of Theorem \ref{thm_main_bounded} and (\ref{eq_keyanasztaz}), we have
\begin{equation*}
\mathsf{C}_1(\widehat{L}_+-L_+)=\mathsf{C}_2\mathcal{X}+\rO_{\mathbb{P}}(n^{-1}). 
\end{equation*} 
\end{lemma}

Armed with the Lemma \ref{lem_keyfinalfinalkey}, we can easily prove parts (2) and (3). Recall the definition of $\mathsf{v}$ in \eqref{eq_phasetransition}.  It is clear from (\ref{def4}) $\mathcal{X}=\rO_{\mathbb{P}}(n^{-1/2}),$ and from central limit theorem that $\mathsf{C}_2/\mathsf{C}_1\mathcal{X}$ is asymptotically Gaussian with variance $n^{-1} \mathsf{v}$.   
We decompose
\begin{equation*}
\lambda_1-L_+=\lambda_1-\widehat{L}_++\widehat{L}_+-L_+.
\end{equation*}
According to \cite{DX_ell,9779233} and Lemma \ref{localestimate2}, we find that on the event $\Omega_D,$ $|\lambda_1-\widehat{L}_+| \prec n^{-2/3}$ and $n^{2/3} \gamma(\lambda_1-\widehat{L}_+)$ follows the type-1 Tracy-Widom law. Moreover, it is easy to see that when $\mathsf{v} \gtrsim n^{-1/3+\iota}$ for some sufficiently small constant $\iota>0$, the Gaussian component dominates, while when $\mathsf{v}=\mathrm{O}\bigl(n^{-1/3-\iota}\bigr)$, the Tracy–Widom component dominates.

 This concludes the general results in part (3) using an argument similar to those between (\ref{eq_conventionsYnZn}) and (\ref{eq_ffffuuuuuu}). Moreover, for part (2), it is easy to see that when $d>1,$ by the Cauchy-Schwarz inequality, $\mathsf{v}$ is bounded from the blow so that the Gaussian part dominates the Tracy-Widom part, and hence we conclude the proof using an argument similar to those between (\ref{eq_conventionsYnZn}) and (\ref{eq_ffffuuuuuu}). To complete Step one, we now prove Lemma \ref{lem_keyfinalfinalkey}. 
\begin{proof}[\bf Proof of Lemma \ref{lem_keyfinalfinalkey}]

Using the first parts in equations (\ref{eq_edgeequationstwodecide}) and (\ref{eq_edgeequationstwodecide2}), we see that 
    \begin{align}\label{eq_generalgeneraldecomposition}
       &m_{1n,c}(L_{+})-m_{1n}(\widehat{L}_{+})\\
       &=\frac{1}{n}\sum_i\frac{\sigma_i}{\widehat{L}_{+}-\frac{\sigma_i}{n}\sum_j\frac{\xi^2_j}{1+\xi^2_jm_{1n}(\widehat{L}_{+})}}-\frac{1}{n}\sum_i\frac{\sigma_i}{L_{+}-\frac{\sigma_i}{n}\sum_j\frac{\xi^2_j}{1+\xi^2_jm_{1n,c}(L_{+})}} \nonumber \\
       &+\frac{1}{n}\sum_i\frac{-\sigma_i\int\frac{s}{1+sm_{1n,c}(L_{+})}\mathrm{d}F(s)+\frac{\sigma_i}{n}\sum_j\frac{\xi^2_j}{1+\xi^2_jm_{1n,c}(L_{+})}}{(L_{+}-\frac{\sigma_i}{n}\sum_j\frac{\xi^2_j}{1+\xi^2_jm_{1n,c}(L_{+})})(L_{+}-\sigma_i\int\frac{s}{1+sm_{1n,c}(L_{+})}\mathrm{d}F(s))} \nonumber \\
%        &=\frac{1}{n}\sum_i\frac{\sigma_i(L_{+}-\widehat{L}_{+})}{(\widehat{L}_{+}-\frac{\sigma_i}{n}\sum_j\frac{\xi^2_j}{1+\xi^2_jm_{1n,c}(L_{+})})(L_{+}-\frac{\sigma_i}{n}\sum_j\frac{\xi^2_j}{1+\xi^2_jm_{1n,c}(L_{+})})} \nonumber \\
%        &+\frac{1}{n}\sum_i\frac{\sigma_i}{\widehat{L}_{+}-\frac{\sigma_i}{n}\sum_j\frac{\xi^2_j}{1+\xi^2_jm_{1n}(\widehat{L}_{+})}}-\frac{1}{n}\sum_i\frac{\sigma_i}{\widehat{L}_{+}-\frac{\sigma_i}{n}\sum_j\frac{\xi^2_j}{1+\xi^2_jm_{1n,c}(L_{+})}} \nonumber \\
%        &+\frac{1}{n}\sum_i\frac{\sigma_i\int\frac{s}{1+sm_{1n,c}(L_{+})}\mathrm{d}F(s)-\frac{\sigma_i}{n}\sum_j\frac{\xi^2_j}{1+\xi^2_jm_{1n,c}(L_{+})}}{(L_{+}-\frac{\sigma_i}{n}\sum_j\frac{\xi^2_j}{1+\xi^2_jm_{1n,c}(L_{+})})(L_{+}-\sigma_i\int\frac{s}{1+sm_{1n,c}(L_{+})}\mathrm{d}F(s))} \nonumber \\
        &=\frac{1}{n}\sum_i\frac{\sigma_i(L_{+}-\widehat{L}_{+})}{(\widehat{L}_{+}-\frac{\sigma_i}{n}\sum_j\frac{\xi^2_j}{1+\xi^2_jm_{1n,c}(L_{+})})(L_{+}-\frac{\sigma_i}{n}\sum_j\frac{\xi^2_j}{1+\xi^2_jm_{1n,c}(L_{+})})} \nonumber \\
        &+\frac{1}{n}\sum_i\frac{-\frac{\sigma_i^2}{n}\sum_j\frac{\xi^4_j(m_{1n}(\widehat{L}_{+})-m_{1n,c}(L_{+}))}{(1+\xi^2_jm_{1n}(\widehat{L}_{+}))(1+\xi^2_jm_{1n,c}(L_{+}))}}{(\widehat{L}_{+}-\frac{\sigma_i}{n}\sum_j\frac{\xi^2_j}{1+\xi^2_jm_{1n}(\widehat{L}_{+})})(\widehat{L}_{+}-\frac{\sigma_i}{n}\sum_j\frac{\xi^2_j}{1+\xi^2_jm_{1n,c}(L_{+})})} \nonumber \\
        &+\frac{1}{n}\sum_i\frac{-\sigma^2_i\int\frac{s}{1+sm_{1n,c}(L_{+})}\mathrm{d}F(s)+\frac{\sigma^2_i}{n}\sum_j\frac{\xi^2_j}{1+\xi^2_jm_{1n,c}(L_{+})}}{(L_{+}-\frac{\sigma_i}{n}\sum_j\frac{\xi^2_j}{1+\xi^2_jm_{1n,c}(L_{+})})(L_{+}-\sigma_i\int\frac{s}{1+sm_{1n,c}(L_{+})}\mathrm{d}F(s))} \nonumber \\
        &:=\mathsf{T}_1+\mathsf{T}_2+\mathsf{T}_3. 
\end{align}
For the term $\mathsf{T}_1,$ by (\ref{eq_keyanasztaz}), Assumption \ref{assum_additional_techinical} and (\ref{def4}), we can see that 
\begin{equation}\label{eq_T1controlcontrol}
\mathsf{T}_1=\mathsf{C}_1(L_+-\widehat{L}_+)+\rO_{\mathbb{P}}(n^{-1}).
\end{equation}

For the term $\mathsf{T}_2,$ we see that 
\begin{align}\label{eq_T2controlcontrol}
        &\mathsf{T}_2=\frac{1}{n}\sum_i\frac{-\frac{\sigma_i^2}{n}\sum_j\frac{\xi^4_j(m_{1n}(\widehat{L}_{+})-m_{1n,c}(L_{+}))}{(1+\xi^2_jm_{1n}(\widehat{L}_{+}))(1+\xi^2_jm_{1n,c}(L_{+}))}}{(\widehat{L}_{+}-\frac{\sigma_i}{n}\sum_j\frac{\xi^2_j}{1+\xi^2_jm_{1n}(\widehat{L}_{+})})^2} \nonumber \\
        &+\frac{1}{n}\sum_i\frac{\big(\frac{\sigma_i^2}{n}\sum_j\frac{\xi^4_j}{(1+\xi^2_jm_{1n}(\widehat{L}_{+}))(1+\xi^2_jm_{1n,c}(L_{+}))}\big)^2(m_{1n}(\widehat{L}_{+})-m_{1n,c}(L_{+}))^2}{(\widehat{L}_{+}-\frac{\sigma_i}{n}\sum_j\frac{\xi^2_j}{1+\xi^2_jm_{1n}(\widehat{L}_{+})})^2(\widehat{L}_{+}-\frac{\sigma_i}{n}\sum_j\frac{\xi^2_j}{1+\xi^2_jm_{1n,c}(L_{+})})} \nonumber \\
        &=\frac{1}{n}\sum_i\frac{-\frac{\sigma_i^2}{n}\sum_j\frac{\xi^4_j(m_{1n}(\widehat{L}_{+})-m_{1n,c}(L_{+}))}{(1+\xi^2_jm_{1n}(\widehat{L}_{+}))^2}}{(\widehat{L}_{+}-\frac{\sigma_i}{n}\sum_j\frac{\xi^2_j}{1+\xi^2_jm_{1n}(\widehat{L}_{+})})^2}+\frac{1}{n}\sum_i\frac{-\frac{\sigma_i^2}{n}\sum_j\frac{\xi^4_j(m_{1n}(\widehat{L}_{+})-m_{1n,c}(L_{+}))^2}{(1+\xi^2_jm_{1n}(\widehat{L}_{+}))^2(1+\xi^2_jm_{1n,c}(L_{+}))}}{(\widehat{L}_{+}-\frac{\sigma_i}{n}\sum_j\frac{\xi^2_j}{1+\xi^2_jm_{1n,c}(L_{+})})^2} \nonumber \\
        &+\frac{1}{n}\sum_i\frac{\big(\frac{\sigma_i^2}{n}\sum_j\frac{\xi^4_j}{(1+\xi^2_jm_{1n}(\widehat{L}_{+}))(1+\xi^2_jm_{1n,c}(L_{+}))}\big)^2(m_{1n}(\widehat{L}_{+})-m_{1n,c}(L_{+}))^2}{(\widehat{L}_{+}-\frac{\sigma_i}{n}\sum_j\frac{\xi^2_j}{1+\xi^2_jm_{1n}(\widehat{L}_{+})})(\widehat{L}_{+}-\frac{\sigma_i}{n}\sum_j\frac{\xi^2_j}{1+\xi^2_jm_{1n,c}(L_{+})})^2}  \\
        &=-(m_{1n}(\widehat{L}_{+})-m_{1n,c}(L_{+}))+\rO_{\mathbb{P}}(n^{-1}), \nonumber
\end{align}
where in the last step we used the second equation of (\ref{eq_edgeequationstwodecide}) for the first term of (\ref{eq_T2controlcontrol}), and (\ref{eq_keyanasztaz}), Theorem \ref{thm_boundedcaselocallaw}, (\ref{def4}) and Assumption \ref{assum_additional_techinical}  for the second and third terms.  Similarly, for $\mathsf{T}_3,$ we have that
\begin{equation}\label{eq_T3controlcontrol}
\mathsf{T}_3=\mathsf{C}_2 \mathcal{X}+\rO_{\mathbb{P}}(n^{-1}). 
\end{equation}

Insert (\ref{eq_T1controlcontrol}), (\ref{eq_T2controlcontrol}) and (\ref{eq_T3controlcontrol}) into (\ref{eq_generalgeneraldecomposition}), we can conclude the proof. 
\end{proof}

Then we prove (\ref{eq_keyanasztaz}) to complete step two and the proof of the theorem. 

\vspace{3pt} 

\noindent{\bf Step two:} To prove (\ref{eq_keyanasztaz}), we first rewrite (\ref{eq_edgeequationstwodecide}) and (\ref{eq_edgeequationstwodecide2}). Recall (\ref{eq:F(m,z)}). We find that (\ref{eq_edgeequationstwodecide}) can be rewritten as 
\begin{gather*}
    F_n(m_{1n}(\widehat{L}_{+}),\widehat{L}_{+})=0,\quad \frac{\partial F_n}{\partial x}(m_{1n}(\widehat{L}_{+}),\widehat{L}_{+})=0,
\end{gather*}
where we denote 
\begin{gather}\label{eq_Fnxyoriginaldefinition}
    F_n(x,y)=\frac{1}{n}\sum_{i=1}^p\frac{\sigma_i}{-y+\frac{\sigma_i}{n}\sum_{j=1}^n\frac{\xi^2_j}{1+x\xi^2_j}}-x.
\end{gather}
Similarly, (\ref{eq_edgeequationstwodecide2}) can be rewritten as 
\begin{gather*}
    F_{n,c}(m_{1n,c}(L_{+}),L_{+})=0,\quad \frac{\partial F_{n,c}}{\partial x}(m_{1n,c}(L_{+}),L_{+})=0,
\end{gather*}
where we denote
\begin{gather*}
    F_{n,c}(x,y)=\frac{1}{n}\sum_{i=1}^p\frac{\sigma_i}{-y+\sigma_i\int\frac{s}{1+xs}dF(s)}-x.
\end{gather*}  
For pair $(\widetilde{x}, \widetilde{y})$ so that $\widetilde{x}>-l^{-1}$ (recall (\ref{rem1nbound})), as long as they satisfy Assumption \ref{assum_additional_techinical} in the sense that $\min_{1 \leq i \leq p} |\Tilde{y}-\sigma_i\int\frac{s}{1+\Tilde{x}s}dF(s)| \geq \tau, $ by (\ref{def4}), we find that 
\begin{gather}\label{eq_controlboundveryveryuseful}
    \left|F_{n,c}(\Tilde{x},\Tilde{y})-F_n(\Tilde{x},\Tilde{y})\right|+\left|\frac{\partial F_{n,c}}{\partial x}(\Tilde{x},\Tilde{y})-\frac{\partial F_{n}}{\partial x}(\Tilde{x},\Tilde{y})\right|+
    \left|\frac{\partial F_{n,c}}{\partial y}(\Tilde{x},\Tilde{y})-\frac{\partial F_{n}}{\partial y}(\Tilde{x},\Tilde{y})\right|=\rO_{\mathbb{P}}(n^{-1/2}).
\end{gather}
Set $(x_0, y_0)=(m_{1n,c}(L_+), L_+).$ Then we have 
\begin{gather}\label{eq_assumptionequation}
    F_{n,c}(x_0,y_0)=0,\quad \frac{\partial F_{n,c}}{\partial x}(x_0,y_0)=0,\quad  0<\frac{\partial F_{n,c}}{\partial y}(x_0,y_0)<\infty, \quad \frac{\partial^2 F_{n}}{\partial y^2}(x_0,y_0)<0.
\end{gather}
It suffices to prove the following lemma. 
\begin{lemma}\label{lem_stabilityargument}
 There exists a pair $(x_1,y_1)$ with condition $|x_1-x_0|+|y_1-y_0|=\mathrm{O}_{\mathbb{P}}(n^{-1/2})$ such that with probability $1-\ro(1)$
    \begin{gather}\label{eq_final_result}
    F_n(x_1,y_1)=0,\quad \frac{\partial F_{n}}{\partial x}(x_1,y_1)=0. 
\end{gather}
\end{lemma}  
  
With Lemma \ref{lem_stabilityargument}, according to (\ref{eq_edgeequationstwodecide}) and Theorem \ref{lem_solutionsystem}, we see that (\ref{eq_keyanasztaz}) holds. In the rest, we prove Lemma \ref{lem_stabilityargument} using (\ref{eq_controlboundveryveryuseful}).

\begin{proof}[\bf Proof of Lemma \ref{lem_stabilityargument}]For some small $\epsilon>0,$ we consider the probability event $\Xi$ so that (\ref{def4}) holds and (\ref{eq_controlboundveryveryuseful}) reads as
\begin{equation}\label{eq_reinterpretteresult}
   \left|F_{n,c}(\Tilde{x},\Tilde{y})-F_n(\Tilde{x},\Tilde{y})\right|+\left|\frac{\partial F_{n,c}}{\partial x}(\Tilde{x},\Tilde{y})-\frac{\partial F_{n}}{\partial x}(\Tilde{x},\Tilde{y})\right|+
    \left|\frac{\partial F_{n,c}}{\partial y}(\Tilde{x},\Tilde{y})-\frac{\partial F_{n}}{\partial y}(\Tilde{x},\Tilde{y})\right|=\rO(n^{-1/2+\epsilon}).
\end{equation}
We have seen that $\mathbb{P}(\Xi)=1-\mathrm{o}(1).$ Now we fix a realization $\{\xi_i^2\} \in \Xi$ so that the discussions below are purely deterministic.

For the above fixed constant $\epsilon>0,$ we set the region
\begin{gather*}
    \mathcal{N}(x,y):=\{(x,y): |x-x_0|+|y-y_0|\leq n^{-1/2+\epsilon}\},
\end{gather*}
To prove the first part of (\ref{eq_final_result}), it suffices to prove that there exists a solution of $F_n(x,y)=0$ in the region $\mathcal{N}(x,y).$ By Bolzano's theorem, we see that for sufficiently large $n,$ we can find two points $(x_{11},y_{11})$ and $(x_{12},y_{12})$ on $\mathcal{N}(x,y)$ so that $F_{n,c}(x_{11},y_{11})<0, \ F_{n,c}(x_{12},y_{12})>0.$ Together with (\ref{eq_reinterpretteresult}), we see that $ F_{n}(x_{11},y_{11})<0,\ F_{n}(x_{12},y_{12})>0.$ Therefore, by continuity, we can find some point $(x',y')$ so that $F_n(x',y')=0.$ Repeating the above procedure, by implicit function theorem, we find that there exists a curve $x \equiv x(y)$ on $\mathcal{N}(x,y)$ so that $F_n(x,y)=0.$ Similarly, we can show that there exists another curve $\widehat{x} \equiv \widehat{x}(\widehat{y})$ on $\mathcal{N}(\widehat{x},\widehat{y})$ so that the second part of (\ref{eq_final_result}) holds in the sense that $\partial F_n (\widehat{x}, \widehat{y})/\partial \widehat{x}=0.$ 

In order to show (\ref{eq_final_result}), we need to prove that the curves $(x,y)$ and $(\widehat{x}, \widehat{y})$ must have at least one intersection in the region $\mathcal{N}(x,y).$ We prove by contradiction. Otherwise, the curve $(x,y)$ will lie in one of the areas separated by $(\widehat{x},\widehat{y})$ with strictly $\partial F_n(x,y)/ \partial x<0$ or $\partial F_n(x,y)/ \partial x>0$. By (\ref{eq_assumptionequation}), we see that $\mathcal{N}(x,y),$  $\partial F_n(x,y)/\partial y>0$. Without loss of generality, we assume $\partial F_n(x,y)/ \partial x<0$. On the one hand, as $F_n(x,y)=0$, one may conclude that for small neighbor around the points on $(x,y)$, it holds that $\mathrm{d} x/\mathrm{d} y>0$. On the other hand, taking the derivative $F_n(x,y)$ with respect to $y$, we obtain that 
\begin{gather*}
    \frac{\mathrm{d} x}{\mathrm{d} y}\times \Big(\frac{1}{n}\sum_{i=1}^p\frac{\frac{ \sigma_i^2}{n}\sum_j\frac{\xi^4_j}{(1+x\xi^2_j)^2}}{(-y+\frac{\sigma_i}{n}\sum_j\frac{\xi^2_j}{1+x\xi^2_j})^2}-1\Big)=0,
\end{gather*}
which implies $\partial F_n(x,y)/\partial x=0$ and gives the contradiction. This concludes our proof.  
\end{proof}

\begin{remark}\label{rmk_varaincecontrol}
We note that, by applying a Taylor expansion to the function $\frac{1}{1+\xi^2 m_{1n,c}(L_{+})}$ around $\mathbb{E}\xi^2$, and in view of Assumption \ref{assum_D}, an argument similar to that in \eqref{rem1nbound} yields
\[
m_{1n,c}(L_{+}) > -l^{-1}.
\]
Consequently, for $\mathsf{v}$ in (\ref{eq_defnvariance}), it follows that $\mathsf{v}$ is of the same order as $\operatorname{Var}(\xi^2)$.
\end{remark}

\subsubsection{Proof of Theorem \ref{thm_main_bounded_conditional}}\label{sec_conditional12121212121}

In this section, we prove Theorem \ref{thm_main_bounded_conditional} by modifying the proof of Theorem \ref{thm_main_bounded} in the same manner as in Section \ref{sec_proof33}, where the proof of Theorem \ref{thm_main_unbounded} in Section \ref{sec_proofoftheorem31} was adapted to establish Theorem \ref{thm_main_unbounded_conditional}. We recall the probabilistic notation conventions introduced in Remark \ref{rmk_notationconventions}.

 For part $(1)$ of Theorem \ref{thm_main_bounded_conditional}, when restricted on $\Omega_X$ and $\Omega_D$, it follows from Proposition \ref{prop_boundedsetting_conditional} that when $d>1$,
\begin{align*}
    \left|n^{1/(d+1)}\big(\lambda_1-\widehat{L}_{+}\big)-n^{1/(d+1)}\big(\frac{1-\phi\widehat{\mathsf{s}}_3}{\widehat{\mathsf{s}}_4}\frac{l-\xi^2_{(1)}}{l^2}\big)\right|=\mathrm{o}(1).
\end{align*}
Applying \eqref{eq_closenessequation} of Lemma \ref{localestimate2}, we have on $\Omega_X\cap\Omega_D,$ when $d>1,$
\begin{align*}
    \left|n^{1/(d+1)}\big(\lambda_1-L_{+}\big)-n^{1/(d+1)}\big(\frac{1-\phi \mathsf{s}_3}{\mathsf{s}_4}\frac{l-\xi^2_{(1)}}{l^2}\big)\right|=\mathrm{o}(1),
\end{align*}
where we refer to Remark \ref{rmk_notationconventions} for the precise definitions of $\Omega_X\cap\Omega_D$. Denote 
\begin{align*}
    \widetilde{Y}_n:=\frac{l^2(\mathfrak{b}n)^{1/(d+1)}}{\mathsf{s}_4^{-1}(1-\phi \mathsf{s}_3)}(\lambda_1-L_{+}),\quad \widetilde{Z}_n:=(\mathfrak{b}n)^{1/(d+1)}(l-\xi^2_{(1)}),
\end{align*}
where $\mathfrak{b}$ is defined in \eqref{eq_defnbfrak}. Applying an argument similar to \eqref{eq_prf_unboundedconditional_1}, we have that for any $\delta>0$, there exists $N_{\delta}$ such that for all $n\geq N_{\delta}$,
\begin{align}\label{eq_prf_boundedconditional_1}
    |\widetilde{Y}_n-\widetilde{Z}_n|\leq\delta,\qquad \text{on} \ \Omega_X \cap \Omega_D.
\end{align}
Then, for all $x\leq 0$, we can denote the following events
\begin{align*}
    \widetilde{A}_n(x):=\{\widetilde{Y}_n\leq x\},\quad \widetilde{B}_n(x):=\{\widetilde{Z}_n\leq x\}.
\end{align*}
Then on the event $\Omega_X\cap\Omega_D$, for $\delta\equiv\delta(x)\in(0,-x)$, by \eqref{eq_prf_boundedconditional_1}, we have
\begin{align*}
    |\widetilde{Y}_n-\widetilde{Z}_n|\leq\delta,
\end{align*}
which immediately implies that
\begin{align*}
    \widetilde{B}_n(x-\delta)\cap\Omega_X\cap\Omega_D\subset\widetilde{A}_n(x),\quad \widetilde{A}_n(x)\cap\Omega_X\cap\Omega_D\subset\widetilde{B}_n(x+\delta).
\end{align*}
Applying conditional expectation and using the monotonicity property, we obtain similarly to \eqref{eq_keyone} that
\begin{align*}
\mathbf{1}_{\Omega_X}\,\mathbb{P}\big(\widetilde{B}_n(x-\delta)\cap \Omega_D \mid X\big)
\leq
\mathbf{1}_{\Omega_X}\,\mathbb{P}\big(\widetilde{A}_n(x)\mid X\big)
\quad \text{a.s.}
\end{align*} 
and      
\begin{align*}
\mathbf{1}_{\Omega_X}\,\mathbb{P}\big(\widetilde{A}_n(x)\mid X\big)
\leq
\mathbf{1}_{\Omega_X}\,\mathbb{P}\big(\widetilde{B}_n(x+\delta)\mid X\big)
+
\mathbf{1}_{\Omega_X}\,\mathbb{P}\big(\Omega_D^c\mid X\big)
\quad \text{a.s.}
\end{align*}
Then, a parallel argument to the proof of Theorem \ref{thm_main_unbounded_conditional}, we are ready to see 
\begin{align*}
\mathbf{1}_{\Omega_X}\Big[
\mathbb{P}\big(\widetilde{B}_n(x-\delta)\mid X\big)
-\mathbb{P}\big(\Omega_D^c\mid X\big)
\Big]
\leq
\mathbf{1}_{\Omega_X}\,\mathbb{P}\big(\widetilde{A}_n(x)\mid X\big)
\leq
\mathbf{1}_{\Omega_X}\Big[
\mathbb{P}\big(\widetilde{B}_n(x+\delta)\mid X\big)
+\mathbb{P}\big(\Omega_D^c\mid X\big)
\Big] \ \text{a.s.}
\end{align*}
Since both $\widetilde{B}_n(\cdot)$ and $\Omega_D$ are $\sigma(D)$-measurable and $D$ is independent with $X$, it holds on $\Omega_X$, almost surely,
\begin{align*}
    \mathbb{P}(\widetilde{B}_n(x-\delta))-\mathbb{P}(\Omega_D^c)
\leq
\mathbb{P}(\widetilde{A}_n(x)\mid X)
\leq
\mathbb{P}(\widetilde{B}_n(x+\delta))+\mathbb{P}(\Omega_D^c).
\end{align*}
We conclude that on $\Omega_X$, almost surely, for any $\delta\in(0,-x)$,
\begin{align*}
    \exp(-|x-\delta|^{d+1})+\mathrm{o}(1)\leq\mathbb{P}\Big(\frac{l^2(\mathfrak{b}n)^{1/(d+1)}}{\mathsf{s}_4^{-1}(1-\phi \mathsf{s}_3)}(\lambda_1-L_{+})\leq x|X\Big)\leq\exp(-|x+\delta|^{d+1})+\mathrm{o}(1).
\end{align*}
This implies that for any fixed $x<0$ by letting $\delta\downarrow0$, we have on $\Omega_X$, almost surely,
\begin{align*}
    \widetilde{\Delta}_n:=\Big|\mathbb{P}\Big(\frac{l^2(\mathfrak{b}n)^{1/(d+1)}}{\mathsf{s}_4^{-1}(1-\phi \mathsf{s}_3)}(\lambda_1-L_{+})\leq x|X\Big)-\exp(-|x|^{d+1})\Big|=\mathrm{o}(1).
\end{align*}
Furthermore, since $\mathbb{P}(\Omega_X^c)=\mathrm{o}(1)$, we can get rid of $\Omega_X$ as the argument in \eqref{eq_unboundedconditional_getridofOmegaX}, which implies for all $x<0$ and any $\epsilon>0$,
\begin{align}\label{eq_dldldldldlddl}
    \mathbb{P}(\widetilde{\Delta}_n>\epsilon)=\mathrm{o}(1).
\end{align}
Finally, following a parallel argument, using the continuous property of $\exp(-|x|^{d+1})$ for $x\uparrow0$ as in the proof of Theorem \ref{thm_main_unbounded_conditional}, we can complete the proof.

%Then, a parallel argument as the proof of Theorem \ref{thm_main_unbounded_conditional}, we can conclude that  with probability $1-\mathrm{o}(1)$,
%\begin{align*}
%    \mathbb{P}\big(l^2 \frac{(\mathfrak{b}n)^{1/(d+1)}}{\mathsf{s}_4^{-1}(1-\phi \mathsf{s}_3)}(\lambda_1-L_+) \leq x|X\big)=\exp(-|x|^{d+1})+\mathrm{o}(1).
%\end{align*}

Next, we prove $(2)$ in Theorem \ref{thm_main_bounded_conditional}. The counterpart of Lemma \ref{lem_keyfinalfinalkey} can be summarized as follows. 
    \begin{lemma}\label{lem_edgeconvergencestructure_bounded_input4}
        Suppose that the assumptions of Theorem \ref{thm_main_bounded_conditional} hold. When restricted to the event $\Omega_D$, we have
        \begin{align*}
            \big|m_{1n,c}(L_{+})-m_{1n}(\widehat{L}_{+})\big|=\mathrm{O}(n^{-1/2}),\quad \big|L_{+}-\widehat{L}_{+}\big|=\mathrm{O}(n^{-1/2}).
        \end{align*}
        As a consequence, it yields that
        \begin{align*}
            \mathsf{C}_1(\widehat{L}_{+}-L_{+})=\mathsf{C}_2\mathcal{X}+\mathrm{O}(n^{-1}).
        \end{align*}
    \end{lemma}
    Both of the quantities $\mathsf{C}_1, \mathsf{C}_2 \asymp 1$ and $\mathcal{X}$ are defined in the Lemma \ref{lem_keyfinalfinalkey}. The proof of Lemma \ref{lem_edgeconvergencestructure_bounded_input4} follows lines of the proof the Lemma \ref{lem_keyfinalfinalkey}, where the error controls turn to be purely deterministic when restricted to the event $\Omega_D$. We omit the details. 

Similar to the decomposition presented below Lemma \ref{lem_keyfinalfinalkey}, we decompose
    \begin{align}\label{eq_boundedconditional_Gaussain_decomp1}
        \lambda_{1}-L_{+}&=\lambda_{1}-\widehat{L}_{+}+\widehat{L}_{+}-L_{+}.
    \end{align}
    The basic idea is that, when $\mathsf{v} \gtrsim n^{-1/3+\iota}$, $\widehat{L}_{+}-L_{+}$ will dominate the fluctuation of $\lambda_1-\widehat{L}_{+}$ and it will be asymptotically Gaussian when conditional on $X$. To this end, we define
    \begin{align*}
        \mathtt{Y}_n:=\sqrt{n/\mathsf{v}}(\lambda_1-L_{+}),\quad \mathtt{Z}_n:=\sqrt{n/\mathsf{v}} \frac{\mathsf{C}_2}{\mathsf{C}_1} \mathcal{X},
    \end{align*}
where $\mathsf{v}$ is defined in \eqref{eq_defnvariance} and $\operatorname{Var}(\mathcal{X})=(\mathsf{C}_1^2/\mathsf{C}_2^2)\mathsf{v}n^{-1}$. According to Theorem 3.7 of \cite{ding2021spiked}, on the event $\Omega_X\cap\Omega_D$, for $\epsilon<\iota/2,$ we have that 
    \begin{align*}
        \lambda_1-\widehat{L}_{+}=\mathrm{O}(n^{-2/3+\epsilon}).
    \end{align*}
In addition, according to Lemma \ref{lem_edgeconvergencestructure_bounded_input4}, on the event $\Omega_D$, 
    \begin{align*}
        \widehat{L}_{+}-L_{+}=\frac{\mathsf{C}_2}{\mathsf{C}_1}\mathcal{X}+\mathrm{O}(n^{-1}).
    \end{align*}
    Therefore, we obtain on $\Omega_X\cap\Omega_D$,
    \begin{align*}
        \mathtt{Y}_n-\mathtt{Z}_n=\sqrt{n/\mathsf{v}}(\lambda_1-\widehat{L}_{+})+\mathrm{O}\left(\frac{1}{\sqrt{n\mathsf{v}}}\right)=\mathrm{O}\left(\frac{n^{-1/6+\epsilon}}{\sqrt{\mathsf{v}}}\right)+\mathrm{O}\left(\frac{1}{\sqrt{n\mathsf{v}}}\right).
    \end{align*}
    When $\mathsf{v}\gtrsim n^{-1/3+\iota}$, both terms on the right-hand side tend to zero. Therefore, for every fixed $\delta>0$, there exists $N_{\delta}$ such that for all $n\geq N_{\delta}$,
    \begin{align*}
        |\mathtt{Y}_n-\mathtt{Z}_n|\leq\delta \qquad\text{on }\Omega_X\cap\Omega_D.
    \end{align*}
    Moreover, $\mathcal{X}$ depends only on the multipliers, and hence $\mathtt{Z}_n$ is $\sigma(D)$-measurable. Applying Lyapunov-Linderberg central limit theorem for the i.i.d. bounded multipliers, we have for every fixed $x\in\mathbb{R}$, when $n$ is sufficiently large
    \begin{align*}
        |\mathbb{P}(\mathtt{Z}_n\leq x)-\Phi(x)|=\mathrm{o}(1).
    \end{align*}
The rest of the proof follows from an argument similar to those between (\ref{eq_prf_boundedconditional_1}) and (\ref{eq_dldldldldlddl}), using $\mathsf{Y}_n$ and $\mathsf{Z}_n$. We omit the details.

% Now, using the fact that $\mathtt{Z}_n$ depends only on $D$ and $D$ is independent of $X$, for very fixed $x\in\mathbb{R}$, we have
%%    \begin{align*}
%%        \mathbb{P}(\mathtt{Z}_n\le x|X)=\mathbb{P}(\mathtt{Z}_n\le x),\qquad \text{a.s.}
%%    \end{align*}
%%    Therefore,
%    \begin{align*}
%        |\mathbb{P}(\mathtt{Z}_n\le x|X)-\Phi(x)|=\mathrm{o}(1) \qquad \text{a.s.}
%    \end{align*}
% 
%
%    To transfer the limit from $\mathtt{Z}_n$ to $\mathtt{Y}_n$, we have for any fixed $x\in\mathbb{R}$ and small $\delta>0$,
%    \begin{align*}
%    \{\mathtt{Z}_n\le x-\delta\}\cap\Omega_X\cap\Omega_D\subset\{\mathtt{Y}_n\le x\},
%\end{align*}
%and
%\begin{align*}
%    \{\mathtt{Y}_n\le x\}\cap\Omega_X\cap\Omega_D\subset\{\mathtt{Z}_n\le x+\delta\}.
%\end{align*}
%Then, using an argument similar to the proof of $(1)$, we can conclude that for every fixed $x\in\mathbb{R}$, it holds with probability $1-\mathrm{o}(1)$, 
%\begin{align*}
%    \big|\mathbb{P}\big(\sqrt{n/\mathsf{v}^{-1/2}}(\lambda_1-L_+)\le x\big|X\big)-\Phi(x)\big|=\mathrm{o}(1).
%\end{align*}

Finally, we turn to part (3) of Theorem \ref{thm_main_bounded_conditional}.  Under the assumptions of part (3) that when $\mathsf{v}$ is smaller, the term $\widehat{L}_{+}-L_{+}$ in \eqref{eq_boundedconditional_Gaussain_decomp1} no longer dominates the fluctuation of $\lambda_1-\widehat{L}_{+}$ since it vanishes at a faster rate. Consequently, the decomposition in \eqref{eq_boundedconditional_Gaussain_decomp1} is no longer sufficient to characterize the conditional limiting behavior. To overcome this, we instead work with the following decomposition:
    \begin{align}\label{eq_ddd}
        (\lambda_1-L_{+})/\mathbb{E}\xi^2=\widehat{\lambda}_1-E_{+}+\lambda_1/\mathbb{E}\xi^2-\widehat{\lambda}_1+E_{+}-L_{+}/\mathbb{E}\xi^2.
    \end{align}
    Recall that $\widehat{\lambda}_1$ is the largest eigenvalue of the sample covariance matrix $S$, and $E_{+}$ is the right edge of its limiting spectral distribution; both quantities are defined in \eqref{eq_twresultoriginal}. 
    
    %{\color{green}[This has been summiarzed in Definition \ref{def_OmegaX}. Notations needed to be summarized.]}
   
Note that for the first part on the right-hand side of (\ref{eq_ddd}), on the event $\Omega_X$, by (3) of Definition \ref{def_OmegaX}, we have that 
    \begin{align}\label{eq_lowerbound_conditional_edge}
        |\widehat{\lambda}_1-E_{+}|\geq C_{X,3}n^{-2/3-\varepsilon_3},
    \end{align}
     for some constant $C_{X,3}>0$ and sufficiently small constant $\varepsilon_3>0$ defined in Definition \ref{def_OmegaX}. 
    %In the following, to simplify the argument, we assume $\mathbb{E}\xi^2=1$ without loss of generality. Moreover, we introduce the convention
    % \[
    %\mathsf{P}^*(\cdot):=\mathbb P(\cdot\mid X).
    %\]
    %$\mathsf P^*(\cdot)$ measures the randomness only from the multipliers $\{\xi_i^2\}_{i=1}^n$  when we fix a realization of $X\in\Omega_X$.
   
Then we analyze the second term $\lambda_1/\mathbb{E}\xi^2-\widehat{\lambda}_1$ on the right-hand side of (\ref{eq_ddd}) on the event $\Omega_X\cap\Omega_D$. Let $\nu_k$ denote the eigenvectors of the sample covariance matrix $S$ associated with $\widehat{\lambda}_k$. Apply the Rayleigh–Schrödinger expansion to the top eigenvalue perturbation from $\lambda_1/\mathbb{E}\xi^2$ to $\widehat{\lambda}_1$, we have
    \begin{align}\label{eq_boundedconditional_indicator_expansion}
        \lambda_1/\mathbb{E} \xi^2-\widehat{\lambda}_1=\nu_1^*\Sigma^{1/2}X(D^2/\mathbb{E}\xi^2-I)X^*\Sigma^{1/2}\nu_1+\mathrm{O}\left(\sum_{k\geq 2}\frac{|\nu_k^*\Sigma^{1/2}X(D^2/\mathbb{E}\xi^2-I)X^*\Sigma^{1/2}\nu_1|^2}{\widehat{\lambda}_1-\widehat{\lambda}_k}\right).
    \end{align}
    For the quadratic term in the remainder, we write 
    \begin{align*}
        \nu_k^*\Sigma^{1/2}X(D^2/\mathbb{E}\xi^2-I)X^*\Sigma^{1/2}\nu_1=\sum_{i=1}^n(\xi^2_i/\mathbb{E}\xi^2-1)(\nu_k^*\Sigma^{1/2}\mathbf{x}_i)(\mathbf{x}_i^*\Sigma^{1/2}\nu_1).
    \end{align*}
On the event $\Omega_X$, according to parts (3) and (4) of Definition \ref{def_OmegaX}, we have for $1\leq k\leq p$,
    \begin{align*}
        (\nu_k^*\Sigma^{1/2}\mathbf{x}_i)(\mathbf{x}_i^*\Sigma^{1/2}\nu_1)&=\Big(\nu_k^*\big(\Sigma^{1/2}\mathbf{x}_i/(\mathbf{x}_i^*\Sigma\mathbf{x}_i)^{1/2}\big)\times \nu_1\big(\Sigma^{1/2}\mathbf{x}_i/(\mathbf{x}_i^*\Sigma\mathbf{x}_i)^{1/2}\big)\Big) \mathbf{x}_i^*\Sigma\mathbf{x}_i\\
        &=\mathbf{x}_i^*\Sigma\mathbf{x}_i\times\mathrm{O}(n^{-1+\varepsilon_4})=\bar{\sigma}_1\times\mathrm{O}(n^{-1+\varepsilon_4})+\mathrm{O}(n^{-3/2+\varepsilon_1+\varepsilon_4}).
    \end{align*}
     for some sufficiently small constants $\iota/3>\varepsilon_1, \varepsilon_4>0$ defined in Definition \ref{def_OmegaX} and $\bar{\sigma}_1$ in (\ref{eq_somenotations}). Moreover, by Remark \ref{rmk_varaincecontrol} that $\mathsf{v} \asymp \operatorname{Var}(\xi^2)$, {we have on the event $\Omega_D$, by (c) of Definition \ref{defn_probset}
     \begin{align}\label{eq_boundedcondition_sumxi}
         \sum_{i=1}^n(\xi^2_i/\mathbb{E}\xi^2-1)=\mathrm{O}(\sqrt{\mathsf{v}n\log^{c_{\mathsf{b},0}}n}).
     \end{align} }
     Combining these estimates, we obtain on the  event $\Omega_X\cap\Omega_D$,
     \begin{align}\label{eq_boundedcondition_indicator_est1}
         \nu_k^*\Sigma^{1/2}X(D^2/\mathbb{E}\xi^2-I)X^*\Sigma^{1/2}\nu_1=\mathrm{O}(\sqrt{\mathsf{v}n^{-1+\widetilde{\epsilon}_1}}),
     \end{align}
     for some sufficiently small constant $\iota/3>\widetilde{\epsilon}_1>\varepsilon_4$. %}
     
% Therefore, for some small constant $\epsilon_3>0$, it holds on $\Omega_X$ that
%\begin{align*}
%    \mathbb E\Big(
%    \Big|
%    \nu_k^*\Sigma^{1/2}X(D^2-I)X^*\Sigma^{1/2}\nu_1
%    \Big|^2
%    \,\Big|\,X
%    \Big)
%    &=\mathsf v \sum_{i=1}^n
%    (\nu_k^*\Sigma^{1/2}\mathbf x_i)^2(\mathbf x_i^*\Sigma^{1/2}\nu_1)^2=\mathrm{O}\Big(\frac{\mathsf v\,n^{\epsilon_3}}{n}\Big),
%\end{align*}
%uniformly in $k$. By Chebyshev's inequality under the conditional law $\mathsf P^*$, it follows that
%\begin{align}\label{eq_boundedcondition_indicator_est1}
%    \nu_k^*\Sigma^{1/2}X(D^2-I)X^*\Sigma^{1/2}\nu_1
%=\mathrm{O}_{\mathsf P^*}\Big(\sqrt{\mathsf v\,n^{-1+\epsilon_3}}\Big)
%\end{align}
%on $\Omega_X$.

On the other hand, to estimate the summation $\sum_{k\geq 2}(\widehat{\lambda}_1-\widehat{\lambda}_k)^{-1}$, we decompose the eigenvalues $\widehat{\lambda}_k$ into two parts:
    \begin{align*}
        \sum_{k\geq 2}\frac{1}{\widehat{\lambda}_1-\widehat{\lambda}_k}=\sum_{k\geq 2}^{\lfloor(1-\omega)p\rfloor}\frac{1}{\widehat{\lambda}_1-\widehat{\lambda}_k}+\sum_{k\geq \lfloor(1-\omega)p\rfloor+1}^p\frac{1}{\widehat{\lambda}_1-\widehat{\lambda}_k},
    \end{align*}
    for some fixed $\omega\in(0,1)$ {in (3) of Definition \ref{def_OmegaX}. By part (3) of Definition \ref{def_OmegaX}}, on the event $\Omega_X$, we have that for $\{\widehat{\lambda}_k\}_{k=2}^{\lfloor(1-\omega)p\rfloor}$, 
    \begin{align*}
        \frac{1}{\widehat{\lambda}_1-\widehat{\lambda}_k}\leq C_{X,2}^{-1}\frac{n^{2/3+\varepsilon_3}}{k^{2/3}},
    \end{align*}
    for some constant $C_{X,2}>0$ and sufficiently small constant $\iota/3>\varepsilon_3>0$ as in Definition \ref{def_OmegaX}. Moreover, for $\{\widehat{\lambda}_k\}_{k=\lfloor(1-\omega)p\rfloor+1}^p$, we have
    \begin{align*}
        \frac{1}{\widehat{\lambda}_1-\widehat{\lambda}_k}\leq C_2,
    \end{align*}
    for some constant $C_2>0$. Therefore,
    \begin{align*}
        \frac{1}{n}\sum_{k\geq 2}\frac{1}{\widehat{\lambda}_1-\widehat{\lambda}_k}&=\frac{1}{n}\sum_{k\geq 2}^{\lfloor(1-\omega)p\rfloor}\frac{1}{\widehat{\lambda}_1-\widehat{\lambda}_k}+\frac{1}{n}\sum_{k\geq \lfloor(1-\omega)p\rfloor+1}^p\frac{1}{\widehat{\lambda}_1-\widehat{\lambda}_k}\\
        &\leq C_{X,2}^{-1}\frac{n^{\varepsilon_3}}{n^{1/3}}\sum_{k\geq 2}^{\lfloor(1-\omega)p\rfloor}\frac{1}{k^{2/3}}+C_2\leq C_3n^{\varepsilon_3},
    \end{align*}
    for some constant $C_3>0$ under $\Omega_X$.

    Combining the above bound with \eqref{eq_boundedcondition_indicator_est1}, we deduce that on $\Omega_X\cap\Omega_D$,
    \begin{align*}
        \sum_{k\geq 2}\frac{|\nu_k^*\Sigma^{1/2}X(D^2/\mathbb{E}\xi^2-I)X^*\Sigma^{1/2}\nu_1|^2}{\widehat{\lambda}_1-\widehat{\lambda}_k}=\mathrm{O}\big(\mathsf{v}n^{\widetilde{\epsilon}_1+\varepsilon_3}\big).
    \end{align*}
    Consequently,
    \begin{align}\label{eq_herehrehrehrhehrehrhhere}
        \lambda_1/\mathbb{E}\xi^2-\widehat{\lambda}_1=\nu_1^*\Sigma^{1/2}X(D^2/\mathbb{E}\xi^2-I)X^*\Sigma^{1/2}\nu_1+\mathrm{O}(\mathsf{v}n^{\widetilde{\epsilon}_2}),
    \end{align}
   on $\Omega_X\cap\Omega_D$, where $\widetilde{\epsilon}_2:=\widetilde{\epsilon}_1+\varepsilon_3<2\iota/3$. It remains to estimate the leading term in \eqref{eq_boundedconditional_indicator_expansion}.  For this purpose, we write
    \begin{align*}
        \nu_1^*\Sigma^{1/2}X(D^2/\mathbb{E}\xi^2-I)X^*\Sigma^{1/2}\nu_1=\sum_{i=1}^n(\xi^2_i/\mathbb{E}\xi^2-1)(\nu_1^*\Sigma^{1/2}\mathbf{x}_i)^2.
    \end{align*}
    Denote $\widetilde{w}=X^*\Sigma^{1/2}\nu_1$ and recall $\mathcal{S}=X^*\Sigma X$. Since $S\nu_1=\widehat\lambda_1\nu_1$, we have $\mathcal{S}\widetilde{w}=X^*\Sigma^{1/2}S\nu_1=\widehat{\lambda}_1X^*\Sigma^{1/2}\nu_1=\widehat{\lambda}_1\widetilde{w}$. Hence, $\widetilde{w}$ is the eigenvector of the companion matrix $\mathcal{S}$ associated with the eigenvalue $\widehat{\lambda}_1$. Moreover, 
    \begin{align*}
    \|\widetilde w\|^2=\nu_1^*\Sigma^{1/2}XX^*\Sigma^{1/2}\nu_1=\nu_1^*S\nu_1=\widehat\lambda_1.
    \end{align*}
    Therefore, after normalization, $w_1=\widetilde{w}/\|\widetilde{w}\|$ a unit eigenvector of $\mathcal{S}$ corresponding to $\widehat{\lambda}_1$.
    By part (5) of Definition \ref{def_OmegaX}, on the event $\Omega_X$,
    \begin{align*}
        \max_{1\leq i\leq n}|w_{i1}|^2=\mathrm{O}(n^{-1+\varepsilon_5}),
    \end{align*}
    for some sufficiently small constants $\iota/3>\varepsilon_5>0$ defined in Definition \ref{def_OmegaX}, where $w_{i1}$ is the $i$-th element in $w_1$. On the other hand, by (2) of Definition \ref{def_OmegaX}, $\|\widetilde{w}\|=\sqrt{\widehat{\lambda}_1}=\mathrm{O}(1)$ on the event $\Omega_X$. Combining the last two estimates, we have that 
    \begin{align*}
        \max_{1\leq i\leq n}|\nu_1^*\Sigma^{1/2}\mathbf{x}_i|^2=\mathrm{O}(n^{-1+\varepsilon_5}).
    \end{align*}
    Consequently, together with \eqref{eq_boundedcondition_sumxi}, on the event $\Omega_X\cap\Omega_D$,
    %\begin{align*}
    %    &\mathbb{E}\big(\sum_{i=1}^n(\xi^2_i-1)(\nu_1^*\Sigma^{1/2}\mathbf{x}_i)^2|X\big)=0,\quad \mathbb{E}\big((\sum_{i=1}^n(\xi^2_i-1)(\nu_1^*\Sigma^{1/2}\mathbf{x}_i)^2)^2|X\big)=\mathrm{O}(\frac{\mathsf{v}n^{\epsilon_5}}{n}),
    %\end{align*}
    \begin{align*}
        \nu_1^*\Sigma^{1/2}X(D^2/\mathbb{E}\xi^2-I)X^*\Sigma^{1/2}\nu_1=\mathrm{O}(\sqrt{\mathsf{v}n^{\widetilde{\epsilon}_3}/n}),
    \end{align*}
    for some sufficiently small constant $\widetilde{\epsilon}_3>\varepsilon_5$. 
    %It implies that $\nu_1^*\Sigma^{1/2}X(D^2-I)X^*\Sigma^{1/2}\nu_1=\mathrm{O}_{\mathsf{P}^*}(\sqrt{\mathsf{v}n^{\epsilon_5}/n})$, under $\Omega_X$. In summary, we get that under $\Omega_X$,
    Together with (\ref{eq_herehrehrehrhehrehrhhere}), we conclude that on  $\Omega_X\cap\Omega_D$ 
    {
    \begin{align}\label{eq_estdifference_lambda1hatlambda1}
        \lambda_1/\mathbb{E}\xi^2-\widehat{\lambda}_1=\mathrm{O}\big(\sqrt{\mathsf{v}n^{\widetilde{\epsilon}_3}/n}+\mathsf{v}n^{\widetilde{\epsilon}_2}\big)\ll n^{-2/3-\varepsilon_3},
    \end{align}
    when $\mathsf{v}\ll n^{-2/3-\iota}$ with $\widetilde{\epsilon}_2<2\iota/3$ and $\varepsilon_3<\iota/3$.
    }

 Finally,  it remains to control the last term $E_{+}-L_{+}/\mathbb{E}\xi^2$ in (\ref{eq_ddd}). 
By an argument similar to \eqref{eq_keyanasztaz}, we can obtain for $\widetilde{L}_{+}:=L_{+}/\mathbb{E}\xi^2$
    \begin{align}\label{eq_estdifference_mQmS}
        |m_{2n,c}(\widetilde{L}_{+})-m_{n}^{\mathtt{S}}(E_{+})|=\mathrm{O}(\mathsf{v}),\quad |\widetilde{L}_{+}-E_{+}|=\mathrm{O}(\mathsf{v}).
    \end{align}
    {
    Consequently, on $\Omega_X\cap\Omega_D$,
    \begin{align}\label{eq_estdifference_E+L+}
        |E_{+}-L_{+}/\mathbb{E}\xi^2|=\mathrm{O}(\mathsf{v})\ll n^{-2/3-\varepsilon_3},
    \end{align}
     when $\mathsf{v}\ll n^{-2/3-\iota}$ with $\varepsilon_3<\iota/3$.

In summary, combining \eqref{eq_lowerbound_conditional_edge}, \eqref{eq_estdifference_lambda1hatlambda1} and \eqref{eq_estdifference_E+L+}, when $\mathsf{v}\ll n^{-2/3-\iota}$, we conclude that on the event $\Omega_X\cap\Omega_D$,
\begin{align*}
    (\lambda_1-L_{+})/\mathbb{E}\xi^2=\widehat{\lambda}_1-E_{+}+\mathrm{o}(n^{-2/3-\varepsilon_3}).
\end{align*}}
Hence, it follows that on $\Omega_X\cap\Omega_D$,
\begin{align*}
    \frac{n^{2/3}\gamma_0}{\mathbb{E}\xi^2}(\lambda_1-L_{+})=n^{2/3}\gamma_0(\widehat{\lambda}_1-E_{+})+\mathrm{o}(1).
\end{align*}
In addition, when $\mathsf{v}\ll n^{-2/3-\iota}$,
by a discussion similar to the proof of Lemma \ref{lem_stabilityargument}, using (\ref{eq_estdifference_mQmS}), we can show that
% \eqref{eq_estdifference_mQmS} implies that the self-consistent system for $(m_{2n,c}(\widetilde L_+),\widetilde L_+)$ is an $\mathrm{O}(\mathsf v)$ perturbation of that for $(m_n^{\mathtt S}(E_+),E_+)$. Since the edge coefficient depends smoothly on the corresponding system parameters under the regular edge condition,
for $\gamma$ defined in (b) of Lemma \ref{localestimate2}, we have
$\gamma=(\gamma_0/\mathbb{E}\xi^2)(1+\mathrm{o}(1))$. Therefore, on $\Omega_X\cap\Omega_D$, we have
\begin{align*}
    n^{2/3}\gamma(\lambda_1-L_{+})=n^{2/3}\gamma_0(\widehat{\lambda}_1-E_{+})+\mathrm{o}(1).
\end{align*}
With the above results on $\Omega_X \cap \Omega_D$, the arguments for the conditional probability can be established by following the same strategy as in Section~\ref{sec_conditional12121212121}. We omit the details due to their similarity.

\subsection{Proof of Corollary \ref{cor_joint_gaussian_edge}}\label{sec_applications_nonspike}

\begin{proof}
We work on the event $\Omega_X$ defined in Definition \ref{def_OmegaX}. For every fixed
integer $k\geq1$, edge rigidity implies that, on
$\Omega_X\cap\Omega_D$,
\begin{equation}\label{eq_conditional_joint_edge_rigidity}
\max_{1\leq i\leq k}
|\lambda_i-\widehat{L}_+|
=
\rO(n^{-2/3+\epsilon})
\end{equation}
for any sufficiently small constant $\epsilon>0$. Recall that
\[
\mathtt{Z}_n:=\sqrt{n\mathsf{v}^{-1}}(\widehat{L}_+-L_+)
\]
depends only on the multipliers. Since $D$ is independent of $X$, the
edge-shift central limit theorem gives
\begin{equation}\label{eq_conditional_common_gaussian_clt}
\sup_{x\in\mathbb{R}}
\left|
\mathbb{P}(\mathtt{Z}_n\leq x\mid X)-\Phi(x)
\right|
\longrightarrow0.
\end{equation}

For $1\leq i\leq k$, write
\begin{equation*}
\sqrt{n\mathsf{v}^{-1}}(\lambda_i-L_+)
=
\mathtt{Z}_n+
\sqrt{n\mathsf{v}^{-1}}(\lambda_i-\widehat{L}_+).
\end{equation*}
In the Gaussian regimes under consideration, we may choose
$\epsilon>0$ sufficiently small so that
\begin{equation*}
\sqrt{n\mathsf{v}^{-1}}\,n^{-2/3+\epsilon}
=
\ro(1).
\end{equation*}
Indeed, this follows immediately when $\mathsf{v}\asymp1$, whereas
for vanished variance it follows from
$\mathsf{v}\gtrsim n^{-1/3+\iota}$ by choosing
$\epsilon<\iota/2$. Consequently, for every fixed $\delta>0$, with probability at least $1-\mathrm{o}(1)$,
\begin{equation}\label{eq_conditional_joint_remainder}
\mathbb{P}\left(
\max_{1\leq i\leq k}
\left|
\sqrt{n\mathsf{v}^{-1}}(\lambda_i-L_+)-\mathtt{Z}_n
\right|>\delta
\Big|X
\right)
\rightarrow 0.
\end{equation}

Fix $(x_1,\ldots,x_k)\in\mathbb{R}^k$ and set
$x_*:=\min_{1\leq i\leq k}x_i$. It follows from
\eqref{eq_conditional_joint_remainder} that, for every $\delta>0$,
\begin{align*}
&\mathbb{P}(\mathtt{Z}_n\leq x_*-\delta\mid X)
-
\mathbb{P}\left(
\max_{1\leq i\leq k}
\left|
\sqrt{n\mathsf{v}^{-1}}(\lambda_i-L_+)-\mathtt{Z}_n
\right|>\delta
\Big|X
\right)\\
&\qquad\leq
\mathbb{P}\left(
\sqrt{n\mathsf{v}^{-1}}(\lambda_i-L_+)\leq x_i,\
1\leq i\leq k
\Big|X
\right)\\
&\qquad\leq
\mathbb{P}(\mathtt{Z}_n\leq x_*+\delta\mid X)
+
\mathbb{P}\left(
\max_{1\leq i\leq k}
\left|
\sqrt{n\mathsf{v}^{-1}}(\lambda_i-L_+)-\mathtt{Z}_n
\right|>\delta
\Big|X
\right).
\end{align*}
Applying \eqref{eq_conditional_common_gaussian_clt} and
\eqref{eq_conditional_joint_remainder}, and then letting
$\delta\downarrow0$, yields
\begin{equation*}
\mathbb{P}\left(
\sqrt{n\mathsf{v}^{-1}}(\lambda_i-L_+)\leq x_i,\
1\leq i\leq k
\Big|X
\right)
\rightarrow
\Phi(x_*),
\end{equation*}
with probability at least $1-\mathrm{o}(1)$.

Finally, the common edge-shift term cancels when two leading
eigenvalues are subtracted. By
\eqref{eq_conditional_joint_edge_rigidity}, when $k\geq2$,
\begin{equation*}
\max_{1\leq i\leq k-1}
(\lambda_i-\lambda_{i+1})
\leq
2\max_{1\leq i\leq k}
|\lambda_i-\widehat{L}_+|
=
\rO(n^{-2/3+\epsilon})
\end{equation*}
on $\Omega_X\cap\Omega_D$. Since both events have probability tending
to one, the same eigengap estimate holds conditionally on $X$.

%Since $\{\sigma_i\}_{1\le i\le p}$ and $\{\widehat{\sigma}_{i}\}_{1\le i\le p}$ are uniformly bounded, the dominated convergence theory implies that 
%    \begin{align}\label{eq_QuEST_convergence}
%        \mathbb{E}\big[p^{-1}\sum_{i=1}^p(\widehat{\sigma}_i-\sigma_i)^2\big]\rightarrow 0,
%    \end{align}
%Then, denote the Green function and the Stieltjes transform of ESD of $\widehat{\Sigma}^{1/2}WW^*\widehat{\Sigma}^{1/2}$ by
%\begin{align*}
%    G^{\mathtt{R}}(z)=(R-zI)^{-1},\quad m_R=\frac{1}{p}\operatorname{tr}G^{\mathtt{R}}(z),\quad z=E+\mathrm{i}\eta\in\mathbb{C}_{+}.
%\end{align*}
%We have the conditions in \cite{knowles2017anisotropic} hold and   the anisotropic local law of $R$ implies that $m_R$ is closed to a deterministic quantity $m_n^{\mathtt{R}}(z)$. By \eqref{eq_QuEST_convergence}, it can be shown that 
%\begin{align*}
%    |m_n^{\mathtt{R}}(z)-m_n^{\mathtt{S}}(z)|=\mathrm{o}(1).
%\end{align*}
%By the edge universality of the large sample covariance matrix \cite{knowles2017anisotropic}, we can conclude that the limiting Tracy-Widom distribution holds for $\zeta_1$. Consequently, the main results of Corollary \ref{coro_conditionalboostrap} follow immediately.
%}
\end{proof}

\section{Proof of the results of Section \ref{sec_boostrapeffect_spike}}\label{sec_proof_spike}
%In this section, we prove the results of Section \ref{sec_boostrapeffect_spike}.
%{\color{green}[PLZ PLZ PLZ PLZ PLZ!!!!!CHECK WHAT YOU WROTE. PLZ AVOID THOSE PLAIN WRONG TYPOS.]}
%\subsection{Proof of the results in Section \ref{sec_spike_thegood}}
{
\subsection{Proof of Theorem \ref{thm_sample}}
The proof of this theorem is largely parallel to those of Theorem 2.1 and equation (2.1) in \cite{CHP}. We therefore only outline the main ideas and omit the routine details. The key step is to introduce auxiliary quantities $\theta_i, 1 \leq i \leq r,$ defined as the solutions to 
\begin{gather}\label{eq_def_theta}
    \frac{\theta_i}{\widetilde{\sigma}_i}=\Big(1-\frac{1}{n\theta_i}\sum_{j=r+1}^p\frac{\sigma_j}{1-\widetilde{\sigma}_i^{-1}\sigma_j}\Big)^{-1},
\end{gather}  
for each $1\leq i\leq r$. Under the restriction that $\theta_i \in [\widetilde{\sigma}_i, 2 \widetilde{\sigma}_i],$ the existence and uniqueness of $\theta_i$ were established in \cite{CHP,yu2024testing}. Moreover, under the assumptions of Theorem \ref{thm_sample}, for each $1 \leq i \leq r$,
\begin{equation}\label{eq_thetaipriorbound}
\begin{split}
    \theta_i/\widetilde{\sigma}_i=1+\frac{1}{n}\sum_{j=r+1}^p\frac{\sigma_j/\widetilde{\sigma}_i}{1-\sigma_j/\widetilde{\sigma}_i}=1+\frac{1}{n}\sum_{j=r+1}^p\frac{\sigma_j}{\widetilde{\sigma}_i}+\mathrm{o}(n^{-1/2}).
\end{split}
\end{equation}

Following the arguments in the proof of Lemma 3.4 of \cite{yu2024testing} or Theorem 2.1 of \cite{CHP}, we obtain that, for $1\leq i\leq r$,
\begin{gather}\label{eq_limitingrepresentation_1}
    \frac{\widehat{\mu}_i}{\theta_i}=1+\mathbf{k}_i^{*}(V_1^{*}XX^{*}V_1-I_r)\mathbf{k}_i+\ro_{\mathbb{P}}(n^{-1/2}),
\end{gather}
where $\mathbf{k}_i, 1 \leq i \leq r,$ denotes the $i$-th standard basis vector in $\mathbb{R}^r.$ Since $\mathbf{v}_i^*XX^*\mathbf{v}_i$ has mean one with variance $(2+\sum_{k=1}^pv_{ki}^4(\mathfrak{m}_4-3))/n$, the representation in \eqref{eq_limitingrepresentation_1} yields the following central limit theorem for $\widehat{\mu}_i/\theta_i, 1\leq i\leq r$,
\begin{align}\label{eq_unconditionalspike_hatmutheta}
    \lim_{n\rightarrow\infty}\mathbb{P}\Big(\sqrt{\frac{n}{\mathsf{V}_i^{\mathtt{S}}}}\big(\frac{\widehat{\mu}_i}{\theta_i}-1\big)\leq x\Big)=\Phi(x),
\end{align}
where $\mathsf{V}_i^{\mathtt{S}}=2+\sum_{k=1}^pv_{ki}^4(\mathfrak{m}_4-3).$

On the other hand, by the definition of $\theta_i$ and \eqref{eq_thetaipriorbound}, we have
 \begin{align}\label{eq_keyrepresentation}
        \frac{\widehat{\mu}_i}{\widetilde{\sigma}_i}&=\frac{\widehat{\mu}_i}{\theta_i}\cdot\frac{\theta_i}{\widetilde{\sigma}_i}=\frac{\widehat{\mu}_i}{\theta_i}+\frac{\widehat{\mu}_i}{\theta_i}\Big(\frac{1}{n}\sum_{j=r+1}^p\frac{\sigma_j}{\widetilde{\sigma}_i}+\mathrm{o}(n^{-1/2})\big)\Big)\nonumber\\
        &=\frac{\widehat{\mu}_i}{\theta_i}+\frac{1}{n}\sum_{j=r+1}^p\frac{\sigma_j}{\widetilde{\sigma}_i}+\mathrm{o}_{\mathbb{P}}(n^{-1/2})=1+\frac{1}{n}\sum_{j=r+1}^p\frac{\sigma_j}{\widetilde{\sigma}_i}+\mathrm{o}_{\mathbb{P}}(n^{-1/2}),
    \end{align}
    where we used the fact $\widehat{\mu}_i/\theta_i=1+\mathrm{O}_{\mathbb{P}}(n^{-1/2})$ and the assumption $\widetilde{\sigma}_i\gg n^{1/4}$. Recall $ \mathsf{M}^\mathtt{S}_i$ in (\ref{eq_keynotations}).
%    \begin{align}\label{eq_def_MS}
%        \mathsf{M}^\mathtt{S}_i=1+\frac{1}{n}\sum_{k=r+1}^p \frac{\sigma_k}{\widetilde{\sigma}_i},
%    \end{align}
%    for $1\leq i\leq r$. 
The proof follows from the weak convergence in \eqref{eq_unconditionalspike_hatmutheta}, (\ref{eq_keyrepresentation}), and  the Slutsky's theorem. 
%, we have 
%    \begin{align}
%        \lim_{n\rightarrow\infty}\mathbb{P}\Big(\sqrt{\frac{n}{\mathsf{V}_i^{\mathtt{S}}}}(\frac{\widehat{\mu}_i}{\widetilde{\sigma}_i}-\mathsf{M}_i^{\mathtt{S}})\le x\Big)=\Phi(x),
%    \end{align}
%    for $1\leq i\leq r$ and every fixed $x\in\mathbb{R}$. The desired conclusion follows.
}

\subsection{Proof of Theorem \ref{thm_main_spike_unconditional}}
%{\color{green}[some arguments are missing; in E.2.1, change $\zeta_i$ to $\pi_i.$]}

\subsubsection{Limiting representation of the spiked eigenvalues of $Q$}
  Recall the bootstrapped sample covariance matrix $\widetilde{Q}$, and denote $\widetilde{Y}=\widetilde{\Sigma}^{1/2}XD$. We then write $\widetilde{Q}:=\widetilde{Y} \widetilde{Y}^*$ and $\widetilde{\mathcal Q}:=\widetilde{Y}^* \widetilde{Y}.$ Since these two matrices share the same nonzero eigenvalues, it suffices to work with the latter for simplicity. Under the spiked covariance model in \eqref{eq_truemodelspiked}, we decompose $\widetilde{\Sigma}$ as follows.
\begin{equation*}
 \widetilde{\Sigma}:=\Sigma_s+\Sigma_o,
\end{equation*}
where we denote the two $p \times p$ matrices as 
\begin{equation}\label{eq_twomatricesdecomposition}
\Sigma_s:=\sum_{i=1}^r\Tilde{\sigma}_i\mathbf{v}_i\mathbf{v}_i^{*} \equiv V_1\Lambda_sV_1^{*},\quad \Sigma_o:=\sum_{i=r+1}^p \sigma_i\mathbf{v}_i\mathbf{v}_i^{*} \equiv V_2\Lambda_oV_2^{*}.
\end{equation}
Consequently, we can decompose $\widetilde{\mathcal Q}$ as follows
\begin{equation*}
\widetilde{\mathcal Q}=DX^{*}\Tilde{\Sigma}XD=DX^{*}\Sigma_sXD+DX^{*}\Sigma_oXD.
\end{equation*}
Note that with high probability 
 \begin{equation*}
        \|DX^{*}\Sigma_oXD\|=\| D^2 X^* \Sigma_0 X \| \leq \| D^2 \| \|X^{*} \Sigma_0 X\| \leq \sigma_{r+1} \xi_{(1)}^2 \| X^* X \| \sim \xi^2_{(1)},
    \end{equation*}
    where in the last step we used \cite{Wen2021} that $\|X^* X \|$ is bounded from above with high probability. Using (\ref{def1}) and (\ref{def3}) as well as Weyl's inequality, we see that from the assumption of (\ref{spiked_assumption}) that, for $1 \leq i \leq r,$ 
    \begin{equation}\label{eq_proof4.1firstpartone}
   \frac{ \mu_i-\lambda_i(DX^* \Sigma_sXD)}{\widetilde{\sigma}_i}=\ro_{\mathbb{P}}(1). 
    \end{equation}
Then we consider the first few largest eigenvalues of $DX^* \Sigma_s XD,$ or equivalently those of $\Sigma_s^{1/2} XD^2X^* \Sigma_s^{1/2}.$ By a discussion similar to Lemma D.1 of \cite{ding2021spiked}, we find that if $\lambda$ is an eigenvalue of $\Sigma_s^{1/2} XD^2X^* \Sigma_s^{1/2},$ recalling (\ref{eq_twomatricesdecomposition}), we have 
    \begin{gather}\label{eq_masterequationeigenvalues}
        \operatorname{det}(V_1^{*}XD^2X^{*}V_1-\lambda\Lambda_s^{-1})=0.
    \end{gather}   
Note that 
\begin{equation*}
    V_1^{*}XD^2X^{*}V_1-\mathbb{E}\xi^2I_r=V_1^{*}X(D^2-\mathbb{E}\xi^2I)X^{*}V_1+\mathbb{E}\xi^2\times(V_1^{*}XX^{*}V_1-I_r),
\end{equation*}
where $I_r$ is the $r\times r$ identity matrix. For the first term in above decomposition, we observe that 
    \begin{gather*}
        |\mathbf{v}_i^{*}X(D^2-\mathbb{E}\xi^2I)X^{*}\mathbf{v}_i|=\mathrm{O}_{\mathbb{P}}(n^{-1/2}),\quad |\mathbf{v}_i^*X(D^2-\mathbb{E}\xi^2I)X^{*}\mathbf{v}_j|=\mathrm{O}_{\mathbb{P}}(n^{-1/2}),
    \end{gather*}  
    for $1\leq i\neq j\leq r$. By Assumptions \ref{assum_model} and \ref{assum_D}, and notice that $r$ is finite, straightforward calculations indicate that 
 \begin{equation*}
     \|V_1^{*}X(D^2-\mathbb{E}\xi^2I)X^{*}V_1\|=\mathrm{O}_{\mathbb{P}}(n^{-1/2}).
 \end{equation*}
 On the other hand, by Theorem 7.1 of \cite{BaiandYao2008}, one has
 \begin{equation*}
     \|V_1^{*}XX^{*}V_1-I_r\|=\mathrm{O}_{\mathbb{P}}(n^{-1/2}).
 \end{equation*}
As a consequence, we conclude that 
 \begin{equation*}
   \|V_1^{*}XD^2X^{*}V_1-\mathbb{E} \xi^2 I_r\|=\mathrm{O}_{\mathbb{P}}(n^{-1/2}), 
\end{equation*}    
 Together with (\ref{eq_masterequationeigenvalues}), we conclude that for $1 \leq i \leq r$
\begin{equation}\label{eq_proof.1firstparttwo}
\frac{\lambda_i(DX^* \Sigma_s XD)}{\widetilde{\sigma}_i}= \mathbb{E} \xi^2+\mathrm{o}_{\mathbb{P}}(1). 
\end{equation}   
Combining (\ref{eq_proof4.1firstpartone}), we obtain the first order result for the limiting property of $\mu_i/\widetilde{\sigma}_i$. 

Then we proceed with the second order results for $\mu_i$. The proof follows from  strategies similar to Theorem 3.3 of \cite{yu2024testing}. We focus on explaining the main ideas and omit the details. Recall the definition of $\theta_i, 1 \leq i \leq r$ in \eqref{eq_def_theta}. In order to establish the asymptotics of $\mu_i/\theta_i,$ for notational convenience, we now work with the rescaled matrix 
\begin{equation*}
\check{\mathcal Q}:= \check{D} X^* \widetilde{\Sigma} X \check{D}, \ \check{D}^2:=(\mathbb{E} \xi^2)^{-1} D^2, 
\end{equation*}
whose eigenvalues are denoted as $\check{\lambda}_1 \geq \check{\lambda}_2 \geq \cdots \geq \check{\lambda}_{ \{p \wedge n\}}>0.$ {Recall that $\mu_i$'s are ordered eigenvalues of the bootstrapped sample covariance matrix $\widetilde{Q}$ in Table \ref{table_notations}. Note that
\begin{align}\label{eq_def_checklambda}
    \check{\lambda}_i=\frac{\mu_i}{\mathbb{E}\xi^2},\quad i=1,\dots,p\wedge n.
\end{align}
}
 By a discussion similar to (\ref{eq_masterequationeigenvalues}), using (\ref{eq_twomatricesdecomposition}), we find that  $\check{\lambda}_i, 1 \leq i \leq r$ satisfy the equation
    \begin{gather*}
        \operatorname{det}(\Lambda_s^{-1}-V_1^{*}X\check{D}(\check{\lambda}_i I-\check{D}X^{*}\Sigma_oX\check{D})^{-1}\check{D}X^{*}V_1)=0.
    \end{gather*}
Denote $\mathbf{B}(x):=xI-\check{D}X^{*}\Sigma_oX\check{D}$ and $\delta_i=(\check{\lambda}_i-\theta_i)/\theta_i$, the above determinant can be rewritten into 
\begin{gather}\label{eq_generalldeterminant}
    \operatorname{det}(\theta_i\Lambda_s^{-1}-\theta_iV_1^{*}X\check{D}\mathbf{B}^{-1}(\theta_i)\check{D}X^{*}V_1+\delta_i\theta_i^2V_1^{*}X\check{D}\mathbf{B}^{-1}(\check{\lambda}_i)\mathbf{B}^{-1}(\theta_i)\check{D}X^{*}V_1)=0.
\end{gather}
Following the procedure in Section 7.1 of \cite{CHP} or Lemma C.5 of \cite{yu2024testing}, we find that for $1 \leq i,l \leq r$
   \begin{gather*}
        \theta_i\mathbf{e}_i^{*}V_1^{*}X\check{D}\mathbf{B}^{-1}(\theta_i)\check{D}X^{*}V_1\mathbf{e}_l=\mathbf{1}(l=i)\big(\mathbf{v}_i^*X\check{D}^2X^*\mathbf{v}_l-\frac{1}{n}\sum_{j=1}^n\xi^2_j/\mathbb{E}\xi^2+\pi_i\big)+\mathrm{O}_{\mathbb{P}}(\frac{1}{\sqrt{n}}\frac{(n\vee p)}{n\widetilde{\sigma}_i}+\frac{1}{n}),
    \end{gather*}
where $\pi_i$ is a random quantity associated with $\theta_i$ satisfying
\begin{equation}\label{eq_def_zeta}
    \pi_i:=\frac{1}{n}\sum_{j=1}^n\frac{\xi^2_j}{\mathbb{E}\xi^2}\Big(1-\frac{\xi^2_j}{n\theta_i\mathbb{E}\xi^2}\sum_{k=r+1}^p\frac{\sigma_k}{1-\theta_i^{-1}\sigma_k\pi_i}\Big)^{-1}.
\end{equation}
Similarly, by a discussion similar to Lemma C.6 of \cite{yu2024testing}, we conclude that 
    \begin{gather*}
        \delta_i \theta_i^2[V_1^{*}X\check{D}\mathbf{B}^{-1}(\check{\lambda}_i)\mathbf{B}^{-1}(\theta_i)\check{D}X^{*}V_1]_{il}= \delta_i \times (\mathbf{1}(l=i)+\ro_{\mathbb{P}}(1)).
    \end{gather*}
Inserting the above two controls into (\ref{eq_generalldeterminant}), by the assumption of (\ref{eq_separation}), using Leibniz’s formula for determinant, one has that 
    \begin{gather}\label{eq: spike_det_est}
        \delta_i(1+\ro_{\mathbb{P}}(1))=(\mathbb{E}\xi^2)^{-1}\mathbf{v}_i^{*}XD^2X^{*}\mathbf{v}_i-\frac{1}{n}\sum_{j=1}^n\xi^2_j/\mathbb{E}\xi^2-\frac{\theta_i}{\widetilde{\sigma}_i}+\pi_i+\mathrm{O}_{\mathbb{P}}(\frac{1}{\sqrt{n}}\frac{(n\vee p)}{n\widetilde{\sigma}_i}+\frac{1}{n}).
    \end{gather}

By a similar argument as in Lemma 3.2 of \cite{yu2024testing} and notice that $\xi^2_j\ll \widetilde{\sigma}_i, 1\leq j\leq n$, one has for $\widetilde{\sigma}_r\gg n^{1/4}$,
\begin{equation}
\begin{split}
    \pi_i-\frac{\theta_i}{\Tilde{\sigma}_i}&=\frac{1}{n}\sum_{j=1}^n\xi^2_j/\mathbb{E}\xi^2-1+\big(\frac{1}{n}\sum_{k=r+1}^p\frac{\sigma_k}{\widetilde{\sigma}_i}\times\mathbb{E}(\xi^2_1/\mathbb{E}\xi^2-1)^2\big)\\
    &+\Big(\big(\frac{1}{n}\sum_{k=r+1}^p\frac{\sigma_k}{\widetilde{\sigma}_i}\big)^2\times\mathbb{E}(\xi_1^2/\mathbb{E}\xi^2-1)^3\Big)\times(1+\mathrm{o}(1))\\
    &+\mathrm{O}\Big(\big(\frac{1}{n}\sum_{k=r+1}^p\frac{\sigma_k}{\widetilde{\sigma}_i}\big)^3\Big)+\mathrm{o}(n^{-1/2})\\
    &=\frac{1}{n}\sum_{j=1}^n\xi^2_j/\mathbb{E}\xi^2-1+\big(\frac{1}{n}\sum_{k=r+1}^p\frac{\sigma_k}{\widetilde{\sigma}_i}\times\mathbb{E}(\xi^2_1/\mathbb{E}\xi^2-1)^2\big)+\mathrm{o}(n^{-1/2}).
\end{split}
\end{equation}
Combining the above results, we conclude that
    \begin{equation}\label{eq_limitingrepresentation_spike}
    \frac{\check{\lambda}_i-\theta_i}{\theta_i}=(\mathbb{E}\xi^2)^{-1}\mathbf{v}_i^{*}XD^2X^{*}\mathbf{v}_i-1+\delta_{\widetilde{\sigma}_i}/\mathbb{E}\xi^2+\mathrm{O}_{\mathbb{P}}(\frac{1}{\sqrt{n}}\frac{(n\vee p)}{n\widetilde{\sigma}_i}+\frac{1}{n}),
    \end{equation}
where 
\begin{equation*}
    \delta_{\widetilde{\sigma}_i}:=\mathbb{E}\xi^2\times\big(\frac{1}{n}\sum_{k=r+1}^p\frac{\sigma_k}{\widetilde{\sigma}_i}\times\mathbb{E}(\xi^2/\mathbb{E}\xi^2-1)^2\big)+\mathrm{o}(n^{-1/2}),
\end{equation*}
is defined as a deterministic correction quantity. 

{
\subsubsection{Proof of Theorem \ref{thm_main_spike_unconditional}}\label{sec_proof_spike_unconditional}

Recall the definition of $\check{\lambda}_i$ in \eqref{eq_def_checklambda}, by \eqref{eq_limitingrepresentation_spike} and  \eqref{eq_thetaipriorbound}, we have
\begin{align}\label{eq_expansion111}
    \frac{\check{\lambda}_i}{\widetilde{\sigma}_i}&=\frac{\check{\lambda}_i}{\theta_i}\cdot \frac{\theta_i}{\widetilde{\sigma}_i}=\frac{\check{\lambda}_i}{\theta_i}\cdot\big(1+\frac{1}{n}\sum_{j=r+1}^p\frac{\sigma_j}{\widetilde{\sigma}_i}+\mathrm{o}(n^{-1/2})\big) \nonumber\\
    &=\Big((\mathbb{E}\xi^2)^{-1}\mathbf{v}_i^{*}XD^2X^{*}\mathbf{v}_i-1+1+\delta_{\widetilde{\sigma}_i}/\mathbb{E}\xi^2+\mathrm{o}_{\mathbb{P}}(n^{-1/2})\Big)\cdot\big(1+\frac{1}{n}\sum_{j=r+1}^p\frac{\sigma_j}{\widetilde{\sigma}_i}+\mathrm{o}(n^{-1/2})\big).
\end{align}
We first analyze the random term $(\mathbb{E}\xi^2)^{-1}\mathbf{v}_i^*XD^2X^*\mathbf{v}_i-1$. Let $\mathfrak{m}_4=\mathbb{E}(\sqrt{n}x_{11})^4$. It is straightforward to verify that
\begin{align*}
    \mathbb{E}\big[(\mathbb{E}\xi^2)^{-1}\mathbf{v}_i^*XD^2X^*\mathbf{v}_i-1\big]=0,
\end{align*}
and
\begin{align*}
    &\operatorname{Var}\big((\mathbb{E}\xi^2)^{-1}\mathbf{v}_i^*XD^2X^*\mathbf{v}_i\big)=(\mathbb{E}\xi^2)^{-2}\mathbb{E}[\mathbf{v}_i^*XD^2X^*\mathbf{v}_i]^2-1\\
    &=\frac{1}{n}\Big[\frac{\mathbb{E}\xi^4}{(\mathbb{E}\xi^2)^2}\big(3+(\sum_{k=1}^pv_{ki}^4(\mathfrak{m}_4-3))\big)-1\Big].
\end{align*}
Recall  $\mathsf{M}_i^{\mathtt{S}}$ from \eqref{eq_keynotations} and  $\mathsf{M}_i^{\mathtt{Unc}}$ in (\ref{eq_biasedmean_spike1}). We have that 
\begin{align*}
    &\mathbb{E}\Big(\big((\mathbb{E}\xi^2)^{-1}\mathbf{v}_i^{*}XD^2X^{*}\mathbf{v}_i-1+1+\delta_{\widetilde{\sigma}_i}/\mathbb{E}\xi^2\big)\cdot\big(1+\frac{1}{n}\sum_{j=r+1}^p\frac{\sigma_j}{\widetilde{\sigma}_i}+\mathrm{o}(n^{-1/2})\big)\Big)\\
   % &=(1+\delta_{\widetilde{\sigma}_i}/\mathbb{E}\xi^2)\cdot\big(1+\frac{1}{n}\sum_{j=r+1}^p\frac{\sigma_j}{\widetilde{\sigma}_i}+\mathrm{o}(n^{-1/2})\big)\\
    &=1+\frac{1}{n}\sum_{k=r+1}^p\frac{\sigma_k}{\widetilde{\sigma}_i}\cdot\big(\mathbb{E}(\xi^2/\mathbb{E}\xi^2-1)^2+1\big)+\mathrm{o}(n^{-1/2})\\
    &=1+\frac{\operatorname{Var}(\xi^2)}{(\mathbb{E}\xi^2)^2}(\mathsf{M}_i^{\mathtt{S}}-1)+\mathsf{M}_i^{\mathtt{S}}-1+\mathrm{o}(n^{-1/2})=\mathsf{M}_i^{\mathtt{Unc}}+\mathrm{o}(n^{-1/2}).
\end{align*}
Similarly, for $\mathsf{V}_i^{\mathtt{Unc}}$ in \eqref{eq_biasedmean_spike2}
\begin{align*}
    &\operatorname{Var}\Big(\big((\mathbb{E}\xi^2)^{-1}\mathbf{v}_i^{*}XD^2X^{*}\mathbf{v}_i-1+1+\delta_{\widetilde{\sigma}_i}/\mathbb{E}\xi^2\big)\cdot\big(1+\frac{1}{n}\sum_{j=r+1}^p\frac{\sigma_j}{\widetilde{\sigma}_i}+\mathrm{o}(n^{-1/2})\big)\Big)\\
    &=\operatorname{Var}\Big(\big((\mathbb{E}\xi^2)^{-1}\mathbf{v}_i^{*}XD^2X^{*}\mathbf{v}_i-1\big)\cdot\big(1+\mathrm{o}(1)\big)\Big)\\
    &=\frac{1}{n}\Big(\frac{\mathbb{E}\xi^4}{(\mathbb{E}\xi^2)^2}\big(3+(\sum_{k=1}^pv_{ki}^4(\mathfrak{m}_4-3))\big)-1\Big)\cdot(1+\mathrm{o}(1)) \\
    &=\frac{\mathsf{V}_i^{\mathtt{Unc}}}{n}\cdot(1+\mathrm{o}(1)).
\end{align*}

%Moreover,
%\begin{align*}
%    &\frac{\check{\lambda}_i}{\widetilde{\sigma}_i}=\frac{\mu_i}{\mathbb{E}\xi^2\cdot\widetilde{\sigma}_i}\\
%    &=\Big(\big((\mathbb{E}\xi^2)^{-1}\mathbf{v}_i^{*}XD^2X^{*}\mathbf{v}_i-1+1+\delta_{\widetilde{\sigma}_i}/\mathbb{E}\xi^2\big)\cdot\big(1+\frac{1}{n}\sum_{j=r+1}^p\frac{\sigma_j}{\widetilde{\sigma}_i}+\mathrm{o}(n^{-1/2})\big)\Big)+\mathrm{o}_{\mathbb{P}}(n^{-1/2}).
%\end{align*}
Recall (\ref{eq_expansion111}). Combining the above calculations, the proof follows by applying the central limit theorem to $\mathbf{v}_i^* X D^2 X^* \mathbf{v}_i$ and invoking Slutsky’s theorem. This gives the desired results of $\check{\lambda}_i/\widetilde{\sigma}_i=\mu_i/(\mathbb{E}\xi^2\widetilde{\sigma}_i)$ for $i=1,\dots,r$.
%Therefore, and Slutsky's theorem, we conclude that 
%\begin{equation}\label{eq_unconditionalspike_musigma}
%\lim_{n \rightarrow \infty} \mathbb{P} \left( \sqrt{\frac{n}{\mathsf{V}^{\mathtt{Unc}}_i}}\left(\frac{\mu_i}{\mathbb{E}\xi^2\cdot\widetilde{\sigma}_i}-\mathsf{M}^{\mathtt{Unc}}_i\right) \leq x \right)=\Phi(x),
%\end{equation}
%where
%\begin{align*}
%    \mathsf{M}^{\mathtt{Unc}}_i=1+\frac{\operatorname{Var}(\xi^2)}{(\mathbb{E}\xi^2)^2}(\mathsf{M}_i^{\mathtt{S}}-1)+\mathsf{M}_i^{\mathtt{S}}-1,
%\end{align*}
%and 
%\begin{align*}
%    \mathsf{V}^{\mathtt{Unc}}_i=\frac{\mathbb{E}\xi^4}{(\mathbb{E}\xi^2)^2}\big(3+(\sum_{k=1}^pv_{ki}^4(\mathfrak{m}_4-3))\big)-1.
%\end{align*}
%This completes the proof.
}

\subsection{Proof of Theorem \ref{thm_main_spike_conditional}}
%{\color{green}1. if notations exist, cite them. Do not create new notations. 2. recall from time to time if the logic becomes clear.}

\subsubsection{Proof of \eqref{eq_main_spike_conditional_theta}}\label{sec_prf_main_spikeconditional_theta}

Following the proof of \eqref{eq_limitingrepresentation_spike}, we have for some small constant $\epsilon_0>0$, on $\Omega_X\cap\Omega_D$,
\begin{align*}
    \frac{\check\lambda_i-\theta_i}{\theta_i}=
    (\mathbb{E}\xi^2)^{-1}\mathbf{v}_i^*XD^2X^*\mathbf{v}_i-1+\delta_{\widetilde\sigma_i}/\mathbb{E}\xi^2+
    \mathrm{O}\left(
    \frac{n^{\epsilon_0}}{\sqrt n}\frac{n\vee p}{n\widetilde\sigma_i}+\frac{n^{\epsilon_0}}{n}
    \right),
\end{align*}
where recall that $\check{\lambda}_i=\mu_i/\mathbb{E}\xi^2$ as in \eqref{eq_def_checklambda} and $n \vee p=\max\{n,p\}.$ By an argument similar to (\ref{eq_expansion111}), we obtain on $\Omega_X\cap\Omega_D$ that,
\begin{align}\label{eq_kakakakadddkkkexpress}
    \frac{\check\lambda_i}{\widetilde\sigma_i}&=\big(\mathbf{v}_i^*XX^*\mathbf{v}_i+(\mathbb{E}\xi^2)^{-1}\mathbf{v}_i^*X(D^2-(\mathbb{E}\xi^2)I)X^*\mathbf{v}_i+\frac{1}{n}\sum_{k=r+1}^p\frac{\sigma_k}{\widetilde{\sigma_i}}\times\mathbb{E}(\xi^2/\mathbb{E}\xi^2-1)^2\big) \nonumber \\
    &\times(1+\frac{1}{n}\sum_{k=r+1}^p\frac{\sigma_k}{\widetilde{\sigma}_i})+\mathrm{O}\left(
    \frac{n^{\epsilon_0}}{\sqrt n}\frac{n\vee p}{n\widetilde\sigma_i}+\frac{n^{\epsilon_0}}{n}
    \right)+\mathrm{o}(n^{-1/2}) \nonumber \\
    &=\big(\mathbf{v}_i^*XX^*\mathbf{v}_i-1+1+\frac{1}{n}\sum_{k=r+1}^p\frac{\sigma_k}{\widetilde{\sigma_i}}\times\mathbb{E}(\xi^2/\mathbb{E}\xi^2-1)^2+(\mathbb{E}\xi^2)^{-1}\mathbf{v}_i^*X(D^2-(\mathbb{E}\xi^2)I)X^*\mathbf{v}_i\big)\\
    &\times(1+\frac{1}{n}\sum_{k=r+1}^p\frac{\sigma_k}{\widetilde{\sigma}_i})+\mathrm{O}\left(
    \frac{n^{\epsilon_0}}{\sqrt n}\frac{n\vee p}{n\widetilde\sigma_i}+\frac{n^{\epsilon_0}}{n}
    \right)+\mathrm{o}(n^{-1/2}) \nonumber \\
    &=\big(\mathbf{v}_i^*XX^*\mathbf{v}_i-1+1+\frac{1}{n}\sum_{k=r+1}^p\frac{\sigma_k}{\widetilde{\sigma}_i}+\frac{1}{n}\sum_{k=r+1}^p\frac{\sigma_k}{\widetilde{\sigma_i}}\times\mathbb{E}(\xi^2/\mathbb{E}\xi^2-1)^2+(\mathbb{E}\xi^2)^{-1}\mathbf{v}_i^*X(D^2-(\mathbb{E}\xi^2)I)X^*\mathbf{v}_i\big) \nonumber \\
    &+(\mathbf{v}_i^*XX^*\mathbf{v}_i-1+(\mathbb{E}\xi^2)^{-1}\mathbf{v}_i^*X(D^2-(\mathbb{E}\xi^2)I)X^*\mathbf{v}_i)\times\frac{1}{n}\sum_{k=r+1}^p\frac{\sigma_k}{\widetilde{\sigma}_i} \nonumber \\
    &+\big(\frac{1}{n}\sum_{k=r+1}^p\frac{\sigma_k}{\widetilde{\sigma}_i}\big)^2+\mathrm{O}\left(
    \frac{n^{\epsilon_0}}{\sqrt n}\frac{n\vee p}{n\widetilde\sigma_i}+\frac{n^{\epsilon_0}}{n}
    \right)+\mathrm{o}(n^{-1/2}) \nonumber \\
    &=\mathbf{v}_i^*XX^*\mathbf{v}_i-1+\mathsf{M}_i^{\mathtt{Unc}}+(\mathbb{E}\xi^2)^{-1}\mathbf{v}_i^*X(D^2-(\mathbb{E}\xi^2)I)X^*\mathbf{v}_i \nonumber \\
    &+(\mathbf{v}_i^*XX^*\mathbf{v}_i-1+(\mathbb{E}\xi^2)^{-1}\mathbf{v}_i^*X(D^2-(\mathbb{E}\xi^2)I)X^*\mathbf{v}_i)\cdot\frac{1}{n}\sum_{k=r+1}^p\frac{\sigma_k}{\widetilde{\sigma}_i} \nonumber \\
&    +\mathrm{O}\left(
    \frac{n^{\epsilon_0}}{\sqrt n}\frac{n\vee p}{n\widetilde\sigma_i}+\frac{n^{\epsilon_0}}{n}
    \right)+\mathrm{o}(n^{-1/2}),
\end{align}
where we used the assumption $\widetilde{\sigma}_i\gg n^{1/4}$ in the last step and recall the definition of $\mathsf{M}_i^{\mathtt{Unc}}$ in \eqref{eq_biasedmean_spike1}.
By part (6) of Definition \ref{def_OmegaX} and part (c) of Definition \ref{defn_probset}, we have on the event $\Omega_X\cap\Omega_D$
\begin{align*}
    (\mathbf{v}_i^*XX^*\mathbf{v}_i-1+(\mathbb{E}\xi^2)^{-1}\mathbf{v}_i^*X(D^2-(\mathbb{E}\xi^2)I)X^*\mathbf{v}_i)\cdot\frac{1}{n}\sum_{k=r+1}^p\frac{\sigma_k}{\widetilde{\sigma}_i}=\mathrm{O}(\frac{n^{\epsilon_1}}{\sqrt{n}\widetilde{\sigma}_i}),
\end{align*}
for some small constant $\epsilon_1>\varepsilon_6$ defined in Definition \ref{def_OmegaX}. It follows that for $\widetilde{\sigma}_i\gg n^{1/4}$, the expression of $\check{\lambda}_i/\widetilde{\sigma}_i$ in (\ref{eq_kakakakadddkkkexpress}) can be simplified to 
\begin{align}\label{eq_spike_conditional_representation_lambdasigma}
    \frac{\check{\lambda}_i}{\widetilde{\sigma}_i}&=\mathbf{v}_i^*XX^*\mathbf{v}_i-1+\mathsf{M}_i^{\mathtt{Unc}}+(\mathbb{E}\xi^2)^{-1}\mathbf{v}_i^*X(D^2-(\mathbb{E}\xi^2)I)X^*\mathbf{v}_i+\mathrm{O}\left(
    \frac{n^{\epsilon_2}}{\sqrt n}\frac{n\vee p}{n\widetilde\sigma_i}+\frac{n^{\epsilon_2}}{n}
    \right)+\mathrm{o}(n^{-1/2})\nonumber \\
    &=\mathsf{M}_i^{\mathtt{Con1}}+\frac{1}{\mathbb{E}\xi^2}\sum_{j=1}^n \omega_j^2(\xi_j^2-\mathbb{E}\xi^2)+r_n,
\end{align}
where $\mathsf{M}_i^{\mathtt{Con1}}$ is defined in \eqref{eq_biasedmean_conditional} for $i=1,\dots, r$, we denote $\omega_j:=\sum_{k=1}^p v_{ki}x_{kj}, j=1,\dots, n$, and the remainder $r_n$ satisfies
\begin{align}\label{eq_spike_conditional_remainder_bound_final}
    r_n=\mathrm{O}\left(\frac{n^{\epsilon_2}}{\sqrt n}\frac{n\vee p}{n\widetilde\sigma_i}+\frac{n^{\epsilon_2}}{n}\right)+\mathrm{o}(n^{-1/2}),
\end{align}
on $\Omega_X\cap\Omega_D$, for some small constant $\epsilon_2\geq\max\{\epsilon_0,\epsilon_1\}$. Since $\widetilde{\sigma}_i\gg n^{1/4}$, the right-hand side of \eqref{eq_spike_conditional_remainder_bound_final} is of order $\mathrm{o}(n^{-1/2})$. Therefore, for every fixed $\delta>0$, there exists $N_\delta$ such that for all $n\geq N_\delta$,
\begin{align}\label{eq_spike_conditional_remainder_small_final}
    \sqrt{\frac{n}{\mathsf{V}_i^{\mathtt{Con}}}}\,|r_n|\leq \delta
    \qquad\text{on }\Omega_X\cap\Omega_D,
\end{align}
where recall the definition of $\mathsf{V}_i^{\mathtt{Con}} \asymp 1$ in \eqref{eq_biasedvar_conditional}. Recall that $\omega_j:=\sum_{k=1}^p v_{ki}x_{kj}.$
%
% as
%\begin{align}
%    \mathsf{V}_i^{\mathtt{Con}}=\frac{\operatorname{Var}\xi^2}{(\mathbb{E}\xi^2)^2}(\mathsf{V}_i^{\mathtt{S}}+1).
%\end{align}
Denote
\begin{align}\label{eq_definitionsimportant}
    T_n
    :=
    \sqrt{\frac{n}{\mathsf V_i^{\mathtt{Con}}}}
    \left(
    \frac{\check\lambda_i}{\widetilde\sigma_i}-\mathsf M_i^{\mathtt{Con1}}
    \right), \  \
    \eta_{j,n}
    :=
    \frac{\sqrt n}{(\mathbb E\xi^2)\sqrt{\mathsf V_i^{\mathtt{Con}}}}
    \,\omega_j^2(\xi_j^2-\mathbb E\xi^2).
\end{align}
Then on $\Omega_X\cap\Omega_D$, \eqref{eq_spike_conditional_representation_lambdasigma} can be wriiten as
\begin{align}\label{eq_spike_conditional_Tn_decomp_final}
    T_n=\sum_{j=1}^n \eta_{j,n}+R_n,
\end{align}
where $R_n:=\sqrt{\frac{n}{\mathsf V_i^{\mathtt{Con}}}}\,r_n .$ By \eqref{eq_spike_conditional_remainder_small_final}, $R_n$ satisfies that for sufficiently large $n$
\begin{align}\label{eq_spike_conditional_Rn_small_final}
    |R_n|\leq\delta
    \quad \text{on} \ \ \Omega_X\cap\Omega_D.
\end{align}

We first establish the conditional CLT for $\sum_{j=1}^n\eta_{j,n}$ in (\ref{eq_definitionsimportant}) using the Lindeberg-Feller CLT. Since the multipliers are independent of $X$, conditional on $X,$ the random variables $\{\eta_{j,n}\}_{j=1}^n$ are independent and centered. Consequently, we have that 
\begin{align*}
    \mathbb{E}(\eta_{j,n}\mid X)=\frac{\sqrt{n}w_j^2}{(\mathbb{E}\xi^2)\sqrt{\mathsf{V}_i^{\mathtt{Con}}}}\mathbb{E}(\xi^2_j-\mathbb{E}\xi^2)=0, \ \ 
    \sum_{j=1}^n \mathbb{E}(\eta_{j,n}^2\mid X)
    &=
    \frac{n}{(\mathbb{E}\xi^2)^2\mathsf{V}_i^{\mathtt{Con}}}
    \sum_{j=1}^n \omega_j^4\,\mathbb{E}\big((\xi_j^2-\mathbb{E}\xi^2)^2\big).
\end{align*}
By part (6) of Definition \ref{def_OmegaX}, on $\Omega_X$,
\begin{align*}
    \sum_{j=1}^n \omega_j^4=
\frac{1}{n}\Big(3+\sum_{k=1}^p v_{ki}^4(\mathfrak{m}_4-3)\Big)+\mathrm{o}(n^{-1}),
\end{align*}
we obtain
\begin{align*}
    \sum_{j=1}^n \mathbb{E}(\eta_{j,n}^2\mid X)=1+\mathrm{o}(1)
\quad\text{on }\Omega_X.
\end{align*}

We verify the Lindeberg conditions for the triangular array $\{\eta_{j,n}\}_{j=1}^n$. Fix $\varepsilon>0$, and define
\begin{align*}
    L_n(X):=\sum_{j=1}^n\mathbb{E}\Big(\eta_{j,n}^2 \mathbf{1}(|\eta_{j,n}|>\varepsilon)\,\Big|\,X\Big).
\end{align*}
On $\Omega_X$, we decompose
\begin{align}\label{eq_decomposition_lindeberg}
    L_n(X)=L_{n,1}(X)+L_{n,2}(X),
\end{align}
where
\begin{align*}
    L_{n,1}(X)&:=\sum_{j=1}^n\mathbb{E}\Big(
    \eta_{j,n}^2 \mathbf{1}(|\eta_{j,n}|>\varepsilon)\mathbf{1}_{\Omega_D}\,\Big|\,X\Big),\
    L_{n,2}(X):=\sum_{j=1}^n\mathbb{E}\Big(
    \eta_{j,n}^2 \mathbf{1}(|\eta_{j,n}|>\varepsilon)\mathbf{1}_{\Omega_D^c}\,\Big|\,X\Big).
\end{align*}
For $L_{n,1}(X)$, by Markov's inequality,
\begin{align*}
    L_{n,1}(X)&\leq\frac{1}{\varepsilon^2}\sum_{j=1}^n
    \mathbb{E}\big(|\eta_{j,n}|^4\mathbf{1}_{\Omega_D}\mid X\big).
\end{align*}
%Recall that
%\begin{align*}
%    \eta_{n,j}=
%    \frac{\sqrt n}{(\mathbb{E}\xi^2)\sqrt{\mathsf{V}_i^{\mathtt{Con}}}}
%    \,\omega_j^2(\xi_j^2-\mathbb{E}\xi^2),
%\end{align*}
%so that
Note that
\begin{align*}
    \sum_{j=1}^n
    \mathbb{E}\big(|\eta_{j,n}|^4\mathbf{1}_{\Omega_D}\mid X\big)
    =\frac{n^2}{(\mathbb{E}\xi^2)^4(\mathsf{V}_i^{\mathtt{Con}})^2}
    \sum_{j=1}^n\omega_j^8\,
    \mathbb{E}\Big((\xi^2-\mathbb{E}\xi^2)^4\mathbf{1}_{\Omega_D}\Big).
\end{align*}
By part (6) of Definition \ref{def_OmegaX},
\begin{align*}
    \sum_{j=1}^n \omega_j^8=
    \mathrm{O}\big(n^{-3}\log^{\varepsilon_7}n\big)
    \qquad\text{on }\Omega_X,
\end{align*}
for some small constant $\varepsilon_7>0$ defined in Definition \ref{def_OmegaX}. One the one hand, if the multipliers are bounded-support, exponential-tail, or polynomial-tail with parameter $\alpha\ge4$, then Assumption \ref{assum_D} implies $\mathbb{E}(\xi^2-\mathbb{E}\xi^2)^4<\infty, $ and hence
\begin{align*}
    \sum_{j=1}^n
    \mathbb{E}\big(|\eta_{j,n}|^4\mathbf{1}_{\Omega_D}\mid X\big)=
    \mathrm{O}\big(n^{-1}\log^{\varepsilon_7}n\big)=
    \mathrm{o}(1)
    \qquad\text{on }\Omega_X.
\end{align*}
On the other hand, if the multipliers have polynomial tail with parameter $\alpha\in(2,4)$, then part (a) of Definition \ref{defn_probset} gives, on $\Omega_D$,
\begin{align*}
    |\xi_j^2-\mathbb{E}\xi^2|
    \leq C n^{1/\alpha}\log n,
    \qquad 1\leq j\leq n,
\end{align*}
for some constant $C>0$. Therefore,
\begin{align*}
    \mathbb{E}\Big((\xi^2-\mathbb{E}\xi^2)^4\mathbf{1}_{\Omega_D}\Big)=
    \mathrm{O}\big(n^{2/\alpha}\log^2 n\big),
\end{align*}
and hence
\begin{align*}
    \sum_{j=1}^n\mathbb{E}\big(|\eta_{j,n}|^4\mathbf{1}_{\Omega_D}\mid X\big)=
    \mathrm{O}\big(n^{-1+2/\alpha}\log^{2+\varepsilon_7}n\big)
    =\mathrm{o}(1)
    \qquad\text{on }\Omega_X.
\end{align*}
Combining the above estimates, we conclude that
\begin{align}\label{eq_spike_Ln1_small}
    L_{n,1}(X)=\mathrm{o}(1)
    \qquad\text{on }\Omega_X.
\end{align}

For $L_{n,2}(X),$ since $\mathbf{1}(|\eta_{j,n}|>\varepsilon)\leq1$,
\begin{align*}
    L_{n,2}(X)
    \leq
    \sum_{j=1}^n
    \mathbb{E}\big(\eta_{j,n}^2\mathbf{1}_{\Omega_D^c}\mid X\big)=
    \frac{n}{(\mathbb{E}\xi^2)^2\mathsf{V}_i^{\mathtt{Con}}}
    \sum_{j=1}^n \omega_j^4\,
    \mathbb{E}\Big((\xi_j^2-\mathbb{E}\xi^2)^2\mathbf{1}_{\Omega_D^c}\,\Big|\,X\Big).
\end{align*}
Since $D$ is independent of $X$, we have that 
\begin{align*}
    \mathbb{E}\Big((\xi_j^2-\mathbb{E}\xi^2)^2\mathbf{1}_{\Omega_D^c}\,\Big|\,X\Big)
    =
    \mathbb{E}\Big((\xi^2-\mathbb{E}\xi^2)^2\mathbf{1}_{\Omega_D^c}\Big).
\end{align*}
Therefore, we have that 
\begin{align*}
    L_{n,2}(X)
    &\leq
    \frac{n}{(\mathbb{E}\xi^2)^2\mathsf{V}_i^{\mathtt{Con}}}
    \Big(\sum_{j=1}^n\omega_j^4\Big)
    \mathbb{E}\Big((\xi^2-\mathbb{E}\xi^2)^2\mathbf{1}_{\Omega_D^c}\Big).
\end{align*}
By part (6) of Definition \ref{def_OmegaX}, we have that 
\begin{align*}
    \sum_{j=1}^n \omega_j^4=
    \frac{1}{n}\Big(3+\sum_{k=1}^p v_{ki}^4(\mathfrak{m}_4-3)\Big)+\mathrm{o}(n^{-1}),
\end{align*}
and hence $n\sum_{j=1}^n\omega_j^4=\mathrm{O}(1)\, \text{on }\Omega_X.$ It follows that
\begin{align*}
    L_{n,2}(X)
    \leq
    C\,\mathbb{E}\Big((\xi^2-\mathbb{E}\xi^2)^2\mathbf{1}_{\Omega_D^c}\Big)
   \ \text{on }\Omega_X,
\end{align*}
for some constant \(C>0\). 
%Since $\alpha>2$, we have
%\begin{align*}
%    \mathbb{E}(\xi^2-\mathbb{E}\xi^2)^2<\infty.
%\end{align*}
Moreover, since $\mathbf{1}_{\Omega_D^c}\rightarrow0$ almost surely and
\[
0\leq (\xi^2-\mathbb E\xi^2)^2\mathbf{1}_{\Omega_D^c}\leq(\xi^2-\mathbb{E}\xi^2)^2,
\]
by dominated convergence theorem, we have that 
$\mathbb{E}\Big((\xi^2-\mathbb{E}\xi^2)^2\mathbf{1}_{\Omega_D^c}\Big)\rightarrow0.$ Therefore
\begin{align}\label{eq_spike_Ln2_small}
    L_{n,2}(X)=\mathrm{o}(1)
    \ \text{on }\Omega_X.
\end{align}

Combining \eqref{eq_spike_Ln1_small} and \eqref{eq_spike_Ln2_small}, we conclude from \eqref{eq_decomposition_lindeberg} that
\begin{align*}
    \sum_{j=1}^n\mathbb{E}\Big(
    \eta_{j,n}^2\mathbf{1}(|\eta_{j,n}|>\varepsilon)
    \,\Big|\,X\Big)\rightarrow0
    \ \text{on }\Omega_X.
\end{align*}
That is, the Lindeberg condition holds on \(\Omega_X\). Consequently, for every fixed \(x\in\mathbb R\),
\begin{align}\label{eq_spike_conditional_partialsum_clt_final}
    \Big|\mathbb{P}\Big(\sum_{j=1}^n\eta_{j,n}\leq x\,\Big|\,X\Big)-\Phi(x)\Big|\rightarrow0
    \ \text{on} \ \Omega_X.
\end{align}

We next work with $T_n$ in (\ref{eq_definitionsimportant}). Recall that for any fixed $x\in\mathbb R$ and $\delta>0$, we have by \eqref{eq_spike_conditional_Rn_small_final}, for all sufficiently large $n$, on $\Omega_X\cap\Omega_D$,
\begin{align*}
    \Big|T_n-\sum_{j=1}^n\eta_{j,n}\Big|\leq\delta.
\end{align*}
Hence, for any fixed $x\in\mathbb{R}$,
\begin{align*}
    \Big\{\sum_{j=1}^n\eta_{j,n}\leq x-\delta\Big\}\cap\Omega_X\cap\Omega_D
    \subset
    \{T_n\leq x\},
\end{align*}
and
\begin{align*}
    \{T_n\leq x\}\cap\Omega_X\cap\Omega_D
    \subset
    \Big\{\sum_{j=1}^n\eta_{j,n}\leq x+\delta\Big\}.
\end{align*}
Equivalently, on $\Omega_X$,
\begin{align*}
    \Big\{\sum_{j=1}^n\eta_{j,n}\leq x-\delta\Big\}\cap\Omega_D
    \subset
    \{T_n\leq x\},
    \quad
    \{T_n\leq x\}
    \subset
    \Big\{\sum_{j=1}^n\eta_{j,n}\leq x+\delta\Big\}\cup\Omega_D^c.
\end{align*}
Then, by an argument similar to (\ref{eq_prf_boundedconditional_1}), we can obtain on that $\Omega_X$,
\begin{align*}
    \mathbb{P}\Big(\sum_{j=1}^n\eta_{j,n}\leq x-\delta\,\Big|\,X\Big)-\mathbb{P}(\Omega_D^c\mid X)
    \le
    \mathbb{P}(T_n\leq x\mid X),
\end{align*}
and
\begin{align*}
    \mathbb{P}(T_n\leq x\mid X)
    \le
    \mathbb{P}\Big(\sum_{j=1}^n\eta_{j,n}\leq x+\delta\,\Big|\,X\Big)+\mathbb{P}(\Omega_D^c\mid X).
\end{align*}
Since $\Omega_D\in\sigma(D)$ and $D$ is independent of $X$, we have $\mathbb{P}(\Omega_D^c\mid X)=\mathbb{P}(\Omega_D^c) \
\text{almost surely}. $ Therefore, by \eqref{eq_spike_conditional_partialsum_clt_final} and $\mathbb{P}(\Omega_D^c)\rightarrow0$ from Lemma \ref{lem_probabilitycontrol}, we obtain on $\Omega_X$,
\begin{align*}
    \Phi(x-\delta)+\mathrm{o}(1)
    \le
    \mathbb{P}(T_n\leq x\mid X)
    \le
    \Phi(x+\delta)+\mathrm{o}(1).
\end{align*}
Letting first $n\rightarrow\infty$ and then $\delta\downarrow0$, and using the continuity of $\Phi$, we conclude that on $\Omega_X$, for every fixed $x\in\mathbb{R}$
\begin{align}\label{eq_spike_conditional_clt_on_good_event_final}
    \Big|\mathbb{P}\Big(\sqrt{\frac{n}{\mathsf{V}_i^{\mathtt{Con}}}}
    \Big(\frac{1}{\mathbb{E} \xi^2}\frac{\mu_i}{\widetilde{\sigma}_i}-\mathsf{M}_i^{\mathtt{Con1}}\Big)\leq x\,\Big|\,X\Big)-\Phi(x)\Big|=\mathrm{o}(1),
\end{align}
where we recall the definition in (\ref{eq_def_checklambda}). 

With the above results on $\Omega_X$, the arguments for the conditional probability can be established by following the same strategy as in Section~\ref{sec_conditional12121212121}. We omit the details due to their similarity. 

%Finally, following a similar strategy as in Section~\ref{sec_conditional12121212121}, define
%\begin{align*}
%    \Delta_n(X):=\Big|\mathbb{P}\Big(\sqrt{\frac{n}{\mathsf{V}_i^{\mathtt{Con}}}}
%    \Big(\frac{\check{\lambda}_i}{\widetilde{\sigma}_i}-\mathsf{M}_i^{\mathtt{Con1}}
%    \Big)\le x\,\Big|\,X\Big)-\Phi(x) \Big|.
%\end{align*}
%Then $\Delta_n(X)$ is $\sigma(X)$-measurable. For every $\varepsilon>0$,
%\begin{align*}
%    \mathbb{P}(\Delta_n(X)>\varepsilon)\le
%    \mathbb{P}(\Omega_X^c)+\mathbb{P}(\{\Delta_n(X)>\varepsilon\}\cap\{\Omega_X\}).
%\end{align*}
%The first term tends to zero because $\mathbb{P}(\Omega_X)=1-\mathrm{o}(1)$, while the second term tends to zero by \eqref{eq_spike_conditional_clt_on_good_event_final}. Therefore, 
%\begin{align*}
%    \mathbb{P}(\Delta_n(X)>\epsilon)=\mathrm{o}(1).
%\end{align*}
%Equivalently, for every fixed $x\in\mathbb{R}$, with probability at least 
%$1-\mathrm{o}(1)$,
%\begin{align*}
%    \Big|\mathbb{P}\Big(
%    \sqrt{\frac{n}{\mathsf{V}_i^{\mathtt{Con}}}}
%    \Big(\frac{\check\lambda_i}{\widetilde\sigma_i}-\mathsf{M}_i^{\mathtt{Con1}}\Big)\le x\,\Big|\,X\Big)-\Phi(x)\Big|
%    =\mathrm{o}(1).
%\end{align*}
%This proves the desired central limit theorem for $\check{\lambda}_i/\widetilde{\sigma}_i=\mu_i/(\mathbb{E}\xi^2\widetilde{\sigma}_i), i=1,\dots, r$.

\subsubsection{Proof of \eqref{eq_main_spike_conditional_hatmu}}

The proof follows similar arguments to those in the previous section, with the main additional task being to control the difference between $\widetilde{\sigma}_i$ and $\widehat{\mu}_i$. We focus on explaining the main ideas and omit the details. In the following, we use the convention $\mathsf{P}^*=\mathbb{P}(\cdot|X)$ to simplify the notation. The core inputs of the proof are the results in \eqref{eq_limitingrepresentation_spike} and \eqref{eq_limitingrepresentation_1}. Recall $\check{\lambda}_i=\mu_i/\mathbb{E}\xi^2$ in \eqref{eq_def_checklambda}. We have
\begin{gather*}
    \frac{\check{\lambda}_i}{\theta_i}-1=(\mathbb{E}\xi^2)^{-1}\mathbf{v}_i^{*}XD^2X^{*}\mathbf{v}_i-1+\delta_{\widetilde{\sigma}_i}/\mathbb{E}\xi^2+\mathrm{o}_{\mathsf{P}^*}(n^{-1/2})\\
    \frac{\widehat{\mu}_i}{\theta_i}-1=\mathbf{v}_i^{*}XX^*\mathbf{v}_i-1+\mathrm{o}_{\mathbb{P}}(n^{-1/2}).
\end{gather*}
Then, on $\Omega_X$, it holds that 
\begin{align*}
    \frac{\widehat{\mu}_i}{\theta_i}-1=\mathbf{v}_i^{*}XX^*\mathbf{v}_i-1+\mathrm{o}(n^{-1/2}).
\end{align*}
Consequently,
\begin{equation*}
    \begin{split}
        \frac{\check{\lambda}_i}{\widehat{\mu}_i}&=\frac{\check{\lambda}_i}{\theta_i}\times\frac{\theta_i}{\widehat{\mu}_i}=\frac{1+(\mathbb{E}\xi^2)^{-1}\mathbf{v}_i^{*}XD^2X^{*}\mathbf{v}_i-1+\delta_{\widetilde{\sigma}_i}/\mathbb{E}\xi^2+\mathrm{o}_{\mathsf{P}^*}(n^{-1/2})}{1+\mathbf{v}_i^{*}XX^*\mathbf{v}_i-1+\mathrm{o}(n^{-1/2})}\\
        &=\frac{(\mathbb{E}\xi^2)^{-1}\mathbf{v}_i^{*}X(D^2-\mathbb{E}\xi^2\cdot I)X^{*}\mathbf{v}_i+\delta_{\widetilde{\sigma}_i}/\mathbb{E}\xi^2}{\mathbf{v}_i^{*}XX^*\mathbf{v}_i+\mathrm{o}(n^{-1/2})}+1+\mathrm{o}_{\mathsf{P}^*}(n^{-1/2}).
    \end{split}
\end{equation*}
By part (6) of Definition \ref{def_OmegaX}, we have on the event $\Omega_X$, $\mathbf{v}_i^*XX^*\mathbf{v}_i=1+\mathrm{O}(n^{-1/2+\varepsilon_6})$ for some small constant $\varepsilon_6>0$. It yields that 
\begin{align*}
     \frac{\check{\lambda}_i}{\widehat{\mu}_i}&=(\mathbb{E}\xi^2)^{-1}\mathbf{v}_i^{*}X(D^2-\mathbb{E}\xi^2\cdot I)X^{*}\mathbf{v}_i+\delta_{\widetilde{\sigma}_i}/\mathbb{E}\xi^2+1+\mathrm{o}(n^{-1/2})+\mathrm{o}_{\mathsf{P}^*}(n^{-1/2})\\
     &=(\mathbb{E}\xi^2)^{-1}\mathbf{v}_i^{*}X(D^2-\mathbb{E}\xi^2\cdot I)X^{*}\mathbf{v}_i+\frac{1}{n}\sum_{k=r+1}^p\frac{\sigma_k}{\widetilde{\sigma}_i}\times\mathbb{E}(\xi^2/\mathbb{E}\xi^2-1)^2+1+\mathrm{o}(n^{-1/2})+\mathrm{o}_{\mathsf{P}^*}(n^{-1/2})\\
     &=(\mathbb{E}\xi^2)^{-1}\mathbf{v}_i^{*}X(D^2-\mathbb{E}\xi^2\cdot I)X^{*}\mathbf{v}_i+\mathsf{M}_i^{\mathtt{Con2}}+\mathrm{o}(n^{-1/2})+\mathrm{o}_{\mathsf{P}^*}(n^{-1/2}),
\end{align*}
where we used the fact $\widetilde{\sigma}_i\gg n^{1/4}$ and recall the definition of $\mathsf{M}_i^{\mathtt{Con2}}$ in \eqref{eq_biasedmean_spike1}. 

With the above representation, the remainder of the proof follows by arguments similar to those used in the analysis of (\ref{eq_spike_conditional_representation_lambdasigma}) in Section \ref{sec_prf_main_spikeconditional_theta}. We omit the details due to the similarity.

\section{Proof of some auxiliary lemmas}\label{appendix_last}

\subsection{Preliminary estimates: Proof of Lemmas \ref{lem: basic bounds} and \ref{localestimate2}}\label{sec_proofpreliminarystieltjes}

\subsubsection{Proof of Lemma \ref{lem: basic bounds}}

Due to similarity, we only prove the results for the separable covariance i.i.d. data model when $\xi^2$ decays  polynomially, i.e., when (\ref{eq:F(m,z)}) and (\ref{ass3.1}) hold. The other cases can be proved analogously, and we omit the details.    

\begin{proof}[{\bf Proof}]
We start with the first statement.   We now abbreviate $F_n(m_{1n}(z)) \equiv F_{n}(m_{1n}(z),z)$ throughout the proof. 
For the real part,  it suffices to prove that with high probability for some $0<C_2<1<C_1$
\begin{equation}\label{eq_realcontrol}
\operatorname{Re} m_{1n}(z) \in \left[-C_1  \frac{\phi \bar{\sigma}_1 E}{E^2+\eta^2}, -C_2  \frac{\phi \bar{\sigma}_1 E}{E^2+\eta^2}\right].
\end{equation}
Moreover, by continuity and Theorem \ref{lem_solutionsystem}, it suffices to prove the following inequalities
\begin{equation}\label{eq)ccc}
 \operatorname{Re}F_n(-C_2 \phi \bar{\sigma}_1 E (E^2+\eta^2)^{-1}+\mathrm{i} \operatorname{Im} m_{1n}(z))<0, \  \operatorname{Re}F_n(-C_1 \phi \bar{\sigma}_1 E (E^2+\eta^2)^{-1}+\mathrm{i} \operatorname{Im} m_{1n}(z))>0.
\end{equation} 
We only focus on the first part. By definition, we have that   
 \begin{align}\label{eq_expansion}
 \operatorname{Re} & F_n(m_{1n}(z))=-\operatorname{Re}m_{1n}(z) \\
&  -\frac{1}{n}\sum_{i=1}^p\frac{\sigma_i\operatorname{Re}(z-\frac{\sigma_i}{n}\sum_{j=1}^n\frac{\xi^2_j}{1+m_{1n}(z)\xi^2_j})}{\operatorname{Re}^2(z-\frac{\sigma_i}{n}\sum_{j=1}^n\frac{\xi^2_j}{1+m_{1n}(z)\xi^2_j})+\operatorname{Im}^2(z-\frac{\sigma_i}{n}\sum_{j=1}^n\frac{\xi^2_j}{1+m_{1n}(z)\xi^2_j})}. \nonumber
 \end{align}
Note that 
 \begin{gather*}
     \operatorname{Re}\Big(z-\frac{\sigma_i}{n}\sum_{j=1}^n\frac{\xi^2_j}{1+m_{1n}\xi^2_j} \Big)=E-\frac{\sigma_i}{n}\sum_{j=1}^n\frac{\xi^2_j(1+\xi^2_j\operatorname{Re}m_{1n})}{(1+\xi^2_j\operatorname{Re}m_{1n})^2+\xi^4_j\operatorname{Im}^2m_{1n}},\\
     \operatorname{Im}\Big(z-\frac{\sigma_i}{n}\sum_{j=1}^n\frac{\xi^2_j}{1+m_{1n}\xi^2_j}\Big)=\eta+\frac{\sigma_i}{n}\sum_{j=1}^n\frac{\xi^4_j\operatorname{Im}m_{1n}}{(1+\xi^2_j\operatorname{Re}m_{1n})^2+\xi^4_j\operatorname{Im}^2m_{1n}}.
 \end{gather*}
By a discussion similar to (\ref{eq: 2.2 beta=2}),  if $\operatorname{Re}m_{1n}=-C_2 (\phi \bar{\sigma}_1 E)/(E^2+\eta^2)$, we have that
 \begin{equation}\label{eq_controlcontrolddd}
 \begin{split}
    \operatorname{Re}\Big(z-\frac{\sigma_i}{n}\sum_{j=1}^n\frac{\xi^2_j}{1+m_{1n}\xi^2_j} \Big)
%    &=E-\frac{\sigma_i}{n}\sum_{j=1}^n\frac{\xi^2_j(1+\xi^2_j\operatorname{Re}m_{1n})}{(1+\xi^2_j\operatorname{Re}m_{1n})^2+\xi^4_j\operatorname{Im}^2m_{1n}}\\
%    &\ge E-\frac{\sigma_i}{n}\sum_{j=1}^n\frac{\xi^2_j}{1+\xi^2_j\operatorname{Re}m_{1n}}\\
    &\geq E(1-\ro(1)),\\
     \operatorname{Im}\Big(z-\frac{\sigma_i}{n}\sum_{j=1}^n\frac{\xi^2_j}{1+m_{1n}\xi^2_j} \Big)
%     &=\eta+\frac{\sigma_i}{n}\sum_{j=1}^n\frac{\xi^4_j\operatorname{Im}m_{1n}}{(1+\xi^2_j\operatorname{Re}m_{1n})^2+\xi^4_j\operatorname{Im}^2m_{1n}}\\
%     &\le \eta+\frac{\sigma_i}{2n}\sum_{j=1}^n\frac{\xi^2_j}{1+\xi^2_j\operatorname{Re}m_{1n}}\\
     &\leq \eta+ E\times \ro(1).
 \end{split}
 \end{equation}
% where we use the inequality $a^2+b^2\ge2ab$ for positive constants $a,b$. 
Therefore, together with (\ref{eq_expansion}), we see that 
\begin{align}\label{eq: 2.3}
  %\begin{split}
   \operatorname{Re}F_n(-C_2 \phi \bar{\sigma}_1 E (E^2+\eta^2)^{-1}+\ri \operatorname{Im} m_{1n}(z))
   &\leq C_2\phi\bar{\sigma}_1\frac{E}{(E^2+\eta^2)}-\frac{1}{n}\sum_{i=1}^p\frac{\sigma_i\times E(1-\ro(1))}{E^2+(\eta+E \times \ro(1))^2} \nonumber\\
   &\leq (C_2-1+\ro(1)) \frac{\phi\bar{\sigma}_1E}{(E^2+\eta^2)}<0,
%\end{split}  
\end{align}
for sufficient large $n$. This completes the discussion for the real part. For the complex part, the idea is similar and it suffices to prove that when $z \in \mathbf{D}_{u},$ for some constants $C_1, C_2>0$ 
\begin{equation}\label{eq_imaginarycontrol}
\operatorname{Im} m_{1n}(z) \in \left[ C_1 \frac{ \eta \phi \bar{\sigma}_1}{E^2+\eta^2}, C_2 \eta \left| \operatorname{Re} m_{1n}(z) \right| \right].
\end{equation} 
Equivalently, it suffices to prove that
\begin{equation*}
\operatorname{Im} F_n(\operatorname{Re} m_{1n}(z)+\ri C_2 \eta \left| \operatorname{Re} m_{1n}(z) \right|)<0,  \  \operatorname{Im} F_n\Big(\operatorname{Re} m_{1n}(z)+\ri C_1 \frac{ \eta \phi \bar{\sigma}_1}{E^2+\eta^2} \Big)>0,  
\end{equation*}
where by definition
\[
\operatorname{Im}F_n(m_{1n},z)=-\operatorname{Im}m_{1n}+\frac{1}{n}\sum_{i=1}^p\frac{\sigma_i\operatorname{Im}(z-\frac{\sigma_i}{n}\sum_{j=1}^n\frac{\xi^2_j}{1+m_{1n}\xi^2_j})}{\operatorname{Re}^2(z-\frac{\sigma_i}{n}\sum_{j=1}^n\frac{\xi^2_j}{1+m_{1n}\xi^2_j})+\operatorname{Im}^2(z-\frac{\sigma_i}{n}\sum_{j=1}^n\frac{\xi^2_j}{1+m_{1n}\xi^2_j})}.
\]
The proof of the above inequalities is similar to (\ref{eq)ccc}) using (\ref{eq_realcontrol}). We briefly discuss the proof of the first inequality by an argument similar to (\ref{eq_controlcontrolddd}) 
\[
\begin{split}
   \operatorname{Im}F_n(m_{1n},z)&=\eta\operatorname{Re}m_{1n}+\frac{\phi\bar{\sigma}_1\eta(1+E)}{E^2+\eta^2(1+E\times \ro(1))^2}\\
   &\leq-C_2\frac{\phi\bar{\sigma}_1\eta E}{(E^2+\eta^2)}+\frac{\phi\bar{\sigma}_1\eta(1+E)}{E^2(1+\eta^2\times \ro(1))+\eta^2(1+2E\times \ro(1))}<0.
\end{split}
\]
This completes our proof.

For the second statement, from the first statement, we see that it is valid to write $m_{1n}(E).$ Since $\vartheta_1 \gg d_1$ holds with high probability (see (\ref{eq_lowerboundmu1})), it suffices to prove that for some constants $0<C_2<1<C_1,$ when $|E-\vartheta_1| \leq C d_1,$  
\begin{equation}\label{eq_boundbound}
m_{1n}(E) \in \left[-C_1 \frac{\phi \bar{\sigma}_1}{E},-C_2 \frac{\phi \bar{\sigma}_1}{E}\right].
\end{equation}
Due to simplicity, we again only focus on the proof of the upper bound. According to Theorem \ref{lem_solutionsystem} and (\ref{eq_functionFequal}), we shall have that $F_n(m_{1n}(E))=0.$ Moreover, since $F_{n}(m_{1n}(\vartheta_1))=0,$ to prove (\ref{eq_boundbound}), it suffices to prove
\begin{equation}\label{eq_reduceddiscussion}
F_{n}(-C_2 \phi \bar{\sigma}_1/E)<0, \ F_{n}(-C_1 \phi \bar{\sigma}_1/E)>0.
\end{equation}
Due to similarity, we focus our discussion on the first inequality. By definition, we have that 
\[
 \begin{split}
     F_n(-C_2 \phi \bar{\sigma}_1/E)&=C_2\frac{\phi\bar{\sigma}_1}{E}-\frac{1}{n}\sum_{i=1}^p\frac{\sigma_i}{E-\frac{\sigma_i}{n}\sum_{j=1}^n\frac{\xi^2_j}{1-\xi^2_j(C_2\frac{\phi\bar{\sigma}_1}{E})}}\\
     &=C_2 \frac{\phi\bar{\sigma}_1}{E}-\frac{1}{n}\sum_{i=1}^p\frac{\sigma_i}{E(1-\frac{\sigma_i}{n} \sum_{j=1}^n \frac{\xi_j^2}{E-C_2\xi_j^2 \phi \bar{\sigma}_1})}.
 \end{split}
 \]
By a discussion parallel to (\ref{eq: 2.2 beta=2}), we further have that
\begin{equation}\label{eq_similarcontrolone}
 F_n(-C_2 \phi \bar{\sigma}_1/E)=C_2 \frac{\phi \bar{\sigma}_1}{E}-\frac{1}{n} \sum_{i=1}^p \frac{\sigma_i}{E(1-\ro(1))}=(C_2-1+\ro(1)) \frac{\phi \bar{\sigma}_1}{E}<0.  
\end{equation}
This completes the proof of the first statement.

Finally, we prove the third statement using the first two statements. For $z \in \mathbf{D}_{u},$ by definition, we have that
\begin{align}\label{eq: m_2n}
 %  \begin{split}
   m_{2n}(z)&=\frac{1}{n}\sum_{j=1}^n\frac{\xi^2_j}{-E-\ri\eta-(E+\ri\eta)\xi^2_j(\operatorname{Re}m_{1n}+\ri \operatorname{Im}m_{1n})}\nonumber\\
   &=\frac{1}{n}\sum_{j=1}^n\frac{\xi^2_j \left[(-E-\xi^2_j(E\operatorname{Re}m_{1n}-\eta\operatorname{Im}m_{1n})+\ri\eta+\ri\xi^2_j(E\operatorname{Im}m_{1n}+\eta\operatorname{Re}m_{1n})) \right]}{(-E-\xi^2_j(E\operatorname{Re}m_{1n}-\eta\operatorname{Im}m_{1n}))^2+(-\eta-\xi^2_j(E\operatorname{Im}m_{1n}+\eta\operatorname{Re}m_{1n}))^2}.
%\end{split} 
\end{align}
According to the results in the first two statements, the definition of $\mathbf{D}_{u}$ in (\ref{eq_spectraldomainone}) and the elementary relation that $|m_{1n}(z)|=\rO(1),$ we find that for some constants $C_1, C_2>0,$ when $n$ is sufficiently large 
 \begin{equation*}
 |(-E-\xi^2_j(E\operatorname{Re}m_{1n}-\eta\operatorname{Im}m_{1n})+\ri\eta+\ri\xi^2_j(E\operatorname{Im}m_{1n}+\eta\operatorname{Re}m_{1n}))| \leq C_1 E, 
 \end{equation*}
and
\begin{equation*}
(-E-\xi^2_j(E\operatorname{Re}m_{1n}-\eta\operatorname{Im}m_{1n}))^2+(-\eta-\xi^2_j(E\operatorname{Im}m_{1n}+\eta\operatorname{Re}m_{1n}))^2 \geq C_2 (E+\xi_j^2)^2. 
\end{equation*}
Together with a discussion similar to (\ref{eq: 2.2 beta=2}), we readily see that for some large constant $C>0$
%then when $\eta\le O(1)$, by \eqref{eq: interval for Re} and \eqref{eq: interval for Im} we have,
\begin{equation}\label{eq_rem2n}
\begin{split}
    |m_{2n}(z)|&\leq \frac{C}{n}\sum_{j=1}^n\frac{E \xi^2_j}{(E+\xi_j^2)^2} \leq \frac{C}{n}\sum_{j=1}^n\frac{\xi^2_j}{E-\xi^2_j}=\rO(e_{2}).
\end{split}
\end{equation}
Then together with the definition of $m_n(z),$ we see that 
\[
\begin{split}
   |m_n(z)| \leq \frac{1}{p|z|}\sum_{i=1}^p\frac{1}{|1+\sigma_i m_{2n}(z)|} \le\frac{1}{p |z|}\sum_{i=1}^p\frac{1}{1+\sigma_i|m_{2n}(z)|}
\leq \frac{1}{|z|}=\rO(E^{-1}).
\end{split}
\]

To control the imaginary part, by \eqref{eq: m_2n}, we can write
\[
\operatorname{Im}m_{2n}(z)=\frac{1}{n}\sum_{j=1}^n\frac{\xi^2_j(\eta+\xi^2_j(E\operatorname{Im}m_{1n}+\eta\operatorname{Re}m_{1n}))}{(-E-\xi^2_j(E\operatorname{Re}m_{1n}-\eta\operatorname{Im}m_{1n}))^2+(-\eta-\xi^2_j(E\operatorname{Im}m_{1n}+\eta\operatorname{Re}m_{1n}))^2}.
\]
Combining with (\ref{eq_realcontrol}) and (\ref{eq_imaginarycontrol}), we see that for some constant $C>0$ 
\begin{align}\label{eq_imm2n}
%\begin{split}
   \operatorname{Im}m_{2n}(z)& \leq \frac{C}{n}\sum_{j=1}^n\frac{\xi^2_j(\eta+\xi^2_j\eta)}{(E+\xi^2_j(\rO(1)+\eta^2\times \rO(E^{-1})))^2+(\eta+\xi^2_j\times \rO(1)+\eta\times \rO(E^{-1}))^2} \nonumber \\
   &=\rO\Big(\frac{1}{n}\sum_{j=1}^n\frac{\eta\xi^4_j}{\rO(E^2)} \Big)= \rO\big(\frac{\eta}{E} \big),
%\end{split}
\end{align}
where in the last step we used (\ref{def1}). 
Moreover, using the definition of $m_n(z)$ in (\ref{eq_systemequationsm1m2}), we can write
\[
\operatorname{Im}m_n(z)=\frac{1}{p}\sum_{i=1}^p\frac{\eta+\sigma_i\eta\operatorname{Re}m_{2n}+\sigma_i E\operatorname{Im}m_{2n}}{(E+\sigma_i E\operatorname{Re}m_{2n}-\sigma_i\eta\operatorname{Im}m_{2n})^2+(\eta+\sigma_i\eta\operatorname{Re}m_{2n}+\sigma_i E\operatorname{Im}m_{2n})^2}.
\]
Together with (\ref{eq_rem2n}) and (\ref{eq_imm2n}), we can easily see that 
\begin{equation*}
\operatorname{Im} m_n(z)=\rO\big( \frac{\eta }{E^2} \big).
\end{equation*}
This completes our proof. 
%\end{proof}
\end{proof}

\begin{remark}\label{rem_JHX}
We may observe from the proof of Lemma \ref{lem: basic bounds} that in many cases we can directly write $m_{1n}(\vartheta_1)$ without considering its imaginary part. For example, when (\ref{ass3.1}) holds, for $z_0=\vartheta_1+\ri\eta$, using (\ref{eq:F(m,z)}) and (\ref{eq: def of vartheta_1}), we see that 
\begin{align*}
    \lim_{\eta\downarrow0}\operatorname{Im}m_{1n}(z_0)&=\lim_{\eta\downarrow0}\frac{1}{n}\sum_i\frac{\sigma_i(\eta+\frac{\sigma_i}{n}\sum_{j=1}^n\frac{\xi^4_j\operatorname{Im}m_{1n}(z_0)}{|1+\xi^2_jm_{1n}(z_0)|^2})}{|z_0-\frac{\sigma_i}{n}\sum_{j=1}^n\frac{\xi^2_j}{1+\xi^2_jm_{1n}(z_0)}|^2}\\
    &=\lim_{\eta\downarrow0}\frac{1}{n}\sum_i\frac{\frac{\sigma_i^2}{n}\sum_{j=1}^n\frac{\xi^4_j\operatorname{Im}m_{1n}(z_0)}{|1+\xi^2_jm_{1n}(z_0)|^2}}{|\vartheta_1-\frac{\sigma_i}{n}\sum_{j=1}^n\frac{\xi^2_j}{1+\xi^2_jm_{1n}(z_0)}|^2}\\
    &=\Big(\frac{1}{n}\sum_i\frac{\frac{\sigma_i^2}{n}\sum_{j=1}^n\frac{(\xi^2_{(1)}+d_2)^2\xi^4_j}{|\xi^2_{(1)}+d_2-\xi^2_j|^2}}{|\vartheta_1-\frac{\sigma_i}{n}\sum_{j=1}^n\frac{(\xi^2_{(1)}+d_2)\xi^2_j}{\xi^2_{(1)}+d_2-\xi^2_j}|^2}\Big)\times\lim_{\eta\downarrow0}\operatorname{Im}m_{1n}(z_0).
\end{align*}
Then by a discussion similar to (\ref{eq_imm2n}), using the fact that $\alpha> 2$, we find that for any $\vartheta_1\gtrsim\xi^2_{(1)}$ 
\begin{gather*}
    \lim_{\eta\downarrow0}\operatorname{Im}m_{1n}(z_0)=\mathrm{o}(1)\times \lim_{\eta\downarrow0}\operatorname{Im}m_{1n}(z_0),
\end{gather*}
which holds true if and only if $\lim_{\eta\downarrow0}\operatorname{Im}m_{1n}(z_0)=0$. This shows that $m_{1n}(\vartheta_1)$ is well-defined for $\vartheta_1\gtrsim\xi^2_{(1)}$.

\end{remark}

\subsubsection{Proof of Lemma \ref{localestimate2}}

\begin{proof}[\bf Proof] We first prove the results on the event $\Omega_D$. Let $\Omega_D$ be the event that satisfies (c) of Definition \ref{defn_probset}. According to Lemma \ref{lem_probabilitycontrol}, we find that $\mathbb{P}(\Omega_D)=1-\mathrm{o}(1).$ Now we choose a realization $\{\xi_i^2\} \in \Omega_D$ so that the proofs of parts (a) and (b) of Lemma \ref{localestimate2} are purely deterministic. 

\begin{proof}[\bf Proof of (a)] We start with (\ref{eq_conditionaledgedefinition}). According to (\ref{eq_systemequationsm1m2}), we have that
\begin{gather*}
    m_{1n}(z)=\frac{1}{n}\sum_{i=1}^p\frac{\sigma_i}{-z+\frac{\sigma_i}{n} \sum_{j=1}^n\frac{\xi^2_j}{1+\xi^2_jm_{1n}(z)}}.
\end{gather*}
To characterize the bulk of the spectrum, we take the imaginary part on both sides of the above equation and let $\eta \downarrow 0$ to obtain that 
\begin{gather}
\begin{split}
    \operatorname{Im}m_{1n}(z)
    &=\frac{1}{n}\sum_{i=1}^p \frac{\sigma_i^2 \left(\frac{1}{n}\sum_j\frac{\xi^4_j\operatorname{Im}(m_{1n})}{\operatorname{Re}^2(1+\xi^2_jm_{1n})+\xi^4_j\operatorname{Im}^2(m_{1n})}\right)}{(E-\operatorname{Re}(\frac{\sigma_i}{n}\sum_j\frac{\xi^2_j}{1+\xi^2_jm_{1n}}))^2+\operatorname{Im}^2(\frac{\sigma_i}{n}\sum_j\frac{\xi^2_j}{1+\xi^2_jm_{1n}})}.
\end{split}
\end{gather}
The above equation can be further rewritten as 
%We may rewrite the above equation as
\begin{gather}\label{eq: supp for m1phi}
    0=\operatorname{Im}m_{1n}(z)(1-g (m_{1n},E)),   
\end{gather}
where $g(m_{1n},E)$ is denoted as 
\begin{gather*}
    g(m_{1n},E):=\frac{1}{n}\sum_{i=1}^p\frac{\sigma_i^2 \left(\frac{1}{n}\sum_j\frac{\xi^4_j}{\operatorname{Re}^2(1+\xi^2_jm_{1n})+\xi^4_j\operatorname{Im}^2(m_{1n})}\right)}{(E-\operatorname{Re}(\frac{\sigma_i}{n}\sum_j\frac{\xi^2_j}{1+\xi^2_jm_{1n}}))^2+\operatorname{Im}^2(\frac{\sigma_i}{n}\sum_j\frac{\xi^2_j}{1+\xi^2_jm_{1n}})}.
\end{gather*}
Similarly to the arguments used in \cite{Kwak2021,lee2016extremal}, it is easy to see that for any fixed $\operatorname{Re}m_{1n}<-l^{-1}$ and $E$, $g(m_{1n},E)\rightarrow 1$ when $|\operatorname{Im}m_{1n}|\rightarrow\infty,$ and $g(m_{1n},E)\rightarrow+\infty$ to satisfy (\ref{eq: supp for m1phi}) when $|\operatorname{Im}m_{1n}|\rightarrow 0.$ 

Therefore, by monotonicity, there exists a unique $\operatorname{Im}m_{1n}>0$ such that \eqref{eq: supp for m1phi} holds, which corresponds to the bulk of the spectrum. 

Furthermore, for any fixed $\operatorname{Re} m_{1n} >-l^{-1}$ and fixed  $E$ so that Theorem \ref{lem_solutionsystem} holds, we see that $g(m_{1n},E)$ is monotonically decreasing in terms of $|\operatorname{Im} m_{1n}|.$

Let $E_+$ be defined according to $m_{1n}(E_+)=-l^{-1}.$ In view of (\ref{eq:F(m,z)}) and (\ref{eq_functionFequal}), we have
\begin{equation*}
l^{-1}=\frac{1}{n} \sum_{i=1}^p  \frac{ \sigma_i}{E_+-\sigma_i \widehat{\mathsf{s}}_2}. 
\end{equation*}
Let $\widetilde{\mathsf{s}}_3$ be defined similarly to $\widehat{\mathsf{s}}_3$ in (\ref{eq_finitesample123}) by replacing $\widehat{L}_+$ with $E_+.$ Based on the above arguments and definitions, it is easy to see that  
\[
\begin{split}
    &\sup_{\operatorname{Re}m_{1n}\in(-l^{-1},\infty)} g(m_{1n}, E)=g(-l^{-1}, E_+)=\phi \widetilde{\mathsf{s}}_3.
\end{split}
\]
Assuming that $\phi \widetilde{\mathsf{s}}_3<1,$ we conclude that (\ref{eq: supp for m1phi}) holds only if $\operatorname{Im} m_{1n}(z)=0$, which corresponds to the outside region of the spectrum. This shows that $m_{1n}=-l^{-1}$ is at the right edge of the spectrum and gives the expression of the end point $\widehat{L}_{+}$ as in (\ref{eq_conditionaledgedefinition}). Therefore, $\widehat{L}_+=E_+$ and $\widehat{\mathsf{s}}_3=\widetilde{\mathsf{s}}_3.$ This completes the proof. 

% in \eqref{eq: def of L+}.

\quad Second, the proof of \eqref{eq: concave decay of rho_Q} follows from an argument similar to Lemma 8.4 of \cite{lee2016extremal} that uses the estimate (\ref{ass3.4}), and we omit the details. 

\quad Third, we prove (\ref{eq_closenessequation}). The closeness of $\mathsf{s}_k$ and $\widehat{\mathsf{s}}_k, k=1,2,3,4,$  follows from arguments similar to the third-to-last equation of (\ref{def4}). Now we proceed to the proof of the second equation in (\ref{eq_closenessequation}).
According to the proof of (\ref{eq_conditionaledgedefinition}) and an analogous argument, we found that $m_{1n}(\widehat{L}_+)=-l^{-1}$ and $m_{1n,c}(L_+)=-l^{-1}.$ Together with the definitions of $\mathsf{s}_2, \widehat{\mathsf{s}}_2,$ $m_{2n}$ and $m_{2n,c},$ we find that 
\begin{equation}\label{eq_deterministicrelation}
\mathsf{s}_2=-m_{2n,c}(L_+) L_+, \ \widehat{\mathsf{s}}_2=-m_{2n}(\widehat{L}_+) \widehat{L}_+. 
\end{equation} 
Next, by (\ref{eq_conditionaledgedefinition}) and an analogous argument for $L_+$, we have 
\begin{align}\label{eq_denomimatorcontrol}
  %  \begin{split}
        0&=\frac{1}{n}\sum_{i=1}^p \frac{-l\sigma_i}{(-L_{+}+\sigma_i\mathsf{s}_2)}+\frac{1}{n}\sum_{i=1}^p \frac{l\sigma_i}{(-\widehat{L}_{+}+\sigma_i\widehat{\mathsf{s}}_2)} \nonumber \\
        &=\frac{1}{n}\sum_i\frac{-l\sigma_i}{(-L_{+}+\sigma_i\mathsf{s}_2)}+\frac{1}{n}\sum_i\frac{l\sigma_i}{(-L_{+}+\sigma_i\widehat{\mathsf{s}}_2)}+\frac{1}{n}\sum_i\frac{l\sigma_i(\widehat{L}_{+}-L_{+})}{(-\widehat{L}_{+}+\sigma_i \widehat{\mathsf{s}}_2)(-L_{+}+\sigma_i\widehat{\mathsf{s}}_2)}.
   % \end{split}
\end{align} 
By the first equation of (\ref{eq_closenessequation}), (\ref{eq_deterministicrelation}) and Assumption \ref{assum_additional_techinical}, we can conclude our proof.

\quad Fourth, we work with (\ref{eq_expansionlinear}) and (\ref{eq_a11coro}). Due to similarity, we only prove (\ref{eq_expansionlinear}). According to (\ref{lem_solutionsystem}), we have that 
\begin{gather*}
    m_{1n}(z)=\frac{1}{n}\sum_{i=1}^p \frac{\sigma_i}{-z+\frac{\sigma_i}{n}\sum_{j=1}^n\frac{\xi^2_j}{1+\xi^2_jm_{1n}(z)}}.
\end{gather*}
Consequently, it is easy to see that for $z= \widehat{L}_{+}-\kappa+\ri\eta\in\mathbf{D}_b,$ 
\begin{gather}\label{eq_decompositionmmmmmm}
\begin{split}
    m_{1n}(\widehat{L}_{+})-m_{1n}(z)
    =R_1(\widehat{L}_{+}-z)+R_2(m_{1n}(\widehat{L}_{+})-m_{1n}(z)),
\end{split}
\end{gather}
where we denote 
\begin{gather*}
    R_1:=\frac{1}{n}\sum_{i=1}^p \frac{\sigma_i}{(-\widehat{L}_{+}+\frac{\sigma_i}{n}\sum_j\frac{\xi^2_j}{1+\xi^2_jm_{1n}(\widehat{L}_{+})})(-z+\frac{\sigma_i}{n}\sum_j\frac{\xi^2_j}{1+\xi^2_jm_{1n}(z)})}\\
    R_2:=\frac{1}{n}\sum_{i=1}^p \frac{\frac{\sigma_i^2}{n}\sum_j\frac{\xi^4_j}{(1+\xi^2_jm_{1n}(\widehat{L}_{+}))(1+\xi^2_jm_{1n}(z))}}{(-\widehat{L}_{+}+\frac{\sigma_i}{n}\sum_j\frac{\xi^2_j}{1+\xi^2_jm_{1n}(\widehat{L}_{+})})(-z+\frac{\sigma_i}{n}\sum_j\frac{\xi^2_j}{1+\xi^2_jm_{1n}(z)})}.
\end{gather*}
To study the terms $R_1$ and $R_2,$ we will need the following control whose proof follows from equations (4.24)-(4.28) of \cite{lee2016extremal} 
\begin{equation}\label{eq_boundused}
\frac{1}{n} \sum_{j=1}^n \frac{\xi_j^4}{(1-\xi_j^2 l^{-1})(1+\xi_j^2 m_{1n}(z))}=
\begin{cases}
\rO\left( \log n \right),  & d \geq 2; \\
\rO \left( |l^{-1}+m_{1n}(z)|^{d-2}\log n \right), & 1<d \leq 2.
\end{cases}
\end{equation}
For the denominator of $R_1$, since $z \in \mathbf{D}_b,$ by a discussion similar to (\ref{eq_denomimatorcontrol}), we find that they are bounded from below, so that $R_1=\rO(1).$  Furthermore, since $m_{1n}(\widehat{L}_+)=-l^{-1},$ by a straightforward calculation, using the definition of $\widehat{\mathsf{s}}_2$ in (\ref{eq_finitesample123}) and the control (\ref{eq_boundused}), we observe that 
 \begin{gather}\label{eq_R1control}
    \begin{split}
        R_1&=\widehat{\mathsf{s}}_4-\frac{1}{n}\sum_{i=1}^p \frac{\sigma_i (z-\widehat{L}_{+})+\frac{\sigma_i^2}{n}\sum_j\frac{\xi^4_j(m_{1n}(z)+l^{-1})}{(1-\xi^2_jl^{-1})(1+\xi^2_jm_{1n}(s))}}{(-\widehat{L}_{+}+\sigma_i\widehat{\mathsf{s}}_2)^2(-z+\frac{\sigma_i}{n}\sum_j\frac{\xi^2_j}{1+\xi^2_jm_{1n}(z)})}\\
        &=\widehat{\mathsf{s}}_4+\rO(|z-\widehat{L}_{+}|)+\rO(|m_{1n}(z)+l^{-1}|^{ \min\{d-1,1\}}\log n),
%        &=\frac{ \widehat{\varsigma}_3}{ \widehat{\varsigma}_1}+\rO(|z-\widehat{L}_{+}|)+\rO(|m_{1n}(z)+l^{-1}|^{\min\{d-1,1\}}\log n).
    \end{split}
\end{gather}
where in the second step, we used again an argument similar to (\ref{eq_denomimatorcontrol}). Similarly, for $R_2,$ we find that 
\begin{gather}\label{eq_R2control}
\begin{split}
R_2&=\phi \widehat{\mathsf{s}}_3+\rO(|z-\widehat{L}_{+}|)+\rO(|m_{1n}(z)+l^{-1}|^{\min\{d-1,1\}}\log n).
\end{split}
\end{gather}
Next, we provide a useful deterministic control. Using a discussion similar to \cite[Lemma A.4]{Kwak2021}, we find from (\ref{eq_systemequationsm1m2}) that 
\begin{equation}\label{eq_expansionusefullessorequaltoone}
\frac{1}{n} \sum_{i=1}^p \frac{\sigma_i^2 \frac{1}{n} \sum_{j=1}^n \frac{\xi_j^4}{(1+\xi_j^2 m_{1n}(z))^2} }{|-z+\frac{\sigma_i}{n} \sum_{j=1}^n \frac{\xi_j^2}{1+\xi_j^2 m_{1n}(z)}|^2}=1-\frac{1}{n}\sum_i\frac{\sigma_i\eta/\operatorname{Im}m_{1n}(z)}{|-z+\frac{\sigma_i}{n}\sum_j\frac{\xi^2_j}{1+\xi^2_jm_{1n}(z)}|^2}=1-\eta \frac{|m_{1n}(z)|^2}{\operatorname{Im} m_{1n}(z)}.
\end{equation}
Since $\operatorname{Im} m_{1n}(z)>0$ and $\eta>0$, this implies that
\begin{equation*}
0 \leq 1-\frac{1}{n}\sum_i\frac{\sigma_i\eta/\operatorname{Im}m_{1n}(z)}{|(-z+\frac{\sigma_i}{n}\sum_j\frac{\xi^2_j}{1+\xi^2_jm_{1n}(z)})|^2} < 1. 
\end{equation*} 
Together with the Cauchy-Schwarz inequality, we see that 
\begin{gather*}
    \begin{split}
       |R_2| \leq (\phi\widehat{\mathsf{s}}_3)^{1/2}\left(1-\frac{1}{n}\sum_i\frac{\sigma_i\eta/\operatorname{Im}m_{1n}(z)}{|(-z+\frac{\sigma_i}{n}\sum_j\frac{\xi^2_j}{1+\xi^2_jm_{1n}(z)})|^2}\right)^{1/2} <1,
    \end{split}
\end{gather*} 
where we used the fact $\operatorname{Im}m_{1n}(z)>0$, $\eta>0$, and assumption $\phi\widehat{\mathsf{s}}_3<1$. Using (\ref{eq_decompositionmmmmmm}), we find that $m_{1n}(\widehat{L}_+)-m_{1n}(z) \asymp \widehat{L}_+-z.$ Then we can conclude our proof using (\ref{eq_decompositionmmmmmm}), (\ref{eq_R1control}) and (\ref{eq_R2control}).

\quad Finally, we prove the controls for the imaginary parts. For (\ref{eq_zopointrate11}), the discussion is  similar to that of Lemma 4.5 in \cite{Kwak2021}. According to (\ref{eq_functionFequal}), we see that 
\begin{gather}\label{eq_diudiudiudiudiu}
   \begin{split}
       -m_{1n}(z)&=\frac{1}{n}\frac{\sigma_{1}}{z-\frac{\sigma_{1}}{n}\sum_{j=1}^n\frac{\xi^2_j}{1+\xi^2_jm_{1n}(z)}}+\frac{1}{n}\sum_{i=2}^{p}\frac{\sigma_i}{z-\frac{\sigma_i}{n}\sum_{j=1}^n\frac{\xi^2_j}{1+\xi^2_jm_{1n}(z)}}\\
       &=\rO(\frac{1}{n\eta})+\frac{1}{n}\sum_{i=2}^{p}\frac{\sigma_i}{z-\frac{\sigma_i}{n}\sum_{j=1}^n\frac{\xi^2_j}{1+\xi^2_jm_{1n}(z)}},
   \end{split} 
\end{gather}
where in the step we the fact that $\operatorname{Im} m_{1n}(z) >0$ and the trivial bound that 
\begin{equation}\label{eq_trivialcontroleta}
\frac{1}{n} \Big|\Big(z-\frac{\sigma_{1}}{n}\sum_{j=1}^n\frac{\xi^2_j}{1+\xi^2_jm_{1n}(z)}\Big)^{-1} \Big| \leq n^{-1} \Big(\eta+\frac{\sigma_1}{n} \sum_{j=1}^n \frac{\xi_j^4 \operatorname{Im} m_{1n}(z)}{|1+\xi_j^2 m_{1n}(z)|^2} \Big)^{-1} \leq (n \eta)^{-1}. 
\end{equation}
Taking the imaginary part on both sides of (\ref{eq_diudiudiudiudiu}), we see that for some constant $0<c<1,$
\begin{gather*}
\begin{split}
    \operatorname{Im}m_{1n}(z)&=\frac{1}{n}\sum_{i=2}^{p}\frac{\sigma_i(\eta+\frac{\sigma_i}{n}\sum_{j=1}^n\frac{\xi^4_j\operatorname{Im}m_{1n}(z)}{|1+\xi^2_jm_{1n}(z)|})}{|z-\frac{\sigma_i}{n}\sum_{j=1}^n\frac{\xi^2_j}{1+\xi^2_jm_{1n}(z)}|^2}+\rO(\frac{1}{n\eta})\\
    &=\frac{1}{n}\sum_{i=2}^{p}\frac{\sigma_i\eta}{|z-\frac{\sigma_i}{n}\sum_{j=1}^n\frac{\xi^2_j}{1+\xi^2_jm_{1n}(z)}|^2}+\frac{1}{n}\sum_{i=2}^{p}\frac{\frac{\sigma_i^2}{n}\sum_{j=1}^n\frac{\xi^4_j\operatorname{Im}m_{1n}(z)}{|1+\xi^2_jm_{1n}(z)|}}{|z-\frac{\sigma_i}{n}\sum_{j=1}^n\frac{\xi^2_j}{1+\xi^2_jm_{1n}(z)}|^2}+\rO(\frac{1}{n\eta})\\
    &=\rO(\eta)+\rO(\frac{1}{n\eta})+ c \operatorname{Im} m_{1n}(z),
    \end{split}
\end{gather*}
where in the last step we used \eqref{eq_expansionusefullessorequaltoone} and discussions similar to (\ref{eq_L1bound}) and (\ref{eq_L3BOUND}) below. This concludes the proof. Then we prove (\ref{eq_oneregimeedgecontrol}) and (\ref{eq_zopointrate}) following \cite[Lemma 5.2]{lee2016extremal}. Due to similarity, we focus our analysis on $m_{1n}(z)$ and discuss $m_n(z)$ briefly in the end. In what follows, for notational simplicity, without loss of generality, we assume that on $\Omega_D,$ $\xi_{(i)}^2=\xi_i^2.$ In what follows, we identify $\eta \equiv \eta_0$ till the end of the proof of the lemma.  According to (\ref{eq_systemequationsm1m2})
\begin{align}\label{eq_decompositionl1l2l3}
 &   \operatorname{Im}m_{1n}(z)=\frac{1}{n}\sum_{i=1}^p\frac{\sigma_i(\eta+\frac{\sigma_i}{n}\sum_{j=1}^n\frac{\xi^4_j\operatorname{Im}m_{1n}(z)}{|1+\xi^2_jm_{1n}(z)|^2})}{|z-\frac{\sigma_i}{n}\sum_{j=1}^n\frac{\xi^2_j}{1+\xi^2_jm_{1n}(z)}|^2} \nonumber \\
    &=\frac{1}{n}\sum_{i=1}^p \frac{\sigma_i\eta}{|z-\frac{\sigma_i}{n}\sum_{j=1}^n\frac{\xi^2_j}{1+\xi^2_jm_{1n}(z)}|^2}+\frac{1}{n}\sum_{i=1}^p \frac{\frac{\sigma_i^2}{n}\frac{\xi^4_1\operatorname{Im}m_{1n}(z)}{|1+\xi^2_1m_{1n}(z)|^2}}{|z-\frac{\sigma_i}{n}\sum_{j=1}^n\frac{\xi^2_j}{1+\xi^2_jm_{1n}(z)}|^2}+ \frac{1}{n} \sum_{i=1}^p \frac{\frac{\sigma_i^2}{n}\sum_{j=2}^n\frac{\xi^4_j\operatorname{Im}m_{1n}(z)}{|1+\xi^2_jm_{1n}(z)|^2}}{|z-\frac{\sigma_i}{n}\sum_{j=1}^n\frac{\xi^2_j}{1+\xi^2_jm_{1n}(z)}|^2} \nonumber \\
    &=\mathsf{L}_1+\mathsf{L}_2+\mathsf{L}_3.  
    \end{align}
For the denominator, by the results and arguments in Section \ref{sec_proofpartiboundedlocallaw}, we observe that when $z\in \mathbf{D}_b^\prime$  
\begin{align*}
z-\frac{\sigma_i}{n} \sum_{j=1}^n \frac{\xi_j^2}{1+\xi_j^2 m_{1n}(z)}& = z-\frac{\sigma_i}{n} \sum_{j=2}^n \frac{\xi_j^2}{1+\xi_j^2 m_{1n}(z)}+\frac{\sigma_i}{n} \frac{\xi_1^2}{1+\xi_1^2 m_{1n}(z)} \\
&=z-\frac{\sigma_i}{n} \sum_{j=2}^n \frac{\xi_j^2}{1+\xi_j^2 m_{1n,c}(z)}+\rO((n \eta)^{-1}+n^{-1/2-1/(d+1)}). 
\end{align*} 
Together with Assumption \ref{assum_additional_techinical} and (\ref{def4}), we find that for some small constant $c'>0,$ when $n$ is sufficiently large, 
\begin{equation*}
\left| z-\frac{\sigma_i}{n} \sum_{j=1}^n \frac{\xi_j^2}{1+\xi_j^2 m_{1n}(z)} \right| \geq c'. 
\end{equation*}
This implies that 
\begin{equation}\label{eq_L1bound}
\mathsf{L}_1 \asymp \eta. 
\end{equation}

For $\mathsf{L}_2,$ on the one hand, when $|z-z_0| \geq Cn^{-1/2+3 \epsilon_{\mathsf{d}}},$ by (\ref{eq_a11coro}), we conclude that on $\Omega_D,$ for some constant $C>0$
\begin{equation}\label{eq_L2bound1}
|\mathsf{L}_2| \leq n^{-1/2-3 \epsilon_{\mathsf{d}}} \operatorname{Im} m_{1n}(z).  
\end{equation}
On the other hand, when $z=z_0$ so that $\operatorname{Re} m_{1n}(z)=-\xi_1^2,$ we can rewrite $\mathsf{L}_2$ as 
\begin{equation}\label{eq_L2bound2}
\mathsf{L}_2=\frac{1}{n}\sum_{i=1}^ p\frac{\frac{\sigma_i}{n\operatorname{Im}m_{1n}(z)}}{|z-\frac{\sigma_i}{n}\sum_{j}\frac{\xi^2_j}{1+\xi^2_jm_{1n}(z)}|^2}=\mathrm{O}(\frac{1}{n}(\operatorname{Im}m_{1n}(z))^{-1}). 
\end{equation}

Next, for $\mathsf{L}_3,$ by (\ref{eq_defnmathsfW}), (\ref{def4}) and the results and arguments in Section \ref{sec_proofpartiboundedlocallaw}, using the trivial bound for $\mathsf{L}_2$ such that $|\mathsf{L}_2|=\rO((n\eta)^{-1}),$ we conclude that when $z \in \mathbf{D}_b^{\prime},$ for some constant $0<\mathfrak{c}<1$
\begin{equation}\label{eq_L3BOUND}
|\mathsf{L}_3| \leq \mathfrak{c} \operatorname{Im} m_{1n}(z). 
\end{equation}

\quad Consequently, we find that (\ref{eq_oneregimeedgecontrol}) follows from (\ref{eq_decompositionl1l2l3}), (\ref{eq_L1bound}), (\ref{eq_L2bound1}) and (\ref{eq_L3BOUND}). Moreover, (\ref{eq_zopointrate}) follows from (\ref{eq_decompositionl1l2l3}), (\ref{eq_L1bound}), (\ref{eq_L2bound2}) and (\ref{eq_L3BOUND}) by solving the associated quadratic equation of $\operatorname{Im}m_{1n}(z)$. Finally, we mention that the results for $\operatorname{Im}m_n(z)$ essentially follow from  (\ref{eq_systemequationsm1m2}) 
\begin{gather*}
    \operatorname{Im}m_n(z)=\frac{1}{n}\sum_{i=1}^p \frac{\eta}{|z-\frac{\sigma_i}{n}\sum_j\frac{\xi^2_j}{1+\xi^2_jm_{1n}(z)}|^2}+\frac{1}{n}\sum_{i=1}^p \frac{\frac{1}{n}\sum_{j}\frac{\xi^4_j\operatorname{Im}m_{1n}(z)}{|1+\xi^2_jm_{1n}(z)|^2}}{|z-\frac{\sigma_i}{n}\sum_j\frac{\xi^2_j}{1+\xi^2_jm_{1n}(z)}|^2},
\end{gather*}
with the results of $\operatorname{Im}m_{1n}(z)$. This completes our proof.

\end{proof}

\begin{proof}[\bf Proof of Part (b)] 

For (\ref{eq_boundedfrombelowimportant}), on the one hand, when when $d>1$ and $\phi^{-1}<\widehat{\mathsf{s}}_3,$ the result has been proved in (\ref{rem1nbound}).  On the other hand, when $-1<d \leq 1,$  we employ the proof idea as in the proof of Lemma A.3 of \cite{lee2013local11111} using a continuity argument. Recall (\ref{eq_Fnxyoriginaldefinition}). Denote 
\begin{equation*}
   g(x,y) \equiv \frac{\partial F_n(x,y)}{\partial x}+1=\frac{1}{n}\sum_{i=1}^p\frac{\frac{\sigma_i^2}{n}\sum_{j=1}^n\frac{\xi^4_j}{(1+x\xi^2_j)^2}}{(-y+\frac{\sigma_i}{n}\sum_{j=1}^n\frac{\xi^2_j}{1+x\xi^2_j})^2}.
\end{equation*}
From our assumption that $-1<d \leq 1$ and (\ref{ass3.4}), we find that there exist constants $C,C_0>0$ such that $\mathrm{d}F(x)\geq C(l-x)^{d}\geq C_0(l-x)$ for $x\in(0,l)$. Let $D$ be a sufficiently large constant and choose a sufficiently small constant $0<\epsilon<D^{-1}$, we have that when $n$ is sufficiently large, there exists some constants $C_1, C_2, C_3>0$
    \begin{align*}
            g(-(l+\epsilon)^{-1},\widehat{L}_{+})&=\frac{1}{n}\sum_{i=1}^p\frac{\frac{\sigma_i^2}{n}\sum_{j=1}^n\frac{(l+\epsilon)^2\xi^4_j}{(l+\epsilon-\xi^2_j)^2}}{\left(\widehat{L}_{+}-\frac{\sigma_i}{n}\sum_{j=1}^n\frac{(l+\epsilon)\xi^2_j}{l+\epsilon-\xi^2_j} \right)^2}\\
            &\geq\frac{C_1}{n}\sum_{i=1}^p\frac{\sigma_i^2\int_{l-(D-1)\epsilon}^l\frac{(l+\epsilon)^2x^2}{(l+\epsilon-x)^2}\mathrm{d}F(x)}{(\widehat{L}_{+}-\sigma_i\rO(1))^2} \geq \frac{1}{n}\sum_{i=1}^p\frac{C_2\int_{l-(D-1)\epsilon}^l\frac{(l-x)}{(l+\epsilon-x)^2}\mathrm{d}x}{(\widehat{L}_{+}-\sigma_i\rO(1))^2}\\
            &=\frac{1}{n}\sum_{i=1}^p\frac{C_2\int_{\epsilon}^{D\epsilon}\frac{(t-\epsilon)}{t^2}\mathrm{d}t}{(\widehat{L}_{+}-\sigma_i\rO(1))^2} \geq C_3(\log D-1+\frac{1}{D})>1,
    \end{align*}
    for sufficiently large $D>0.$ Similar arguments apply to $g(-(l-\epsilon),\widehat{L}_+).$ Consequently, by the continuity of $g(x,y)$, we obtain that $ \partial F_n (-l^{-1},\widehat{L}_{+}) /\partial x>0. $ Since $\partial F_n (m_{1n}(\widehat{L}_+),\widehat{L}_{+}) /\partial x=0,$ we can conclude that (\ref{rem1nbound}) still holds. That is, $m_{1n}(\widehat{L}_{+})>-l^{-1}$. This finishes the proof of (\ref{eq_boundedfrombelowimportant}). For (\ref{thm_main_squared_root_bounded}) and (\ref{eq_gammadefinition}), using (\ref{eq_boundedfrombelowimportant}),   by  a discussion similar to (\ref{eq_assumptionequation}), we see that 
\begin{equation}\label{eq_secondmomentbounded}
\frac{\partial^2 F_n(m_{1n}(\widehat{L}_+), \widehat{L}_+)}{\partial x^2} \asymp 1.
\end{equation} 
Armed with this input, the square root behavior of $\rho$ at $\widehat{L}_+$ can be obtained in the same way as Lemma A.1 of \cite{lee2013local11111}. Due to similarity, we omit the details.
\end{proof}

Finally, it is easy to check that we can follow the lines of the proofs of parts (a) and (b) to prove the unconditional results by replacing the related quantities verbatim. 
\end{proof}

\subsection{Control of some bad probability events: proof of Lemma \ref{lem_probabilitycontrol}}\label{sec_appendxi_goodevent}

In this subsection, we prove Lemma \ref{lem_probabilitycontrol} case by case. 
%We first prove Case (a) in Definition \ref{defn_probset}. 

%{\color{blue}
%We first recall the assumption of the unbounded multiplier:
%\begin{assumption}
% We assume that $\xi^2$ has an unbounded support and $\xi^2$ has polynomial decay tail such  that 
%\begin{equation}
%	\mathbb{P}(\xi^2>x)=\frac{L}{x^{\alpha}},\quad x\rightarrow\infty,
%\end{equation}
%	for some $\alpha\in[2,+\infty)$ and some constant $L>0$.
%\end{assumption}
%}

\begin{proof}[\bf Proof of Case (a)] First, the last statement of (\ref{def1}) follows directly from strong law of large number. In fact, the result holds almost surely. 

{

Then,  we prove the second statement of (\ref{def1}). For the upper bound, since $\{\xi_i^2\}$ are independent, we readily see that when $n$ is sufficiently large,  
%{\color{blue}[the use of $\asymp$ is wrong]}
%For any positive constants $\kappa_0$ and $\delta$,
\begin{equation*}
\begin{split}
    \mathbb{P}(\xi_{(1)}^2 \leq C_{\mathsf{p},1}n^{1/\alpha}
    \log n)&=\left(1-\mathbb{P}(\xi^2> C_{\mathsf{p},1}n^{1/\alpha}
    \log n) \right)^n =\left(1-\frac{L}{(C_{\mathsf{p},1} n^{1/\alpha}\log n)^{\alpha}}\right)^n\\
    &=(1-LC_{\mathsf{p},1}^{-\alpha}n^{-1}\log^{-\alpha} n)^n{=\exp\left( -1/(C' \log^{\alpha} n)\right) (1+\mathrm{o}(1))}= 1-\rO(\log^{-\alpha}n) 
\end{split}
\end{equation*}
where in the fourth step we set
%we used the Taylor expansion for $\ln(1-LC^{-\alpha}n^{-1}\log^{-\alpha}n)^n$ with the dominated term being $n\cdot (C^{\prime}n)^{-1}\log^{-\alpha}n$ for 
$C^{\prime}=L^{-1}C_{\mathsf{p},1}^{\alpha}$; and in the last step we used the Taylor expansion for $\exp(x)$ near zero.
% with $x=-(C^{\prime}\log^{\alpha}n)^{-1}$. 
This proves the upper bound. For the lower bound, we can show that  
\begin{align}\label{eq_bbbbboneoneoneone}
%\begin{split}
    \mathbb{P}(\xi^2_{(1)}\geq C_{\mathsf{p},1}^{-1}n^{1/\alpha}\log^{-1}n)& =1-\mathbb{P}(\xi^2_{(1)}< C_{\mathsf{p},1}^{-1}n^{1/\alpha}\log^{-1}n)=1-\big(\mathbb{P}(\xi^2< C_{\mathsf{p},1}^{-1}n^{1/\alpha}\log^{-1}n)\big)^n \nonumber \\
    &=1-\big(1-\mathbb{P}(\xi^2>C_{\mathsf{p},1}^{-1}n^{1/\alpha}\log^{-1}n)\big)^n=1-\left(1-\frac{LC_{\mathsf{p},1}^{\alpha}}{(n^{1/\alpha}\log^{-1}n)^{\alpha}}\right)^n\\
    &=1-\exp(-C^{\prime \prime}\log^{\alpha}n){(1+\mathrm{o}(1))}\geq 1-\mathrm{O}(n^{-C^{\prime \prime}}), \nonumber
%\end{split}
\end{align}
where in the second-to-last step, we used  $C^{\prime \prime}=LC_{\mathsf{p},1}^{\alpha}$. In the last step, we used the fact  that $\exp(-C^{\prime \prime}\log^{\alpha}n)\ll\exp(-C^{\prime \prime}\log n)=n^{-C^{\prime \prime}}$. This concludes the proof of the second statement. 

Next, we prove the first statement. By the joint distribution of the order statistics \cite{david2004order}, $(\xi^2_{(1)},\xi^2_{(2)})$ has the density function 
\begin{align}\label{eq_jointdensity1and2}
    f_{\xi^2_{(1)},\xi^2_{(2)}}(x,y)=n(n-1)(F_{\xi^2}(y))^{n-2}f_{\xi^2}(x)f_{\xi^2}(y),\quad x>y>0,
\end{align}
where $F_{\xi^2}(x)$ and $f_{\xi^2}(x)$ are CDF and PDF of $\xi^2$, respectively. Then, we have
\begin{align*}
    \mathbb{P}(\xi^2_{(1)}-\xi^2_{(2)}\leq\epsilon)=\iint_{0<x-y\le\epsilon}f_{\xi^2_{(1)},\xi^2_{(2)}}(x,y)\mathrm{d}x\mathrm{d}y.
\end{align*}
After the change of variables $u = y$ and $v = x - y$, we obtain that
\begin{align}\label{eq_joint}
    \mathbb{P}(\xi^2_{(1)}-\xi^2_{(2)}\leq\epsilon)=\int_0^{\infty}\left(\int_{0}^{\epsilon}n(n-1)f_{\xi^2}(u+v)f_{\xi^2}(u)\mathrm{d}v\right)(F_{\xi^2}(u))^{n-2}\mathrm{d}u.
\end{align}

Before utilizing (\ref{eq_joint}) for the analysis, we first provide some useful controls. By \eqref{ass3.1} and the fact that $\{\xi^2_i\}$ are independent and continuous random variables, we have that 
\begin{align}\label{eq_hererepeating}
    &\mathbb{P}(\xi^2_{(2)}\geq C_{\mathsf{p},0}^{-1}n^{1/\alpha}\log^{-1}n)=1-\mathbb{P}(\xi^2_{(2)}<C_{\mathsf{p},0}^{-1}n^{1/\alpha}\log^{-1}n)\\
    &=1-\mathbb{P}(\xi^2<C_{\mathsf{p},0}^{-1}n^{1/\alpha}\log^{-1}n)^n-n\mathbb{P}(\xi^2\geq C_{\mathsf{p},0}^{-1}n^{1/\alpha}\log^{-1}n)\mathbb{P}(\xi^2<C_{\mathsf{p},0}^{-1}n^{1/\alpha}\log^{-1}n)^{n-1} \nonumber\\
    &\ge1-\big(1-\mathbb{P}(\xi^2>C_{\mathsf{p},0}^{-1}n^{1/\alpha}\log^{-1}n)\big)^n-n(1-\mathbb{P}(\xi^2>C_{\mathsf{p},0}^{-1}n^{1/\alpha}\log^{-1}n))^{n-1} \nonumber \\
    &=1-(1-\frac{LC_{\mathsf{p},0}^{\alpha}}{(n^{1/\alpha}\log^{-1}n)^{\alpha}})^n-n(1-\frac{LC_{\mathsf{p},0}^{\alpha}}{(n^{1/\alpha}\log^{-1}n)^{\alpha}})^{n-1} \nonumber \\
    &=1-{(\exp(-LC_{\mathsf{p},0}^{\alpha}\log^{\alpha}n)+n\exp(-LC_{\mathsf{p},0}^{\alpha}\log^{\alpha}n))(1+\mathrm{o}(1))} \nonumber \\
   & =1-\rO\left( n\exp(-LC_{\mathsf{p},0}^{\alpha}\log^{\alpha}n) \right). \nonumber
\end{align}

%where we estimated $\ln(1-LC^{\alpha}n^{-1}\log^{\alpha}n)^n$ at $-LC^{\alpha}\log^{\alpha}n$ using the Taylor expansion.
%It indicates that the integral region of $u$ in (\ref{eq_joint}) should be greater than $C^{-1}n^{1/\alpha}\log^{-1}n$ with probability at least $1-n\exp(-LC^{\alpha}\log^{\alpha}n)$.

With the above control, we proceed to study (\ref{eq_joint}).  By the law of total probability \cite{ross1998first}, we have
\begin{align}\label{eq_finalreponseoneone}
    \mathbb{P}(\xi^2_{(1)}-\xi^2_{(2)}\leq\epsilon)&=\mathbb{P}(\xi^2_{(1)}-\xi^2_{(2)}\leq \epsilon|\xi^2_{(2)}\geq C_{\mathsf{p},0}^{-1}n^{1/\alpha}\log^{-1}n)\cdot\mathbb{P}(\xi^2_{(2)}\geq C_{\mathsf{p},0}^{-1}n^{1/\alpha}\log^{-1}n) \nonumber\\
    &+\mathbb{P}(\xi^2_{(1)}-\xi^2_{(2)}\leq\epsilon|\xi^2_{(2)}< C_{\mathsf{p},0}^{-1}n^{1/\alpha}\log^{-1}n)\cdot\mathbb{P}(\xi^2_{(2)}< C_{\mathsf{p},0}^{-1}n^{1/\alpha}\log^{-1}n)\\
=\mathbb{P}(\xi^2_{(1)}-\xi^2_{(2)}\leq\epsilon &|\xi^2_{(2)}\geq C_{\mathsf{p},0}^{-1}n^{1/\alpha}\log^{-1}n)\cdot(1-\rO \left(n\exp(-LC_{\mathsf{p},0}^{\alpha}\log^{\alpha}n) \right))+\mathrm{O}(n\exp(-LC_{\mathsf{p},0}^{\alpha}\log^{\alpha}n)) \nonumber \\
=\mathbb{P}(\xi^2_{(1)}-\xi^2_{(2)}\leq \epsilon &|\xi^2_{(2)}\geq C_{\mathsf{p},0}^{-1}n^{1/\alpha}\log^{-1}n)+\mathrm{O}(n\exp(-LC_{\mathsf{p},0}^{\alpha}\log^{\alpha}n)). \nonumber
\end{align}
Then, it is sufficient to calculate $\mathbb{P}(\xi^2_{(1)}-\xi^2_{(2)}\leq \epsilon |\xi^2_{(2)}\geq C_{\mathsf{p},0}^{-1}n^{1/\alpha}\log^{-1}n)$. For notional simplicity, we denote $$a_n:=C_{\mathsf{p},0}^{-1}n^{1/\alpha}\log^{-1}n.$$ Consequently, we can write
\begin{align}\label{eq_gapstatistics_prob1}
    \mathbb{P}(\xi^2_{(1)}-\xi^2_{(2)}\leq \epsilon & |\xi^2_{(2)}\geq a_n) \nonumber \\
  &  =\int_{a_n}^{\infty}\big(\int_{0}^{\epsilon}n(n-1)f_{\xi^2}(u+v)f_{\xi^2}(u)\mathrm{d}v\big)(F_{\xi^2}(u))^{n-2}\mathrm{d}u.
\end{align}
To prove our result, we now choose $\delta:=\log^{-c_{\mathsf{p},0}}n$ and set
 \begin{equation*}
 \epsilon=a_n\delta.
 \end{equation*}
In what follows, in terms of (\ref{eq_gapstatistics_prob1}), we shall only focus on the case when $u \gg \epsilon,$ where we have that  
\begin{align}\label{eq_calculatehere}
    &\int_0^{\epsilon}f_{\xi^2}(u+v)\mathrm{d}v=\int_u^{u+\epsilon}f_{\xi^2}(x)\mathrm{d}x=F_{\xi^2}(u+\epsilon)-F_{\xi^2}(u) \nonumber \\
    &=Lu^{-\alpha}-L(u+\epsilon)^{-\alpha}=Lu^{-\alpha}(1-(\frac{u+\epsilon}{u})^{-\alpha})=Lu^{-\alpha}\big(1-(1+\frac{\epsilon}{u})^{-\alpha}\big) \nonumber \\
    &= Lu^{-\alpha}\big(1-(1-\alpha\frac{\epsilon}{u})(1+\mathrm{o}(1))\big)=L\alpha\epsilon u^{-\alpha-1}(1+\mathrm{o}(1))=\epsilon f_{\xi^2}(u){(1+\mathrm{o}(1))},
\end{align}
where we used Taylor expansion to approximate {$(1+\epsilon/u)^{-\alpha}=1-(\alpha\epsilon/u)(1+\mathrm{o}(1))$} for $u\gg\epsilon$. Moreover, by the assumption of \eqref{ass3.1}, we have
\begin{align*}
    1-F_{\xi^2}(u)=Lu^{-\alpha},\quad f_{\xi^2}(u)=L\alpha u^{-\alpha-1},\quad (F_{\xi^2}(u))^{n-2}=(1-Lu^{-\alpha})^{n-2}= \exp(-nLu^{-\alpha})(1+\mathrm{o}(1)),
\end{align*}
 where the last approximation holds for sufficiently large $u \gg 1$. 
Inserting all the above into  (\ref{eq_gapstatistics_prob1}), we have that
% On the other hand, taking for some $\delta\downarrow0$, we can calculate the inner integral of \eqref{eq_gapstatistics_prob1} as
\begin{align*}
    &\mathbb{P}(\xi^2_{(1)}-\xi^2_{(2)}\leq \epsilon|\xi^2_{(2)}\geq a_n) \leq n^2\epsilon\int_{a_n}^{\infty}f^2_{\xi^2}(u)(F_{\xi^2}(u))^{n-2}\mathrm{d}u\times(1+\mathrm{o}(1)) \\
  &  = n^2(L\alpha)^2\epsilon\int_{a_n}^{\infty}u^{-2\alpha-2}\exp(-nLu^{-\alpha})\mathrm{d}u\times(1+\mathrm{o}(1)).
\end{align*}
Substituting $s=nLu^{-\alpha}$, we have $u=n^{1/\alpha}L^{1/\alpha}/s^{1/\alpha}$ and $\mathrm{d}u=-(1/\alpha)(nL)^{1/\alpha}s^{-1/\alpha-1}\mathrm{d}s$. It follows that when $n$ is sufficiently large 
{
\begin{align*}
    \mathbb{P}(\xi^2_{(1)}-\xi^2_{(2)}\leq \epsilon|\xi^2_{(2)}\geq a_n) & =\alpha(nL)^{-1/\alpha}\epsilon\int_{0}^{nL/a_n^{\alpha}}s^{1+1/\alpha}\exp(-s)\mathrm{d}s\times(1+\mathrm{o}(1))\\
 &   \leq \alpha(nL)^{-1/\alpha}\epsilon\int_{0}^{\infty}s^{1+1/\alpha}\exp(-s)\mathrm{d}s\times(1+\mathrm{o}(1)).
\end{align*}
}
Note the remaining integral is of the form of the Gamma function that
\begin{align*}
    \int_0^{\infty}s^{1+1/\alpha}\exp(-s)\mathrm{d}s=\Gamma(2+\frac{1}{\alpha}).
\end{align*}
Consequently, we have that for some constant $C'>0$
\begin{align}\label{eq_gapstatistics_prob2}
    \mathbb{P}(\xi^2_{(1)}-\xi^2_{(2)}\leq \epsilon|\xi^2_{(2)}\geq a_n) & \leq C' \alpha L^{-1/\alpha}\Gamma(2+\frac{1}{\alpha})n^{-1/\alpha}\epsilon \nonumber \\
 &   = C' \alpha L^{-1/\alpha}\Gamma(2+\frac{1}{\alpha})n^{-1/\alpha}a_n\delta=\mathrm{O}(\log^{-1-c_{\mathsf{p},0}}).
\end{align}
Together with (\ref{eq_finalreponseoneone}), we can therefore conclude our proof. 

For the third statement, due to similarity, we focus on the case $i = K := \lfloor \log n \rfloor$, which is in fact the most challenging scenario. We will explain how to handle all cases $1 \leq i \leq K$ at the end of the proof. We start with the preparation on some controls for $\xi^2_{(K)}.$ 
%
%Recall from the second statement of equation (\ref{def1}) (previously (A.25)) that for some constant $C>1$ and $D>0$ 
%\begin{equation}\label{eq_1stuseful}
%\mathbb{P}\left( \xi_{(1)}^2 \leq C n^{1/\alpha} \log n \right)=1-\mathrm{O}\left( \log^{-D}n \right). 
%\end{equation}
For some $C_K>C_{\mathsf{p},1}$ in the second statement, we have that 
\begin{align*}
     &\mathbb{P}(\xi^2_{(K)}\geq C_K^{-1} n^{1/\alpha}\log^{-1}n)\\
    &=\sum_{k=K}^n\binom{n}{k}\Big(\mathbb{P}(\xi^2\geq C_K^{-1}n^{1/\alpha}\log^{-1}n)\Big)^k\Big(1-\mathbb{P}(\xi^2\geq C_K^{-1}n^{1/\alpha}\log^{-1}n)\Big)^{n-k}.
\end{align*}
Denote 
\begin{equation*}
S_n=\left|\{1 \leq j \leq n: \xi_j^2 \geq C_K^{-1} n^{1/\alpha}\log^{-1}n \} \right|. 
\end{equation*}
It is clear that $S_n$ follows a Binomial distribution $\mathsf{B}(n,p_n)$ with $p_n:=\mathbb{P}(\xi^2\geq C_K^{-1}n^{1/\alpha}\log^{-1}n)$. According to \eqref{ass3.1}, for sufficiently large $n,$ we have 
\begin{align*}
    p_n= L C^{\alpha}_K n^{-1}\log^{\alpha}n, \ \ \text{and} \ \ \mu_n:=np_n \gg K.
\end{align*}
Therefore, the Chernoff left-tail bound yields that  
\begin{align*}
    \mathbb{P}(S_n<K) = \mathbb{P}\left(S_n < \frac{K}{\mu_n}\mu_n\right) \leq \exp\left(-\frac{\mu_n}{2}\left(1-\frac{K}{\mu_n}\right)^2\right) \sim \exp\left(-\frac{1}{2}\log^{\alpha}n\right).
\end{align*}

Consequently, we have that
\begin{align*}
    \mathbb{P}(\xi^2_{(K)}\geq C_K^{-1} n^{1/\alpha}\log^{-1}n) = \mathbb{P}(S_n\geq K) = 1 - \mathrm{O}\left(\exp\left(-\frac{1}{2}\log^{\alpha}n\right)\right).
\end{align*}

On the other hand, as $\xi^2_{(K)}\leq \xi^2_{(1)}$, together with the second statement, we have that 
\begin{align*}
    \mathbb{P}(\xi^2_{(K)}\leq C_K n^{1/\alpha}\log n)=1-\mathrm{O}(\log^{-D}n).
\end{align*}
Combining the above arguments, we see that
\begin{align}\label{eq_b12134567890123}
    \mathbb{P}(C_K^{-1}n^{1/\alpha}\log^{-1}n\leq \xi^2_{(K)}\leq C_K n^{1/\alpha}\log n)=1-\mathrm{O}(\log^{-D}n).
\end{align} 

With the above preparation, we now proceed to conclude the proof. Denote $c_{n,-}:=C_K^{-1}n^{1/\alpha}\log^{-1}n$, $c_{n,+}:=C_K n^{1/\alpha}\log n$ and
\begin{align*}
    I_{c_n}:=[c_{n,-},c_{n,+}].
\end{align*}
For notional simplicity, till the end of the proof in this comment, we denote 
\begin{equation*}
\epsilon:=C_{\mathsf{p},2}^{-1}n^{\epsilon_{\mathsf{p},1}}\log^{-1}n.
\end{equation*} 
We have that 
\begin{align}\label{eq_important}
     \mathbb{P}(\xi^2_{(K-1)}-\xi^2_{(K)}\leq\epsilon)&=\int_{I_{c_n}}\int_{0<x-y\leq\epsilon}f_{\xi^2_{(K-1)},\xi^2_{(K)}}(x,y)\mathrm{d}x\mathrm{d}y+\int_{I_{c_n}^c}\int_{0<x-y\leq\epsilon}f_{\xi^2_{(K-1)},\xi^2_{(K)}}(x,y)\mathrm{d}x\mathrm{d}y \nonumber \\
     &=\int_{I_{c_n}}\int_{0<x-y\leq\epsilon}f_{\xi^2_{(K-1)},\xi^2_{(K)}}(x,y)\mathrm{d}x\mathrm{d}y+\mathrm{O}(\log^{-D}n),
\end{align}
where we used (\ref{eq_b12134567890123}) to obtain
\begin{align*}
    &\int_{I_{c_n}^c}\int_{0<x-y\leq \epsilon}f_{\xi^2_{(K-1)},\xi^2_{(K)}}(x,y)\mathrm{d}x\mathrm{d}y\leq\int_{I_{c_n}^c}\int_{y<x}f_{\xi^2_{(K-1)},\xi^2_{(K)}}(x,y)\mathrm{d}x\mathrm{d}y\\
    &=\mathbb{P}(\xi^2_{(K)}\notin I_{c_n})=\mathrm{O}(\log^{-D}n).
\end{align*}

For the first term in (\ref{eq_important}), using the explicit density function for $(\xi^2_{(K-1)}, \xi^2_{(K)}),$ we have that 
\begin{align*}
    &\int_{I_{c_n}}\int_{0<x-y\leq \epsilon}f_{\xi^2_{(K-1)},\xi^2_{(K)}}(x,y)\mathrm{d}x\mathrm{d}y\\
    &=\int_{I_{c_n}}\int_{0<x-y\leq \epsilon}\frac{n!}{(K-2)!(n-K)!}(1-F_{\xi^2}(x))^{K-2}f_{\xi^2}(x)f_{\xi^2}(y)(F_{\xi^2}(y))^{n-K}\mathrm{d}x\mathrm{d}y.
\end{align*}
Using the change of variables of $u=y$ and $v=x-y$, we can further obtain
\begin{align*}
    &\int_{I_{c_n}}\int_{0<x-y\leq \epsilon}f_{\xi^2_{(K-1)},\xi^2_{(K)}}(x,y)\mathrm{d}x\mathrm{d}y\\
    &=\int_{c_{n,-}}^{c_{n,+}}\int_0^{\epsilon}\frac{n!}{(K-2)!(n-K)!}(1-F_{\xi^2}(u+v))^{K-2}f_{\xi^2}(u+v)f_{\xi^2}(u)(F_{\xi^2}(u))^{n-K}\mathrm{d}v\mathrm{d}u.
\end{align*}
Since for $u\gg 1$, from (\ref{ass3.1}), we have
\begin{align}\label{eq_changchangechange}
    1-F_{\xi^2}(u)=Lu^{-\alpha},\quad f_{\xi^2}(u)=L\alpha u^{-\alpha-1},
\end{align}
which are decreasing functions of $u$. This leads to 
\begin{align*}
    &\int_{I_{c_n}}\int_{0<x-y\leq \epsilon}f_{\xi^2_{(K-1)},\xi^2_{(K)}}(x,y)\mathrm{d}x\mathrm{d}y\\
    &=\int_{c_{n,-}}^{c_{n,+}}\int_0^{\epsilon}\frac{n!}{(K-2)!(n-K)!}(1-F_{\xi^2}(u+v))^{K-2}f_{\xi^2}(u+v)f_{\xi^2}(u)(F_{\xi^2}(u))^{n-K}\mathrm{d}v\mathrm{d}u\\
    &\leq \int_{c_{n,-}}^{c_{n,+}}\int_0^{\epsilon}\frac{n!}{(K-2)!(n-K)!}(1-F_{\xi^2}(u))^{K-2}f^2_{\xi^2}(u)(F_{\xi^2}(u))^{n-K}\mathrm{d}v\mathrm{d}u\\
    &=\int_{c_{n,-}}^{c_{n,+}}\frac{\epsilon n!}{(K-2)!(n-K)!}(1-F_{\xi^2}(u))^{K-2}f^2_{\xi^2}(u)(F_{\xi^2}(u))^{n-K}\mathrm{d}u.
\end{align*}

By further substituting $t=1-F_{\xi^2}(u)$ and $\mathrm{d}t=-f_{\xi^2}(u)\mathrm{d}u,$ we have that 
\begin{align}\label{eq_hahahadueduedue}
    &\int_{I_{c_n}}\int_{0<x-y\leq \epsilon}f_{\xi^2_{(K-1)},\xi^2_{(K)}}(x,y)\mathrm{d}x\mathrm{d}y \nonumber \\
    &\leq\int_{Lc_{n,+}^{-\alpha}}^{Lc_{n,-}^{-\alpha}}\frac{\epsilon n!}{(K-2)!(n-K)!}t^{K-2}(1-t)^{n-K}f_{\xi^2}(F^{-1}_{\xi^2}(1-t))\mathrm{d}t.
\end{align}
Note that by construction and (\ref{eq_changchangechange}), we have 
\begin{equation*}
F_{\xi^2}^{-1}(1-t)=F^{-1}_{\xi^2}(F_{\xi^2}(u))=u=(L/t)^{1/\alpha}.
\end{equation*}
Moreover, we readily see that $Lc_{n,+}^{-\alpha}=LC_k^{-\alpha}n^{-1}\log^{-\alpha}n\ll 1$ and $Lc_{n,-}^{-\alpha}=LC_k^{\alpha}n^{-1}\log^{\alpha}n\ll 1.$ The above yield that for
$t\in[Lc_{n,+}^{-\alpha},Lc_{n,-}^{-\alpha}],$ when $n$ is large enough
\begin{align*}
    f_{\xi^2}(F^{-1}_{\xi^2}(1-t))=f_{\xi^2}((L/t)^{\alpha})=L^{-1/\alpha}\alpha t^{1/\alpha+1}.
\end{align*}
Inserting the above back into (\ref{eq_hahahadueduedue}), we obtain that 
\begin{align*}
    &\int_{I_{c_n}}\int_{0<x-y\leq \epsilon}f_{\xi^2_{(K-1)},\xi^2_{(K)}}(x,y)\mathrm{d}x\mathrm{d}y\\
    &\leq\int_{Lc_{n,+}^{-\alpha}}^{Lc_{n,-}^{-\alpha}}\frac{\epsilon n!}{(K-2)!(n-K)!}L^{-1/\alpha}\alpha t^{K-1+1/\alpha}(1-t)^{n-K}\mathrm{d}t \\
    & \leq\int_{0}^{1}\frac{\epsilon n!}{(K-2)!(n-K)!}L^{-1/\alpha}\alpha t^{K-1+1/\alpha}(1-t)^{n-K}\mathrm{d}t \\
    &=L^{-1/\alpha}\alpha \epsilon \frac{n!}{(K-2)!(n-K)!}\frac{\Gamma(K+1/\alpha)\Gamma(n-K+1)}{\Gamma(n+1+1/\alpha)},
\end{align*}
where we used the fact that in the last but two step, the integral of $t$ is exactly the Beta function $B(x,y)=\int_0^1t^{x-1}(1-t)^{y-1}\mathrm{d}t$ with $x=K+1/\alpha$ and $y=n-K+1.$  

Recall $\Gamma(x+a)/\Gamma(x)\asymp x^{a}$ for $x\rightarrow\infty$. Moreover, we have that
\begin{align*}
    \frac{n!}{\Gamma(n+1+1/\alpha)}&=\frac{\Gamma(n+1)}{\Gamma(n+1+1/\alpha)}\asymp (n+1)^{-1/\alpha}\asymp n^{-1/\alpha},\\
    \frac{\Gamma(K+1/\alpha)}{(K-2)!}&=\frac{\Gamma(K+1/\alpha)}{\Gamma(K-1)}\asymp (K-1)^{1+1/\alpha}\asymp K^{1+1/\alpha},\\
    \frac{\Gamma(n-K+1)}{(n-K)!}&=\frac{(n-K)!}{(n-K)!}=1.
\end{align*}
Combining the above discussions, we obtain that 
\begin{align*}
    &\int_{I_{c_n}}\int_{0<x-y\leq \epsilon}f_{\xi^2_{(K-1)},\xi^2_{(K)}}(x,y)\mathrm{d}x\mathrm{d}y\leq L^{-1/\alpha}\alpha \epsilon n^{-1/\alpha}K^{1+1/\alpha}=\mathrm{O}(n^{\epsilon_{\mathsf{p},1}-1/\alpha}\log^{1/\alpha}n).
\end{align*}
Together with (\ref{eq_important}), we readily complete the proof. For the lower indices $1 \leq i < \lfloor \log n \rfloor$, the proof proceeds similarly by establishing an analogous result to (\ref{eq_b12134567890123}) and using the corresponding explicit joint density function for $(\xi^2_{(i-1)}, \xi^2_{(i)})$ given in (\ref{eq_important}).}
% Recall that $a_n=C^{-1}n^{1/\alpha}\log^{-1}n$, we may choose $\delta=\log^{-c}n$ for some constant $c>0$. Then \eqref{eq_gapstatistics_prob2} reads
% \begin{align*}
%    \mathbb{P}(\xi^2_{(1)}-\xi^2_{(2)}\le\epsilon|\xi^2_{(2)}\ge a_n)\le C_1\log^{-1-c}n,
%\end{align*}
%where we denote $C_1=C^{-1}\alpha L^{-1/\alpha}\Gamma(2+1/\alpha)$.
 
%Combining the above calculation and notice that $\epsilon=a_n\delta=C^{-1}n^{1/\alpha}\log^{-1-c}n$, we finally obtain
%\begin{align*}
%    &\mathbb{P}(\xi^2_{(1)}-\xi^2_{(2)}\ge \epsilon)=
%    \mathbb{P}(\xi^2_{(1)}-\xi^2_{(2)}\ge C^{-1}n^{1/\alpha}\log^{-c_0}n)=1-\mathbb{P}(\xi^2_{(1)}-\xi^2_{(2)}< C^{-1}n^{1/\alpha}\log^{-c_0}n)\\
%    &=1-\mathbb{P}(\xi^2_{(1)}-\xi^2_{(2)}< C^{-1}n^{1/\alpha}\log^{-c_0}n|\xi^2_{(2)}\ge a_n)-C^{\prime}n\exp(-LC^{\alpha}\log^{\alpha}n)\\
%    &\ge 1-C_1\log^{-1-c}n-C^{\prime}n\exp(-LC^{\alpha}\log^{\alpha}n),
%\end{align*}
%for some constant $C^{\prime}>0$ and we denote $c_0=1+c$. This concludes the first statement.

Finally, we justify the fourth statement. For the constant $C_{\mathsf{p},1}>0$ in the second statement, we suppose a positive constant $C_{\mathsf{p},3}$ such that $C_{\mathsf{p},3}>C_{\mathsf{p},1}$ and $1/2<b_{\mathsf{p}}\leq 1$. Then, we have
\begin{align*}
%\begin{split}
   \mathbb{P}( & \xi^2_{(1)}-\xi^2_{(\lceil n^{b_{\mathsf{p}}}\rceil)}<C_{\mathsf{p},3}^{-1}n^{1/\alpha}\log^{-1}n)=\mathbb{P}(\xi^2_{(\lceil n^{b_{\mathsf{p}}}\rceil)}>\xi^2_{(1)}-C_{\mathsf{p},3}^{-1}n^{1/\alpha}\log^{-1}n)\\
   &=\mathbb{P}(\xi^2_{(\lceil n^{b_{\mathsf{p}}}\rceil)}>\xi^2_{(1)}-C_{\mathsf{p},3}^{-1}n^{1/\alpha}\log^{-1}n|\xi^2_{(1)}\geq C_{\mathsf{p},1}^{-1}n^{1/\alpha}\log^{-1}n)\cdot\mathbb{P}(\xi^2_{(1)}\geq C_{\mathsf{p},1}^{-1}n^{1/\alpha}\log^{-1}n)\\
   &+\mathbb{P}(\xi^2_{(\lceil n^{b_{\mathsf{p}}}\rceil)}>\xi^2_{(1)}-C_{\mathsf{p},3}^{-1}n^{1/\alpha}\log^{-1}n|\xi^2_{(1)}< C_{\mathsf{p},1}^{-1}n^{1/\alpha}\log^{-1}n)\cdot\mathbb{P}(\xi^2_{(1)}< C_{\mathsf{p},1}^{-1}n^{1/\alpha}\log^{-1}n)\\
   &=\mathbb{P}(\xi^2_{(\lceil n^{b_{\mathsf{p}}}\rceil)}>\xi^2_{(1)}-C_{\mathsf{p},3}^{-1}n^{1/\alpha}\log^{-1}n|\xi^2_{(1)}\geq C_{\mathsf{p},1}^{-1}n^{1/\alpha}\log^{-1}n)\cdot(1-\exp(-LC_{\mathsf{p},1}^{\alpha}\log^{\alpha}n))\\
   &+\mathrm{O}(\exp(-LC_{\mathsf{p},1}^{\alpha}\log^{\alpha}n))\\
   &=\mathbb{P}(\{\xi^2_{(\lceil n^{b_{\mathsf{p}}}\rceil)}>\xi^2_{(1)}-C_{\mathsf{p},3}^{-1}n^{1/\alpha}\log^{-1}n\}\cap\{\xi^2_{(1)}\geq C_{\mathsf{p},1}^{-1}n^{1/\alpha}\log^{-1}n\})+\mathrm{O}(\exp(-LC_{\mathsf{p},1}^{\alpha}\log^{\alpha}n))\\
   &=\mathbb{P}(\xi^2_{(\lceil n^{b_{\mathsf{p}}}\rceil)}\geq(C_{\mathsf{p},1}^{-1}-C_{\mathsf{p},3}^{-1})n^{1/\alpha}\log^{-1}n)+\mathrm{O}(\exp(-LC_{\mathsf{p},1}^{\alpha}\log^{\alpha}n))\\
   &=\sum_{k=\lceil n^{b_{\mathsf{p}}}\rceil}^{n}\binom{n}{k}\left[\mathbb{P}(\xi^2\geq C^{\prime}n^{1/\alpha}\log^{-1}n)\right]^k \left[\mathbb{P}(\xi^2\leq C^{\prime}n^{1/\alpha}\log^{-1}n)\right]^{n-k}+\mathrm{O}(\exp(-LC_{\mathsf{p},1}^{\alpha}\log^{\alpha}n))\\
   &=\sum_{k=\lceil n^{b_{\mathsf{p}}}\rceil}^{n}\binom{n}{k}\left(\frac{L\log^{\alpha}n}{(C^{\prime})^{\alpha}n} \right)^k \left(1-\frac{L\log^{\alpha}n}{(C^{\prime})^{\alpha}n}\right)^{n-k}+\mathrm{O}(\exp(-LC_{\mathsf{p},1}^{\alpha}\log^{\alpha}n))\\
   &\leq\sum_{k=\lceil n^{b_{\mathsf{p}}}\rceil}^{n}\left(\frac{en}{k} \right)^k \left(\frac{L\log^{\alpha}n}{(C^{\prime})^{\alpha}n}\right)^k \left(1-\frac{L\log^{\alpha}n}{(C^{\prime})^{\alpha}n} \right)^{n-k} +\mathrm{O}(\exp(-LC_{\mathsf{p},1}^{\alpha}\log^{\alpha}n))\\
   &=\sum_{k=\lceil n^{b_{\mathsf{p}}}\rceil }^{n}\left(\frac{eL\log^{\alpha}n}{(C^{\prime})^{\alpha}k}\right)^k \exp(-C^{\prime\prime}\log^{\alpha}n(1-k/n))(1+\mathrm{o}(1))+\mathrm{O}(\exp(-LC_{\mathsf{p},1}^{\alpha}\log^{\alpha}n))\\
   &\leq n\left(\frac{eL\log^{\alpha}n}{(C^{\prime})^{\alpha}\lceil n^{b_{\mathsf{p}}}\rceil}\right)^{\lceil n^{b_{\mathsf{p}}}\rceil}(1+\mathrm{o}(1))+\mathrm{O}(\exp(-LC_{\mathsf{p},1}^{\alpha}\log^{\alpha}n))=\mathrm{O}(1/n^{C_1}), 
%\end{split}
\end{align*}
where in the second step we used the law of total probability, in the third step we used the results in (\ref{eq_bbbbboneoneoneone}), in the third-to-last step we used the Stirling's formula such that $k!\geq (k/e)^k$, in the last step we uniformly bounded $\exp(-C^{\prime\prime}\log^{\alpha}n(1-k/n))\leq 1$ for $\lceil n^{b_{\mathsf{p}}}\rceil \leq k\leq n$. Here we denote $C^{\prime}=C_{\mathsf{p},1}^{-1}-C_{\mathsf{p},3}^{-1}$ and $C^{\prime\prime},C_1>0$ are some constants. This concludes the fourth statement.

\end{proof}

Then we prove Case (b) of Definition \ref{defn_probset}. 
\begin{proof}[\bf Proof of Case (b)] 

{

Recall that we assume
\begin{align}\label{eq_def_exponentialdecaytail}
    \mathbb{P}(\xi^2>x)=Le^{-tx^{\beta}}, \ x \uparrow \infty,
\end{align}
for some constants $\beta,t, L>0.$

We start with the proof of the second statement that
\begin{align}\label{eq_firstsecond}
    \mathbb{P}(C_{\mathsf{e},1}^{-1} \log^{1/\beta} n\leq\xi^2_{(1)}\leq C_{\mathsf{e},1} \log^{1/\beta} n)=1-\mathrm{O}(n^{-c}),
\end{align}
for $C_{\mathsf{e},1}>\max\{(\max(t,t^{-1}))^{1/\beta},1\}$ and some constant $c>0$. For the upper bound, Since $C_{\mathsf{e},1}^{\beta} t>1,$ using the i.i.d. assumption of $\{\xi^2_i\}$, we have from (\ref{eq_def_exponentialdecaytail}) that
\[
\begin{split}
    \mathbb{P}(\xi^2_{(1)}\leq C_{\mathsf{e},1} \log^{1/\beta} n)&=(1-\mathbb{P}(\xi^2> C_{\mathsf{e},1}\log^{1/\beta} n))^n= \left(1-\frac{L}{e^{tC_{\mathsf{e},1}^{\beta}\log n}}\right)^n\\
    &=\left(1-\frac{L}{n^{tC_{\mathsf{e},1}^{\beta}}} \right)^n= \exp(-L/n^{{tC_{\mathsf{e},1}^{\beta}}-1})(1+\mathrm{o}(1))=1-\rO(n^{-({tC_{\mathsf{e},1}^{\beta}}-1)}),
\end{split}
\]
where we used an argument similar to the proof of the second statement of (\ref{def1}).
%for some constant $\mathsf{C}_1>0$, where we used the Taylor expansion of $\ln(1-C/n^{tC^{\beta}_t})^n$ in the fourth step and we used Taylor expansion for $\exp(-C/n^{{tC^{\beta}_t}-1})$ in the last step.

 For the lower bound, based on the above discussions, we have that for any large constant $D>0,$ when $n$ is sufficiently large
\begin{align*}
\mathbb{P}(\xi^2_{(1)} \leq C_{\mathsf{e},1}^{-1} \log^{1/\beta} n)=\left(1-\frac{L}{n^{tC_{\mathsf{e},1}^{-\beta}}}\right)^n\leq\exp(-Ln^{1-tC_{\mathsf{e},1}^{-\beta}})=\rO(n^{-D}),
\end{align*} 
where we used the inequality $1-x \leq e^{-x}$ in the last step. According to our assumption, we have that $tC_{\mathsf{e},1}^{-\beta}<1$. Consequently, we have that  
\begin{align*}
    \mathbb{P}(\xi^2_{(1)}\geq C_{\mathsf{e},1}^{-1} \log^{1/\beta} n)=1-\mathbb{P}(\xi^2_{(1)}< C_{\mathsf{e},1}^{-1} \log^{1/\beta} n)\geq 1-\exp(-L n^{1-tC_{\mathsf{e},1}^{-\beta}})=1-\rO(n^{-D}).
\end{align*}
This completes the proof. 
% of the second statement for the constant $C^{\beta}_t>\max(t,t^{-1})$.

Now, we prove the first statement that
\begin{align*}
    \mathbb{P}(\xi^2_{(1)}-\xi^2_{(2)}\geq C_{\mathsf{e},0}^{-1} \log^{{1/\beta}-1-c_{\mathsf{e},0}}n)=1-\rO(\log^{-c_{\mathsf{e},0}}n).
\end{align*}
Using the first two identities in~\eqref{eq_hererepeating} and an argument similar to that used in the proof of the second statement above, we can show that (\ref{eq_firstsecond}) also holds for $\xi_{(2)}^2.$ For notional simplicity, in what follows, we denote 
 $b_{n,-}:=C_{\mathsf{e},1}^{-1}\log^{1/\beta}n, \ b_{n,+}:=C_{\mathsf{e},1} \log^{1/\beta}n$ and
% , it follows that $\xi^2_{(2)}\in[b_{n,-},b_{n,+}]$ with probability higher than $1-\mathsf{C}_2n^{-(tC_{t_2}^{\beta}-1)}(1+\mathrm{o}(1))$.
$$I_{b_n}:=[b_{n,-},\ b_{n,+}].$$ Using (\ref{eq_jointdensity1and2}), we have that 
% Since $F_{\xi^2}(x)$ is continuous for large $x$, the joint distribution of the ordered statistics, $(\xi^2_{(1)},\xi^2_{(2)})$ has the density function 
%\begin{align*}
%    f_{\xi^2_{(1)},\xi^2_{(2)}}(x,y)=n(n-1)(F_{\xi^2}(y))^{n-2}f_{\xi^2}(x)f_{\xi^2}(y),\quad x>y>0,
%\end{align*}
%where $F_{\xi^2}(x)$ and $f_{\xi^2}(x)$ are CDF and PDF of $\xi^2$, respectively. Then, we have
\begin{align}\label{eq_twotwotwo}
    \mathbb{P}(\xi^2_{(1)}-\xi^2_{(2)}\leq\epsilon)&=\iint_{0<x-y\leq\epsilon}f_{\xi^2_{(1)},\xi^2_{(2)}}(x,y)\mathrm{d}x\mathrm{d}y \nonumber \\
    &=\int_{I_{b_n}}\int_{0<x-y\leq\epsilon}f_{\xi^2_{(1)},\xi^2_{(2)}}(x,y)\mathrm{d}x\mathrm{d}y+\int_{I_{b_n}^c}\int_{0<x-y\leq\epsilon}f_{\xi^2_{(1)},\xi^2_{(2)}}(x,y)\mathrm{d}x\mathrm{d}y\\
    &=\int_{I_{b_n}}\int_{0<x-y\leq\epsilon}f_{\xi^2_{(1)},\xi^2_{(2)}}(x,y)\mathrm{d}x\mathrm{d}y+\mathrm{O}(n^{-c}), \nonumber
\end{align}
where we used the fact that 
\begin{align*}
    &\int_{I_{b_n}^c}\int_{0<x-y\leq\epsilon}f_{\xi^2_{(1)},\xi^2_{(2)}}(x,y)\mathrm{d}x\mathrm{d}y\leq\int_{I_{b_n}^c}\int_{y<x}f_{\xi^2_{(1)},\xi^2_{(2)}}(x,y)\mathrm{d}x\mathrm{d}y\\
    &=\mathbb{P}(\xi^2_{(2)}\notin I_{b_n})=\mathrm{O}(n^{-c}).
\end{align*}
 
 Using the change of variables as in (\ref{eq_joint}), we have that 
\begin{align*}
\int_{I_{b_n}}\int_{0<x-y\leq\epsilon}f_{\xi^2_{(1)},\xi^2_{(2)}}(x,y)\mathrm{d}x\mathrm{d}y=\int_{b_{n,-}}^{b_{n,+}}\big(\int_{0}^{\epsilon}n(n-1)f_{\xi^2}(u+v)f_{\xi^2}(u)\mathrm{d}v\big)(F_{\xi^2}(u))^{n-2}\mathrm{d}u,
\end{align*}
where for $u \gg 1,$ we have from (\ref{eq_def_exponentialdecaytail}) that when $u \gg 1$
\begin{align}\label{eq_densityexponential}
    &f_{\xi^2}(u)=t\beta u^{\beta-1}L\exp(-tu^{\beta}),\quad 1-F_{\xi^2}(u)=L\exp(-tu^{\beta}).  
\end{align}
To prove our result, we now choose $\delta:=C'\log^{-1-c_{\mathsf{e},0}}n$ for some constant $C'>0$ and set
 \begin{equation}\label{eq_epsilondefinition11111}
 \epsilon=b_{n,-}\delta.
 \end{equation}

Similar to the discussions of (\ref{eq_calculatehere}), using (\ref{eq_densityexponential}), we have that for $u \gg \epsilon$ 
\begin{align*}
    &\int_0^{\epsilon}f_{\xi^2}(u+v)\mathrm{d}v=\int_u^{u+\epsilon}f_{\xi^2}(x)\mathrm{d}x=F_{\xi^2}(u+\epsilon)-F_{\xi^2}(u)\\
    &=L\exp(-tu^{\beta})-L\exp(-t(u+\epsilon)^{\beta}) =\epsilon f_{\xi^2}(u)(1+\mathrm{o}(1)),
%    =C\exp(-tu^{\beta})\big(1-\exp(-t(u+\epsilon)^{\beta}+tu^{\beta})\big)\\
%    &=C\exp(-tu^{\beta})\big(1-\exp(-tu^{\beta}((1+\epsilon/u)^{\beta}-1))\big)= C\exp(-tu^{\beta})\big(1-\exp(-t\beta\epsilon u^{\beta-1})(1+\mathrm{o}(1))\big)\\
%    &=\epsilon t\beta u^{\beta-1}C\exp(-tu^{\beta})(1+\mathrm{o}(1))=\epsilon f_{\xi^2}(u)(1+\mathrm{o}(1)). 
\end{align*}
This implies that
%where we applied the Taylor expansion for $(1+\epsilon/u)^{\beta}$ and $\exp(-t\beta\epsilon u^{\beta-1})$ for $u\gg\epsilon\gg1$. It gives that
\begin{align*}
    &\int_{I_{b_n}}\int_{0<x-y\leq\epsilon}f_{\xi^2_{(1)},\xi^2_{(2)}}(x,y)\mathrm{d}x\mathrm{d}y\leq n^2\epsilon\int_{b_{n,-}}^{b_{n,+}}f^2_{\xi^2}(u)(F_{\xi^2}(u))^{n-2}\mathrm{d}u(1+\mathrm{o}(1))\\
    &= \epsilon\int_{b_{n,-}}^{b_{n,+}}(nf_{\xi^2}(u))^2\exp(-nLe^{-tu^{\beta}})\mathrm{d}u(1+\mathrm{o}(1)),
\end{align*}
where in the last step, we used that for $u\geq b_{n,-}$
\begin{align*}
    (F_{\xi^2}(u))^{n-2}=(1-L\exp(-tu^{\beta}))^{n-2}=\exp\big(-nL\exp(-tu^{\beta})\big)(1+\mathrm{o}(1)).
\end{align*}
For the calculation of the above integration, 
using the change of variables that $s=Ln\exp(-tu^{\beta}), u(s)=(-t^{-1}\log(s/Ln))^{1/\beta}$ and $$\mathrm{d}s=-tn\beta u^{\beta-1}L\exp(-tu^{\beta})\mathrm{d}u=-nf_{\xi^2}(u)\mathrm{d}u,$$
we have that
\begin{align}\label{eq_oneoneone}
\int_{I_{b_n}}\int_{0<x-y\leq\epsilon}f_{\xi^2_{(1)},\xi^2_{(2)}}(x,y)\mathrm{d}x\mathrm{d}y \leq \epsilon\int_{Cn\exp(-tb_{n,+}^{\beta})}^{Cn\exp(-tb_{n,-}^{\beta})}ne^{-s}f_{\xi^2}(u(s))\mathrm{d}s(1+\mathrm{O}(1)).
\end{align}

Denote $h(u):=f_{\xi^2}(u)/(1-F_{\xi^2}(u))$. Using the fact that $1-F_{\xi^2}(u)=s/n,$ we have  
\begin{align*}
    f_{\xi^2}(u)=h(u)\frac{s}{n}.
\end{align*}
Moreover, for $h(u(s)), \ s\in[Cn\exp(-tb_{n,+}^{\beta}),Cn\exp(-tb_{n,-}^{\beta})],$ when $\beta\geq 1$, we have
\begin{align*}
    h(u(s))=t\beta (u(s))^{\beta-1} \leq t\beta C^{\beta-1}\log^{1-1/\beta}n, 
\end{align*}
and for $\beta<1$, we have
\begin{align*}
    h(u(s))=t\beta (u(s))^{\beta-1}\leq t\beta C^{1-\beta}\log^{1-1/\beta}n.  
\end{align*}

With the above preparation, we have that 
%we have
%\begin{align*}
%    \mathbb{P}(\xi^2_{(1)}-\xi^2_{(2)}\le\epsilon)\le \epsilon\int_{Cn\exp(-tb_{n,+}^{\beta})}^{Cn\exp(-tb_{n,-}^{\beta})}se^{-s}h(u(s))\mathrm{d}s+\mathrm{O}(n^{-(tC_{t_2}^{\beta}-1)}).
%\end{align*}
%Therefore, we can estimate 
\begin{align*}
& \epsilon\int_{Cn\exp(-tb_{n,+}^{\beta})}^{Cn\exp(-tb_{n,-}^{\beta})}ne^{-s}f_{\xi^2}(u(s))\mathrm{d}s =\epsilon\int_{Cn\exp(-tb_{n,+}^{\beta})}^{Cn\exp(-tb_{n,-}^{\beta})}se^{-s}h(u(s))\mathrm{d}s  \\
& =\mathrm{O} \left( \epsilon \log^{1-1/\beta}n\int_{Cn\exp(-tb_{n,+}^{\beta})}^{Cn\exp(-tb_{n,-}^{\beta})}se^{-s}\mathrm{d}s \right)\\
    &= \mathrm{O}\left( \epsilon \log^{1-1/\beta}n\int_0^{\infty}se^{-s}\mathrm{d}s \right)=\rO\left(\Gamma(2)\epsilon \log^{1-1/\beta}n \right)=\rO\left(\log^{-c_{\mathsf{e},0}}n \right),
\end{align*}
where we used $\int_0^{\infty}se^{-s}\mathrm{d}s=\Gamma(2)$ and the definition in (\ref{eq_epsilondefinition11111}). Together with (\ref{eq_oneoneone}) and (\ref{eq_twotwotwo}), we can therefore conclude the proof.  
}
% As a consequence, we take $\epsilon=C\log^{1/\beta-1-c}n$ for some small constant $c>0$, it follows
%\begin{align*}
%    &\mathbb{P}(\xi^2_{(1)}-\xi^2_{(2)}\ge C\log^{1/\beta-1-\mathrm{c}}n)=1-\mathbb{P}(\xi^2_{(1)}-\xi^2_{(2)}<\epsilon)\\
%    &\ge1-\mathsf{C}_4\log^{-c}n-\mathsf{C}_5n^{-(tC_{t_2}^{\beta}-1)},
%\end{align*}
%for some constants $\mathsf{C}_4,\mathsf{C}_5>0$. This concludes the proof of the first statement.

%Next, we are going to show that it also holds that 
%\begin{align*}
%    \xi^2_{(1)}-\xi^2_{(\lceil n^{1-c_{\mathsf{e},1}}\rceil)}\ge C_{\mathsf{e},2}^{-1}\log^{1/\beta}n,
%\end{align*}
%with probability tending to one, where $0<c_{\mathsf{e},1}<\min\{t(C_{\mathsf{e},1}^{-1}-C_{\mathsf{e},2}^{-1})^{\beta},1\}$. 
Then we prove the third statement. As we assume $C_{\mathsf{e},2}>C_{\mathsf{e},1}$, we have 
\begin{align*}
    &\mathbb{P}(\xi^2_{(1)}-\xi^2_{(\lceil n^{1-c_{\mathsf{e},1}}\rceil)}<C_{\mathsf{e},2}^{-1}\log^{1/\beta}n)=\mathbb{P}(\xi^2_{(\lceil n^{1-c_{\mathsf{e},1}}\rceil)}>\xi^2_{(1)}-C_{\mathsf{e},2}^{-1}\log^{1/\beta}n)\\
    &=\mathbb{P}(\xi^2_{(\lceil n^{1-c_{\mathsf{e},1}}\rceil)}>\xi^2_{(1)}-C_{\mathsf{e},2}^{-1}\log^{1/\beta}n|\xi^2_{(1)}\geq C_{\mathsf{e},1}^{-1}\log^{1/\beta}n)\cdot\mathbb{P}(\xi^2_{(1)}\geq C_{\mathsf{e},1}^{-1}\log^{1/\beta}n)\\
    &+\mathbb{P}(\xi^2_{(\lceil n^{1-c_{\mathsf{e},1}}\rceil)}>\xi^2_{(1)}-C_{\mathsf{e},2}^{-1}\log^{1/\beta}n|\xi^2_{(1)}< C_{\mathsf{e},1}^{-1}\log^{1/\beta}n)\cdot\mathbb{P}(\xi^2_{(1)}<C_{\mathsf{e},1}^{-1}\log^{1/\beta}n)\\
    &=\mathbb{P}(\xi^2_{(\lceil n^{1-c_{\mathsf{e},1}}\rceil)}>\xi^2_{(1)}-C_{\mathsf{e},2}^{-1}\log^{1/\beta}n|\xi^2_{(1)}\geq C_{\mathsf{e},1}^{-1}\log^{1/\beta}n)\cdot(1-\mathsf{C}_6\exp(-Cn^{1-tC_{\mathsf{e},1}^{-\beta}}))\\
    &+\mathsf{C}_7\exp(-Cn^{1-tC_{\mathsf{e},1}^{-\beta}})\\
    &=\mathbb{P}(\{\xi^2_{(\lceil n^{1-c_{\mathsf{e},1}}\rceil)}>\xi^2_{(1)}-C_{\mathsf{e},2}^{-1}\log^{1/\beta}n\}\cap\{\xi^2_{(1)}\geq C_{\mathsf{e},1}^{-1}\log^{1/\beta}n\})+\mathsf{C}_8\exp(-Cn^{1-tC_{\mathsf{e},2}^{-\beta}})\\
    &=\mathbb{P}\big(\xi^2_{(\lceil n^{1-c_{\mathsf{e},1}}\rceil)}>(C_{\mathsf{e},1}^{-1}-C_{\mathsf{e},2}^{-1})\log^{1/\beta}n\big)+\mathsf{C}_8\exp(-Cn^{1-tC_{\mathsf{e},1}^{-\beta}})\\
    &=\sum_{k=\lceil n^{1-c_{\mathsf{e},1}}\rceil}^n\binom{n}{k}\big[\mathbb{P}(\xi^2>(C_{\mathsf{e},1}^{-1}-C_{\mathsf{e},2}^{-1})\log^{1/\beta}n)\big]^k\big[\mathbb{P}(\xi^2\leq (C_{\mathsf{e},1}^{-1}-C_{\mathsf{e},2}^{-1})\log^{1/\beta}n)\big]^{n-k}+\mathsf{C}_8\exp(-Cn^{1-tC_{\mathsf{e},1}^{-\beta}})\\
    &=\sum_{k=\lceil n^{1-c_{\mathsf{e},1}}\rceil}^n\binom{n}{k}\big(C\exp(-t(C^{\prime})^{\beta}\log n)\big)^k\big(1-C\exp(-t(C^{\prime})^{\beta}\log n)\big)^{n-k}+\mathsf{C}_8\exp(-Cn^{1-tC_{\mathsf{e},2}^{-\beta}})\\
    &\leq \sum_{k=\lceil n^{1-c_{\mathsf{e},1}}\rceil}^n\Big(\frac{Cen}{k}\Big)^k\big(\frac{1}{n^{t(C^{\prime})^{\beta}}}\big)^k\big(1-\frac{1}{n^{t(C^{\prime})^{\beta}}}\big)^{n-k}+\mathsf{C}_8\exp(-Cn^{1-tC_{\mathsf{e},2}^{-\beta}})\\
    &=\sum_{k=\lceil n^{1-c_{\mathsf{e},1}}\rceil}^n\Big(\frac{Cen^{1-t(C^{\prime})^{\beta}}}{k}\Big)^k\exp\big(-n^{1-t(C^{\prime})^{\beta}}(1-\frac{k}{n})\big)+\mathsf{C}_8\exp(-Cn^{1-tC_{\mathsf{e},2}^{-\beta}})\\
    &\leq n\Big(\frac{Cen^{1-t(C^{\prime})^{\beta}}}{\lceil n^{1-c_{\mathsf{e},1}}\rceil}\Big)^{\lceil n^{1-c_{\mathsf{e},1}}\rceil}+\mathsf{C}_8\exp(-Cn^{1-tC_{\mathsf{e},2}^{-\beta}})=\mathrm{O}(1/n^{\mathsf{C}_9}),
\end{align*}
for some positive constants $\mathsf{C}_6,\mathsf{C}_7,\mathsf{C}_8,\mathsf{C}_9>0$, where $C^{\prime}>0$ be defined such that $C^{\prime}=C_{\mathsf{e},1}^{-1}-C_{\mathsf{e},2}^{-1}$. In the above procedures, we used the law of total probability in the second step; we used the results of the second statement in the third step; we used Stirling's formula in the third-to-last step and in the last step, we used condition $c_{\mathsf{e},1}<t(C^{\prime})^{\beta}$. This gives our desired result.

Finally,  the last statement follows directly from the strong law of large number. In fact, the result holds almost surely. 
\end{proof}

Finally, we prove Case (c) of Definition \ref{defn_probset}. 

\begin{proof}[\bf Proof of Case (c)] 
Note that the fourth statement holds trivially and surely. 

For the first statement, under the assumption of  (\ref{ass3.4}), we see that the lower bound follows from that  
\[
\begin{split}
   \mathbb{P}(l-\xi^2_{(1)}>n^{-1/(d+1)-\epsilon_{\mathsf{d}}})&=\big(1-\mathbb{P}(l-\xi^2_{(1)}<n^{-1/(d+1)-\epsilon_{\mathsf{d}}})\big)^n\\
&\geq (1-Cn^{-\epsilon_{\mathsf{d}}(d+1)-1})^n\\
&\geq 1-Cn^{-\epsilon_{\mathsf{d}}(d+1)}. 
\end{split}
\]
Similarly, for the upper bound, we find that when $n$ is sufficiently large, for some constant $C'>0$
\[
\begin{split}
    \mathbb{P}(l-\xi^2_{(1)}>n^{-1/(d+1)}\log n)&\leq n\big(1-\mathbb{P}(l-\xi^2\leq n^{-1/(d+1)}\log n)\big)^{n-1}\\
    &\leq n\big(1-C^{-1}n^{-1}\log^{d+1}n\big)^{n-1}\\
    &\leq ne^{-C^{-1}\log^{d+1}n} \leq n^{-C'}.
\end{split}
\]
This completes the proof of the first statement. 

For the third statement, we prove by contradiction, i.e., there exists some sequence $\mathtt{a}_n=\ro(1),$ $l-\xi_{\lfloor b n \rfloor} \leq \mathtt{a}_n$ holds with high probability.  In fact, by a discussion similar to (\ref{eq_hererepeating}) using (\ref{ass3.4}), we have that as long as $c \equiv c_n > n/\mathtt{a}_n,$
\[
\begin{split}
    \mathbb{P}(l-\xi^2_{(c)} \leq \mathtt{a}_n)&=\mathbb{P}(\xi^2_{(c)} \geq l-\mathtt{a}_n)\\
    &=\sum_{k=c+1}^{n}\binom{n}{k}\mathbb{P}(\xi^2>l-\mathtt{a}_n)^k\mathbb{P}(\xi^2\leq l-\mathtt{a}_n)^{n-k}=\rO(n^{-C}),
\end{split}
\]
for some constant $C>0$ when $n$ is sufficiently large. This completes our proof for the third statement.

For the second statement, its discussion is similar to the proof of \cite[Theorem 8.1]{lee2016extremal} . In this case, we will define the partition of the intervals as $I_k=[l-(k+1)n^{-1/(d+1)-\epsilon_{\mathsf{d}}},l-kn^{-1/(d+1)-\epsilon_{\mathsf{d}}}]$ for $k=[\![1,n^{\epsilon_{\mathsf{d}}}\log n]\!].$  We observe that 
\[
\mathsf p_k=\mathbb{P}(\xi^2\in I_k)\leq Cn^{-\epsilon_{\mathsf{d}}}n^{-1/(d+1)}(n^{-1/(d+1)}\log n)^d=Cn^{-1-\epsilon_{\mathsf{d}}}\log^dn. 
\]  
Using the above control together with the fact
\begin{equation}
\begin{split}
    &\mathbb{P}(\xi^2_{(i)}-\xi^2_{(i+1)}\leq n^{-1/(d+1)-\epsilon_{\mathsf{d}}})\\
    &=\mathbb{P}(\xi^2_{(i)},\xi^2_{(i+1)}\in I_k(y))+\mathbb{P}(\xi^2_{(i)}\in I_k(y)\;\text{and}\;\xi^2_{(i+1)}\in I_{k+1}(y)),
\end{split}
\end{equation}
we readily obtain that 
\[
\mathbb{P}(\xi^2_{(1)}-\xi^2_{(2)}\leq\ n^{-1/(d+1)-\epsilon_{\mathsf{d}}})\leq n^2\mathsf p_k^2\leq Cn^{-2\epsilon_{\mathsf{d}}}\log^{2d}n.
\]
This completes the proof of the second statement. 

Next, we proceed to the proof of the third-to-last statement. Denote the random variable $\tau_{\xi_i}$ as follows
\begin{gather*}
    \tau_{\xi^2_i}:=\frac{\xi^2_i}{1+\xi^2_im_{1n,c}(z)}-\int\frac{t}{1+tm_{1n,c}(z)}\mathrm{d}F(t).
\end{gather*}
By definition $\mathbb{E}\tau_{\xi^2_i}=0$. On the one hand, according to the discussion around (\ref{eq_defnmathsfW}), we find that 
\begin{gather*}
    \frac{1}{n}\sum_{i=1}^p\frac{\sigma_i^2\int\frac{t^2}{|1+tm_{1n,c}(z)|^2}\mathrm{d}F(t)}{|z-\sigma_i\int\frac{t}{1+tm_{1n,c}(z)}\mathrm{d}F(t)|^2}<1.
\end{gather*}
Together with Assumption \ref{assum_additional_techinical} and the continuity of $m_{2n,c},$ we can therefore conclude that for some constant $C_0>0,$
\begin{gather*}
    \int\frac{t^2}{|1+tm_{1n,c}(z)|^2}\mathrm{d}F(t)<C_0.
\end{gather*}
As a consequence, by Cauchy-Schwarz inequality, we find that for some constants $C_1, C_2>0$
\begin{gather*}
  \mathbb{E}|\tau_{\xi^2}|^2\leq C_1  \int\frac{t^2}{|1+tm_{1n,c}(z)|^2}\mathrm{d}F(t) < C_2<\infty. 
\end{gather*}
Since $\tau_{\xi_i^2}, 1 \leq i \leq n,$ are independent, we can conclude our proof using Markov inequality.

Finally, for the last two statements, they can be obtained straightly by Chebyshev's inequality and Markov's inequality.
\end{proof}

\subsection{Fluctuation averaging arguments: Proof of Lemma \ref{lem_fa}}\label{sec_FAlemma}
In this section, we prove the fluctuation averaging results in Lemma \ref{lem_fa} following the strategies of Section 6 of \cite{lee2016extremal}. Fluctuation averaging is a common step in the proof of local laws for random matrix models, especially when the LSD has a square root decay behavior near the edge so that the entries of the resolvents can be controlled under some ansatz; see the monograph \cite{erdHos2017dynamical} for a review. However, in our setting, due to the lack of square root decay as in (\ref{eq: concave decay of rho_Q}), many entries of the resovelents, even the off-diagonal ones can be large when $\eta \asymp n^{-1/2}.$ To address this issue, we will follow the strategies of \cite{lee2016extremal} to focus on the resolvent fractions instead of the entries themselves; see the discussion above Sections 6.1 of \cite{Kwak2021,lee2016extremal}. In what follows, due to similarity, we focus on the parts which deviate from \cite[Section 6]{lee2016extremal} the most.

   \begin{proof}[\bf Proof of Lemma \ref{lem_fa}]
In what follows, with loss of generality, we assume that $\xi_1^2 \geq \xi_2^2 \geq \cdots \geq \xi_n^2.$   
   
  We start with part (1). Recall (\ref{eq: decomp m_2}). Using Theorem \ref{thm_boundedcaselocallaw} and Remark \ref{remk_zibound}, we have that 
 \begin{gather*}
\begin{split}
    |m_2-m_2^{(1)}| & \leq \left|\frac{1}{n}\frac{\xi^2_{1}}{z(1+\xi^2_{1}m_{1n}+\rO_{\prec}((n \eta_0)^{-1}))} \right| \\
 &   +\left|\frac{1}{n}\sum_{i=2}^{p}\frac{ \rO_{\prec}((n\eta_0)^{-1})}{z(1+\xi^2_im_{1n}+\rO_{\prec}((n \eta_0)^{-1}))(1+\xi^2_im_{1n}+\rO_{\prec}((n\eta_0)^{-1}))}\right|.
\end{split}
\end{gather*} 
For the first term on the right-hand side of the equation,  it can be trivially bounded by $(n \eta_0)^{-1}$ by a discussion similar to (\ref{eq_trivialcontroleta}) using (\ref{eq_zopointrate11}).  The second term can also be controlled by $(n \eta_0)^{-1}$ using a discussion similar to (\ref{eq_L1bound}). The proves the first equation in (\ref{eq_c5first}).  For the second equation, due to similarity, we focus on $|m_2-m_2^{(i)}|.$ Using  (\ref{eq_decompositionleavoneout}), we have that
\begin{gather}\label{eq_differenceshouldbehere}
    |m_{2}-m_2^{(i)}|\leq\frac{|\mathcal{G}_{ii}|}{n}+\frac{1}{n}\sum_{j \neq i}|\mathcal{G}_{jj}-\mathcal{G}_{jj}^{(i)}|.
\end{gather}
For $\mathcal{G}_{ii},$ by  Lemma \ref{lem: Resolvent}, Theorem \ref{thm_boundedcaselocallaw} and the assumption that $z \in \mathbf{D}_b^\prime$ in (\ref{eq_spectralparameterprime}), we conclude that with high probability, for some constant $C>0$
\begin{gather}\label{eq_giiboundinverse}
    |\mathcal{G}_{ii}|=\frac{1}{|z(1+\xi^2_im^{(i)}_1+Z_i)|}\leq Cn^{1/(d+1)+\epsilon_{\mathsf{d}}}. 
\end{gather} 
For $\mathcal{G}_{jj}-\mathcal{G}_{jj}^{(i)},$ by  Lemma \ref{lem: Resolvent}, (\ref{lem:Wald}) and Lemma \ref{lem:large deviation}, 
\begin{gather*}
    |\mathcal{G}_{jj}-\mathcal{G}_{jj}^{(i)}|=|\frac{\mathcal{G}_{ij}\mathcal{G}_{ji}}{\mathcal{G}_{ii}}| \prec |\mathcal{G}_{ii}||\mathcal{G}_{jj}^{(i)}|^2\frac{\operatorname{Im}m_1^{(ij)}}{n\eta_0} \prec \frac{n^{2 \epsilon_{\mathsf{d}}}}{n }|\mathcal{G}_{ii}||\mathcal{G}_{jj}^{(i)}|^2,
\end{gather*}  
where in the last step we used (\ref{eq_zopointrate11}) and Theorem \ref{thm_boundedcaselocallaw}. Inserting all the above bounds back to (\ref{eq_differenceshouldbehere}) and use the trivial bound that $|\mathcal{G}_{jj}^{(i)}| \leq \eta_0,$ we can conclude the proof. This completes the proof of part (1). 

We now proceed to the proof of parts (2) and (3). Due to similarity, we focus on the details of part (2) and briefly mention how to prove (3) in the end. For simplicity, following the conventions in \cite{Kwak2021,lee2016extremal}, we denote the operator $P_i:=\mathbf{1}-\mathbb{E}_i,$
where $\mathbb{E}_i$ is the conditional expectation with respect to $\mathbf{y}_i$. Using Lemma \ref{lem: Resolvent}, we see that on $\Omega_D$
\begin{gather}\label{eq_pigiiinverse}
    \frac{1}{n}\sum_{i=2}^n P_i(\frac{1}{\mathcal{G}_{ii}})=\frac{1}{n}\sum_{i=2}^n P_i(-z-z\mathbf{y}_i^{*}G^{(i)}(z)\mathbf{y}_i)=-\frac{z}{n}\sum_{i=2}^n Z_i.
\end{gather}
Consequently, it suffices to show that 
\begin{gather*}
    \left|\frac{1}{n}\sum_{i}P_i(\frac{1}{\mathcal{G}_{ii}})\right|\prec n^{-1/2-\frac{1}{2}(\frac{1}{2}-\frac{1}{d+1})+2 \epsilon_{\mathsf{d}}}.
\end{gather*}  
By Chebyshev's inequality, it suffices to prove the following lemma. 
\begin{lemma}\label{lem_fakeycomponents}
Under the assumptions of Lemma \ref{lem_fa}, for any $z\in\mathbf{D}_b^{\prime}$ and fixed even number $M\in\mathbb{N}$, we have
\begin{gather*}
    \mathbb{E}^X\left|\frac{1}{n}\sum_{i=2}^n P_i(\frac{1}{\mathcal{G}_{ii}(z)})\right|^M\prec n^{M(-1/2-\frac{1}{2}(\frac{1}{2}-\frac{1}{d+1})+2 \epsilon_{\mathsf{d}})}.
\end{gather*}
\end{lemma} 
\begin{proof}
The proof strategy and technique follows closely from Section 6 of \cite{lee2016extremal}. In what follows, we adopt the way how \cite[Section 6.3]{Kwak2021} generalizes \cite[Section 6.2]{lee2016extremal} and only check the core estimates that have been used in \cite{lee2016extremal}. We first provide some notations following the conventions of \cite[Section 6.1]{lee2016extremal}. For any subset $\mathcal{T},\mathcal{T}^{\prime}\subset\{1,\dots, n\}$ with $i,j\notin\mathcal{T}$ and $j\notin\mathcal{T}^{\prime}$, we set
\begin{gather*}
  F_{ij}^{(\mathcal{T},\mathcal{T}^{\prime})} \equiv   F_{ij}^{(\mathcal{T},\mathcal{T}^{\prime})}(z):=\frac{\mathcal{G}_{ij}^{(\mathcal{T})}(z)}{\mathcal{G}_{jj}^{(\mathcal{T}^{\prime})}(z)}. 
\end{gather*}
In case $\mathcal{T}=\mathcal{T}^\prime=\emptyset$, we simply write $F_{ij}=F_{ij}^{(\mathcal{T},\mathcal{T}^{\prime})}$. With Lemma \ref{lem: Resolvent}, according to \cite[Lemma 6.1]{lee2016extremal}, we have that for any subset $\mathcal{T},\mathcal{T}^{\prime}\subset\{1,\dots, p\}$ with $i,j\notin\mathcal{T}$ and $j \notin\mathcal{T}^{\prime}$, and $\gamma\notin\mathcal{T}\bigcup\mathcal{T}^{\prime}$
\begin{equation*}
F_{ij}^{(\mathcal{T},\mathcal{T}^{\prime})}=F_{ij}^{(\mathcal{T}\gamma,\mathcal{T}^{\prime})}+F_{i\gamma}^{(\mathcal{T},\mathcal{T}^{\prime})}F_{\gamma j}^{(\mathcal{T},\mathcal{T}^{\prime})}, \ 
\end{equation*}
and 
\begin{equation*}
F_{ij}^{(\mathcal{T},\mathcal{T}^{\prime})}=F_{ij}^{(\mathcal{T},\mathcal{T}^{\prime}\gamma)}-F_{ij}^{(\mathcal{T},\mathcal{T}^{\prime}\gamma)}F_{j\gamma}^{(\mathcal{T},\mathcal{T}^{\prime})}F_{\gamma j}^{(\mathcal{T},\mathcal{T}^{\prime})}.
\end{equation*}
Moreover, we have that for $\gamma \notin \mathcal{T}$
\begin{equation*}
 \frac{1}{\mathcal{G}_{ii}^{(\mathcal{T})}}=\frac{1}{\mathcal{G}_{ii}^{(\mathcal{T\gamma})}} \left(1-F_{i\gamma}^{(\mathcal{T},\mathcal{T})}F_{\gamma i}^{(\mathcal{T},\mathcal{T})}\right).
\end{equation*}
In order to apply the techniques of \cite[Section 6.2]{lee2016extremal}, we need to prove the following estimates 
\begin{align}\label{eq_keybounds}
% \begin{split}
    |m_1(z)-m_{1n}(z)|\prec (n \eta_0)^{-1}, & \ \quad \operatorname{Im}m_1(z)\prec (n \eta_0)^{-1}, \ \left|P_i(\frac{1}{\mathcal{G}_{ii}}) \right|\prec (n \eta_0)^{-1}, \ i \neq 1, \nonumber \\
   & \max_{i\neq j}|F_{ij}(z)|\prec n^{-(1/2-1/(d+1))/2+\epsilon_{\mathsf{d}}},\quad i,j \neq 1,\\
   & \max_{i\neq j} \left|\frac{F_{ij}^{(\emptyset,i)}(z)}{\mathcal{G}_{ii}(z)} \right|\prec (n \eta_0)^{-1},\quad i, j\neq 1, \nonumber
   % \end{split}
\end{align}
First, the first part of (\ref{eq_keybounds}) follows from Theorem \ref{thm_boundedcaselocallaw}, Lemma \ref{lem: est for Im m_1} and Remark \ref{remk_zibound} (recall (\ref{eq_pigiiinverse})). Second, for the second part of (\ref{eq_keybounds}), by a discussion similar to (\ref{eq_giiboundinverse}), for $i \neq j$ and $i,j \neq 1,$ we have that for some constant $C>0,$ with high probability
\begin{equation}\label{eq_priorpriorbound}
 |\mathcal{G}_{ii}^{(j)}| \leq Cn^{1/(d+1)+\epsilon_{\mathsf{d}}}.
\end{equation}
Together with Lemma \ref{lem: Resolvent}, we see that for some constant $C>0$
\begin{gather*}
    \begin{split}
        |F_{ij}|&=|z\mathcal{G}_{ii}^{(j)}\mathbf{y}_i^{*}G^{(ij)}\mathbf{y}_j| \prec \left|z\mathcal{G}_{ii}^{(j)}\frac{1}{n}\|G^{(ij)}\Sigma\|_F \right| \leq C \Big|\mathcal{G}_{ii}^{(j)}\Big(\frac{\operatorname{Im}m_1^{(ij)}}{n\eta} \Big)^{1/2} \Big|\\
        &\prec n^{1/(d+1)+\epsilon_{\mathsf{d}}}\frac{1}{n\eta_0}=n^{1/(d+1)-1/2+2\epsilon_{\mathsf{d}}},
    \end{split}
\end{gather*}
where in the second step we used (\ref{lem:Wald}) and in the third step we used (\ref{eq_priorpriorbound}) and the fact $z \in \mathbf{D}_b^\prime.$ Finally, for the third part of (\ref{eq_keybounds}), using Lemma \ref{lem: Resolvent}, Lemma \ref{lem:large deviation} and (\ref{lem:Wald}), we see that 
\begin{equation*}
 \Big| \frac{F_{ij}^{(\emptyset,i)}}{\mathcal{G}_{ii}}\Big|=\Big|\frac{\mathcal{G}_{ij}}{\mathcal{G}_{jj}^{(i)}\mathcal{G}_{ii}} \Big|=\Big| z\mathbf{y}_i^{*}G^{(ij)}\mathbf{y}_j \Big| \prec \sqrt{\frac{\operatorname{Im} m^{(ij)}_1(z)}{n \eta}}.
\end{equation*}
We can therefore conclude our proof using Lemma \ref{lem: est for Im m_1}, Remark \ref{remk_zibound} and (\ref{lem:trace_difference}). 

Using (\ref{eq_keybounds}) and Assumption \ref{assum_model}, we can follow the proof of Corollary 6.4 of \cite{lee2016extremal} verbatim  and conclude that for any $\mathcal{T},\mathcal{T}^{\prime},\mathcal{T}^{\prime\prime}\in\{2,\dots,n\}$ with $|\mathcal{T}|,|\mathcal{T}^{\prime}|,|\mathcal{T}^{\prime\prime}|\leq M,$ where $M$ is some large positive even integer, and for $z\in\mathbf{D}_b^{\prime}$, we have that when $i \neq j, i, j \neq 1,$
\begin{gather}\label{eq_keybounds2}
\begin{split}
   & |F_{ij}^{(\mathcal{T},\mathcal{T}^{\prime})}(z)|\prec n^{-(1/2-1/(d+1))/2+\epsilon_{\mathsf{d}}},\\
   & \Big|\frac{F_{ij}^{(\mathcal{T}^{\prime},\mathcal{T}^{\prime\prime})}(z)}{G^{(\mathcal{T})}_{ii}(z)} \Big|\prec (n \eta_0)^{-1}, \  \ \Big|P_i\Big(\frac{1}{\mathcal{G}_{ii}^{(\mathcal{T})}} \Big)\Big|\prec (n \eta_0)^{-1}.
    \end{split}
\end{gather}
Once  (\ref{eq_keybounds}) and (\ref{eq_keybounds2}) have been proved, we can follow lines of \cite[Lemma 6.6]{lee2016extremal} or \cite[Lemma 6.11]{Kwak2021} to conclude the proof. Due to similarity, we omit the details.
\end{proof}

\quad This completes the proof of part (2). The proof of part (3) is similar except we need to following the proof of Lemma \ref{lem_fakeycomponents} and \cite[Lemma 6.12]{lee2016extremal}  to show 
\begin{equation*}
   \mathbb{E}^X\Big| \frac{1}{n} \sum_{i=2}^n \frac{1}{(1+\xi_i^2 m_{1n}(z))^2} P_i(\frac{1}{\mathcal{G}_{ii}(z)})\Big|^M\prec n^{M(-1/2-\frac{1}{2}(\frac{1}{2}-\frac{1}{d+1})+2 \epsilon_{\mathsf{d}})}.
\end{equation*}
We omit the proof and refer the readers to the proof of \cite[Lemma 6.12]{lee2016extremal} for more details. This completes the proof of Lemma \ref{lem_fa}.
\end{proof}